\documentclass[10pt,reqno,twosdie]{amsart}

\usepackage{amsmath,amssymb,mathrsfs}

\usepackage[asymmetric,top=3.2cm,bottom=3.9cm,left=2cm,right=2cm]{geometry}
\usepackage{amsmath,amsfonts,amsthm,mathrsfs,amssymb,cite}
\usepackage[usenames]{color}

\usepackage{mathtools} 
\usepackage{graphics}
\usepackage{framed}

\usepackage{enumerate}

\usepackage{subfigure}

\usepackage{hyperref}
\usepackage{xcolor}
\hypersetup{
  colorlinks,
  linkcolor={blue!80!black},
  urlcolor={blue!80!black},
  citecolor={blue!80!black}
}
\usepackage[perpage]{footmisc}

\usepackage{graphicx}
\usepackage{latexsym} 
\usepackage{enumerate}

\theoremstyle{plain}
\newtheorem{theorem}{Theorem}[section]
\newtheorem{lemma}[theorem]{Lemma}
\newtheorem{corollary}[theorem]{Corollary}
\newtheorem{proposition}[theorem]{Proposition}

\newtheorem{remark}[theorem]{Remark}

\newtheorem*{assumptions}{Assumptions}
\renewcommand{\eqref}[1]{\textnormal{(\ref{#1})}}

\numberwithin{equation}{section}

\usepackage{color}

\def\Oh{{\mathcal  O}}

\usepackage{booktabs}

\newcommand{\diam}{\mathrm{diam}}      
\renewcommand{\div}{\mathrm{div}\,} 

\title[From All-dielectrics to Plasmonics]{From all-dieletric nanoresonators\\ to extended quasi-static plasmonic resonators}

\author[Cao, Ghandriche and Sini]{Xinlin Cao $^*$ Ahcene Ghandriche  $^{**}$ and Mourad Sini$^{\ddag}$}
\thanks{$^*$ Department of Applied Mathematics, The Hong Kong Polytechnic University. Email: xinlin.cao@polyu.edu.hk; xlcao.pdeip@gmail.com.}
\thanks{$^{**}$ Nanjing Center for Applied Mathematics, Nanjing, 211135, People’s Republic of China. Email: gh.hsen@njcam.org.cn.}
\thanks{$^{\ddag}$ RICAM, Austrian Academy of Sciences, Altenbergerstrasse 69, A-4040, Linz, Austria. Email: mourad.sini@oeaw.ac.at. This author is partially supported by the Austrian Science Fund (FWF): P 32660}

\date{\today}

\begin{document}

\begin{abstract}
We derive the electromagnetic medium equivalent to a cluster of all-dielectric nanoparticles (i.e. enjoying high refractive indices), distributed periodically in a smooth domain $\Omega$, while excited at nearly resonating dielectric incident frequencies (i.e. subwavelength Mie-resonant frequencies). This effective medium is an alteration of the permeability that keeps the permittivity unchanged. We provide regimes under which the effective permeability can be positive or negative valued. In addition, if the incident frequency is {\it{close}} to any of the subwavelength all-dielectric resonances, then the distributed cluster behaves as an extended  quasi-static plasmonic resonator. Therefore, exciting the cluster of all-dielectric nanoresonators with nearly resonating incident frequencies, we can generate an extended quasi-static plasmonic resonator which creates {\it{giant electromagnetic fields}} in its surrounding.

\medskip 

 

\end{abstract}

\maketitle
\hypersetup{linkcolor=black}
\tableofcontents 

\section{Introduction and statement of the results}\label{sec:Intro}
\subsection{Background and motivation}\label{Background and motivation}

The nanoscale optics structures are usually
manufactured using building-blocks made of plasmonics, with
metals such as gold or silver, enjoying negative real part of the permittivity, or dielectrics, as silicon. Each type of material (metal or dielectric) has advantages and disadvantages, see \cite{All-dielectric-1, All-dielectric-2} for related discussions. Such properties are due mainly to their resonant effects. Indeed, both material can resonate at specific incident frequencies. The metals are known to resonate at the so-called plasmonic resonances which are related to the eigenvalues of the Magnetization operator while the dielectrics resonate at the so-called all-dielectric resonances (or subwavelenght Mie-resonances) which are related to the vector-Newtonian operator. A natural way of representing the electromagnetic field generated by such structures is the Lippmann-Schwinger system of equations, via the Maxwell dyadic fundamental tensor. This integral operator is given as a sum of the two operators we just mentioned, namely the Magnetization and vector-Newtonian operators. Such a point of view provides a unified way for studying the resonant effects generated by the types of contrast of the permittivity/permeability of the structure as compared to the surrounding background. This approach was used, with eventually different strategies and applications, to study the subwavelenght plasmonic resonances in \cite{GS} and dielectric resonances in \cite{ALZ, CGS} while  created by a single nanoparticle and in  \cite{Ammari-Li-Li-Zou, AM-1, Cao-Sini, AM} while created by a dense cluster of such resonators. Let us mention a notable difference between the field created by plasmonic resonators and the ones created by dielectric ones when the contrast appears in the electric permittivity (considering non-magnetic materials). In the former case, only the electric permittivity is averaged and changed while the permeability is kept the same. In the latter case, however, the the opposite holds, namely the permittivity is unchanged while the effective permeability is changed. This last phenomenon is quite surprising as it means that we can generate purely magnetic materials using a nanostructure of purely electric (i.e. non-magnetic) building blocks. This last observation is already made in the engineering literature for spherically shaped nanoparticles using formally the classical Mie-computations, \cite{Videen-Bickel}, see also \cite{All-dielectric-2} for a review on the related applied liteature. In addition, it was justified using the homogenization techniques in \cite{BBF} assuming the nanoparticles to be absorbing (i.e. with a non-vanishing imaginary part of the permittivity) and tacking the incident frequencies away from the dielectric resonant frequencies (or the subwavelenght Mie resonant frequencies).  In the current work, we go one step forward by taking incident frequencies close to these resonant frequencies, regardless of the absorption assumption, and characterize the equivalent medium. 
The outcome can be summarized as follows
\begin{enumerate}
\item[] 
\item As we allow incident frequencies to be close to the resonant frequencies then the volumetric concentration of the nano-particles can be low. Such a property might be very useful in practice as we do not need to (or we cannot) inject highly dense clusters of nano-particles to get a desired effective field. This can be seen in Assumptions (\ref{condition-on-k}) and  (\ref{def-cr}) of \textbf{Theorem \ref{thm-main-posi}} which means that the number of nano-particles is of the order $a^{h-3}$ with $h$ close to $1$ while the usual homogenization, with frequencies away from the resonant frequencies, need to take this number of the order $a^{-3}$ where $a$,  $a\ll 1$, models the maximum radius of the nano-particles.    
\item[] 
\item The effective permeability is changed, in contrast to the permittivity, and it can be positive or negative signed depending on the choice of the incident frequency. Therefore, the equivalent and extended medium can behave as a plasmon or as dielectric again depending on the choice of the incident frequency. This is described in \textbf{Theorem \ref{thm-main-posi}} and discussed more in \textbf{Remark \ref{remark-remove-inver}}. 
\item[] 
\item We show that the equivalent extended medium can generate low frequency plasmonic resonances. In other words, this result means that we can generate extended quasi-static plasmonic resonators from structures of all-dielectric nanoresonators. This is described in \textbf{Theorem \ref{coro-plas-resonance}} and discussed more in \textbf{Remark \ref{R-7}} and \textbf{Remark \ref{R-8}}.
\item[] 
\end{enumerate}

An important consequence of these results is that we can use all-dielectric nanoparticles, with well tuned optical properties, to generate high and even giant electromagnetic fields, see the discussions in \textbf{Remark \ref{R-7}} and \textbf{Remark \ref{R-8}} for more details. In addition to its mathematical interest, the 
 generation of high or eventually giant electromagnetic fields by designs coming from nanostructures is highly attractive and desirable, see the discussions in \cite{MZM} and the references therein. 
\bigskip

To justify these results, we proceed in two main steps. In the first step, we derive the point-interaction approximation (also called the Foldy-Lax approximation) of the fields generated by the all-dielectric nanoparticles. Therefore, we reduce the Lippmann-Schwinger equations (L.S.E) stated on the union of the nanoparticles by a linear algebraic system having as a matrix, to invert, defined with the dyadic fundamental solution taken on the locating centers of these nanoparticles. The justification of this step, taking into account the whole cluster of nanoparticles and their high contrasting values of the indices of refraction, is highly non trivial. These approximations, in those regimes, are secured in \cite{CGS}. The second step, which is the object of the current work, aims at deriving the limit of this linear algebraic system to an eventual system of integral equations which will play the role of the L.S.E modeling the wave propagation of the effective electromagnetic field (with the generated effective permittivity/permeability).  Characterizing this limiting problem is also highly non trivial. In addition to the technicalities coming from the Maxwell system, other difficulties come from the resonant regimes we are handling. For instance, a key step in the analysis is to derive the well-posedness and smoothing property for the effective Maxwell system. In regimes where the effective medium is positive, this smoothing property is derived by restating the corresponding L.S.E as $Id+K$ where $Id$ is the identity operator and $K$ is 'smaller' in related Hölder spaces where the unknowns are the electromagnetic fields inside the domain and their normal traces on the boundary, see \cite{Cao-Sini}. In the case where the effective medium is negative, we rewrite the corresponding L.S.E as  $Inv+K$ where $Inv:=Id-M$ with $Id$ is the identity operator and $M$ is related to the Magnetization operator, here $K$ is 'relatively smaller' as compared $Inv$ in some regimes. The operator $Inv$ can be inverted, with the needed control of its lower bound, using the spectral decomposition of the Magnetization operator projected on the three subspaces splitting the energy space via the Helmholtz decomposition: $\mathbb{L}^2=\mathbb{H}_0(\div=0)\overset{\perp}{\otimes}\mathbb{H}_0(Curl=0)\overset{\perp}{\otimes}\nabla \mathcal{H}armonic$, see  (\ref{hel-decomp}).  The control of the lower bounds of $Inv$ is key in justifying the point (3) above, precisely to show how the field becomes enhanced and eventually derive giant electromagnetic fields. This property suggest that the Maxwell operator, defined with the mentioned negative permeability, supports low frequency resonances, in the resolvent sense, which are related to the eigenvalues of the Magnetization operator. A rigorous justification of the last assertion can be achieved in the lines of \cite{MPS-JMPA} where the acoustic model was studied and a low frequency resonance, through the Minnaert frequency, of the related Hamiltonian is shown.  
\bigskip

In this work, we distribute the all-dielectric nanoparticles in 3D-domains. The proposed approach to analyse the elecromagnetic field can also be applied to distribution in more general manifolds. Of a particular interest to us are 2D-type nanostructures where, we distribute the nanoparticles in superpositions of 2D-surfaces (flat or not). The following two situations are of special importance:
\begin{enumerate}
       \item[] 
	\item {\it{Moir\'e metamaterials}}, \cite{W-Z:2018}. Briefly, the $3$ D Moir\'e metamaterials are built as superpositions and 'rotations' of surfaces while Moir\'e metasurfaces are build as superpositions and 'rotations' of lines. As, we can generate $2$ D and 1 $D$ blocks (i.e. metasurfaces and metawires), we do believe that we can provide with rigorous mathematical models for the Moir\'e metamaterials or metasurfaces. 
   \item[] 
	\item {\it{van der Waals heterostructures}}, \cite{van-der-Waals-paper}. The van der Waals heterostructures are built as superposition of 2D-sheets, of subwavelenght nanoparticles, which are globally periodic but locally formed of heterogeneous clusters of nanoparticles. Example of such distributions as trigonal prismatic, octahidral, chalcogenides, Boron nitrids, etc. see \cite{van-der-Waals-paper, van-der-Waals-book}.   
 \item[] 
\end{enumerate}
Event though these phenomena are reported with the modeling at the quantum scale, we are tempted to analyze this challenging issue, first based on the Maxwell models, which is known to be a good model at scales of tens to hundreds of nanometers.

\subsection{The electromagnetic fields generated by a cluster of all-dielectric nano-particles.}\label{subsec1}

We denote by $D:= \underset{m=1}{\overset{\aleph}{\cup}} \, D_m$ a collection of $\aleph$ connected, bounded and $C^1$--smooth particles of $\mathbb{R}^3$. Consider the electromagnetic scattering of time-harmonic electromagnetic plane waves, which can be described by the following Maxwell system
\begin{equation}\label{model-m}
	\begin{cases}
	\mathrm{Curl} \, E^T-i k \mu_rH^T=0\quad\mbox{in}\ \mathbb{R}^3,\\	
	\mathrm{Curl} \, H^T+i k \epsilon_r E^T=0\quad\mbox{in}\ \mathbb{R}^3,\\
	E^T=E^{in}+E^s,\quad H^T=H^{in}+H^s,\\
	\sqrt{\mu_0\epsilon_0^{-1}}H^s\times\frac{x}{|x|}-E^s=\Oh({\frac{1}{|x|^2}}),\quad \mbox{as} \ |x|\rightarrow\infty,
	\end{cases}
\end{equation}
where $\epsilon_0$ and $\mu_0$ are respectively the electric permittivity and the magnetic permeability of the vacuum outside $D$, $\epsilon_r \, := \, \dfrac{\epsilon}{\epsilon_0}$ and $\mu_r \, := \, \dfrac{\mu}{\mu_0}$ are the relative permittivity and permeability fulfilling that $\epsilon_r=1$ outside $D$ while $\mu_r=1$ in the whole space $\mathbb{R}^3$. Taking the incident plane wave 
\begin{equation*}
  E^{Inc}(x) \, = \, \mathrm{p} \, e^{i \, k \, \theta \cdot x}  \quad \text{and} \quad  H^{Inc}(x) \, = \, \left( \theta \times \mathrm{p} \right) \, e^{i \, k \, \theta \cdot x},
\end{equation*}
satisfying
\begin{equation}\label{IncWave}
	\begin{cases}
	\mathrm{Curl}\left(E^{Inc}\right) \, - \, i \, k \, H^{Inc} \, = \, 0, \quad \text{in} \quad \mathbb{R}^{3}, \\
                                           \\
	\mathrm{Curl}\left(H^{Inc}\right) \, + \, i \, k \, E^{Inc}=0, \quad \text{in} \quad \mathbb{R}^{3},
	\end{cases}
\end{equation}
with $\theta, \mathrm{p} \in \mathbb{S}^2$, such that $\mathrm{p} \cdot \theta \, = \, 0$, where $\mathbb{S}^2$ is the unit sphere, as a direction of incidence, then this problem is well-posed in appropriate Sobolev spaces, see \cite{colton2019inverse, Mitrea}, and the scattered wave $(E^s, H^s)$ possesses the coming behaviors

\begin{equation}\label{far-def}
E^s(x)=\frac{e^{i k|x|}}{|x|}\left(E^\infty(\hat{x}, \theta, p)+O(|x|^{-1})\right),\quad \mbox{as}\quad |x|\rightarrow \infty,
\end{equation}
and
\begin{equation}\notag
H^s(x)=\frac{e^{i k|x|}}{|x|}\left(H^\infty(\hat{x}, \theta, p)+O(|x|^{-1})\right),\quad \mbox{as}\quad |x|\rightarrow \infty,
\end{equation}
where $(E^\infty(\hat{x}, , \theta, p), H^\infty(\hat{x}, \theta, p))$ is the corresponding electromagnetic far-field pattern of \eqref{model-m} in the propagation direction $\hat{x} \, := \, \dfrac{x}{|x|}$. Suppose that 
\begin{equation}\label{scaling motion}
	D_m \, = \, a \, {B}_m \, + \, {z}_m, \, m=1, \cdots, \aleph,
\end{equation}
where $D_m$ are the $\aleph$ connected small components of the medium $D$, which are characterized by the parameter $a>0$ and the locations ${z}_m$. Each ${B}_m$ containing the origin is a bounded Lipschitz domain and fulfills that ${B}_m\subset B(0, 1)$, where $B(0,1)$ is the ball centered at the origin with radius 1. We denote
\begin{equation}\label{basic-ad}
a:=\max\limits_{1\leq m \leq \aleph}\mathrm{diam}(D_m)\quad\mbox{and}\quad
d:=\min\limits_{{1\leq m,j\leq \aleph}\atop{m\neq j}}d_{mj}:=\min\limits_{{1\leq m,j\leq \aleph}\atop{m\neq j}} \mathrm{dist}(D_m, D_j).
\end{equation}

\bigskip

On the basis of the Foldy-Lax approximation presented in \cite{CGS}, we investigate the corresponding effective permittivity and permeability generated by the cluster of nano-particles with high relative permittivity contrast under the following assumptions.
\begin{assumptions}
\smallskip
\phantom{}
\begin{enumerate}
    \item[]
    \item[I).] Assumption on the shape of $B_m$. 
Assume that the shapes of $B_m$'s introduced in \eqref{scaling motion} are the same and we denote
\begin{equation}\label{def-B}
	B:=B_m\quad\mbox{for}\quad m=1,\cdots, \aleph.
\end{equation}
The domain $B$ is a bounded and $C^1$-smooth domain that contains the origin. Recall the Helmholtz decomposition of $\mathbb{L}^2:=\mathbb{L}^2(B)$ space, for any bounded domain $B$,  
\begin{equation}\label{hel-decomp}
\mathbb{L}^2=\mathbb{H}_0(\div=0)\overset{\perp}{\otimes}\mathbb{H}_0(Curl=0)\overset{\perp}{\otimes}\nabla \mathcal{H}armonic.
\end{equation} 
Then, the projection of an arbitrary vector field $F$ onto three subspaces $\mathbb{H}_0(\div=0), \; \mathbb{H}_0(Curl=0)$ and $\nabla \mathcal{H}armonic$ can be respectively represented as $\overset{1}{\mathbb{P}}\left( F \right), \overset{2}{\mathbb{P}}\left( F \right)$ and $\overset{3}{\mathbb{P}}\left( F \right)$. We define the vector Newtonian potential operator ${\bf N} \, := \, {\bf N}_B$ for any vector function $F$ as
\begin{equation}\label{def-new-operator}
{\bf N}(F)(x) \, := \, \int_B\Phi_0(x, z) \, F(z)\,dz,\quad\mbox{where}\quad \Phi_0(x, z)=\frac{1}{4\pi}\frac{1}{|x-z|}\quad (x\neq z).
\end{equation}
For the Newtonian operator, under the previous decomposition, we denote $\left(\lambda_{n}^{(2)}(B), e_{n}^{(2)} \right)_{n \in \mathbb{N}}$ as the related eigen-system (of its projection) over the subspace $\mathbb{H}_{0}\left( Curl = 0 \right)$, and we denote $\left(\lambda_{n}^{(1)}(B), e_{n}^{(1)} \right)_{n \in \mathbb{N}}$ as the related eigen-system (of its projection) over the subspace $\mathbb{H}_{0}(\div=0)$. Since 
\begin{equation*}
\mathbb{H}_{0}\left( \div = 0 \right) \equiv Curl \left( \mathbb{H}_{0}\left( Curl \right) \cap \mathbb{H}\left( \div = 0 \right) \right),    
\end{equation*}
see for instance \cite{{amrouche1998vector}}, then for any $n$, there exists $\phi_n\in \mathbb{H}_0(Curl)\cap \mathbb{H}(\div=0)$, such that 
\begin{equation}\label{pre-cond}
e_{n}^{(1)}= Curl (\phi_{n}) \,\, \mbox{ with } \,\, \nu\times \phi_{n}=0 \,\, \text{and} \,\, \div(\phi_{n})=0.
\end{equation}	
We assume that, for a certain $n_0$, 
\begin{equation}\label{Intphin0}
\int_{B} \phi_{n_{0}}(y) \, dy \, \neq \, 0,
\end{equation}
where $\phi_{n_0}$ fulfills \eqref{pre-cond} with $n=n_0$. Then we can define a constant tensor ${\bf P_0}$ by
\begin{equation}\label{defP0}
{\bf{P}}_0 := \sum_{\ell=1}^{\ell_{\lambda_{n_{0}}}} \int_{B}\phi_{n_{0}, \ell}(y)\,dy \otimes \overline{ \int_{B}\phi_{n_{0}, \ell}(y)\,dy}, \quad \ell_{\lambda_{n_{0}}} \in \mathbb{N}^{\star},
\end{equation}
with $\phi_{n_{0}, \ell}$ fulfilling
\begin{equation}\notag
e^{(1)}_{n_0, \ell}=Curl(\phi_{n_{0}, \ell}),\, \div(\phi_{n_{0}, \ell})=0, \, \nu\times \phi_{n_{0}, \ell}=0,\mbox{ and }
{\bf N}(e^{(1)}_{n_0, \ell})=\lambda_{n_0}^{(1)}(B) e^{(1)}_{n_0, \ell}.
\end{equation}  
    \item[]
    \item[II).] Assumptions on the permittivity and permeability of each particle. In order to investigate the electromagnetic scattering by all-dielectric nano-particles with high contrast electric permittivity parameter, we assume that for a positive constant $\eta_0$, independent of $a$, 
\begin{equation}\label{contrast-epsilon}
\eta:=\epsilon_{r} - 1=\eta_0 \; a^{-2},\; \mbox{ with } \;\; \; ~~ a \ll 1,
\end{equation}
and the relative magnetic permeability $\mu_r$ to be moderate, namely $\mu_r=1$. 
    \item[]
    \item[III).] Assumption on the used incident frequency $k$. There exists a positive constant $c_0$, independent of $a$, such that
\begin{equation}\label{condition-on-k}
1 \, - \, k^2 \, \eta \, a^2 \, \lambda_{n_{0}}^{(1)}(B) \, = \, \pm \; c_0\; a^h,\;\; a \ll 1,
\end{equation}
where $\lambda_{n_{0}}^{(1)}(B)$ is the eigenvalue corresponding to $e_{n_0}^{(1)}$ in $(\ref{pre-cond})$ with $n=n_0$.
    \item[]
    \item[IV).] Assumptions on the distribution of the cluster of particles. Let $\Omega$ be a bounded domain of unit volume, containing the particles $D_m$, $m=1, 2, \cdots, \aleph$. We divide $\Omega$ into $[d^{-3}]$ { periodically distributed} subdomains $\Omega_m$, $m=1, 2, \cdots, \aleph$, such that $z_m\in D_m\subset\Omega_m$ with $z_m$ centering at $\Omega_m$.
Then we see that $\aleph=\Oh(d^{-3})$. Let each $\Omega_m$ be a cube of volume $d^3$.
To further describe the relation between $d$ and $a$, we take a positive constant $c_r$ to be the dilution parameter such that 
\begin{equation}\label{def-cr}
	{a^{3-h}}=c_r^{-3}{d^3}.
\end{equation}  

\begin{figure}[htbp]
	\centering
	\includegraphics[width=0.48\linewidth]{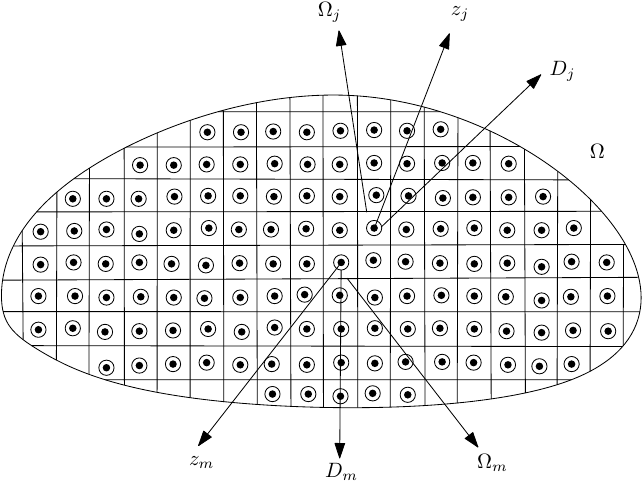}
	\caption{A schematic illustration for a periodic distribution of all-dielectric nano-particles.}
	\label{fig:omega-perodic}
\end{figure}
The above schematization describes a way of subdividing the domain $\Omega$, see Figure \ref{fig:omega-perodic}.
\end{enumerate}
\medskip
\begin{remark}\label{rem-vol-d}
The following remarks are necessary regarding these assumptions.
\begin{enumerate}
    \item[]
    \item An extra condition on $\eta_0$ in \eqref{contrast-epsilon} to be small enough or $c_r$ in \eqref{def-cr} to be large enough has to be proposed if we distribute the maximum numbers of particles, i.e. $\aleph\sim d^{-3}$.
    \item[]
    \item 	Since the shape of $\Omega$ is arbitrary, the intersecting part between the set of the cubes and $\partial \Omega$ is not necessarily to be empty, unless $\Omega$ is a cube. As the volume of $\Omega_m$ satisfies $|\Omega_m|=\Oh(d^3)$, then the maximum radius of $\Omega_m$ is of order $d$. Thus the intersecting surfaces, of the eventual $\Omega_m$'s, with $\partial \Omega$ possess the area of order $d^2$. Since the total area of $\partial \Omega$ is of order one, we deduce that the number of such particles near $\partial \Omega$ will not exceed the order of $d^{-2}$, and therefore the volume of this set is of order $d$ under the regime of \eqref{def-cr}, as $a\rightarrow0$.

    \item []
    \item Among all the available decompositions of the energy space $L^2(\Omega)$, the one given by $(\ref{hel-decomp})$ is natural as we know that for the Newtonian operator, when restricted to $\mathbb{H}_{0}(\div=0)$  and/or $\mathbb{H}_{0}(Curl = 0)$ admits an eigen-system and when restricted to $\nabla \mathcal{H}armonic$ it can be neglected, since its presence in the L.S.E is multiplied by the frequency $k^2$ which we consider as a small parameter. Similarly, for the Magnetization operator, when restricted to $\mathbb{H}_{0}(\div=0)$  is a vanishing one, when restricted to $\mathbb{H}_{0}(Curl=0)$ it will be reduced to the identity operator and, finally, when we restrict it to $\nabla \mathcal{H}armonic$ subspace admits an eigen-system.
\item[] 
\item Equation $(\ref{condition-on-k})$'s sign selection is dependent on how the eigenvalue $\lambda_{n_{0}}^{(1)}(B)$ is approached (i.e. from left or right). Consequently, the obtained results will be given in term of the selected sign. 
\end{enumerate}
\end{remark}
\end{assumptions}

Our starting point is the following proposition.

\begin{proposition}\cite[Theorem 1.3]{CGS}\label{prop-discrete}
	Consider the electromagnetic scattering problem \eqref{model-m} generated by a cluster of nano-particles $D_1, \cdots, D_\aleph$. Under \textbf{Assumptions}, for $\frac{9}{11}<h<1$, the electric far-field of the scattered wave possesses the expansion as 
	\begin{equation}\label{approximation-E}
	E^\infty(\hat{x}, \theta, p) = - \, \frac{i \, k^3 \, \eta}{4\, \pi}  \,  \sum_{m=1}^{\aleph} \, e^{- i \, k \, \hat{x} \cdot z_{m}} \hat{x}\times Q_m+\mathcal{O}\left(a^{\frac{h}{3}}\right),
	\end{equation} 
uniformly in all directions $\hat{x}  \, \in \, \mathbb{S}^2$, where 
	$\left( Q_m \right)_{m=1,\cdots, \aleph}$ is the vector solution to the following algebraic system
	\begin{equation}\label{linear-discrete}
	Q_{m}-\frac{\eta \, k^2}{\pm c_0} \, a^{5-h} \, \sum_{j=1 \atop j \neq m}^\aleph {\bf{P}}_0 \cdot \Upsilon_k(z_{m}, z_j) \cdot {Q}_j
	=\frac{i \, k}{\pm c_0} \, a^{5-h} \; {\bf{P}}_0 \cdot  H^{Inc}(z_{m}),
	\end{equation}
	with $\Upsilon_k(\cdot,\cdot)$ is the dyadic Green's function given by
	\begin{equation}\label{dyadic-Green}
		\Upsilon_k(x,z):=\underset{x}{\nabla} \, \underset{x}{\nabla} \, \Phi_k(x,z) \, + \, k^2 \, \Phi_k(x,z) \, {\bf I}, \quad \text{for} \quad x \neq z,
	\end{equation}
	where 
 \begin{equation}\label{DefFSHE}
  \Phi_k(x, z) \, := \, \dfrac{e^{i \, k \, \left\vert x - z \right\vert}}{4 \, \pi \, \left\vert x - z \right\vert}, \quad \text{for} \quad x \neq z,   
 \end{equation}
  being the fundamental solution to the Helmholtz equation, and ${\bf{P}}_0$ is the polarization tensor introduced in \eqref{defP0}. In particular, \eqref{linear-discrete} is invertible under the condition
	\begin{equation}\label{condi-inver-p1}
	\frac{ k^2 |\eta| a^5}{d^3 \left\vert 1 - k^2 \, \eta \, a^2 \, \lambda_{n_{0}}^{(1)} \right\vert} \left\vert {\bf{P}}_0 \right\vert < 1.
	\end{equation}
\end{proposition}

\begin{remark}

In \textbf{Proposition $\ref{prop-discrete}$}, the condition $(\ref{condi-inver-p1})$ is just a sufficient condition to ensure the invertibility of the algebraic system \eqref{linear-discrete}. However, if we take $-c_0$ in the denominator of \eqref{linear-discrete}, then the condition \eqref{condi-inver-p1} can be released, see \textbf{Remark $\ref{remark-remove-inver}$}.
\end{remark}

\subsection{The effective medium and the generated fields}

On the basis of \textbf{Assumptions} as well as the Foldy-Lax approximation presented in \textbf{Proposition \ref{prop-discrete}}, we shall investigate the corresponding effective medium generated by the cluster of all-dielectric nano-particles. We show in this current work that even if the dielectric nano-particles are merely generated by the contrasts of their permittivity, the effective medium is a perturbation of the permeability and not the permittivity. We denote $\mathring{\epsilon}_r$ and $\mathring{\mu_{r}}$ as the corresponding effective relative permittivity and permeability, respectively.  We state the main results by distinguishing two cases, one related to the positive definite $\mathring{\mu_{r}}$ and the second to the negative definite $\mathring{\mu_{r}}$. 

\begin{theorem}\label{thm-main-posi}
	Let $\Omega$ be a bounded domain with $C^{1,\alpha}$-regularity, for $\alpha\in (0,1)$. Then under \textbf{Assumptions}, there holds the following expansion for the far-fields
 \begin{enumerate}
     \item[]
     \item For positive definite $\mathring \mu_r$.
     \newline There exists a positive constant $c_{\mathrm{reg}}(\alpha, \Omega)$, depending only on $\alpha$ and $\Omega$, such that if
			\begin{equation}\notag
			\left( k^{3+\alpha}+k^3+k^2+k+1 \right) \; c_{\mathrm{reg}}(\alpha, \Omega) \; c_r^{-3} \; \left\vert {\bf P_0} \right\vert \;  < \; 1,
			\end{equation} 
     then
	\begin{align}\label{err-es}
 \notag
E^\infty_{\mathrm{eff}, +}(\hat{x}, \theta, p)-E^\infty(\hat{x}, \theta, p)&=\Oh\bigg(\frac{k^5 \, \eta_{0}^2 \, c_{r}^{-3} \, \lvert {\bf P_0} \rvert^{2}}{c_{0} \, \left( c_0 \, c_r^3 \, - \, k^2 \, \eta_0 \, \lvert{\bf P_0} \rvert \right)} \, \left\vert \left( {\bf I} \, - \, \frac{\eta_0 \, k^2}{\pm \, 3 \, c_0} \, c_r^{-3} \, {\bf P_0}\right)^{-1} \right\vert \\ &\cdot\left([H^{\mathring{\mu_{r}}}]^{2}_{C^{0, \alpha}(\overline\Omega)} \, c_r^{2 \, \alpha} \, a^{2 \, \alpha \, (1-\frac{h}{3})} \, \left\vert \log(a) \right\vert^{2} + \lVert H^{\mathring{\mu_{r}}}\rVert^{2}_{\mathbb L^{\infty}(\Omega)}  \, c_r^{\frac{12}{7}} \, a^{\frac{12}{7}(1-\frac{h}{3})}\right)^{\frac{1}{2}} \bigg) + \mathcal{O}\left(a^{\frac{h}{3}} \right),
\end{align}
uniformly in all directions $\hat{x}  \, \in \, \mathbb{S}^2$, where
$E^\infty_{\mathrm{eff},+}(\cdot)$ is the far-field pattern in the same sense as \eqref{far-def}, associated with the following electromagnetic scattering problem for the effective medium as $a\rightarrow 0$,
	\begin{equation}\label{model-equi}
	\begin{cases}
	\mathrm{curl}E^{{\mathring\epsilon}_r}-i k {\mathring\mu}_r H^{{\mathring\mu}_r}=0,\quad \mathrm{curl} H^{{\mathring\mu}_r}+i k{\mathring\epsilon_r} E^{{\mathring\epsilon}_r}=0\quad\mbox{in} \ \Omega,\\
	\mathrm{curl} E^{{\mathring\epsilon}_r}-i k H^{{\mathring\mu}_r}=0,\quad\mathrm{curl} H^{{\mathring\mu}_r}+i k E^{{\mathring\epsilon}_r}=0\quad\mbox{in}\ \mathbb{R}^3\backslash \Omega,\\
	E^{{\mathring\epsilon}_r}=E_s^{{\mathring\epsilon}_r}+E^{{\mathring\epsilon}_r}_{in}, \quad 	H^{{\mathring\mu}_r}=H_s^{{\mathring\mu}_r}+H^{{\mathring\mu}_r}_{in},\\
	\nu\times E^{{\mathring\epsilon}_r}|_{+}=\nu\times E^{{\mathring\epsilon}_r}|_{-}, \quad \nu\times H^{{\mathring\mu}_r}|_{+}=\nu\times H^{{\mathring\mu}_r}|_{-}\quad\mbox{on} \ \partial \Omega,\\
	H^{{\mathring\mu}_r}_s\times\frac{x}{|x|}-E^{{\mathring\epsilon}_r}_s=\Oh({\frac{1}{|x|^2}}),\quad \mbox{as} \ |x|\rightarrow\infty
	\end{cases}
	\end{equation}
 and $[H^{\mathring{\mu_{r}}}]_{C^{0, \alpha}(\overline\Omega)}$ denotes the H\"{o}lder norm of $H^{\mathring{\mu_{r}}}$. The effective permittivity and permeability are respectively given by
	\begin{equation}\label{new-coeff}
		\mathring{\epsilon}_r={\bf I} \quad \text{and} \quad
		\mathring \mu_r ={\bf I}+\frac{\eta_0 k^2}{\pm c_0}c_r^{-3}\left({\bf I}-\frac{\eta_0 k^2}{\pm 3 c_0}c_r^{-3}{\bf P_0}\right)^{-1}{\bf P_0}\quad \mbox{in}\quad\Omega.
	\end{equation} 
 \item[]
     \item For negative definite $\mathring \mu_r$. \newline
      Let the domain $B$, introduced in \eqref{def-B}, be the unit ball $B(0,1)$. \footnote{The reason why we assume that $B$ is a ball is that, in this case, the tensor ${\bf P_0}$ becomes isotropic (proportional to the identity matrix). Handling the plasmonic resonances with anisotropic permeability (or permittivity) is quite challenging but very interesting. This general case needs certainly to be studied with a different approach than the one we propose in \textbf{Section \ref{sec4:negative}}.} Under  the condition  $k^{2} > \dfrac{\pi^{3}}{4} \, \dfrac{c_{0} \, c_{r}^{3}}{\eta_{0}}$, 
\begin{eqnarray}\label{es-error-neg}
\nonumber
E^\infty_{\mathrm{eff},-}(\hat{x}, \theta, p)-E^\infty(\hat{x}, \theta, p) &=& \Oh\Bigg( \frac{k^5 \eta_0^2c_{r}^{-3}}{c_0\left( c_0 \, c_r^3 \,  \pi^3 \, -12 \, k^2 \, \eta_0 \right)} \left|\left(1 \, \mp \, \frac{4 \, \eta_0 \, k^2}{\pi^3 \, c_0} \, c_r^{-3} \right)^{-1}\right| \\  \nonumber
&\cdot& \left([H^{\mathring{\mu_r}}]^{2}_{C^{0, \alpha}(\overline\Omega)} \, c_r^{2 \, \alpha} \, a^{2 \, \alpha \, \left(1-\frac{h}{3}\right)} \, \left\vert \log\left( a \right) \right\vert^{2} \, + \, \lVert H^{\mathring{\mu_r}}\rVert^{2}_{\mathbb L^{\infty}(\Omega)} \, c_r^{\frac{12}{7}} \, a^{\frac{12}{7}\left(1-\frac{h}{3}\right)}\right)^{\frac{1}{2}} \Bigg) \\ &+& \mathcal{O}\left( a^{\frac{h}{3}} \right),
\end{eqnarray}
 uniformly in all directions $\hat{x}  \, \in \, \mathbb{S}^2$,
	where $[H^{\mathring{\mu_r}}]_{C^{0, \alpha}(\overline{\Omega})}$ is the H\"{o}lder coefficient. $E^\infty_{\mathrm{eff},-}(\cdot)$ is the far-field associated with the electromagnetic scattering problem \eqref{model-equi} where the effective permittivity and permeability are respectively given by	
 \begin{equation}\label{new-coeff-neg}
	\mathring{\epsilon}_r={\bf I} \quad \text{and} \quad  \mathring \mu_{r} \, = \, \frac{\pi^{3} \, \pm 8 \, \dfrac{\eta_0 k^2}{c_0 \,  c_r^{3}}}{\pi^{3} \, \mp 4 \, \dfrac{\eta_0 k^2}{c_0 \,  c_r^{3}}} \, {\bf I}, \quad \mbox{in}\quad\Omega. 
	\end{equation}  
 \end{enumerate}
\end{theorem}
\bigskip
The justification of \textbf{Theorem \ref{thm-main-posi}} highly relies on the $C^{0, \alpha}$-regularity of the solution $(E^{\mathring{\epsilon_{r}}}, H^{\mathring{\mu_{r}}})$ to \eqref{model-equi}, which has been investigated in \cite[Theorem 2.2]{Cao-Sini} for positive definite $\mathring{\mu_{r}}$. For $\mathring{\mu_{r}}$ being negative definite, $C^{0, \alpha}$-regularity of the solution also plays a key role in deriving the corresponding effective medium and we shall later give a separate proposition to guarantee it. The following theorem presents the other main result of the effective medium generated by the cluster of all-dielectric nano-particles with negative definite effective permeability $\mathring{\mu_{r}}$. Before presenting the main theorem with respect to negative definite $\mathring{\mu_r}$, we first introduce some notations. Denote
\begin{equation}\label{def-xi}
	\xi:=\frac{\eta_0 k^2}{c_0 \, c_r^{3}},
\end{equation}
where $\eta_0$ and $c_0$ are given in \eqref{contrast-epsilon} and \eqref{condition-on-k}, respectively. We also define for any vector function $F \in \mathbb{L}^2(\Omega)$, the Magnetization operator $\nabla{\bf M}$ as 
\begin{equation}\label{def-mag-operator}
	\nabla {\bf M}_D(F)(x) := \underset{x}{\nabla}\int_D\underset{z}{\nabla}\Phi_0(x, z) \cdot F(z) \,dz, \quad x \in D,
\end{equation}
where $\Phi_{0}(\cdot,\cdot)$ is given by $(\ref{def-new-operator})$. It is known that the Magnetization operator $\nabla{\bf M}_D: \nabla \mathcal{H}armonic \rightarrow \nabla \mathcal{H}armonic$ is self-adjoint, bounded and its spectrum $\sigma(\nabla{\bf M}_D)$ satisfies $\sigma(\nabla{\bf M}_D) \subset ]0, 1]$ with $\dfrac{1}{2}$ as the only accumulation point of its spectrum. We refer to \cite{Raevskii1994, Dyakin-Rayevskii, friedman1980mathematical, friedman1981mathematical} and \cite{ 10.2307/2008286} for more discussions on other properties of $\nabla{\bf M}$.

\begin{remark}\label{remark-remove-inver}
Three remarks are in order.
\begin{enumerate}
    \item[] 
    \item In the  case where the tensor ${\bf P_0}$ is defined over a ball, and then ${\bf P_0} \; = \; \frac{12}{\pi^3} \; {\bf I}$, see (\ref{exp-p0}), the condition $$ \frac{\eta_0 \, k^2}{c_0 \, c_r^{3}} \, < \, \frac{\pi^{3}}{8}$$ ensures the positive definiteness of the permeability tensor $\mathring \mu_r$, as it can be seen from (\ref{new-coeff}) .
    \item[]
    \item In \textbf{Theorem \ref{thm-main-posi}}, to generate the effective medium with negative definite $\mathring{\mu_{r}}$, the invertibility condition \eqref{condi-inver-p1} can be removed. Indeed, for $\mathring{\mu_{r}}$ being negative definite, it suffices to satisfy
	\begin{equation}\notag
		{\bf I} \, \neq \, \frac{\eta k^2}{\pm 3c_0}c_r^{-3}{\bf P_0} \overset{(\ref{exp-p0})}{=} \, \frac{4\, \eta \, k^2  }{\pm \, c_0 \, \pi^{3} \, c_r^{3}} \,{\bf I},
	\end{equation}
	which is obviously fulfilled, since we have the freedom to choose the parameters $\eta$, $k$, $c_{0}$ and $c_{r}$, which appear on the right hand side (R.H.S) of the previous in-equation.
    \item[]
    \item With \eqref{def-xi} and \eqref{new-coeff-neg}, $\mathring{\mu_{r}}$ can be represented as 
	\begin{equation}\notag
		\mathring{\mu_{r}}=\left(1\pm \frac{12}{\pi^3}\xi\left(1\mp \frac{4}{\pi^3}\xi\right)^{-1}\right){\bf I}.
	\end{equation}
	It is easy to verify that for $\xi \, > \, \dfrac{\pi^3}{4}$, 
	\begin{equation}\notag
		1\pm \frac{12}{\pi^3}\xi\left(1\mp \frac{4}{\pi^3}\xi\right)^{-1}<0,
	\end{equation}
	which indicates that $\mathring{\mu_r}$ is negative definite. \textbf{Theorem \ref{thm-main-posi}} presents the effective medium associated with positive or negative definite $\mathring{\mu_r}$. In both cases, and with the stronger condition that $B$ is a unit ball for the case of negative definite $\mathring{\mu_{r}}$, the $C^{0, \alpha}$ regularity of the solution to the electromagnetic scattering problem \eqref{model-equi} is obtained. Under this H\"{o}lder regularity result, the error term \eqref{es-error-neg} can be regarded as a particular form of the error \eqref{err-es} in \textbf{Theorem \ref{thm-main-posi}} by taking the explicit expression of the constant tensor ${\bf P_0}$ for negative definite $\mathring{\mu_r}$, see \textbf{Section \ref{sec4:negative}} for more detailed discussions.
 \end{enumerate}
\end{remark}  

\subsection{Extended quasi-static plasmonics and the generated giant electromagnetic fields}

 The interesting consequence of this last result is that for the effective permeability $\mathring{\mu_r}$ being negative, we can generate plasmonic resonators under quasi-static regime from the cluster of all-dielectric sub-wavelength resonators. We present this result in the following theorem. 

\begin{theorem}\label{coro-plas-resonance}
We follow the same notations and conditions in \textbf{Theorem \ref{thm-main-posi}}, $(2)$. In addition, we assume $\Omega$ to satisfy the following upper bound for its surface area to volume ratio  
\begin{equation}\label{non-flat}
\frac{\vert \partial \Omega\vert}{8 \;\vert \Omega\vert}<1.
\end{equation}
Then we have the following results.
\begin{enumerate}
    \item[]
    \item Let  $\Omega$ be a $C^{4,\alpha}$-smooth domain \footnote{This extra regularity condition is needed in the proof of \textbf{Proposition \ref{AddedLemma}}.} and $\lambda_{m_0}^{(3)}(\Omega) \, > \,\dfrac{1}{3}$ be any eigenvalue of the Magnetization operator $\nabla {\bf M}$, defined in $\Omega$. We choose the incident frequency $k$ of the form
 \begin{equation}\label{the plasmon-incident-frequency}
k^2=\frac{1}{\eta_0\;  \lambda_{n_0}^{(1)}(B)} \, + \, \beta \, \frac{\left( 3 \, \lambda_{m_0}^{(3)}(\Omega) \, - \,1 \right)}{\eta_0\;\lambda_{n_{0}}^{(1)}(B)},
\end{equation}
where $\beta$ is a small parameter, such that $k^{2} \, \leq \, C^\text{te} \, \left\vert \beta \right\vert$, with $C^\text{te}$ as a constant independent on $k$ and $\beta$. Then, in the regime where $\eta_{0} \gg 1$ and under the assumption of the existence of $m_{0} \in \mathbb{N}$ such that 
\begin{equation*}
 \boldsymbol{\mathcal{X}}_{m_{0}}(\lambda_{m_{0}}^{(3)}(\Omega), \beta) \neq 0,   
\end{equation*}
the far-field $E^{\infty}(\cdot)$  admits the following approximation  
\begin{equation}\label{far-E}
	E^\infty(\hat{x}, \theta, p) = \frac{\pm \, 3 \, k^{2}}{\left\vert \beta \right\vert \; \left( 3 \, \lambda_{m_{0}}^{(3)}\left( \Omega \right) - 1 \right)^{2}}  \, \hat{x} \, \times \, \left( \boldsymbol{\mathcal{X}}_{m_{0}}(\lambda_{m_0}^{(3)}(\Omega), \beta) \cdot \left( \theta \, \times \, \mathrm{p} \right) \right) + \mathcal{O}\left( k^{2} \right) + \mathcal{O}\left(\frac{k^3}{\left\vert \beta \right\vert} \right), 
\end{equation}
uniformly in all directions $\hat{x}  \, \in \, \mathbb{S}^2$, where the tensor $\boldsymbol{\mathcal{X}}_{m_{0}}(\lambda_{m_0}^{(3)}(\Omega), \beta)$ is given by 
\begin{equation*}
    \boldsymbol{\mathcal{X}}_{m_{0}}(\lambda_{m_0}^{(3)}(\Omega), \beta) := \int_{\Omega} e_{m_{0}}^{(3)}(z) \; dz   \otimes \int_{\Omega} \overline{  \begin{LARGE} \textbf{N} \end{LARGE}^{-1}\left(e_{m_0}^{(3)}\right)}(x) \, dx.  
\end{equation*}
with
\begin{equation*}
    \begin{LARGE} \textbf{N} \end{LARGE} := \left({\bf I} \, - \, \dfrac{\pm \, k^{2}}{2 \, \left\vert \beta \right\vert \, \left( 3 \, \lambda_{m_{0}}^{(3)}(\Omega) \, - \, 1 \right)^{2}} \left( N + N^{\prime} \right) \overset{3}{\mathbb{P}} \, \right)
\end{equation*}
where $N(\cdot)$ is the Newtonian operator given by $(\ref{def-new-operator})$, $N^{\prime}(\cdot)$ is the operator defined by 
\begin{equation}\label{DefNprimeTHM}
  N^{\prime}(F)(x) \, := \, \int_{\Omega} \, \Phi_0(x, z) \, \frac{(x - z) \otimes (x - z)}{\lVert x-z \rVert^2} \cdot F(z) \, dz, \quad x \in \Omega.
\end{equation}
    \item[]
    \item 
    \item Let $\Omega$ be a ball. In this case, we know that  $\lambda_{1}^{(3)}(\Omega) = \dfrac{1}{3}$ and $\displaystyle \int_{\Omega} e_{1}^{(3)}(z) \; dz \, \neq \, 0 $, and under the assumption 
    \begin{equation*}
    \langle \left( N \, + \, N^{\prime}\right)(e_{1}^{(3)}); e_{1}^{(3)} \rangle_{\mathbb{L}^{2}(\Omega)} \, \neq \, 0,
    \end{equation*}
    where $N(\cdot)$ is the Newtonian operator given by $(\ref{def-new-operator})$ and  $N^{\prime}(\cdot)$ is the operator defined by $(\ref{DefNprimeTHM})$, we have 
    \begin{eqnarray}\label{Eq126}
    \nonumber
    &  & \\
    E^{\infty}(\hat{x}, \theta, p) \, & = & \, \pm \, i \, 2 \, \pi \, k^{-1} \, \hat{x} \times \left(\theta \times \mathrm{p} \right)  \, + \,  \mathcal{O}\left( 1 \right), \\ \nonumber
    &&
    \end{eqnarray}
    uniformly in all directions $\hat{x} \, \in \, \mathbb{S}^2$, where the frequency $k$ is taken to be of the form
    \begin{equation}\label{Ballkscale}
        k^{4} \, = \, \frac{\pi^{3} \, c_{0} \, c_{r}^{3}}{6 \, \eta_{0} \, \left\langle \left(N + N^{\prime} \right)\left( e_{1}^{(3)} \right) ; e_{1}^{(3)} \right\rangle_{\mathbb{L}^{2}(\Omega)} },
    \end{equation}
    with $\eta_{0} \, \gg \, 1$ (and then $k\ll 1$).
    \item[] 
\end{enumerate}
\end{theorem}

\begin{remark} \label{R-7}   
Regarding \textbf{Theorem \ref{coro-plas-resonance}}, the following remarks are in order. 
\begin{enumerate}
    \item[]
    \item The condition (\ref{non-flat}) is not restrictive as it is satisfied by many shapes (as sphere, cube, polyhedron, etc), see \cite{wiki}. In addition, taking any of these domains and apply a dilation with a magnitude of order larger than one then we get domains satisfying (\ref{non-flat}). Finally, general convex bodies with appropriate affine transformations will also satisfy (\ref{non-flat}), see \cite{Ball}. Let us mention that the surface area to volume ratio appears in many areas of applied sciences as well.
    \item[] 
    \item For $k$ satisfying \eqref{the plasmon-incident-frequency}, with  $\eta_0 \, \gg \, 1$ and $\beta \, \ll \, 1$, we have a giant amplification of the magnetic field $H^{Inc}(\cdot)$. In other words, the incident frequency
\begin{equation}\notag
	k_{n_0} \, := \, \sqrt{\frac{1}{\eta_{0} \; \lambda_{n_0}^{(1)}\left( B \right)}}, 
\end{equation}
is a low-frequency (or quasi-static) plasmonic resonance generated by the cluster of the all-dielectric nano-particles. 
Moreover, the corresponding effective permeability now possesses the form
\begin{equation*}\label{DPmu}
	\mathring{\mu_r} \, = \, \left( \, 1\, - \, \frac{1}{\lambda_{m_0}^{(3)}(\Omega)} \, +  \, \mathrm{T}\left( \beta \right) \right) \; {\bf I}.
\end{equation*}	 
with $\left\vert \mathrm{T}\left( \beta \right) \right\vert \; \ll \; \left\vert \beta \right\vert \ll 1 $. Remember that $\lambda^{(3)}_{n}(\Omega) \, < \, 1$, for every $n$, hence $\mathring{\mu_r}<0$ for $\vert \beta\vert \ll 1$.
 \item[]
   \item If we choose the back-scattering direction, i.e. $\hat{x}=-\theta$, then the formula $(\ref{Eq126})$ will be reduced to 
       \begin{equation*}
    E^{\infty}(\hat{x}, \theta, p) \, = \, \pm \, i \, 2 \, \pi \, k^{-1} \, \mathrm{p}  \, + \,  \mathcal{O}\left( 1 \right),
    \end{equation*}
    where $k$ is given by $(\ref{Ballkscale})$. In particular, the generated electromagnetic field is very large and directed into the polarization direction $p$. In addition, this giant behavior of the electromagnetic field infers that the frequency in $(\ref{Ballkscale})$ behaves as a quasi-static Plasmonic resonance.
   \item[] 
        \item In $(\ref{far-E})$, if  $\displaystyle\int_{\Omega} e_{n}^{(3)}(z) \, dz =0 $,\; $\forall \; n \in \mathbb{N}$, by expanding the function $e^{- \, i \, k \, \hat{x} \cdot z}$ with respect to the frequency $k$, see the far-field expression given by $(\ref{add-coro-E})$, neglecting the zeroth order term, keeping the first order term and repeating the same computations we get the following formula
        \begin{eqnarray*}
            E^{\infty}\left( \hat{x}, \theta, p \right) &=& \frac{k^{2} \, 4 \, \pi}{9 \, \left( 3 \, \lambda_{m_{0}}^{(3)}(\Omega) \, - \, 1  \right)} \; \hat{x} \times \left( \theta \times \mathrm{p} \right)  \\
            &-& \frac{2 \, k^{2}}{\langle \left( N + N^{\prime} \right)\left( e_{m_{0}}^{(3)}\right) ; e_{m_{0}}^{(3)} \rangle_{\mathbb{L}^{2}(\Omega)}} \; \hat{x} \times \left( \boldsymbol{\Lambda}_{m_{0}}  \left( \hat{x}, \theta \right) \, \cdot \, \left( \theta \times \mathrm{p} \right) \right) \; + \; \mathcal{O}\left( k^{3} \right),  
        \end{eqnarray*}
        where $N(\cdot)$ is the Newtonian operator given by $(\ref{def-new-operator})$, $N^{\prime}(\cdot)$ is the operator defined by $(\ref{DefNprimeTHM})$ and $\boldsymbol{\Lambda}_{m_{0}}\left( \hat{x}, \theta \right)$ is the tensor given by
        \begin{equation*}
            \boldsymbol{\Lambda}_{m_{0}}\left( \hat{x}, \theta \right) := \int_{\Omega} \hat{x} \cdot z \; e_{m_{0}}^{(3)}(z) \, dz \otimes \int_{\Omega} \theta \cdot z \; \overline{e_{m_{0}}^{(3)}}(z) \, dz.
        \end{equation*}
        \item[]
        \item To avoid complicated computations and presentation, we have focused on the case $\lambda_{n}^{(3)}(\Omega) \geq \dfrac{1}{3}$. To deal with the other eventual case \footnote{For a general shape $\Omega$, we have $0 \, < \, \lambda^{(3)}_{m}(\Omega) \,< \, 1$ and for $\Omega$ as a ball we have $\frac{1}{3} \, < \, \lambda^{(3)}_{m}(\Omega) \, < \, \frac{2}{3}$. The existence of shapes having such eigenvalues smaller than $\frac{1}{3}$ is shown in \cite[Theorem 3.4]{ahner1994} for instance.}, i.e. $\lambda_{n}^{(3)}(\Omega) < \dfrac{1}{3}$, we need to look for $\xi$ as the solution to the full dispersion equation
        \begin{equation*}
           \left(\pi^{3} \, + \, 4 \, \xi \right) \, - \, 12 \, \xi \, \lambda_{n}^{(3)}(\Omega) \, + \, 12 \, \xi \, \langle K(e_{n}^{(3)}); e_{n}^{(3)} \rangle_{\mathbb{L}^{2}(\Omega)} \; = \; 0,
        \end{equation*}
        which by using $(\ref{ANC})$ becomes 
        \begin{eqnarray*}
           0 &=& \pi^{4} \, + \, \pi \, 4 \, \xi  \, \left( 1 - \, 3  \, \lambda_{n}^{(3)}(\Omega) \right) \, + \pi \, 6 \, \xi k^2 \, \langle \left( N + N^{\prime} \right)(e_{n}^{(3)}); e_{n}^{(3)} \rangle_{\mathbb{L}^{2}(\Omega)}  + 2 \, \xi \, i \, k^{3} \, \left\vert \int_{\Omega} e_{n}^{(3)}(x) \, dx \right\vert^{2} \\ &+&  3 \, \xi  \, \sum_{j \geq 3} \, \frac{(i \, k)^{j+1}}{(j+1)!} \, \langle \int_{\Omega} \nabla \nabla \left\vert \cdot - y \right\vert^{j} \cdot e_{n}^{(3)}(y) \, dy; e_{n}^{(3)} \rangle_{\mathbb{L}^{2}(\Omega)} \\ &+& 3 \, \xi \, k^{2} \, \sum_{j \geq 1} \, \frac{(i \, k)^{j+1}}{(j+1)!} \, \langle \int_{\Omega}  \left\vert \cdot - y \right\vert^{j} \, e_{n}^{(3)}(y) \, dy; e_{n}^{(3)} \rangle_{\mathbb{L}^{2}(\Omega)},         
        \end{eqnarray*}
        which can be seen, recalling that $\xi = \dfrac{k^{2} \, \eta_{0}}{c_{0} \, c_{r}^{3}}$, see $(\ref{def-xi})$, as a polynomial function in terms of the frequency $k$. With similar arguments, we derive the same results as in \textbf{Theorem \ref{coro-plas-resonance}}, for $\lambda^{(3)}_{m}(\Omega) \, < \, \dfrac{1}{3}$.   
    \end{enumerate}
\end{remark}
\medskip

\begin{remark}\label{R-8}
   Thanks to \cite[Theorem 3.4]{ahner1994}, we know that for any $l \, \in \, ]0,1]$, there exists a smooth surface $\partial \Omega^l$, such that $l$ is an eigenvalue of $\nabla {\bf M}$, in $\Omega^l$. In particular, for any $\varsigma \, \ll \, 1$, there existences $\Omega^{\pm \, \varsigma}$ such that $\dfrac{1}{3} \, \pm \, \varsigma \, \in \, \sigma\left( \nabla {\bf M}_{\Omega^{\pm \, \varsigma}} \right)$, and in this case $(\ref{far-E})$ takes the following form
   \begin{equation*}
	E^\infty(\hat{x}, \theta, p) = \frac{- \, 6}{\varsigma^{4}}  \, \hat{x} \, \times \, \left( \boldsymbol{\mathcal{Z}}_{m_{0}} \cdot \left(\mathrm{p} \times \theta \right) \right) + \mathcal{O}\left( \varsigma^{2} \right) + \mathcal{O}\left( k \right), 
\end{equation*}
with
\begin{equation*}
   \boldsymbol{\mathcal{Z}}_{m_{0}} \, := \, \int_{\Omega} e_{m_{0}}^{(3)}(z) \; dz  \, \otimes \, \int_{\Omega} \,   \overline{\left( \left(  N + N^{\prime} \right) \overset{3}{\mathbb{P}} \, \right)^{-1}\left(e_{m_0}^{(3)}\right)} \, dx.
\end{equation*}
Hence, as $\varsigma$ is taken to be small, we observe a huge amplification of the electric field.
\end{remark}

Let us finish this part of the section by stressing that although the results presented in \textbf{ Theorem \ref{thm-main-posi}} and \textbf{Theorem \ref{coro-plas-resonance}} are derived for the farfields $E^{\infty}\left( \hat{x}, \theta, p \right)$, $\hat{x} \in \mathbb{S}^2$, they can be also derived, in a similar way, for the scattered fields $E^s\left( x, \theta, p \right)$ for $x \in \mathbb{R}^3\setminus \overline{\Omega}$.

\subsection{Formal derivation of the effective medium}
Based on \textbf{Proposition \ref{prop-discrete}}, formula $(\ref{linear-discrete})$, we have the following linear algebraic system
\begin{eqnarray}
\nonumber
	U_m \, - \, \frac{\eta\, k^2}{\pm c_0}\, a^{5-h} \,  \sum_{j=1 \atop j \neq m}^{\aleph} \, \Upsilon_{k}(z_m, z_j) \cdot {\bf P_0} \cdot U_{j} \, &=& \, i \, k \, H^{Inc}(z_m) \\ \label{alg-dis-00}
 U_m \, - \, \frac{\eta_{0} \, k^2}{\pm c_0}\, a^{3-h} \,  \sum_{j=1 \atop j \neq m}^{\aleph} \, \Upsilon_{k}(z_m, z_j) \cdot {\bf P_0} \cdot U_{j} \, & \overset{(\ref{contrast-epsilon})}{=} & \, i \, k \, H^{Inc}(z_m), 
\end{eqnarray}
where
\begin{equation*}
	U_m \; := \; \pm \, c_0 \, a^{h-5} \, {\bf P_0^{-1}} \cdot Q_m, \quad \text{for} \quad 1 \leq m \leq \aleph.
\end{equation*} 
Accordingly, the discrete form of the Foldy-Lax approximation given in \textbf{Proposition \ref{prop-discrete}}, formula $(\ref{approximation-E})$, becomes
\begin{eqnarray}
\nonumber
	E^\infty(\hat x) \, &=& \, - \frac{i\, k^3\, \eta}{\pm\, 4\, \pi\, c_0} \, a^{5-h}
\sum_{m=1}^\aleph e^{-ik\hat{x}\cdot z_m}\hat{x}\times {\bf P_0} \cdot U_m \, + \, \Oh(a^{\frac{h}{3}}) \\ \nonumber
& \overset{(\ref{contrast-epsilon})}{=} &  \, - \frac{i\, k^3\, \eta_{0}}{\pm\, 4\, \pi\, c_0} \, a^{3-h}
\sum_{m=1}^\aleph e^{-ik\hat{x}\cdot z_m}\hat{x}\times {\bf P_0} \cdot U_m \, + \, \Oh(a^{\frac{h}{3}}) \\ \label{far-dis-00}
& \overset{(\ref{def-xi})}{=} &  \, - \frac{i\, k \, \xi}{\pm \, 4 \, \pi} \, c_{r}^{3} \, a^{3-h}
\sum_{m=1}^\aleph e^{-ik\hat{x}\cdot z_m}\hat{x}\times {\bf P_0} \cdot U_m \, + \, \Oh(a^{\frac{h}{3}}).
\end{eqnarray}
Now, inspired by \eqref{alg-dis-00}, we have the following integral form for the magnetic field 
\begin{equation}\label{PRP1}
H^{\mathring{\mu}_r}(x) \, - \, \frac{\eta_0 \, k^2}{\pm \, c_0} \, a^{3-h}  \, \int_{\Omega \setminus \Omega_{m}} \Upsilon_k(x,y) \cdot {\bf P_0} \cdot H^{\mathring{\mu}_r}(y) \, dy \, = \, i \, k \, H^{Inc}(x), \quad x \in \Omega_{m}, 
\end{equation}
with $\Upsilon_{k}(\cdot,\cdot)$ is the dyadic Green's function given by $(\ref{dyadic-Green})$. Moreover, under Assumption \uppercase\expandafter{\romannumeral4}, we know that the dilution parameter $c_r$ fulfills
\begin{equation}\label{distribute}
	c_r^{-3} \, = \, \frac{a^{3-h}}{d^3} \quad \text{and} \quad \left\vert \Omega_m \right\vert \,  = \, \Oh\left( d^3 \right) \, = \, \Oh\left(c_r^3 a^{3-h}\right).
\end{equation}
This allows us to rewrite $(\ref{PRP1})$ as 
\begin{equation*}
H^{\mathring{\mu}_r}(x) \, - \, \frac{\eta_0 \, k^2}{\pm \, c_0 \, c_{r}^{3}} \, d^{3}  \, \int_{\Omega \setminus \Omega_{m}} \Upsilon_k(x,y) \cdot {\bf P_0} \cdot H^{\mathring{\mu}_r}(y) \, dy \, = \, i \, k \, H^{Inc}(x), \quad x \in \Omega_{m},
\end{equation*}
or, 
\begin{eqnarray*}
H^{\mathring{\mu}_r}(x) \, &+& \, \frac{\eta_0 \, k^2}{\pm \, c_0 \, c_{r}^{3}} \, d^{3}  \,  \int_{\Omega_{m}}   \Upsilon_k(x,y) \cdot {\bf P_0} \cdot H^{\mathring{\mu}_r}(y) \, dy \\ &-& \, \frac{\eta_0 \, k^2}{\pm \, c_0 \, c_{r}^{3}} \, d^{3}  \,  \int_{\Omega} \,   \Upsilon_k(x,y) \cdot {\bf P_0} \cdot H^{\mathring{\mu}_r}(y) \, dy \, = \, i \, k \, H^{Inc}(x), \quad x \in \Omega_{m}.
\end{eqnarray*}
Then, from the second term on the left hand side (L.H.S), we extract a constant tensor, that we denote by ${\bf T}^{\mathring\mu_{r}}$, related to the periodic distribution of the cluster of nano-particles, we will provide later with further details, and this can be used to derive the coming Lipmann-Schwinger system of Equation (L.S.E) 
\begin{equation}\label{LS-1}
	H^{\mathring{\mu}_r}(x) \, - \, \frac{\eta_0 \, k^2}{\pm \, c_0 \, c_r^{3}}  \, \left[ - \, \nabla {\bf M}^{k}\left( {\bf T}^{\mathring\mu_{r}} \cdot H^{\mathring{\mu}_r}\right)(x) \, + \, k^{2} \, {\bf N^{k}} \left( {\bf T}^{\mathring\mu_{r}} \cdot H^{\mathring{\mu}_r}\right)(x)\right] \, = \, i\, k \, H^{Inc}(x, \theta), \quad  x \in \Omega,
\end{equation}
where ${\bf T}^{\mathring\mu_{r}}$ is a constant tensor defined by
\begin{equation}\label{exp-T}
	{\bf T}^{\mathring\mu_{r}} \, = \, \left({\bf I} \, - \, \frac{\eta_0 \, k^2}{\pm 3c_0}c_r^{-3}{\bf P_0}\right)^{-1}{\bf P_0},
\end{equation}
with ${\bf P_0}$ given by \eqref{defP0}, $\nabla {\bf M}^{k}$ is the Magnetization operator with non vanishing frequency defined by
\begin{equation*}
    \nabla {\bf M}^{k}\left( E \right)(x) \, := \, \nabla \int_{\Omega} \underset{y}{\nabla}\Phi_{k}(x,y) \cdot E(y) \, dy, \quad x \in \Omega, 
\end{equation*}
and ${\bf N}^{k}$ is the Newtonian operator with non vanishing frequency given by 
\begin{equation*}
   {\bf N^{k}}\left( E \right)(x) \, := \, \int_{\Omega} \Phi_{k}(x,y)  E(y) \, dy, \quad x \in \Omega,
\end{equation*}
where $\Phi_{k}(\cdot,\cdot)$ is defined by $(\ref{DefFSHE})$. Furthermore, it is straightforward to verify that $H^{\mathring{\mu}_r}(\cdot)$ in \eqref{LS-1} is the solution to the electromagnetic scattering problem for the effective medium \eqref{model-equi}
with
\begin{equation}\label{effective-permi}
	\mathring{\epsilon}_r={\bf I}\quad\mbox{and}\quad 
	\mathring{\mu}_r={\bf I}+ \frac{\eta_0 k^2}{\pm c_0}c_r^{-3}{\bf T}^{\mathring\mu_{r}}, \quad \text{in} \;\; \Omega.
\end{equation}	
\begin{remark}\label{rem-F}
 We add the following two observations.
 \begin{enumerate}
     \item[] 
     \item By choosing appropriate parameters $c_0, \eta_{0}$ and $c_r$, the invertibility of the matrix
		\begin{equation}\label{P0-introduction}
			{\bf I}-\frac{\eta_0 \, k^2}{\pm 3c_0}c_r^{-3}{\bf P_0}
		\end{equation} 
		in $(\ref{exp-T})$ can be guaranteed. Indeed, it suffices to let $c_r$ to be large enough. If, in particular, ${\bf P_0}$ is constructed for the unit ball $\Omega:=B(0;1)$, then with the explicit form of ${\bf P_0}$ given by $(\ref{exp-p0})$, it is sufficient to fulfill $\dfrac{4\eta_0 k^2}{\pm c_0 \pi^3}c_r^{-3}\neq 1$ to get the invertibility of (\ref{P0-introduction}).
  \item[]
  \item By rearranging terms in \eqref{exp-T}, we  see that ${\bf T}^{\mathring\mu_{r}}$ satisfies the following equation
	\begin{equation*}
			{\bf T}^{\mathring\mu_{r}}\left({\bf I}+\frac{\eta_0 \, k^2}{\pm 3c_0}c_r^{-3}{\bf T}^{\mathring\mu_{r}}\right)^{-1}={\bf P_0},
	\end{equation*}

	where the invertibility of the matrix ${\bf I}+\dfrac{\eta_0 \, k^2}{\pm 3c_0}c_r^{-3}{\bf T}^{\mathring{\mu_r}}$ can be fulfilled for $c_r$ large enough, see the observation (1) above as well.
 \item[] 
 \end{enumerate}
\end{remark}

\bigskip
The rest of the paper is organized as follows. \textbf{Section \ref{sec3:positive}} is devoted to prove \textbf{Theorem \ref{thm-main-posi}} for positive definite $\mathring{\mu_{r}}$ by investigating the effective medium and evaluating the error estimate between the far-fields $E^\infty(\cdot)$ and $E^\infty_{\mathrm{eff},+}(\cdot)$, under some necessary a-priori estimates. In \textbf{Section \ref{sec4:negative}}, we first present a systematical analysis on the behaviours of $\mathring{\mu_{r}}$ by different choices of the parameter $c_0$, in particular, the case for $\mathring{\mu_r}$ being negative definite is studied in detail under certain geometrical setup. Then the $C^{0, \alpha}$-regularity of the solution $(E^{\mathring{\epsilon_{r}}}(\cdot), H^{\mathring{\mu_{r}}}(\cdot))$ for negative definite $\mathring{\mu_{r}}$ is presented and the corresponding proof is given in \textbf{Section \ref{sec5:proof}}. Based on the similar a-priori estimates in \textbf{Section \ref{sec3:positive}}, we can prove \textbf{Theorem \ref{thm-main-posi}} for negative definite $\mathring{\mu_{r}}$, accordingly. Finally, with appropriate choice of the parameter $\xi$, we shall generate the plasmonic resonance under quasi-static regime from the cluster of all-dielectric nanoparticles by proving \textbf{Theorem \ref{coro-plas-resonance}}. \textbf{Section \ref{sec5:proof}} is devoted to provide the detailed proof of the a-priori estimates in \textbf{Section \ref{sec3:positive}} and the  $C^{0, \alpha}(\overline{\Omega})$-regularity of $H^{\mathring{\mu_{r}}}$ in \textbf{Section \ref{sec4:negative}}.
\bigskip

In the rest of the manuscript, the dependency of the electromagnetic fields $(E, H)$ and $(E^{\mathring{\mu}_r}, H^{\mathring{\mu}_r})$, with the incident fields, in terms of the incident direction $\theta$ and polarization direction $\mathrm{p}$ will be omitted to reduce heavy notations.

\section{The effective medium with $\mathring\mu_{r}$-Dielectric}\label{sec3:positive}

In this section, we consider the electromagnetic scattering for the effective medium with positive permeability $\mathring{\mu_r}$, i.e. $\mathring{\mu_r} \, > \, 0$. The analyses will be divided into the following three parts. First, we present the H\"{o}lder regularity of the solution to the L.S.E \eqref{LS-1}, which also indicates the well-posedness for the scattering problem \eqref{model-equi} for $\mathring{\mu_{r}}$ being positive. Then, we give several significant a-priori estimates associated with the L.S.E as well as the discrete linear algebraic system \eqref{linear-discrete}. Finally, we shall prove our main results in \textbf{Theorem \ref{thm-main-posi}}.

\subsection{$\boldsymbol{C^{0, \alpha}}$-regularity for the solution $\boldsymbol{H^{\mathring{\mu_{r}}}}$ with $\boldsymbol{ \mathring{\mu_{r}}}$-Dielectric.}

From \cite{Cao-Sini}, where the effective permittivity and permeability generated by a cluster of moderately contrasting nano-particles are studied, by following a similar argument, we know that for $\mathring{\mu_{r}}$ being positive definite in the scattering problem \eqref{model-equi}, under certain conditions with respect to the frequency $k$, the dilution parameter $c_r$ and the constant tensor ${\bf P_0}$, there also holds the H\"{o}lder regularity,  up to the boundary, of the electromagnetic fields associated with \eqref{model-equi}.

\begin{proposition}\cite[Theorem 2.2]{Cao-Sini}\label{prop-regularity}
		Let $\Omega$ be a bounded domain with $C^{1,\alpha}$-regularity, for $\alpha\in (0,1)$. Regarding the electromagnetic scattering problem  \eqref{model-equi}, with 
 $\mathring{\mu_{r}}$ being positive definite, there exists a positive constant $c_{\mathrm{reg}}(\alpha, \Omega)$, depending only on $\alpha$ and $\Omega$, such that if
			\begin{equation}\notag
			\left( k^{3+\alpha}+k^3+k^2+k+1 \right) \; c_{\mathrm{reg}}(\alpha, \Omega) \; c_r^{-3} \; \left\vert {\bf P_0} \right\vert \;  < \; 1,
			\end{equation} 
		then we have
		\begin{equation}\notag
		(E^{\mathring{\epsilon}_r},H^{\mathring{\mu}_r})\in C^{0,\alpha}({\overline{\Omega}})\times C^{0,\alpha}({\overline{\Omega}}).	
		\end{equation}
	\end{proposition}
	We refer to \cite[Theorem 2.2]{Cao-Sini} for the detailed proof of \textbf{Proposition \ref{prop-regularity}}. The coming corollary is suggested by the previous proposition.
\begin{corollary}\label{CoroHolderH}
For $1 \leq j \leq \aleph$ and $z \in \Omega_{j}$, denote $S_j\subset\Omega_j$ as the largest ball contained in $\Omega_j$, then we have 
\begin{equation}\label{AS0456}
    \left\vert H^{\mathring{\mu}_r}(z) - \frac{1}{\left\vert S_{j} \right\vert} \, \int_{S_{j}} H^{\mathring{\mu}_r}(y) \, dy  \right\vert \; \lesssim \; \left[ H^{\mathring{\mu}_r} \right]_{C^{0,\alpha}\left( \overline{\Omega} \right)} \; d^{\alpha}.  
\end{equation}
\end{corollary}

\subsection{A-priori estimates with $\boldsymbol{ \mathring{\mu_r}}$-Dielectric.}

In order to prove \textbf{Theorem \ref{thm-main-posi}}, we first give some important a-priori estimates related to the L.S.E \eqref{LS-1} and the discrete linear algebraic system \eqref{linear-discrete} in the following proposition, which is derived by the counting lemmas and the Calder\'{o}n-Zygmund inequality. Denote $S_m\subset\Omega_m$ as the largest ball contained in $\Omega_m$, $m=1, \cdots, \aleph$. Undoubtedly, for every $m$ such that $1 \leq m \leq \aleph$, we have $\left\vert S_{m} \right\vert \sim \left\vert \Omega_{m} \right\vert$. Based on \textbf{Proposition \ref{prop-regularity}}, we have. 

\begin{proposition}\label{prop-es-LS}
Under \textbf{Assumptions}, there holds the coming algebraic system associated with the L.S.E \eqref{LS-1},  
\begin{equation*}
    \digamma_{m} \, - \, \pm \, \xi \, d^{3} \, \sum_{j=1 \atop j \neq m}^{\aleph} \Upsilon_{k}\left( z_{m}, z_{j} \right) \cdot {\bf P}_{0} \cdot \digamma_{j} \, = \, i \, k \, H^{Inc}(z_m) \, + \, \mathrm{Error}_m, 
\end{equation*}
where 
\begin{equation*}
    \digamma_{m} \, := \, \left( {\bf I}+\frac{\eta_0 \, k^2}{\pm 3c_0}c_r^{-3} {\bf T}^{\mathring\mu_{r}}\right) \cdot \frac{1}{|S_m|}\int_{S_m}H^{\mathring{\mu}_r}(x)\,dx, \quad m \, = \, 1, \cdots, \aleph,
\end{equation*}
    with  ${\bf T}^{\mathring\mu_{r}}$ as the tensor introduced in $(\ref{exp-T})$, $\eta_0$ is the constant given in \eqref{contrast-epsilon} and $\mathrm{Error}_m$ is the corresponding error term satisfying the following estimate
	\begin{equation*}
	\sum_{m=1}^{\aleph} \left\vert \mathrm{Error}_{m} \right\vert^{2}
	\lesssim \, \frac{\eta_0^2 k^4}{c_0^2} \, \lvert{\bf T}^{\mathring\mu_{r}}\rvert^{2} \; c_r^{-6} \; \left( [H^{\mathring\mu_{r}}]^{2}_{C^{0, \alpha}(\overline{\Omega})}d^{2 \, \alpha \, - \, 3 } \, \left\vert \log(d) \right\vert^{2} \,
	+ \, \lVert H^{\mathring\mu_{r}}\rVert^{2}_{\mathbb L^{\infty}(\Omega)} \, d^{-\frac{9}{7}} \right).
	\end{equation*}
In addition, we have 
	\begin{equation}\label{es-l2}
         \left( \sum_{m=1}^\aleph \left\vert \digamma_{m} \, - \, U_{m} \right\vert^{2}\right)^{\frac{1}{2}} \\ 
		\lesssim  \frac{\eta_0\, k^2 \; \left\vert {\bf T}^{\mathring\mu_{r}} \right\vert}{\left( c_0 \, c_r^3 \, -k^2 \, \eta_0 \, \lvert{\bf P_0}\rvert \right)} \; \left([H^{\mathring\mu_{r}}]^{2}_{C^{0,\alpha}(\overline{\Omega})} \, d^{2 \, \alpha - 3} \, \left\vert \log(d) \right\vert^{2} \, + \, \left\Vert H^{\mathring\mu_{r}} \right\Vert^{2}_{\mathbb{L}^{\infty}(\Omega)} \, d^{-\frac{9}{7}}\right)^{\frac{1}{2}},
	\end{equation}
where $\{ U_m \}_{1 \leq m \leq \aleph}$ is the solution to the algebraic system $(\ref{alg-dis-00})$. 

\end{proposition}

In order to clear the structure of our paper, we postpone the proof of \textbf{Proposition \ref{prop-es-LS}} to \textbf{Section \ref{sec5:proof}}.
\bigskip
\newline 
With the regularity results for the solution as well as the necessary a-priori estimates above, we are now in a position to prove \textbf{Theorem \ref{thm-main-posi}},  for the effective medium with $\mathring{\mu_{r}}$-Dielectric as follows.

\subsection{Proof of Theorem \ref{thm-main-posi}.}
 
We start by deriving a useful integral representation of $E^{\infty}_{eff, +}(\cdot)$ from $(\ref{LS-1})$. For this, we take the Curl operator in both sides of $(\ref{LS-1})$, to get 
\begin{eqnarray*}
	Curl\left( H^{\mathring{\mu}_r} \right)(x) \, \mp \, \xi \, k^{2} \, Curl \, \int_{\Omega} \Phi_{k}(x,z) \, {\bf T}^{\mathring\mu_{r}} \cdot H^{\mathring{\mu}_r}(z) \, dz  \, = \, i\, k \, Curl\left(H^{Inc}\right)(x) \, &=&  \, k^{2} \, E^{Inc}(x) \\
 	Curl\left( H^{\mathring{\mu}_r} \right)(x) \, \mp \, \xi \, k^{2} \, \int_{\Omega} \underset{x}{\nabla} \Phi_{k}(x,z) \, \times  \, {\bf T}^{\mathring\mu_{r}} \cdot H^{\mathring{\mu}_r}(z) \, dz \, - \, k^{2} \, E^{T}(x) &=& - \, k^{2} \, E^{s}(x).
\end{eqnarray*}
Since $H^{\mathring{\mu}_r}(\cdot)$ is the magnetic field generated by a source $i \, k \, H^{Inc}(\cdot)$, instead of $H^{Inc}(\cdot)$, the equations satisfied by $H^{\mathring{\mu}_r}(\cdot)$ are to understood up to the coefficient $i \, k$, in particular, we have
\begin{equation*}
    Curl\left( H^{\mathring{\mu}_r} \right) \, = \, i \, k \, Curl\left( H^{T} \right) \, =\, - \, i \, k \, i \, k \, E^{T}  \, =  \, k^{2} \, E^{T}, \quad\mbox{for}\quad x\in\mathbb{R}^3\backslash\overline{\Omega}. 
\end{equation*}
Then, we reduce the previous equation to
\begin{equation*}
 \, \pm \, \xi  \, \int_{\Omega} \underset{x}{\nabla} \Phi_{k}(x,z) \, \times  \, {\bf T}^{\mathring\mu_{r}} \cdot H^{\mathring{\mu}_r}(z) \, dz \, = \, E^{s}(x), \quad \left\vert x \right\vert \, \rightarrow \, + \, \infty.
\end{equation*}
In addition, we have, 
\begin{equation*}
    \underset{x}{\nabla} \Phi_{k}(x,z) \, = \, \frac{i \, k}{4 \, \pi} \, \frac{x}{\left\vert x \right\vert^{2}} \, e^{i \, k \, \left\vert x \right\vert} \, e^{- \, i \, k \, \hat{x} \cdot z} \, + \, \mathcal{O}\left( \frac{1}{\left\vert x - z \right\vert^{2}} \right), \quad \left\vert x \right\vert \, \rightarrow \, + \, \infty. 
\end{equation*}
Then, by plugging this formula into the previous one, we end up with
\begin{equation}\label{Eq0224Eq}
\frac{ e^{i \, k \, \left\vert x \right\vert}}{\left\vert x \right\vert} \, \left( \, \mp \, \frac{i \, k}{4 \, \pi} \, \xi  \, \int_{\Omega}  \, e^{- \, i \, k \, \hat{x} \cdot z}  \, \hat{x} \, \times  \, {\bf T}^{\mathring\mu_{r}} \cdot H^{\mathring{\mu}_r}(z) \, dz \, + \,  \mathcal{O}\left( \frac{1}{\left\vert x \right\vert} \right) \right)  \, = \, E^{s}(x), \quad \left\vert x \right\vert \, \rightarrow \, + \, \infty.
\end{equation}
Moreover, from $(\ref{far-def})$, we know that 
\begin{equation*}
    E^{s}(x) \, = \, \frac{e^{i \, k \, \left\vert x \right\vert}}{\left\vert x \right\vert} \, \left( E^{\infty}_{eff, +}(\hat{x}) + \mathcal{O}\left(\frac{1}{\left\vert x \right\vert}\right) \right),
\end{equation*}
which, by identification with $(\ref{Eq0224Eq})$, we get the representation

\begin{equation}\label{ls-eff}
    E^{\infty}_{eff, +}(\hat{x}) \, =  \, \pm \, \frac{i \, k}{4 \, \pi} \, \xi  \, \int_{\Omega}  \, e^{- \, i \, k \, \hat{x} \cdot z}  \, \hat{x} \, \times  \, {\bf T}^{\mathring\mu_{r}} \cdot H^{\mathring{\mu}_r}(z) \, dz.
\end{equation}
Next, we use $(\ref{exp-T})$, to rewrite it like
\begin{equation}\label{LS-Far}
	    E^{\infty}_{eff, +}(\hat{x})  \,  =  \, \pm \, \frac{i\, k}{4\, \pi}\, \xi\, \int_{\Omega}\,  e^{-\, i\, k\, \hat{x}\cdot z}\, \hat{x}\times   \left({\bf I} \, - \, \frac{\eta_0\, k^2}{\pm 3c_0}c_r^{-3} {\bf P_0} \right)^{-1} \cdot {\bf P_0} \cdot H^{\mathring\mu_{r}}(z)\,dz.
\end{equation}
By subtracting \eqref{far-dis-00} from \eqref{LS-Far}, and using the fact that $\Omega =  \left( \underset{m=1}{\overset{\aleph}{\cup}}  \Omega_{m} \right) \cup \left( \Omega \setminus \left( \underset{m=1}{\overset{\aleph}{\cup}}  \Omega_{m} \right) \right)$, we get  
\begin{equation}\label{mid1}
	 	E^\infty_{\mathrm{eff}, +}(\hat{x}) \, - \, E^\infty(\hat{x}) \, = \pm \, \frac{i \, k }{4\, \pi} \, \xi\, \left[ \, F^{\infty}(\hat{x}) \, + \, G^{\infty}(\hat{x}) \, \right] + \, \mathcal{O}\left( a^{\frac{h}{3}} \right),
\end{equation}
where 
\begin{eqnarray*}
		F^\infty(\hat{x}) &:=& \sum_{m=1}^{\aleph}\int_{\Omega_m} e^{- i\, k\,\hat{x}\cdot z}\, \hat{x}\times \left({\bf I} \, - \,\frac{\eta_0 k^2}{\pm 3c_0}c_r^{-3} \, {\bf P_0} \right)^{-1} \cdot {\bf P_0} \cdot \left(H^{\mathring\mu_{r}}(z)-\frac{1}{|S_m|}\int_{S_m} H^{\mathring\mu_{r}}(y)\,dy\right)\,dz  \\
		&+& \sum_{m=1}^{\aleph}\int_{\Omega_m}\left( e^{- i\, k\,\hat{x}\cdot z} \hat{x}\times  \left({\bf I} - \frac{\eta_0 k^2}{\pm 3c_0}c_r^{-3}{\bf P_0} \right)^{-1} \cdot {\bf P_0} \cdot \frac{1}{|S_m|}\int_{S_m}H^{\mathring{\mu_r}}(y)dy-e^{-i\, k \, \hat{x}\cdot z_m}\hat{x}\times {\bf P_0} \cdot U_m\right)dz,
\end{eqnarray*}
and
\begin{equation}\label{DefGinf}
	G^\infty(\hat{x}) := \int_{\Omega\backslash\underset{m=1}{\overset{\aleph}{\cup}}\Omega_m}e^{- i\, k\,\hat{x}\cdot z}\, \hat{x}\times  \left({\bf I} \, - \, \frac{\eta_0 k^2}{\pm 3c_0}c_r^{-3}{\bf P_{0}} \right)^{-1} \cdot {\bf P_0} \cdot H^{\mathring\mu_{r}}(z)\,dz.
 \end{equation} 
Next, we estimate $\left\Vert F^{\infty} \right\Vert_{\mathbb{L}^{\infty}\left( \mathbb{S}^2 \right)}$ and $\left\Vert E^{\infty} \right\Vert_{\mathbb{L}^{\infty}\left( \mathbb{S}^2 \right)}$.
\begin{enumerate}
    \item[]
    \item Estimation of $\left\Vert F^{\infty} \right\Vert_{\mathbb{L}^{\infty}\left( \mathbb{S}^2 \right)}$. To accomplish this, we estimate separately the first and the second term on the R.H.S of $F^{\infty}(\cdot)$.
    \begin{enumerate}
        \item[]
        \item Estimation of $F_{1}^{\infty}$. 
        \begin{eqnarray*}
            F_{1}^{\infty}(\hat{x}) &:=& \sum_{m=1}^{\aleph}\int_{\Omega_m} e^{- i\, k\,\hat{x}\cdot z}\, \hat{x}\times \left({\bf I} \, - \,  \frac{\eta_0 k^2}{\pm 3c_0}c_r^{-3}{\bf P_{0}}\right)^{-1} \cdot {\bf P_0} \cdot \left(H^{\mathring\mu_{r}}(z)-\frac{1}{|S_m|}\int_{S_m} H^{\mathring\mu_{r}}(y)\,dy\right)\,dz \\
            \left\vert F_{1}^{\infty}(\hat{x}) \right\vert & \leq & \sum_{m=1}^{\aleph}\int_{\Omega_m}   \left\vert {\bf P_0} \right\vert \, \left\vert \left({\bf I} \, - \,  \frac{\eta_0 k^2}{\pm 3c_0}c_r^{-3}{\bf P_{0}}\right)^{-1} \right\vert \, \left\vert H^{\mathring\mu_{r}}(z)-\frac{1}{|S_m|}\int_{S_m} H^{\mathring\mu_{r}}(y)\,dy \right\vert \,dz \\
            & \overset{(\ref{AS0456})}{\lesssim} & \sum_{m=1}^{\aleph} \left\vert \Omega_m \right\vert  \, \left\vert {\bf P_0} \right\vert \, \left\vert \left({\bf I} \, - \,  \frac{\eta_0 k^2}{\pm 3c_0}c_r^{-3}{\bf P_{0}}\right)^{-1} \right\vert \, \, [H^{\mathring{\mu_{r}}}]_{C^{0, \alpha}(\overline{\Omega})} \, d^{\alpha}, 
        \end{eqnarray*}
        where we have used the fact that ${\bf P_0}$ is a constant tensor. In addition, since $\underset{m=1}{\overset{\aleph}{\sum}} \left\vert \Omega_m \right\vert \sim 1$, we obtain 
	\begin{equation}\label{F1}
		\left\Vert F_{1}^{\infty} \right\Vert_{\mathbb{L}^{\infty}(\mathbb{S}^2)}  \lesssim
		  \lvert {\bf P_0} \rvert \;  \left\vert \left({\bf I} \, - \,  \frac{\eta_0 k^2}{\pm 3c_0}c_r^{-3}{\bf P_{0}}\right)^{-1} \right\vert \, \, [H^{\mathring{\mu_{r}}}]_{C^{0, \alpha}(\overline{\Omega})} d^\alpha.
	\end{equation}
        \item[]
        \item Estimation of $F_{2}^{\infty}$. 
        \begin{equation*}
          F_{2}^{\infty}(\hat{x}) :=  \sum_{m=1}^{\aleph}\int_{\Omega_m}\left( e^{-i\, k\,\hat{x}\cdot z} \, \hat{x}\times \left({\bf I} \, - \, \frac{\eta_0 k^2}{\pm 3c_0}c_r^{-3}{\bf P_0} \right)^{-1} \cdot {\bf P_0} \cdot \frac{1}{|S_m|}\int_{S_m}H^{\mathring{\mu_r}}(y)\,dy-e^{- i k \hat{x}\cdot z_m}\hat{x}\times {\bf P_0} \cdot U_m\right) \,dz.
        \end{equation*}
        By using Taylor expansion near $z_{m}$, we obtain that
        \begin{equation*}
            e^{- i\, k\,\hat{x}\cdot z} = e^{- i\, k\,\hat{x}\cdot z_{m}} - i \, k \, \int_{0}^{1} (z-z_{m}) \cdot \hat{x} \, e^{- i \, k \, \left(z_{m} + t \, (z-z_{m}) \right)} \, dt, \quad z \in \Omega_{m}. 
        \end{equation*}
        Then, 
        \begin{eqnarray*}
		F_2^\infty\left( \hat{x} \right) &=& \sum_{m=1}^{\aleph}\int_{\Omega_m} e^{- i\, k\,\hat{x}\cdot z_m}\, \hat{x}\times {\bf P_0} \cdot \left[ \left({\bf I} \, - \, \frac{\eta_0 k^2}{\pm 3c_0}c_r^{-3}{\bf P_0} \right)^{-1} \cdot \frac{1}{|S_m|}\int_{S_m} H^{\mathring\mu_{r}}(y)\,dy-U_{m} \right] \, dz \\
		&-& \sum_{m=1}^{\aleph}\int_{\Omega_m}  i \, k \, \int_{0}^{1} (z-z_{m}) \cdot \hat{x} \, e^{- i \, k \, \left(z_{m} + t \, (z-z_{m}) \right)} \, dt \\ && \qquad \qquad \qquad \qquad \hat{x}\times {\bf P_0} \cdot \left({\bf I} \, - \, \frac{\eta_0 k^2}{\pm 3c_0}c_r^{-3}{\bf P_0} \right)^{-1} \cdot \frac{1}{|S_m|}\int_{S_m} H^{\mathring\mu_{r}}(y)\,dy\,dz.
	\end{eqnarray*}
        \item[]
    \end{enumerate}
    We evaluate $F^\infty_2(\hat{x})$ by separately estimating the first and the second term on the R.H.S of the previous equation.  
    \begin{enumerate}
        \item[]
        \item Estimation of $F_{2,1}^{\infty}$.
        \begin{eqnarray*}
            F_{2,1}^{\infty}\left( \hat{x} \right) &:=& \sum_{m=1}^{\aleph}\int_{\Omega_m} e^{- i\, k\,\hat{x}\cdot z_m}\, \hat{x}\times {\bf P_0} \cdot \left[ \left({\bf I} \, - \, \frac{\eta_0 k^2}{\pm 3c_0}c_r^{-3}{\bf P_0} \right)^{-1} \cdot \frac{1}{|S_m|}\int_{S_m} H^{\mathring\mu_{r}}(y)\,dy-U_{m} \right] \, dz\\
            \left\vert F_{2,1}^{\infty}\left( \hat{x} \right) \right\vert & \leq & \left\vert \Omega_{m} \right\vert  \, \left\vert {\bf P_0} \right\vert \,  \sum_{m=1}^{\aleph}  \left\vert \left({\bf I} \, - \,  \frac{\eta_0 k^2}{\pm 3c_0}c_r^{-3}{\bf P_0} \right)^{-1} \cdot \frac{1}{|S_m|}\int_{S_m} H^{\mathring\mu_{r}}(y)\,dy-U_{m} \right\vert \\
            & \leq & \left\vert \Omega_{m} \right\vert \, \left\vert {\bf P_0} \right\vert \, \aleph^{\frac{1}{2}} \left( \sum_{m=1}^{\aleph}  \left\vert \left({\bf I} \, - \, \frac{\eta_0 k^2}{\pm 3c_0}c_r^{-3}{\bf P_0} \right)^{-1} \cdot \frac{1}{|S_m|}\int_{S_m} H^{\mathring\mu_{r}}(y)\,dy-U_{m} \right\vert^{2} \, \right)^{\frac{1}{2}}.
        \end{eqnarray*}
        From $(\ref{exp-T})$, we can derive that
        \begin{equation*}
            \left( {\bf I} \, + \, \frac{\eta_0 k^2}{\pm 3c_0}c_r^{-3}{\bf T}^{\mathring{\mu_{r}}} \right) \; = \;  \left({\bf I} \, - \, \frac{\eta_0 k^2}{\pm 3c_0}c_r^{-3}{\bf P_0} \right)^{-1},
        \end{equation*}
        which will be used together with the estimation $(\ref{es-l2})$, and the fact that $\left\vert \Omega_{m} \right\vert = \mathcal{O}\left( d^{3} \right)$ and $\aleph = \mathcal{O}\left( d^{-3} \right)$, to deduce the estimation
        \begin{equation}\label{EstF21}
            \left\Vert F_{2,1}^{\infty} \right\Vert_{\mathbb{L}^{\infty}(\mathbb{S}^2)} \lesssim 	\frac{\eta_0 \, k^2  \, \lvert {\bf P_0} \rvert^{2} \, \left\vert \left( {\bf I} \, - \, \dfrac{\eta_{0} \, k^{2}}{\pm \, 3 \, c_{0} \, c_r^3} \, {\bf P_0} \right)^{-1} \right\vert}{\left(c_0 \, c_r^3 \, - \, k^2 \, \eta_0 \, \lvert {\bf P_0} \rvert \right)} \; \left([H^{\mathring{\mu_{r}}}]^{2}_{C^{0, \alpha}(\overline{\Omega})} \, d^{2 \, \alpha } \left\vert \log(d) \right\vert^{2} \, + \, \lVert H^{\mathring{\mu_r}}\rVert^{2}_{\mathbb L^\infty\left( \Omega \right)} \, d^{\frac{12}{7}}\right)^{\frac{1}{2}}.
        \end{equation}
        \item[]
        \item Estimation of $F_{2,2}^{\infty}$.
        \begin{eqnarray*}
           F_{2,2}^{\infty}\left( \hat{x} \right) & := & -\sum_{m=1}^{\aleph}\int_{\Omega_m}  i \, k \, \int_{0}^{1} (z-z_{m}) \cdot \hat{x} \, e^{- i \, k \, \left(z_{m} + t \, (z-z_{m}) \right)} \, dt \\ && \qquad \qquad \qquad \qquad \hat{x}\times {\bf P_0} \cdot \left({\bf I} \, - \, \frac{\eta_0 k^2}{\pm 3c_0} c_r^{-3}{\bf P_{0}} \right)^{-1} \cdot \frac{1}{|S_m|}\int_{S_m} H^{\mathring\mu_{r}}(y)\,dy\,dz \\
           \left\vert F_{2,2}^{\infty}\left( \hat{x} \right) \right\vert & \lesssim &  k\, \sum_{m=1}^{\aleph} \int_{ \Omega_{m}} \left\vert z - z_{m} \right\vert \, dz  \, \lvert {\bf P_0} \rvert \, \left\vert \left( {\bf I} \, - \, \frac{\eta_0 \, k^2}{\pm \, 3 \, c_0} \, c_r^{-3} \,  {\bf P_0} \right)^{-1} \right\vert \, \lVert H^{\mathring{\mu_{r}}}\rVert_{\mathbb L^\infty\left( S_{m} \right)}. 
        \end{eqnarray*}
        Then, 
\begin{equation}\label{EstF22}
                \left\Vert F_{2,2}^{\infty} \right\Vert_{\mathbb{L}^{\infty}(\mathbb{S}^2)} = \mathcal{O}\left( k\, d \, \lvert {\bf P_0} \rvert \,  \left\vert \left( {\bf I} \, - \, \frac{\eta_0 \, k^2}{\pm \, 3 \, c_0} \, c_r^{-3} \,  {\bf P_0} \right)^{-1} \right\vert \, \lVert H^{\mathring{\mu_{r}}}\rVert_{\mathbb L^\infty\left( \Omega \right)}  \right).
\end{equation}
        \item[] 
    Consequently, by gathering $(\ref{EstF21})$ and $(\ref{EstF22})$, we obtain that
    \begin{equation}\label{EstF2}
            \left\Vert F_{2}^{\infty} \right\Vert_{\mathbb{L}^{\infty}(\mathbb{S}^2)} \lesssim 	\frac{\eta_0\, k^2 \,  \lvert {\bf P_0}\rvert^{2} \,  \left\vert \left( {\bf I} \, - \, \dfrac{\eta_0 \, k^2}{\pm \, 3 \, c_0} \, c_r^{-3} \,  {\bf P_0} \right)^{-1} \right\vert }{\left(c_0 \, c_r^3 \, - \, k^2 \, \eta_0 \, \lvert {\bf P_0} \rvert \right)} \; \left([H^{\mathring{\mu_{r}}}]^{2}_{C^{0, \alpha}(\overline{\Omega})} \, d^{2 \, \alpha } \left\vert \log(d) \right\vert^{2} \, + \, \lVert H^{\mathring{\mu_r}}\rVert^{2}_{\mathbb L^\infty\left( \Omega \right)} \, d^{\frac{12}{7}}\right)^{\frac{1}{2}}.
    \end{equation}
      \end{enumerate}
      Now, by combining with $(\ref{F1})$ and $(\ref{EstF2})$, we obtain, 
          \begin{equation}\label{EstF}
            \left\Vert F^{\infty} \right\Vert_{\mathbb{L}^{\infty}(\mathbb{S}^2)} \lesssim 	\frac{\eta_0 \, k^2 \,  \lvert {\bf P_0}\rvert^{2} \,  \left\vert \left( {\bf I} \, - \, \dfrac{\eta_0 \, k^2}{\pm \, 3 \, c_0} \, c_r^{-3} \,  {\bf P_0} \right)^{-1} \right\vert }{\left(c_0 \, c_r^3 \, - \, k^2 \, \eta_0 \, \lvert {\bf P_0} \rvert \right)} \; \left([H^{\mathring{\mu_{r}}}]^{2}_{C^{0, \alpha}(\overline{\Omega})} \, d^{2 \, \alpha } \left\vert \log(d) \right\vert^{2} \, + \, \lVert H^{\mathring{\mu_r}}\rVert^{2}_{\mathbb L^\infty\left( \Omega \right)} \, d^{\frac{12}{7}}\right)^{\frac{1}{2}}.
    \end{equation}
    \item[]
    \item Estimation of $\left\Vert G^{\infty} \right\Vert_{\mathbb{L}^{\infty}\left( S \right)}$. By taking the modulus on both sides of the equation $(\ref{DefGinf})$, we get: 
\begin{eqnarray*}
    \left\vert G^\infty(\hat{x}) \right\vert & \leq & \left\vert  \int_{\Omega\backslash\underset{m=1}{\overset{\aleph}{\cup}}\Omega_m}e^{- i\, k\,\hat{x}\cdot z}\, \hat{x}\times \left({\bf I} \, - \, \frac{\eta_0 k^2}{\pm 3c_0}c_r^{-3}{\bf P_0} \right)^{-1} \cdot {\bf P_0} \cdot H^{\mathring\mu_{r}}(z)\,dz \right\vert \\
    & \leq & \left\vert   \Omega\backslash\underset{m=1}{\overset{\aleph}{\cup}}\Omega_m \right\vert \,  \left\vert {\bf P_0} \right\vert \, \left\vert \left( {\bf I} \, - \, \frac{ \eta_0 \, k^2}{\pm \, 3 \, c_0} \, c_r^{-3} \, 
     {\bf P_0}  \right)^{-1}  \right\vert \,\left\Vert H^{\mathring\mu_{r}} \right\Vert_{\mathbb{L}^{\infty}(\Omega)}.
 \end{eqnarray*}

Due to \textbf{Remark \ref{rem-vol-d}}, we know that $\left\vert \Omega \backslash \underset{m=1}{\overset{\aleph}{\cup}}  \Omega_{m} \right\vert = \mathcal{O}\left( d \right)$, which implies,
	\begin{equation}\label{G}
		\left\Vert G^\infty \right\Vert_{\mathbb{L}^{\infty}(\Omega)} \lesssim d \,  \lvert {\bf P_0} \rvert \, \left\vert \left( {\bf I} \, - \, \frac{ \eta_0 \, k^2}{\pm \, 3 \, c_0} \, c_r^{-3} \, 
     {\bf P_0}  \right)^{-1}  \right\vert \, \lVert H^{\mathring{\mu_{r}}}\rVert_{\mathbb L^\infty\left( \Omega \right)}.
	\end{equation}
 \end{enumerate}
Now, combining with the estimation of $F^{\infty}(\cdot)$ given by $(\ref{EstF})$ and the estimation of $G^{\infty}(\cdot)$ presented by $(\ref{G})$, under the relation $(\ref{distribute})$, we derive from $(\ref{mid1})$ the coming estimation
	 	\begin{eqnarray}\label{far-end}
    E^{\infty}_{\mathrm{eff}, +}(\hat{x}) - E^{\infty}(\hat{x})  &=& \nonumber \Oh\bigg( \frac{k^5 \, \eta_{0}^2 \, c_{r}^{-3} \, \lvert {\bf P_0} \rvert^{2} \, \left\vert \left( {\bf I} \, - \, \dfrac{ \eta_0 \, k^2}{\pm \, 3 \, c_0} \, c_r^{-3} \, 
     {\bf P_0}  \right)^{-1} \right\vert}{c_{0} \, \left( c_0 \, c_r^3 \, - \, k^2 \, \eta_0 \, \lvert{\bf P_0} \rvert \right)} \\ \nonumber && \left([H^{\mathring{\mu_{r}}}]^{2}_{C^{0, \alpha}(\overline\Omega)} \, c_{r}^{2 \, \alpha} \,  a^{2 \, \alpha \, (1-\frac{h}{3})} \, \left\vert \log(a) \right\vert^{2} \, + \, \lVert H^{\mathring{\mu_{r}}}\rVert^{2}_{\mathbb L^{\infty}\left( \Omega \right)} \, c_r^{\frac{12}{7}} \, a^{\frac{12}{7}(1-\frac{h}{3})}\right)^{\frac{1}{2}} \bigg) 
	 \\ &+& \mathcal{O}\left( a^{\frac{h}{3}} \right),
	 \end{eqnarray}
	 where $\alpha\in (0,1)$. The remaining part of the proof consists in substituting the explicit form of the constant tensor ${\bf T}^{\mathring{\mu_r}}$, given by $(\ref{exp-T})$, into $(\ref{effective-permi})$, then  derive the expressions of the far-field error given by $(\ref{err-es})$  and the effective coefficient as $(\ref{new-coeff})$.
  
  \medskip

Now we prove the second part of \textbf{Theorem \ref{thm-main-posi}}. 
 Using $(\ref{def-xi})$ and substituting $(\ref{exp-T-neg})$ into $(\ref{ls-eff})$, the far-field becomes
\begin{equation}\notag
	E_{\mathrm{eff}, -}^\infty(\hat{x})= \mp \frac{3 \, i\, \eta_0\, k^3}{c_0\, \pi^4}\, c_r^{-3}\, \left(1 \, \mp \, \frac{4\eta_0 k^2}{\pi^3 \, c_0 \, c_r^{3}}\right)^{-1}\int_{\Omega}e^{- i k \hat{x}\cdot z}\hat{x}\times H^{\mathring{\mu_r}}(z)\,dz.
\end{equation} 
Then, similar to the expression of the error for $E^\infty_{\mathrm{eff}, +}(\hat{x})-E^\infty(\hat{x})$, given by \eqref{far-end}, with the help of the explicit representations of the constant tensors ${\bf P_0}$, in $(\ref{exp-p0})$, and ${\bf T}^{\mathring{\mu_r}}$, in $(\ref{exp-T-neg})$, we deduce by direct computations that
\begin{eqnarray*}
\nonumber
  \left\vert E^\infty_{\mathrm{eff}, -}(\hat{x})-E^\infty(\hat{x})\right\vert & \lesssim & \frac{k^5 \, \eta_0^2\, c_r^{-3}}{c_0\left( c_0 \, c_r^3 \, \pi^3 \, - \, 12 \, k^2 \, \eta_0 \right)}\left|\left(1 \mp  \frac{4 \,\eta_0 \, k^2}{\pi^3 \, c_0} \, c_r^{-3} \, \right)^{-1}\right| \\
	& \cdot &\left([H^{\mathring{\mu_r}}]^{2}_{C^{0, \alpha}(\overline\Omega)} \, d^{2 \, \alpha} \, \left\vert \log(d) \right\vert^{2} \, + \, \lVert H^{\mathring{\mu_r}}\rVert^{2}_{\mathbb L^{\infty}(\Omega)} \,  d^{\frac{12}{7}} \right)^{\frac{1}{2}} \, + \, \mathcal{O}\left( a^{\frac{h}{3}} \right),
\end{eqnarray*} 
where $\alpha\in(0, 1)$. Then \eqref{es-error-neg} can be obtained directly by using \eqref{distribute}. The negative definite effective permeability $\mathring{\mu_r}$ is given by the diagonal matrix \eqref{mu-r0-diag}, while $\mathring{\epsilon_{r}}$ keeps unchanged, which proves \eqref{new-coeff-neg}. This ends the proof of \textbf{Theorem \ref{thm-main-posi}}.

\section{The effective medium with $\mathring{\mu_r}$-Plasmonic.}\label{sec4:negative}

In this section, we investigate the electromagnetic scattering for the effective medium with negative permeability $\mathring{\mu_r}$. Precisely, we first present a systematical discussion on the behaviors of the effective permeability $\mathring{\mu_r}$ appearing in \eqref{model-equi}, associated with different values of the parameter $c_0$ given by \eqref{condition-on-k}, particularly for the case with $\mathring{\mu_r}$ being negative definite. Then, similar to the analyses in \textbf{Section \ref{sec3:positive}} for positive permeability $\mathring{\mu_r}$, we study the H\"{o}lder regularity of the solution $(E^{\mathring{\epsilon_{r}}}, H^{\mathring{\mu_r}})$ to \eqref{model-equi}. Indeed, we show that with $\partial\Omega$ is $C^2-$ regular, we obtain the $C^{0, \alpha}$-regularity to the solution for negative permeability $\mathring{\mu_r}$, see \textbf{ Proposition \ref{prop-regu-neg}}. If in addition $\partial\Omega$ is $C^{4, \alpha}-$ regular, then we derive $C^{0, \alpha}$-norm estimates of these solutions, see \textbf{ Lemma \ref{AddedLemma}}. Finally, we give the proof of \textbf{Theorem \ref{thm-main-posi}}, point (2), by utilizing the a-priori estimates for the corresponding L.S.E in a similar way as we did in \textbf{Proposition \ref{prop-es-LS}}.

\medskip

\subsection{Negative definite effective permeability $\boldsymbol{\mathring{\mu_r}}$.}

From the L.S.E \eqref{LS-1}, we deduce that the effective permeability is of the form
\begin{equation}\label{mu-r0}
	\mathring{\mu_r} = {\bf I}+\frac{\eta_0 k^2}{\pm c_0}c_r^{-3}{\bf T}^{\mathring{\mu_r}}
	\overset{(\ref{exp-T})}{=} {\bf I}+\frac{\eta_0 k^2}{\pm c_0}c_r^{-3}\left({\bf I}-\frac{\eta_0 k^2}{\pm 3c_0}c_r^{-3} {\bf P_0}\right)^{-1}{\bf P_0},
\end{equation}
where ${\bf P_0}$ is given by $(\ref{defP0})$. In the subsequent discussions, regarding negative contrast permeability $\mathring{\mu_r}$, we assume that the domain $B$, in $(\ref{defP0})$, is the unit ball $B(0,1)$. In this case, from \cite[Section 5.2]{ALZ}, we know that  
\begin{equation}\label{Eq0849}
    \int_{B(0,1)} \phi_{1,j}(x) \, dx \, = \, \frac{4}{\pi} \, \nabla \left( \left\vert x \right\vert \, Y_{1}^{j}\left( \hat{x} \right) \right), 
\end{equation}
where $\phi_{1,j}(\cdot)$ is the function defined in $(\ref{Intphin0})$, with $n_{0} = 1$, and $Y_{1}^{j}\left( \hat{x} \right)$, with $j=-1,0,1$, are the spherical harmonics given explicitly by
\begin{eqnarray*}
   Y_{1}^{-1}\left( \hat{x} \right) & \, = \, & \frac{1}{2} \, \sqrt{\frac{3}{2 \, \pi}} \, \left( \hat{x}_{1} \, - \, i \, \hat{x}_{2} \right) \\
   Y_{1}^{0}\left( \hat{x} \right) & \, = \, & \frac{1}{2} \, \sqrt{\frac{3}{\pi}} \,  \hat{x}_{3}  \\
   Y_{1}^{1}\left( \hat{x} \right) & \, = \, & - \, \frac{1}{2} \, \sqrt{\frac{3}{2 \, \pi}} \, \left( \hat{x}_{1} \, + \, i \, \hat{x}_{2} \right).
\end{eqnarray*}
Using $(\ref{Eq0849})$, we obtain\footnote{Consequently, 
\begin{equation}\label{PCB}
    \int_{B(0,1)} \phi_{1}(x) \, dx \, = \, \frac{2}{\pi} \, \sqrt{\frac{3}{\pi}} \, \begin{pmatrix}
0 \\
- i \sqrt{2}  \\
1
\end{pmatrix}.
\end{equation}
}
\begin{eqnarray*}
    \int_{B(0,1)} \phi_{1,-1}(x) \, dx \, &=& \, \frac{1}{\pi} \, \sqrt{\frac{6}{\pi}} \, \begin{pmatrix}
1 \\
-i  \\
0
\end{pmatrix} \\
  \int_{B(0,1)} \phi_{1,0}(x) \, dx \, &=& \, \frac{2}{\pi} \, \sqrt{\frac{3}{\pi}} \, \begin{pmatrix}
0 \\
0  \\
1
\end{pmatrix} \\
  \int_{B(0,1)} \phi_{1,1}(x) \, dx \, &=& \, - \, \frac{1}{\pi} \, \sqrt{\frac{6}{\pi}} \, \begin{pmatrix}
1 \\
i  \\
0
\end{pmatrix}. 
\end{eqnarray*}
Now, by using the expression of ${\bf P_0}$, given in $(\ref{defP0})$, we obtain
\begin{equation}\label{exp-p0}
{\bf P_0} \; = \; \frac{12}{\pi^3} \; {\bf I}.
\end{equation}
Recall the definition of $\xi$ given by $(\ref{def-xi})$. By substituting $(\ref{exp-p0})$ and $(\ref{def-xi})$ into $(\ref{mu-r0})$, we have
\begin{equation}\label{mu-r0-diag}
	\mathring{\mu_r}={\bf I}\pm \xi \left(1\mp \xi \frac{4}{\pi^3} \right)^{-1}{\bf I}\cdot \frac{12}{\pi^3}{\bf I} \; = \; \frac{\left(\pi^{3}  \, \pm \, 8 \, \xi \right)}{\left(\pi^{3} \, \mp \, 4 \, \xi \right)} \; {\bf I}.
\end{equation}
And, regarding the permeability contrast, we get
\begin{equation*}
	\mathring{\mu_r} \, - \, {\bf I} \, = \,  \frac{ \pm \, 12 \, \xi}{\left(\pi^{3} \, \mp \, 4 \, \xi \right)} \; {\bf I}.
\end{equation*}
The above formula, as a function of the chosen sign, suggests two cases 
\begin{enumerate}
    \item[]
    \item By taking the lower sign, we obtain
    \begin{equation*}
	\mathring{\mu_r} \, - \, {\bf I} \, = \,  \frac{ - \, 12 \, \xi}{\left(\pi^{3} \, + \, 4 \, \xi \right)} \; {\bf I} \, < \, 0. 
\end{equation*}
 
    \item[] 
    \item By taking the upper sign, we obtain
            \begin{equation*}
	\mathring{\mu_r} \, - \, {\bf I} = \left(  \frac{12 \, \xi}{\pi^3 - 4 \, \xi} \right) \, {\bf I} \, \begin{cases}
		  < 0, & \text{\;\; if $\xi > \dfrac{\pi^3}{4}$ \,\, }\\
                 & \\
            > 0, & \text{\;\; if $ \xi < \dfrac{\pi^3}{4}$ \,\, }
		        \end{cases}.
           \end{equation*}
\end{enumerate}
The above formula, as a functions of the chosen sign, suggests two cases\footnote{$\mathring{\mu_r}$ is positive (resp. negative) definite if and only if $\langle x^{T}; \mathring{\mu_r} \cdot x \rangle \, > \, 0$ (resp. $\langle x^{T}; \mathring{\mu_r} \cdot x \rangle \, < \, 0$), for every $x \in \mathbb{R}^{3}$.}. 
\begin{enumerate}
    \item[] 
    \item By taking the lower sign, we obtain 
    \begin{equation}\label{ls}
	\mathring{\mu_r} = \left(  \frac{\pi^3 - 8 \, \xi}{\pi^3 + 4 \, \xi} \right) \, {\bf I} \, \begin{cases}
		  < 0, & \text{\;\; if $\xi > \dfrac{\pi^3}{8}$}\\
                 & \\
            > 0, & \text{\;\; if $0 < \xi < \dfrac{\pi^3}{8}$}
		        \end{cases}.
\end{equation}
\item[] 
    \item By taking the upper sign, we obtain:
        \begin{equation}\label{us}
	\mathring{\mu_r} = \left(  \frac{\pi^3 + 8 \, \xi}{\pi^3 - 4 \, \xi} \right) \, {\bf I} \, \begin{cases}
		  < 0, & \text{\;\; if $\xi > \dfrac{\pi^3}{4}$}\\
                 & \\
            > 0, & \text{\;\; if $0 < \xi < \dfrac{\pi^3}{4}$}
		        \end{cases}.
\end{equation}
\end{enumerate}
Therefore, for $\xi:=\dfrac{\eta_0 k^2}{c_0}c_r^{-3}>\dfrac{\pi^3}{4}$, which means
\begin{equation}\notag
	c_{0} < \dfrac{4 \, \eta_{0} \, k^2}{\pi^3 \, c_{r}^3},
\end{equation}
 the tensor $\mathring{\mu_r}$ is always negative definite. Particularly, if $c_0\ll 1$, then in \eqref{mu-r0}, $\dfrac{\eta_0 k^2}{\pm c_0}c_r^{-3}{\bf P_0}$ dominates and there holds
\begin{equation*}
	\left({\bf I}-\frac{\eta_0 k^2}{\pm 3c_0}c_r^{-3}{\bf P_0}\right)^{-1} =\left(-\frac{\eta_0 k^2}{\pm 3c_0}c_r^{-3}{\bf P_0}\right)^{-1}\left({\bf I}+\left(-\frac{\eta_0 k^2}{\pm 3c_0}c_r^{-3}{\bf P_0}\right)^{-1}\right)^{-1}
	= \left(-\frac{\eta_0 k^2}{\pm 3c_0}c_r^{-3}{\bf P_0}\right)^{-1}\left({\bf I}+o(1)\right). 
\end{equation*}
Thus
\begin{equation}\notag
	\mathring{\mu_r}={\bf I}+\frac{\eta_0 k^2}{\pm c_0}c_r^{-3}\left(-\frac{\eta_0 k^2}{\pm 3c_0}c_r^{-3}\right)^{-1}{\bf P_0}^{-1}{\bf P_0}\left({\bf I}+o(1)\right)=-2{\bf I}+o(1),
\end{equation}
which indicates that in this case, there will be twice of the enhancement of the negative definite effective permeability. In addition, by the explicit representation of ${\bf P_0}$, given by $(\ref{exp-p0})$, we can also derive from the definition of ${\bf T}^{\mathring{\mu_r}}$, see for instance $(\ref{exp-T})$, that
	\begin{equation}\label{exp-T-neg}
	{\bf T}^{\mathring{\mu_r}} \; = \; \left(1 \, \mp \, \frac{4}{\pi^3} \, \xi \right)^{-1} \, \frac{12}{\pi^3} \, {\bf I},
	\end{equation}
with $\xi$ as given by $(\ref{def-xi})$.

\subsection{$\boldsymbol{C^{0,\alpha}}$-regularity of the solution $\boldsymbol{H^{\mathring{\mu_r}}}$ with negative definite $\boldsymbol{\mathring{\mu_r}}$.}

In order to generate the effective medium with negative definite permeability in \textbf{Theorem \ref{thm-main-posi}}, second point, the $C^{0, \alpha}$-regularity of the solution $H^{\mathring{\mu_r}}$ is needed. However, the method utilized in \textbf{Proposition \ref{prop-regularity}} does not work for $\mathring{\mu_r}$ being negative definite in the L.S.E \eqref{LS-1}. As explained in \textbf{Section \ref{Background and motivation}}, for $\mathring{\mu_r}$ positive definite, we rewrite the operator appearing in the LSE as $Id +A$ where $Id$ is the identity operator and $A$ is shown to be dominated by $Id$ under the conditions in \textbf{Proposition \ref{prop-regularity}}. In short, in this case we have no plasmonic eigenvalues. However, for $\mathring{\mu_r}$ negative definite, such plasmon eigenvalues occur, therefore, we rewrite the related operator as $Inv+A$ where $Inv$ includes the Magnetization operator $\nabla {\bf M}$ (which creates the plasmon eigenvalues).  Therefore, in this subsection, we prove the $C^{0, \alpha}$-regularity of the solution, in this case, by properly analyzing the operator $Inv$. We state these results in the following proposition.

\begin{proposition}\label{prop-regu-neg}
	Let $\Omega$ be a bounded domain which is $C^{2}$- smooth. Regarding the electromagnetic scattering problem \eqref{model-equi}, with $\mathring{\mu_r}$ being negative definite, we have the following regularity result
	\begin{equation}\notag
		(E^{\mathring{\epsilon_{r}}}, H^{\mathring{\mu_r}})\in C^{0, \alpha}(\overline{\Omega})\times C^{0, \alpha}(\overline{\Omega}),
	\end{equation}
 under the condition 
 \begin{equation*}
     2 \, \delta \, \pi^{3} \, < \, \xi \, < \, \frac{\pi^{4}}{12 \, \left\vert \Omega \right\vert \, c(k)},
 \end{equation*}
where $\xi$ is the coefficient given by $(\ref{def-xi})$,  $c(\cdot)$ is the frequency function defined by
\begin{equation}\label{def-ck}
   c(k) : =  k^2 \,  +  \frac{k^3}{6}  \,   + \frac{k^2}{4  \, \diam(\Omega)} \, \sum_{n \geq 2}\frac{\left( k \, \diam(\Omega) \right)^{n}}{n !} + \frac{k^3}{16} \, \sum_{n \geq 1}\frac{\left( k \, \diam(\Omega) \right)^{n}}{n !},
\end{equation}
 and $\delta$ is the parameter given by
 \begin{equation*}
     \delta := 1 + \left[ \frac{1}{12 \; \underset{n \in \mathbb{N}}{\min} \;  \left\vert \lambda_{n}^{(3)}\left( \Omega \right) - \dfrac{1}{3}\right\vert }  \right],
 \end{equation*}
 with the eigenvalues $\{ \lambda_{n}^{(3)}\left( \Omega \right) \}_{n \in \mathbb{N}}$ are such that $\lambda_{n}^{(3)}\left( \Omega \right) \, > \, \dfrac{1}{3}$. 
\end{proposition}

We postpone the detailed proof of \textbf{Proposition \ref{prop-regu-neg}} to \textbf{Section \ref{sec5:proof}}.

\subsection{A-priori estimates with negative definite  $\boldsymbol{\mathring{\mu_r}}$.}

The a-priori estimates corresponding to the L.S.E \eqref{LS-1} for negative definite $\mathring{\mu_r}$ possess the similar expressions to \textbf{Proposition \ref{prop-es-LS}}. Indeed, with the help of the explicit representations of the constant tensors ${\bf P_0}$ and ${\bf T}^{\mathring{\mu_r}}$, which are given by \eqref{exp-p0} and \eqref{exp-T-neg}, respectively, we can get the corresponding a-priori estimates as follows.

\begin{proposition}\label{prop-ls-neg}
	Let ${\bf T}^{\mathring{\mu_r}}$ be the constant tensor defined by $(\ref{exp-T})$. For negative definite permeability $\mathring{\mu_r}$, under \textbf{Assumptions}, we have the following algebraic system associated with the L.S.E $(\ref{LS-1})$,   
	\begin{equation*}
       V_{m} \; \mp \; c_r^3 a^{3-h} \, \frac{12 \, \xi}{\pi^{3}}
		 \, \sum_{j=1 \atop j \neq m}^\aleph\Upsilon_k(z_m, z_j)\cdot V_{j}  = i \, k \, H^{Inc}(z_m)+\Oh(k d)+\widetilde{\mathrm{Error}}_m,
	\end{equation*}
	where $\xi$ is given by $(\ref{def-xi})$ and 
 \begin{equation*}
   V_{m} \; := \;  \frac{\pi^3}{\left( \pi^3 \, \mp \, 4 \, \xi \right)}  \, \frac{1}{|S_m|} \, \int_{S_m}H^{\mathring{\mu_r}}(x)\,dx, \quad \text{for} \quad 1 \leq m \leq \aleph. 
 \end{equation*}
 Moreover, there holds the following $\ell_2$-estimates, 
	\begin{equation*}
 \sum_{m=1}^\aleph\left|\widetilde{\mathrm{Error}}_m\right|^2 \, \lesssim \frac{\xi^{2}}{\left\vert \pi^{3} \, \mp \, 4 \, \xi \right\vert^{2}} \,  \, \left([H^{\mathring{\mu_r}}]^{2}_{C^{0, \alpha}(\overline{\Omega})} \, d^{2 \, \alpha - 3} \, \left\vert \log d \right\vert^{2} +\lVert H^{\mathring{\mu_r}}\rVert^{2}_{\mathbb L^{\infty}(\Omega)} d^{-\frac{9}{7}}\right),
	\end{equation*}
and
	\begin{equation*}
	 \sum_{m=1}^{\aleph} \left| V_{m} \; - \; U_m \right|^{2} 
	\lesssim \frac{\left\vert \xi \right\vert^{2}}{\left\vert \pi^{3} \, - 12 \, \xi \right\vert^{2} \; \left\vert \pi^{3} \, \mp \, 4 \,  \xi \right\vert^{2}} \; \left([H^{\mathring{\mu_r}}]^{2}_{C^{0, \alpha}(\overline{\Omega})}d^{2 \, \alpha \, - \, 3} \, 
 \left\vert \log d \right\vert^{2} \, + \, \lVert H^{\mathring{\mu_r}}\rVert^{2}_{\mathbb L^{\infty}(\Omega)} d^{-\frac{9}{7}}\right).
	\end{equation*}
\end{proposition}

The proof of \textbf{Proposition \ref{prop-ls-neg}} is very similar to the proof of \textbf{Proposition \ref{prop-es-LS}}. Indeed, from the proof of \textbf{Proposition \ref{prop-es-LS}}  presented in \textbf{Section \ref{sec5:proof}}, we notice that the main methods utilized therein are independent of the precise form of the effective permeability $\mathring{\mu_r}$. Therefore, with the help of \textbf{Proposition \ref{prop-regu-neg}}, which shows the $C^{0, \alpha}\left( \overline{\Omega} \right)$-regularity of the solution $H^{\mathring{\mu_r}}$, we obtain the a-priori estimate results directly by substituting the explicit representations of the constant tensors ${\bf P_0}$, given by \eqref{exp-p0}, and ${\bf T}^{\mathring{\mu_r}}$, given by \eqref{exp-T-neg}. To avoid being redundant, we omit the detailed proof here.

\medskip

\subsection{Proof of Theorem \ref{coro-plas-resonance}.}\label{ProofTHM18}

We divide the proof into the following few parts 
to demonstrate the estimation of the far field collected by exciting a medium having a negative permeability. The initial step is dedicated to demonstrating that ${\bf I} \, - \, \dfrac{12 \, \xi}{\pi^{3}+4 \xi} \, \nabla {\bf M}$ is the dominant operator in the L.S.E given by $(\ref{LS-1})$.
The second step involves computing low-frequency resonances that are related to the L.S.E given by $(\ref{LS-1})$. We will prove that they are small and related to the dominant operator given by ${\bf I} \, - \, \dfrac{12 \, \xi}{\pi^{3}+4 \xi} \, \nabla {\bf M}$. The field generated at the previously computed low-frequency plasmonic resonances is estimated in the third part. In the fourth part, we finish estimating the far field that was generated. We conclude the fourth part by mentioning a valuable remark regarding the effective permeability tensor expression and the far field expression which depends on the dominated operator's size in $(\ref{LS-1})$. In the last part, we discuss a particular case when the domain $\Omega$ is a ball and the excitation of the first eigenvalue related to the Magnetization operator defined on a ball.
\newline \smallskip \newline 
We start by recalling the L.S.E given by $(\ref{LS-1})$, 
\begin{equation}\label{LS-vec}
	\left({\bf I}\pm \xi \nabla{\bf M}{\bf T}^{\mathring{\mu_r}}-{\bf K}\right)H^{\mathring{\mu_r}} \, = \, i \, k \, H^{Inc}(x, \theta),
\end{equation}
where the parameter $\xi$ is given by $(\ref{def-xi})$ and the operator ${\bf K}$ is defined by
\begin{eqnarray}\label{TD09}
\nonumber
	{\bf K}(F) \, &:=&  \, \pm \, \xi \, \left[ - \nabla{\bf M^k} \, + \, \nabla{\bf M} \, + \, k^{2} \, {\bf N^k} \right] \left( {\bf T}^{\mathring{\mu_{r}}} \cdot F \right) \\
 &\overset{(\ref{exp-T-neg})}{=}& \, \pm  \, \frac{12 \, \xi }{\left( \pi^3 \, \mp \, 4 \, \xi \right)} \, \left[ - \nabla{\bf M^k} \, + \, \nabla{\bf M} \, + \, k^{2} \, {\bf N^k} \right] \left( F \right).
\end{eqnarray}
To investigate $(\ref{LS-vec})$, we first show that ${\bf I}\pm \xi \nabla{\bf M}{\bf T}^{\mathring{\mu_r}}$ dominates the operator ${\bf K}$. Thanks to the Helmholtz decomposition of $\mathbb{L}^{2}(\Omega)$, given by $(\ref{hel-decomp})$, we have
\begin{equation*}
\left({\bf I}\pm \xi \nabla{\bf M}{\bf T}^{\mathring{\mu_r}}\right)(H)=\overset{1}{\mathbb{P}}\left(\left({\bf I}\pm \xi \nabla{\bf M}{\bf T}^{\mathring{\mu_r}}\right)(H)\right)+\overset{2}{\mathbb{P}}\left(\left({\bf I}\pm \xi \nabla{\bf M}{\bf T}^{\mathring{\mu_r}}\right)(H)\right)+\overset{3}{\mathbb{P}}\left(\left({\bf I}\pm \xi \nabla{\bf M}{\bf T}^{\mathring{\mu_r}}\right)(H)\right).
\end{equation*}
Next, we derive sufficient conditions such that 
\begin{equation}\label{Scdt}
 \left\Vert \overset{j}{\mathbb{P}}\left( {\bf I} \pm \xi \nabla{\bf M}{\bf T}^{\mathring{\mu_r}} \right) \right\Vert_{\mathcal{L}(\mathbb{L}^{2}(\Omega);\mathbb{L}^{2}(\Omega))} \, \gg \, \left\Vert {\bf K} \right\Vert_{\mathcal{L}(\mathbb{L}^{2}(\Omega);\mathbb{L}^{2}(\Omega))}, \quad \text{for} \quad j=1,2,3.   
\end{equation}

To do this, we first recall the Helmholtz decomposition $\mathbb{L}^2(\Omega)=\mathbb{H}_0(\div=0)\overset{\perp}{\otimes}\mathbb{H}_0(Curl=0)\overset{\perp}{\otimes}\nabla \mathcal{H}armonic$ to which we correspond a basis given by three subfamilies $(e^1_n)_{n\in \mathbb{N}}, (e^2_n)_{n\in \mathbb{N}}$ and $(e^3_n)_{n\in \mathbb{N}}$ where the first two sequences are basis of the first two subspaces generated by the projection, on them, of the vector Newtonian operator while the third sequence is a basis of the third space generated by the Magnetization operator $\nabla {\bf M}$, see \cite{CGS}. We use the  fact that 
\begin{equation}\label{M-prop}
	\nabla {\bf M}(e_n^{(1)})=0, \nabla {\bf M}(e_n^{(2)})=e_n^{(2)} \mbox{ and } \nabla{\bf M}(e_n^{(3)})=\lambda_n^{(3)}(\Omega) \; e_n^{(3)},\; \forall \, n,
\end{equation}
to derive the following estimation
\begin{align}\label{P1-h}
	\left\lVert \overset{1}{\mathbb{P}}\left(\left({\bf I}\pm \xi \nabla{\bf M}{\bf T}^{\mathring{\mu_r}}\right) \cdot H \right)\right\rVert^2_{\mathbb L^2(\Omega)}&=\sum_n\left|\langle \left({\bf I}\pm \xi \nabla{\bf M}{\bf T}^{\mathring{\mu_r}}\right) \cdot H, e_n^{(1)}\rangle_{\mathbb{L}^{2}(\Omega)}\right|^2\notag\\
	&=\sum_n \left|\langle H, e_n^{(1)}\rangle_{\mathbb{L}^{2}(\Omega)} \, \pm \, \langle \xi \nabla{\bf M}{\bf T}^{\mathring{\mu_r}}(H), e_n^{(1)}\rangle\right|^2=\sum_n \left| \langle H, e_n^{(1)}\rangle_{\mathbb{L}^{2}(\Omega)} \right|^2,
\end{align}
invoking the fact that $\nabla {\bf M}$ is self-adjoint. Again by $(\ref{M-prop})$, we have
\begin{equation}\notag
	\langle \xi \nabla {\bf M}{\bf T}^{\mathring{\mu_r}}, e_n^{(2)}\rangle_{\mathbb{L}^{2}(\Omega)} \, = \, \langle \xi {\bf T}^{\mathring{\mu_r}}, \nabla{\bf M}(e_n^{(2)})\rangle_{\mathbb{L}^{2}(\Omega)} \, = \, \langle \xi{\bf T}^{\mathring{\mu_r}}, e_n^{(2)}\rangle_{\mathbb{L}^{2}(\Omega)},
\end{equation}
and then, 
\begin{equation*}
	\left\lVert \overset{2}{\mathbb{P}}\left(\left({\bf I}\pm \xi \nabla{\bf M} {\bf T}^{\mathring{\mu_r}}\right)(H)\right)\right\rVert_{\mathbb L^2(\Omega)}^2=\sum_n\left| \langle \left({\bf I}\pm \xi {\bf T}^{\mathring{\mu_r}}\right)(H), e_n^{(2)}\rangle_{\mathbb{L}^{2}(\Omega)} \right|^2\overset{\eqref{mu-r0}}{=}\sum_n\left| \langle \mathring{\mu_r} \cdot H, e_n^{(2)}\rangle_{\mathbb{L}^{2}(\Omega)} \right|^2,
\end{equation*}
with 
\begin{equation}\label{Tmur}
  \mathring{\mu_r} := {\bf I} \pm \xi {\bf T}^{\mathring{\mu_r}} \, \overset{(\ref{mu-r0-diag})}{=} \, \frac{\left(\pi^{3} \pm 8 \, \xi \right)}{\left(\pi^{3} \mp 4 \, \xi \right)} \, {\bf I},
\end{equation}
which gives us a representation of $\mathring{\mu_r}$ in the form proportional to the identity. Thus we have
\begin{equation}\label{P2-hm}
	\left\lVert \overset{2}{\mathbb{P}}\left(\left({\bf I}\pm \xi \nabla{\bf M} {\bf T}^{\mathring{\mu_r}}\right)(H)\right)\right\rVert_{\mathbb L^2(\Omega)}^2  = \left\vert \mathring{\mu_r} \right\vert^{2} \, \sum_n\left| \langle  H, e_n^{(2)}\rangle_{\mathbb{L}^{2}(\Omega)} \right|^2.
\end{equation}
Similarly, by using the relation that
\begin{equation}\notag
	\left\langle \xi \nabla {\bf M}{\bf T}^{\mathring{\mu_r}}(H), e_n^{(3)}\right\rangle_{\mathbb{L}^{2}(\Omega)} = \left\langle \xi {\bf T}^{\mathring{\mu_r}}(H), \nabla {\bf M}(e_n^{(3)})\right\rangle_{\mathbb{L}^{2}(\Omega)} \; = \; \lambda_n^{(3)}(\Omega) \; \left\langle \xi {\bf T}^{\mathring{\mu_r}}(H), e_n^{(3)}\right\rangle_{\mathbb{L}^{2}(\Omega)},
\end{equation}
we obtain
\begin{equation*}
	\left\lVert \overset{3}{\mathbb{P}}\left(\left({\bf I}\pm \xi \nabla{\bf M} {\bf T}^{\mathring{\mu_r}}\right) \cdot H \right)\right\rVert_{\mathbb L^2(\Omega)}^2 = \sum_n \left| \left\langle \left({\bf I}\pm \lambda_n^{(3)}(\Omega) \,  \xi \, {\bf T}^{\mathring{\mu_r}}\right) \cdot H, e_n^{(3)}\right\rangle_{\mathbb{L}^{2}(\Omega)} \right|^2 = \sum_n\left| \left\langle A_{\mathring{\mu_r}, n} \cdot H, e_n^{(3)}\right\rangle_{\mathbb{L}^{2}(\Omega)} \right|^2,
\end{equation*}
with 
\begin{equation}\label{A-mu}
	A_{\mathring{\mu_r}, n} \, := \, {\bf I} \, \pm \lambda_n^{(3)}(\Omega) \, \xi \, {\bf T}^{\mathring{\mu_r}} \, \overset{(\ref{exp-T-neg})}{=}  \left( 1 \, \pm \,  \frac{12 \, \xi \, \lambda_n^{(3)}(\Omega)}{\pi^{3}} \left( 1 \, \mp \frac{4 \xi}{\pi^{3}} \right)^{-1} \right) \, {\bf I}. 
\end{equation}
Similar to $\mathring{\mu_r}$, since the tensor $A_{\mathring{\mu_r}, n}$ is proportional to the identity, we obtain 
  \begin{equation}\label{P3-hm}
	\left\lVert \overset{3}{\mathbb{P}}\left(\left({\bf I}\pm \xi \nabla{\bf M} {\bf T}^{\mathring{\mu_r}}\right) \cdot H \right)\right\rVert_{\mathbb L^2(\Omega)}^2  = \sum_{n} \left\vert A_{\mathring{\mu_r}, n} \right\vert^{2} \, \left| \left\langle   H, e_n^{(3)}\right\rangle_{\mathbb{L}^{2}(\Omega)} \right|^2 \, \geq  \, \left\vert A_{\mathring{\mu_r}} \right\vert^{2} \sum_{n}  \, \left| \left\langle   H, e_n^{(3)}\right\rangle_{\mathbb{L}^{2}(\Omega)} \right|^2, 
\end{equation} 
where 
\begin{equation*}
\left\vert A_{\mathring{\mu_r}} \right\vert := \underset{n}{Inf} \left\vert A_{\mathring{\mu_r}, n}\right\vert.
\end{equation*}
Then, by gathering $(\ref{P1-h})$, $(\ref{P2-hm})$ and $(\ref{P3-hm})$, we deduce that
\begin{equation}\notag
	\left\lVert \left({\bf I}\pm \xi \nabla {\bf M}{\bf T}^{\mathring{\mu_r}}\right)(H)\right\rVert^{2}_{\mathbb L^2(\Omega)}=\sum_{i=1}^3\left\lVert \overset{i}{\mathbb{P}}\left(\left({\bf I}\pm \xi \nabla{\bf M}{\bf T}^{\mathring{\mu_r}}\right)(H)\right)\right\rVert^{2}_{\mathbb L^2(\Omega)} \geq \min\left\{1, \lvert \mathring{\mu_r} \rvert^{2}, \lvert A_{\mathring{\mu_r}} \rvert^{2} \right\}\lVert H \rVert^{2}_{\mathbb L^2(\Omega)},
\end{equation}
which further indicates that the operator norm fulfills
\begin{equation}\label{low-bd-finalm}
	\left\lVert {\bf I}\pm \xi \nabla {\bf M}{\bf T}^{\mathring{\mu_r}}\right\rVert_{\mathcal{L}\left( \mathbb{L}^{2}(\Omega) ; \mathbb{L}^{2}(\Omega) \right)} \geq \min\left\{1, \lvert\mathring{\mu_r}\rvert, \lvert A_{\mathring{\mu_r}}\rvert \right\}.
\end{equation}
Now, we evaluate the operator norm of ${\bf K}$. To achieve this, recall from $(\ref{TD09})$ that 
\begin{equation*}
	{\bf K}(F) \, = \, \pm  \, \frac{12 \, \xi }{\left( \pi^3 \, \mp \, 4 \, \xi \right)} \, \left[ - \nabla{\bf M^k} \, + \, \nabla{\bf M} \, + \, k^{2} \, {\bf N^k} \right] \left( F \right),
\end{equation*}
and, thanks to \cite[formulas (2.8) \& (2.9)]{CGS}, we rewrite the R.H.S as
\begin{eqnarray}\label{mid-diff}
\nonumber
{\bf K}(F)(x) &=& \pm  \, \frac{12 \, \xi }{\left( \pi^3 \, \mp \, 4 \, \xi \right)} \Bigg[ \frac{k^2}{2} \, N(F) + \frac{k^2}{2} \, N^{\prime}(F) + \frac{i k^3}{6\pi} \, \int_{\Omega} F(y) \, dy  \\ &+& \frac{k^2}{4\pi}\sum_{n\geq1}\frac{(ik)^{n+1}}{(n+1)!}\int_{\Omega}\lVert x-z \rVert^n \, F(z) \,dz + \frac{1}{4\pi}\sum_{n\geq3}\frac{(i k)^{n+1}}{(n+1)!}\int_{\Omega} \nabla \nabla \left(\lVert x-z\rVert^n\right) \cdot F(z) \, dz \, \Bigg],
\end{eqnarray}
where $N^{\prime}(\cdot)$ is the operator defined by $(\ref{DefNprimeTHM})$. Now, by taking the $\mathbb{L}^{2}(\Omega)-$norm on the both sides of \eqref{mid-diff} and using the fact that the second term on the R.H.S. can be evaluated in a same way as the first one, we obtain that
\begin{eqnarray}\label{5.46*m}
\left\Vert {\bf K} \right\Vert_{\mathcal{L}\left( \mathbb{L}^{2}(\Omega); \mathbb{L}^{2}(\Omega)\right)} & \leq &  \frac{12 \,  \xi }{\left\vert \pi^3 \, \mp \, 4 \, \xi \right\vert} \, \Bigg[ k^2 \, \left\Vert N \right\Vert_{\mathcal{L}\left(\mathbb{L}^{2}(\Omega);\mathbb{L}^{2}(\Omega) \right)}  + \frac{k^3}{6 \, \pi}  \, \left\vert \Omega \right\vert  \notag\\  
	&+& \frac{k^2 \, \left\vert \Omega \right\vert}{4 \, \pi \, \diam(\Omega)}  \, \sum_{n \geq 2 }\frac{ \left( k \; \diam(\Omega) \right)^{n}}{ n !} \,     + \frac{k^3 \, \left\vert \Omega \right\vert}{16 \, \pi}  \,
 \sum_{n \geq 1}\frac{\left( k \, \diam(\Omega) \right)^{n}}{n \, !} \,    \, \Bigg].
\end{eqnarray}
Moreover, in \cite[Section 5.1]{CGS}, we have justified the convergence of the last two series appearing on the R.H.S of \eqref{5.46*m}. By referring to \cite[Theorem 3.3]{RS}, we have 
\begin{equation}\notag
	\lVert N \rVert_{\mathcal{L}(\mathbb{L}^{2}(\Omega);\mathbb{L}^{2}(\Omega))} \leq \lVert N \rVert_{\mathcal{L}(\mathbb{L}^{2}(\textbf{B});\mathbb{L}^{2}(\textbf{B}))} \leq \left( \frac{3}{4 \, \pi} \, \left\vert \textbf{B} \right\vert \right)^{\frac{2}{3}} \lVert N \rVert_{\mathcal{L}(\mathbb{L}^{2}(B(0,1));\mathbb{L}^{2}(B(0,1)))},
\end{equation}
where $\textbf{B}$ is a ball such that $\left\vert \Omega \right\vert = \left\vert \textbf{B} \right\vert$, and the last inequality is derived from the Newtonian operator's scale. In addition, due to \cite[Corollary 4.2]{A1992S}, we have
\begin{equation*}
    \lVert N \rVert_{\mathcal{L}(\mathbb{L}^{2}(B(0,1));\mathbb{L}^{2}(B(0,1)))} = \frac{4}{\pi^{2}}. 
\end{equation*}
Then, by the use of $\left\vert \Omega \right\vert = \left\vert \textbf{B} \right\vert$ and $\left\vert \Omega \right\vert > 1$, we end up with the following estimation
\begin{equation*}
	\lVert N \rVert_{\mathcal{L}(\mathbb{L}^{2}(\Omega);\mathbb{L}^{2}(\Omega))} \lesssim \frac{\left\vert \Omega \right\vert}{\pi}. 
\end{equation*}
This implies, 
\begin{equation}\label{es-Km}
	\lVert {\bf K} \rVert_{\mathcal{L}(\mathbb{L}^{2}(\Omega);\mathbb{L}^{2}(\Omega))} \leq \frac{12 \,  \xi }{\left\vert \pi^3 \, \mp \, 4 \, \xi \right\vert} \frac{\left\vert \Omega \right\vert}{\pi} \; c(k).
\end{equation}
where $c(\cdot)$ is the frequency function given by
\begin{equation*}
    c(k) : =  k^2 \,  +  \frac{k^3}{6}  \,   + \frac{k^2}{4  \, \diam(\Omega)} \, \sum_{n \geq 2}\frac{\left( k \, \diam(\Omega) \right)^{n}}{n !} + \frac{k^3}{16} \, \sum_{n \geq 1}\frac{\left( k \, \diam(\Omega) \right)^{n}}{n !}.  
\end{equation*}
By combining $(\ref{low-bd-finalm})$ with $(\ref{es-Km})$, we derive
\begin{equation}\label{compa-bdm}
	\frac{12 \; \xi \; c(k)}{\pi \, \left\vert \pi^{3} \mp 4 \, \xi \right\vert } \, \left\vert \Omega \right\vert \, < \, \min\left\{ 1, \lvert \mathring{\mu_{r}}\rvert, \lvert A_{\mathring{\mu_r}} \rvert \right\}.
\end{equation}
This gives us a sufficient condition ensuring $(\ref{Scdt})$, and ends the first part of the proof. \newline \bigskip \newline
Now, even if the tensors $\mathring{\mu_{r}}$ and $A_{\mathring{\mu_r}}$  are both characterized by a chosen sign, which allows us to distinguish two cases, we will focus on the sequel only on the lower sign. 
Taking the lower sign, using the formulas $(\ref{Tmur})$ and $(\ref{A-mu})$, the inequality $(\ref{compa-bdm})$ becomes
    \begin{equation*}
	\frac{12 \, \xi \, c(k)}{\pi \, \left( \pi^{3} + 4 \, \xi \right)} \, \left\vert \Omega \right\vert < \min\left\{
	 1;\;\; \frac{8 \, \xi \, - \, \pi^{3}}{4 \, \xi \, + \, \pi^{3}}; \;\; \underset{n}{Inf} \;\; \frac{\left\vert \pi^3 + 4 \, \xi \, \left(1 \, - \, 3 \, \lambda_n^{(3)}(\Omega) \right) \right\vert}{4 \, \xi \, + \, \pi^{3}} \right\},
\end{equation*}
which under the condition $\xi >  \, \dfrac{\pi^{3}}{2}$, which is satisfied for $\xi > 2 \, \delta \, \pi^{3}$, with $\delta \geq 1$, can be further reformulated as, 
    \begin{equation}\label{eq-main-ineq}
	\frac{12 \, \xi \, c(k)}{\pi \, \left( \pi^{3} + 4 \, \xi \right)} \, \left\vert \Omega \right\vert < \min\left\{
	 1; \;\; \underset{n}{Inf} \;\; \frac{\left\vert \pi^3 + 4 \, \xi \, \left(1 \, - \, 3 \, \lambda_n^{(3)}\left( \Omega \right) \right) \right\vert}{4 \, \xi \, + \, \pi^{3}} \right\},
\end{equation}
where we shall provide more details on the real parameter $\delta$ later. To investigate the more interesting case that 
\begin{equation*}
    \min\left\{
	 1; \;\; \underset{n}{Inf} \;\; \frac{\left\vert \pi^3 + 4 \, \xi \, \left(1 \, - \, 3 \, \lambda_n^{(3)}\left( \Omega \right) \right) \right\vert}{4 \, \xi \, + \, \pi^{3}} \right\}=\underset{n}{Inf} \;\; \frac{\left\vert \pi^3 + 4 \, \xi \, \left(1 \, - \, 3 \, \lambda_n^{(3)}\left( \Omega \right) \right) \right\vert}{4 \, \xi \, + \, \pi^{3}} \neq 0,
\end{equation*}
we should exclude the following sequence of frequencies $k$ such that,
\begin{equation}\label{ExcSeqxn}
    \xi_{n}  = \frac{\pi^{3}}{4 \, \left( 3 \, \lambda_{n}^{(3)}\left( \Omega \right) \, - \, 1 \right)}.
\end{equation}
By a straightforward computation, we know that
\begin{equation}\label{Expressionximo0}
    \xi_{m_0}:= \frac{\pi^3}{4\left( 3\lambda_{m_0}^{(3)}(\Omega)-1\right)}
\end{equation}
is the solution to the (dispersion) equation
\begin{equation*}
    \pi^3+4\xi\left(1-3\lambda_{m_0}^{(3)}(\Omega)\right)=0.
\end{equation*}
In addition, the sequence 
\begin{equation*}
 \{ \xi_{n} \}_{n \in \mathbb{N}} \; \in \; ]2 \, \delta \, \pi^{3}, \infty[,   
\end{equation*}
where we recall that the lower bound, in the previous interval, is the one fixed for $\xi$.
To avoid $(\ref{eq-main-ineq})$ being meaningless (its R.H.S is zero), we introduce a small parameter $\beta$, i.e. $\left\vert \beta \right\vert \ll 1$ and $\left\vert \beta \right\vert \ll \xi_{m_{0}}$, such that\footnote{The parameter $\beta$ should be chosen such that $\xi$, given by $(\ref{xn0=...+betan0})$, satisfies $\xi \neq \xi_{n}$, $\forall \, n \in \mathbb{N}$, where $\{ \xi_{n} \}_{n \in \mathbb{N}}$ is the sequence given by $(\ref{ExcSeqxn})$. In addition, it's multiplication by the constant $\dfrac{\pi^{3}}{4}$ is only for technical reasons.} in \eqref{eq-main-ineq}, 
\begin{equation}\label{xn0=...+betan0}
    \xi = \xi_{m_{0}} + \frac{\pi^{3}}{4} \, \beta = \frac{\pi^{3}}{4} \; \left[  \frac{1}{\left(3 \, \lambda_{m_{0}}^{(3)}\left( \Omega \right) \, - \, 1 \right)} + \beta \right].
\end{equation}
Next, by using $\left(\lambda_{m_{0}}^{(3)}\left( \Omega \right); \, \xi_{m_{0}} \right)$, we get
\begin{eqnarray}\label{Inff}
    \nonumber
    \underset{n}{Inf} \;\; \frac{\left\vert \pi^3 + 4 \, \xi \, \left(1 \, - \, 3 \, \lambda_n^{(3)}\left( \Omega \right) \right) \right\vert}{4 \, \xi \, + \, \pi^{3}} \;\; & \leq & \;\; \frac{\left\vert \pi^3 + 4 \, \xi \, \left(1 \, - \, 3 \, \lambda_{m_{0}}^{(3)}\left( \Omega \right) \right) \right\vert}{4 \, \xi \, + \, \pi^{3}} \\  & \overset{(\ref{xn0=...+betan0})}{=} & \;  \frac{\left\vert \beta \right\vert \; \left\vert 3 \, \lambda_{m_{0}}^{(3)}\left( \Omega \right) - 1 \right\vert^{2}}{\left\vert 3 \, \lambda_{m_{0}}^{(3)}\left( \Omega \right) + \beta \, \left( 3 \, \lambda_{m_{0}}^{(3)}\left( \Omega \right) - 1 \right) \right\vert} = \mathcal{O}\left( \left\vert \beta \right\vert \right).
\end{eqnarray}
In the sequel, we denote by $f\left( \lambda_{m_{0}}^{(3)}\left( \Omega \right), \beta \right)$ the dominant part of the R.H.S of the previous inequality, i.e. 
\begin{equation}\label{WBMX}
    f\left( \lambda_{m_{0}}^{(3)}\left( \Omega \right), \beta \right) \; := \; \frac{\left\vert \beta \right\vert \; \left\vert 3 \, \lambda_{m_{0}}^{(3)}\left( \Omega \right) - 1 \right\vert^{2}}{3 \, \lambda_{m_{0}}^{(3)}\left( \Omega \right)}
\end{equation}
Moreover, we have  
\begin{equation}\label{small-f}
	 f\left( \lambda_{m_0}^{(3)}\left( \Omega \right), \beta \right) = \mathcal{O}\left( \left\vert \beta \right\vert \right) \, \ll \, 1.
\end{equation}
 Then, by combining with $(\ref{eq-main-ineq})$ and $(\ref{Inff})$, near the resonances, we derive that
\begin{equation*}\label{ineq-com}
	\frac{12 \, \xi \, c(k) \, \left\vert \Omega \right\vert}{\pi \, \left(\pi^3 \, + \, 4 \, \xi \right)}  \,< \, f\left( \lambda_{m_0}^{(3)}\left( \Omega \right), \beta \right),
\end{equation*}
which yields
\begin{equation}\label{cond-ck2}
c(k) \; < \; \frac{\pi \, \left( \pi^{3} \, + \, 4 \, \xi \right)}{12 \, \xi \, \left\vert \Omega \right\vert} \; f\left( \lambda_{m_0}^{(3)}\left( \Omega \right), \beta \right)  \overset{(\ref{xn0=...+betan0})}{=}  \; \left[ \frac{\pi \, \lambda_{m_0}^{(3)}\left( \Omega \right)}{\left\vert \Omega \right\vert} \, + \, g\left( \beta \right) \right] \; f\left( \lambda_{m_{0}}^{(3)}\left( \Omega \right), \beta \right),
\end{equation}
where 
\begin{equation}\label{Deffctg}
 g\left( \beta \right) :=   \frac{\pi \, \lambda_{m_0}^{(3)}(\Omega)}{\left\vert \Omega \right\vert} \; \beta \; \frac{\left( 2 \, - \, 3 \, \lambda_{m_0}^{(3)}(\Omega) \, - \, \dfrac{1}{3 \, \lambda_{m_0}^{(3)}(\Omega)}  \right)}{1 \, + \, \beta \, \left(3 \, \lambda_{m_0}^{(3)}(\Omega) \, - \, 1 \right)}, \quad \text{with} \quad g\left( \beta \right) \, = \, \mathcal{O}\left( \left\vert \beta \right\vert \right).
\end{equation}
Now, by recalling the definition of $\xi$, given by $(\ref{def-xi})$, and using the expression of $\xi$ given by $(\ref{xn0=...+betan0})$, we obtain
\begin{equation}\label{exp-k-p}
	k^{2} \, = \, \frac{c_{r}^{3} \, c_{0}}{\eta_{0}} \; \frac{\pi^{3}}{4} \; \left[ \frac{1}{\left( 3 \,  \lambda_{m_0}^{(3)}\left( \Omega \right) \, - \, 1 \right)} \, + \, \beta \right].
\end{equation}
Moreover, since $k$ fulfills the condition given by $(\ref{condition-on-k})$, which implies that
\begin{equation}\label{exp-k-d}
	k^{2} \, = \, \frac{1 \, \mp \, c_{0} \, a^{h}}{\eta \, a^{2} \,  \lambda_{n_0}^{(1)}\left( B \right)} \, \overset{(\ref{contrast-epsilon})}{=} \, \frac{1}{\eta_0 \, \lambda_{n_0}^{(1)}(B)} \, \mp \, \frac{c_0 \, a^h}{\eta_0 \, \lambda_{n_0}^{(1)}(B)},
\end{equation}
and, by gathering $(\ref{exp-k-p})$ with $(\ref{exp-k-d})$, we end up with the following relation
\begin{equation*}
	\frac{c_{0} \, c_{r}^{3}}{\eta_{0}} \, \frac{\pi^3}{4} \, \left[ \frac{1}{\left( 3 \, \lambda_{m_0}^{(3)}\left( \Omega \right) \, - \, 1 \right)} \, + \, \beta \, \right] \, = \, \frac{1}{\eta_{0} \, \lambda_{n_0}^{(1)}\left( B \right) } \mp \frac{c_{0} \, a^{h}}{\eta_{0} \, \lambda_{n_0}^{(1)}\left( B \right)}.
\end{equation*}
This implies, 
\begin{equation*}
c_{0} \, c_{r}^{3} \, = \, \frac{4}{\pi^{3}} \, \left[ \frac{1}{\left( 3 \, \lambda_{m_0}^{(3)}\left( \Omega \right) \, - \, 1 \right)} \, + \, \beta \right]^{-1} \; \left( \frac{1}{\lambda_{n_0}^{(1)}(B)} \, \mp \, \frac{c_{0} \, a^{h}}{\lambda_{n_0}^{(1)}(B)} \right),
\end{equation*}
which can be reduced, using the smallness of $\beta$ and the parameter $a$, to 
\begin{equation}\label{eq-cocr}
c_{0} \, c_{r}^{3} \, = \, \frac{4}{\pi^{3}} \; \frac{\left( 3 \, \lambda_{m_0}^{(3)}\left( \Omega \right) \, - \, 1 \right)}{\lambda_{n_0}^{(1)}(B)} \, + \mathcal{O}\left( a^{h} \right) \, + \, \mathcal{O}\left( \beta \right).
\end{equation}
By keeping the dominant term
of $\, c_{0} \, c_{r}^{3} \,$, given in $(\ref{eq-cocr})$, and substituting it into $(\ref{exp-k-p})$,  we obtain 
\begin{equation}\label{kfreqplasmon}
    k^{2} \, = \, \frac{1}{\eta_{0} \, \lambda_{n_0}^{(1)}(B)} \, + \, \beta \; \frac{\left(3 \, \lambda_{m_0}^{(3)}(\Omega) \, - \, 1 \right)}{\eta_{0} \, \lambda_{n_0}^{(1)}(B)}.  
\end{equation}
Remember that $k$ is the argument of the frequency function $c(\cdot)$ defined by $(\ref{def-ck})$ and satisfying $(\ref{cond-ck2})$. Then, we can get from \eqref{cond-ck2} that
\begin{equation}\label{c(.)Series}
 k^2 \,  +  \frac{k^3}{6}  \,   + \frac{k^2}{4  \, \diam(\Omega)} \, \sum_{n \geq 2}\frac{\left( k \, \diam(\Omega) \right)^{n}}{n !} + \frac{k^3}{16} \, \sum_{n \geq 1}\frac{\left( k \, \diam(\Omega) \right)^{n}}{n !} \; < \;  \left[ \frac{\pi \, \lambda_{m_0}^{(3)}(\Omega)}{\left\vert \Omega \right\vert} \, + \, g\left( \beta \right) \right] \; f\left( \lambda_{m_{0}}^{(3)}(\Omega), \beta \right), 
\end{equation}
and, as all the terms on the L.H.S of \eqref{c(.)Series} are positive, we deduce that 
\begin{equation}\label{kscalebeta}
 k^2  \; < \;  \left[ \frac{\pi \, \lambda_{m_0}^{(3)}\left( \Omega \right)}{\left\vert \Omega \right\vert} \, + \, g\left( \beta \right) \right] \; f\left( \lambda_{m_{0}}^{(3)}\left( \Omega \right), \beta \right) \, \underset{(\ref{small-f})}{\overset{(\ref{Deffctg})}{\lesssim}} \, \left\vert \beta \right\vert \, \ll \,1.
\end{equation}
This implies the smallness of the frequency $k$. Indeed, if $\eta_0$ is large enough such that the inequality $(\ref{c(.)Series})$ holds, then the frequency $k$ generated by $(\ref{kfreqplasmon})$ can induce a plasmonic resonances. This concludes the second part of the proof. \newline \bigskip \newline
Next, to estimate the field generated incident frequencies $k$ satisfying (\ref{kfreqplasmon}), we need to estimate the Fourier coefficients related to the magnetic field $H^{\mathring{\mu_r}}$, i.e. $\langle H^{\mathring{\mu_r}}, e_{n}^{(j)} \rangle_{\mathbb{L}^{2}(\Omega)}$, with $n \in \mathbb{N}$ and $j=1,2,3$. To do this, we start by rewriting the L.S.E given by $(\ref{LS-vec})$ as
\begin{equation}\label{coro-LS}
\left({\bf I}\pm \xi\nabla{\bf M}{\bf T}^{\mathring{\mu_r}}\right) \, H^{\mathring{\mu_r}}={\bf K}H^{\mathring{\mu_r}}+ i k H^{Inc}(x, \theta).
\end{equation}
By repeating the calculations as before, it is obvious that the term on the L.H.S of \eqref{coro-LS} can be written as follows
\begin{equation*}
\left({\bf I}\pm \xi\nabla{\bf M^0}{\bf T}^{\mathring{\mu_r}}\right)H^{\mathring{\mu_r}} = \sum_n\left\langle H^{\mathring{\mu_r}}, e_n^{(1)}\right\rangle_{\mathbb{L}^{2}(\Omega)} e_n^{(1)}+\sum_n\left\langle \mathring{\mu_r} \cdot H^{\mathring{\mu_r}},e_n^{(2)} \right\rangle_{\mathbb{L}^{2}(\Omega)} e_n^{(2)}+\sum_n \left\langle A_{\mathring{\mu_r},n} \cdot H^{\mathring{\mu_r}}, e_n^{(3)} \right\rangle_{\mathbb{L}^{2}(\Omega)} e_n^{(3)},
\end{equation*}
where $\mathring{\mu_r}$ is the tensor given by $(\ref{Tmur})$ and $A_{\mathring{\mu_r}}$ is the tensor defined by $(\ref{A-mu})$. In addition, since $\mathring{\mu_r}$ and  $(\ref{Tmur})$ and $A_{\mathring{\mu_r}}$ are proportional to the identity matrix, we get from the previous formula that
\begin{eqnarray}\label{decomp-left}
\nonumber
\left({\bf I}\pm \xi\nabla{\bf M^0}{\bf T}^{\mathring{\mu_r}}\right)H^{\mathring{\mu_r}} \, & = &  \, \sum_n \, \left\langle H^{\mathring{\mu_r}}, e_n^{(1)}\right\rangle_{\mathbb{L}^{2}(\Omega)} e_n^{(1)} \, + \, \sum_n \, \mathring{\mu_r} \, \left\langle  H^{\mathring{\mu_r}},e_n^{(2)} \right\rangle_{\mathbb{L}^{2}(\Omega)} e_n^{(2)} \\ &+& \sum_n \, A_{\mathring{\mu_r},n} \, \left\langle  H^{\mathring{\mu_r}}, e_n^{(3)} \right\rangle_{\mathbb{L}^{2}(\Omega)} e_n^{(3)}.
\end{eqnarray}
Similarly, by projecting the R.H.S of \eqref{coro-LS} onto the three sub-spaces, we can obtain that
\begin{eqnarray*}\label{decomp-right}
{\bf K}H^{\mathring{\mu_r}}+i k H^{Inc} \, &=&  \sum_{j=1}^{3} \, \sum_{n} \, \left\langle {\bf K}H^{\mathring{\mu_r}}+i k H^{Inc}, e_n^{(j)} \right\rangle_{\mathbb{L}^{2}(\Omega)} \, e_n^{(j)} \\
& \overset{(\ref{TD09})}{=} & \sum_{j=1}^{3} \, \sum_{n} \, \pm  \, \frac{12 \, \xi }{\left( \pi^3 \, \mp \, 4 \, \xi \right)} \left\langle \, \left[ - \nabla{\bf M^k} \, + \, \nabla{\bf M} \, + \, k^{2} \, {\bf N^k} \right] \left( H^{\mathring{\mu_r}}\right), e_n^{(j)} \right\rangle_{\mathbb{L}^{2}(\Omega)} \, e_n^{(j)} \\ &+& \sum_{j=1}^{3} \, \sum_{n} \,i k \, \left\langle  H^{Inc}, e_n^{(j)} \right\rangle_{\mathbb{L}^{2}(\Omega)} \, e_n^{(j)}. 
\end{eqnarray*}
Knowing that $\nabla{\bf M^k}$, and also $\nabla{\bf M}$, when restricted to $\mathbb{H}_{0}(\div=0)$ is a vanishing operator, we get from the previous formula 
\begin{eqnarray*}\label{decomp-right}
{\bf K}H^{\mathring{\mu_r}}+i k H^{Inc} \, &=&   \sum_{n} \, \pm  \, \frac{12 \, \xi }{\left( \pi^3 \, \mp \, 4 \, \xi \right)} \, k^{2} \, \left\langle {\bf N^k}  \left( H^{\mathring{\mu_r}}\right), e_n^{(1)} \right\rangle_{\mathbb{L}^{2}(\Omega)} \, e_n^{(1)} \\ &+& \sum_{j=2}^{3} \, \sum_{n} \, \pm  \, \frac{12 \, \xi }{\left( \pi^3 \, \mp \, 4 \, \xi \right)} \left\langle \, \left[ - \nabla{\bf M^k} \, + \, \nabla{\bf M} \, + \, k^{2} \, {\bf N^k} \right] \left( H^{\mathring{\mu_r}}\right), e_n^{(j)} \right\rangle_{\mathbb{L}^{2}(\Omega)} \, e_n^{(j)} \\ &+& \sum_{j=1}^{3} \, \sum_{n} \,i k \, \left\langle  H^{Inc}, e_n^{(j)} \right\rangle_{\mathbb{L}^{2}(\Omega)} \, e_n^{(j)}. 
\end{eqnarray*}
And, regarding the subspace $\mathbb{H}_{0}\left(Curl = 0 \right)$, we have 
\begin{equation*}
 \nabla{\bf M^{-k}} =   \, k^{2} \, {\bf N^{-k}} + \textbf{I}  \quad \text{and} \quad     \nabla{\bf M} =   \textbf{I}.
\end{equation*}
This implies, 
\begin{eqnarray}\label{decomp-right}
\nonumber
{\bf K}H^{\mathring{\mu_r}}+i k H^{Inc} \, &=&   \sum_{n} \, \pm  \, \frac{12 \, \xi }{\left( \pi^3 \, \mp \, 4 \, \xi \right)} \, k^{2} \, \left\langle {\bf N^k}  \left( H^{\mathring{\mu_r}}\right), e_n^{(1)} \right\rangle_{\mathbb{L}^{2}(\Omega)} \, e_n^{(1)} \\ \nonumber &+&  \sum_{n} \, \pm  \, \frac{12 \, \xi }{\left( \pi^3 \, \mp \, 4 \, \xi \right)} \left\langle \, \left[ - \nabla{\bf M^k} \, + \, \nabla{\bf M} \, + \, k^{2} \, {\bf N^k} \right] \left( H^{\mathring{\mu_r}}\right), e_n^{(3)} \right\rangle_{\mathbb{L}^{2}(\Omega)} \, e_n^{(3)} \\ &+& \sum_{j=1}^{3} \, \sum_{n} \,i k \, \left\langle  H^{Inc}, e_n^{(j)} \right\rangle_{\mathbb{L}^{2}(\Omega)} \, e_n^{(j)}. 
\end{eqnarray}
By combining with \eqref{decomp-left} and \eqref{decomp-right}, together with the expressions of $\mathring{\mu_r}$ and $A_{\mathring{\mu_r}}$, we get
\begin{equation}\label{Equa426}
		\left\langle H^{\mathring{\mu_r}}, e_n^{(1)}\right\rangle_{\mathbb{L}^{2}(\Omega)}  =  i \, k \, \left\langle H^{Inc}, e_n^{(1)} \right\rangle_{\mathbb{L}^{2}(\Omega)} \; \pm \; \frac{12 \, \xi \, k^{2}}{\left(\pi^{3} \, \mp \, 4 \, \xi \right)}  \, \left\langle {\bf N}^{k}\left(   H^{\mathring{\mu_r}}\right), e_n^{(1)}\right\rangle_{\mathbb{L}^{2}(\Omega)}.
   \end{equation}
By projecting onto the second subspace we obtain
\begin{equation*}
            \left\langle H^{\mathring{\mu_r}}, e_n^{(2)}\right\rangle_{\mathbb{L}^{2}(\Omega)}  =  \frac{i \, k}{|\mathring{\mu_r}|} \; \left\langle H^{Inc}, e_n^{(2)}\right\rangle_{\mathbb{L}^{2}(\Omega)}.
\end{equation*}
Since $\div\left( H^{Inc} \right) \, = \, 0$, we get $\left\langle H^{Inc}, e_n^{(2)}\right\rangle_{\mathbb{L}^{2}(\Omega)} = 0$. Hence, 
\begin{equation}\label{Proj2=0}
            \left\langle H^{\mathring{\mu_r}}, e_n^{(2)}\right\rangle_{\mathbb{L}^{2}(\Omega)}  =  0.
\end{equation}
Finally, by projecting onto the third subspace, 
\begin{eqnarray}\label{comp-coro}
\nonumber
\left\langle H^{\mathring{\mu_r}}, e_n^{(3)}\right\rangle_{\mathbb{L}^{2}(\Omega)} &=& \frac{i \, k \, \left\langle  H^{Inc}, e_n^{(3)} \right\rangle_{\mathbb{L}^{2}(\Omega)} \, + \, \dfrac{\pm 12 \, \xi}{\left(\pi^{3} \, \mp \, 4 \, \xi \right)} \, \langle \left[ - \nabla{\bf M^k} \, + \, \nabla{\bf M} \, + \, k^{2} \, {\bf N^k} \right]\left( H^{\mathring{\mu_r}}\right), e_n^{(3)} \rangle_{\mathbb{L}^{2}(\Omega)}}{|A_{\mathring{\mu_r},n}|} \\ \nonumber &\overset{(\ref{A-mu})}{=}& \frac{i \, k \, \left\langle  H^{Inc}, e_n^{(3)} \right\rangle_{\mathbb{L}^{2}(\Omega)} \, }{\left( 1 \pm \dfrac{12}{\pi^3} \, \lambda_n^{(3)}(\Omega) \, \xi \, \left( 1 \, \mp \, \dfrac{4}{\pi^3}\xi\right)^{-1}\right)} \\
&+& \frac{\dfrac{\pm 12 \, \xi}{\left(\pi^{3} \, \mp \, 4 \, \xi \right)} \, \langle \left[ - \nabla{\bf M^k} \, + \, \nabla{\bf M} \, + \, k^{2} \, {\bf N^k} \right]\left( H^{\mathring{\mu_r}}\right), e_n^{(3)} \rangle_{\mathbb{L}^{2}(\Omega)}}{\left( 1 \pm \dfrac{12}{\pi^3} \, \lambda_n^{(3)}(\Omega) \, \xi \, \left( 1 \, \mp \, \dfrac{4}{\pi^3}\xi\right)^{-1}\right)}.
\end{eqnarray}
Now, let $\xi$ be close to $\xi_{m_{0}}$, as in $(\ref{xn0=...+betan0})$, then  we have
\begin{equation}\label{IBA}
\left\vert 1 \, + \, \frac{12}{\pi^3} \, \lambda_{n}^{(3)}\left( \Omega \right) \, \xi \, \left(1 \, - \, \frac{4}{\pi^3} \, \xi \right)^{-1} \right\vert \, = \mathcal{O}(1), \quad \forall \; n \in \mathbb{N}, 
\end{equation}
and
\begin{equation}\label{Equa0758}
\left\vert 1 \, - \, \frac{12}{\pi^3} \, \lambda_{n}^{(3)}(\Omega) \, \xi \left(1 \, + \, \frac{4}{\pi^3} \, \xi \right)^{-1} \right\vert \, = \, \begin{cases}
			\mathcal{O}\left( 1 \right), & \text{if $n \neq m_0$}\\
   & \\
         \mathcal{O}\left( f\left(\lambda_{m_0}^{(3)}(\Omega), \beta\right) \right), \; & \text{if $n = m_0$}
		 \end{cases},
\end{equation}
where $f\left( \lambda_{m_0}^{(3)}\left( \Omega \right), \beta \right)$ fulfills \eqref{small-f}. This explains why we picked the lower sign instead of the upper sign. Hence, from the estimation of the Fourier coefficients, we deduce that
\begin{eqnarray*}
\left\lVert H^{\mathring{\mu_r}}\right\rVert^{2}_{\mathbb{L}^2(\Omega)} \, & \lesssim & \, k^{2} \, \sum_{n} \left\vert \left\langle H^{Inc}, e_n^{(1)} \right\rangle_{\mathbb{L}^{2}(\Omega)} \right\vert^{2} \; + \; k^{4} \, \sum_{n} \left\vert \left\langle {\bf N}^{k}\left(   H^{\mathring{\mu_r}}\right), e_n^{(1)}\right\rangle_{\mathbb{L}^{2}(\Omega)} \right\vert^{2} \\ \nonumber &+& \left| 1-\frac{12}{\pi^3} \, \lambda_{m_0}^{(3)}(\Omega) \, \xi \, \left(1+\frac{4}{\pi^3}\xi \right)^{-1}\right|^{-2} \; k^{2} \, \sum_{n} \left\vert \left\langle  H^{Inc}, e_n^{(3)} \right\rangle_{\mathbb{L}^{2}(\Omega)} \right\vert^{2} \,    \\
&+& \left| 1-\frac{12}{\pi^3} \, \lambda_{m_0}^{(3)}(\Omega) \, \xi \, \left(1+\frac{4}{\pi^3}\xi \right)^{-1}\right|^{-2} \;  \, \sum_{n} \left\vert \langle \left[ - \nabla{\bf M^k} \, + \, \nabla{\bf M} \, + \, k^{2} \, {\bf N^k} \right]\left( H^{\mathring{\mu_r}}\right), e_n^{(3)} \rangle_{\mathbb{L}^{2}(\Omega)}\right\vert^{2}.   \end{eqnarray*}
For the second term on the R.H.S, we use the smallness of $k$ and the boundedness of $\left\Vert {\bf N}^{k} \right\Vert_{\mathcal{L}\left(\mathbb{L}^{2}(\Omega);\mathbb{L}^{2}(\Omega)\right)}$ to move it into the L.H.S. In addition, the third term dominate the first one. With all these concerns, we obtain 
\begin{eqnarray*}
\left\lVert H^{\mathring{\mu_r}}\right\rVert^{2}_{\mathbb{L}^2(\Omega)} & \lesssim & \frac{\left(k^{2} \, \lVert H^{Inc}\rVert^{2}_{\mathbb L^2(\Omega)} \, + \,  \lVert \left[ - \nabla{\bf M^k} \, + \, \nabla{\bf M} \, + \, k^{2} \, {\bf N^k} \right] \rVert^{2}_{\mathcal{L}\left(\mathbb{L}^{2}(\Omega);\mathbb{L}^{2}(\Omega)\right)}  \, 
 \lVert H^{\mathring{\mu_r}}\rVert^{2}_{\mathbb L^2(\Omega)}\right)}{\left| 1-\dfrac{12}{\pi^3} \, \lambda_{m_0}^{(3)}(\Omega) \, \xi \, \left(1+\dfrac{4}{\pi^3}\xi \right)^{-1}\right|^{2}}  \\ 
 & \overset{(\ref{TD09})}{\lesssim}  & \frac{\left(k^{2} \, \lVert H^{Inc}\rVert^{2}_{\mathbb L^2(\Omega)} \, + \,  \lVert {\bf K}  \rVert^{2}_{\mathcal{L}\left(\mathbb{L}^{2}(\Omega);\mathbb{L}^{2}(\Omega)\right)}  \, 
 \lVert H^{\mathring{\mu_r}}\rVert^{2}_{\mathbb L^2(\Omega)}\right)}{\left| 1-\dfrac{12}{\pi^3} \, \lambda_{m_0}^{(3)}(\Omega) \, \xi \, \left(1+\dfrac{4}{\pi^3}\xi \right)^{-1}\right|^{2}}. 
 \end{eqnarray*}
 By using $(\ref{compa-bdm})$, we derive 
 \begin{equation*}
     \frac{\lVert {\bf K}  \rVert^{2}_{\mathcal{L}\left(\mathbb{L}^{2}(\Omega);\mathbb{L}^{2}(\Omega)\right)}}{\left| 1-\dfrac{12}{\pi^3} \, \lambda_{m_0}^{(3)}(\Omega) \, \xi \, \left(1+\dfrac{4}{\pi^3}\xi \right)^{-1}\right|^{2}} \; \lesssim \; 1,
 \end{equation*}
 which can be used to obtain
 \begin{equation}\label{es-coro1}
 \left\lVert H^{\mathring{\mu_r}}\right\rVert_{\mathbb{L}^2(\Omega)} \, \lesssim  \,  \left| 1-\frac{12}{\pi^3} \, \lambda_{m_0}^{(3)}(\Omega) \, \xi \, \left(1+\frac{4}{\pi^3}\xi \right)^{-1}\right|^{-1} \; k \; \lVert H^{Inc}\rVert_{\mathbb L^2(\Omega)}. 
\end{equation}
This ends the proof of the third part. Additionally, since in \textbf{Theorem \ref{thm-main-posi}}, second point, the error term defined by $(\ref{es-error-neg})$, depends on both $\left\lVert H^{\mathring{\mu_r}} \right\rVert_{\mathbb{L}^{\infty}(\Omega)}$ and $[H^{\mathring{\mu_r}}]_{C^{0, \alpha}(\overline{\Omega})}$, studying their reliance on the model parameters $k$ and $\beta$ is necessary. The objective of the upcoming lemma is to achieve this.
\begin{proposition}\label{AddedLemma} We assume here that $\Omega$ is a $C^{4,\alpha}$-regularity bounded domain. 
We have the following estimations,
\begin{equation}\label{ToproveinAppendix}
    \left[ H^{\mathring{\mu}_r} \right]_{\mathcal{C}^{0, \alpha}(\overline{\Omega})} \; \lesssim \;  \frac{k^{3}}{\left\vert \beta \right\vert^{2}}  \, \left\Vert H^{Inc} \right\Vert_{\mathbb{L}^{2}(\Omega)} \, + \,   \frac{k}{\left\vert \beta \right\vert} \, \left\Vert \nu \times H^{Inc} \right\Vert_{\mathbb{H}^{\frac{3}{2}}(\partial \Omega)} \; = \; \mathcal{O}\left( \frac{k}{\left\vert \beta \right\vert} \right),
\end{equation}
and 
\begin{equation*}
\lVert H^{\mathring{\mu_r}} \rVert_{\mathbb L^{\infty}(\Omega)}  \; \lesssim \; [H^{\mathring{\mu_r}}]_{C^{0, \alpha}(\overline{\Omega})}, 
\end{equation*}
where $0 \, < \, \alpha \, < \, 1$ and $k^{2} \lesssim \beta \ll 1$.
\end{proposition}
\begin{proof}
The second point is a straightforward one. For the first one we refer the readers to \textbf{Subsection \ref{AddedSubSecLemma}}. 
\end{proof}
The estimation $(\ref{ToproveinAppendix})$ proves that $[H^{\mathring{\mu_r}}]_{C^{0, \alpha}(\overline{\Omega})}$ is of order one with respect to the parameter $a$. Regarding the error term in $(\ref{es-error-neg})$, thanks to \textbf{Proposition \ref{AddedLemma}}, we are able to deduce that
\begin{equation}\label{add-coro-4}
E^{\infty}_{\mathrm{eff}, -}(\hat{x})-E^\infty(\hat{x}) = \Oh\left( c(k, \eta_0, c_r, c_0)  \; g_{c_r}(a)  \right) \, + \, \mathcal{O}\left( a^{\frac{h}{3}} \right),
\end{equation}
where
\begin{eqnarray}\label{add-constant}
\nonumber
c(k ,\eta_0, c_r, c_0) &=& \frac{k^3 \, \eta_0^2\, c_r^{-3}}{c_0\left( c_0 \, c_r^3 \, \pi^{3} \, - \, 12 \, k^2 \, \eta_{0} \right)}\left|\left(1 \, \mp \, \frac{4 \, \eta_{0} \, k^2}{\pi^3 \, c_{0}} \, c_{r}^{-3} \right)^{-1}\right| \, \frac{k}{\left\vert \beta \right\vert} \\
&& \\ \nonumber
g_{c_r}(a) &=& \left( c_{r}^{2 \, \alpha} \, a^{2 \, \alpha \, \left(1-\frac{h}{3}\right)} \, \left\vert \log\left( a \right) \right\vert^{2} \, + \, c_{r}^{\frac{12}{7}} \, a^{\frac{12}{7}\left(1-\frac{h}{3}\right)} \right)^{\frac{1}{2}}, \quad \mbox{with} \quad  \alpha \, \in (0,1).
\end{eqnarray}
The formula $(\ref{add-coro-4})$ suggests that the far field $E^\infty(\hat{x})$ is a small perturbation of the far field $E^{\infty}_{\mathrm{eff}, -}(\hat{x})$. 
Moreover, we recall that we have 
\begin{eqnarray}\label{Eq455}
\nonumber
E^\infty_{\mathrm{eff}, -}\left( \hat{x} \right) \, &=& \, - \, \frac{i \, k^3 \, \eta_0}{\pm \, c_0} \, c_r^{-3} \, \int_{\Omega} \, e^{- \, i \, k \, \hat{x}\cdot z} \, \hat{x} \, \times \, {\bf T}^{\mathring{\mu_r}} \cdot H^{\mathring{\mu_{r}}}(z)\,dz \\ 
& \overset{(\ref{exp-T-neg})}{=}& \, - \, \frac{i \, k^3 \, \eta_0}{\pm \, c_0} \, c_r^{-3} \, \left(1 \, \mp \, \frac{4}{\pi^3} \, \xi \right)^{-1} \, \frac{12}{\pi^3} \, \int_{\Omega} \, e^{- \, i \, k \, \hat{x}\cdot z} \, \hat{x} \, \times \,  H^{\mathring{\mu_{r}}}(z)\,dz. 
\end{eqnarray}
Suggested by $(\ref{IBA})$ and $(\ref{Equa0758})$, we choose the lower sign for ${\bf T}^{\mathring{\mu_r}}$, and we get 
\begin{eqnarray}\label{add-coro-E}
E^\infty_{\mathrm{eff}, -}\left( \hat{x} \right) \, & \overset{(\ref{Proj2=0})}{=} & - \, \frac{i \, k^3 \, \eta_0}{\pm \, c_0} \, c_r^{-3} \, \left(1 \, + \, \frac{4}{\pi^3} \, \xi \right)^{-1} \, \frac{12}{\pi^3} \, \int_{\Omega} \, e^{- \, i \, k \, \hat{x}\cdot z} \, \hat{x} \, \times \overset{3}{\mathbb{P}} \, H^{\mathring{\mu_{r}}}(z)\,dz \notag\\
& - &  \, \underbrace{\frac{i \, k^3 \, \eta_0}{\pm \, c_0} \, c_r^{-3} \, \left(1 \, + \, \frac{4}{\pi^3} \, \xi \right)^{-1} \, \frac{12}{\pi^3} \, \int_{\Omega} \, e^{- \, i \, k \, \hat{x}\cdot z} \, \hat{x} \, \times \overset{1}{\mathbb{P}} \, H^{\mathring{\mu_{r}}}(z)\,dz}_{I_{1,2}}.
\end{eqnarray}
To estimate $I_{1,2}$, we use the fact that $k$ is a small parameter and by Taylor expansion of the function $e^{- \, i \, k \, \hat{x} \cdot z}$, we get 
\begin{equation}\label{Texpk}
    e^{- \, i \, k \, \hat{x} \cdot z} \, = \, 1 \, + \, \sum_{n \geq 1} \frac{(- \, i \, k \, \hat{x} \cdot z)^n}{n!},
\end{equation}
and knowing that 
\begin{equation*}
    \int_{\Omega} e_n^{(1)}(x) \, dx \, = \, 0, \quad \forall n \in \mathbb{N},
\end{equation*}
we deduce that the dominant part of $I_{1,2}$ behaves as 
\begin{eqnarray*}
    I_{1,2} & \simeq & \frac{- \, k^3 \, \eta_0}{\pm \, c_0} \, c_r^{-3} \, \left(1 \, + \, \frac{4}{\pi^3} \, \xi \right)^{-1} \, \frac{12}{\pi^3} \, \int_{\Omega} \,  k \, \hat{x}\cdot z \, \hat{x} \, \times \overset{1}{\mathbb{P}} \, H^{\mathring{\mu_{r}}}(z)\,dz \\
   \left\vert I_{1,2} \right\vert & \lesssim  & \frac{k^4 \, \eta_0}{c_r^{3} \, c_0} \; \left\Vert \overset{1}{\mathbb{P}} \,H^{\mathring{\mu_{r}}} \right\Vert_{\mathbb{L}^{2}(\Omega)} \overset{(\ref{Equa426})}{\lesssim} \frac{k^4 \, \eta_0}{c_r^{3} \, c_0} \; k \, \left\Vert \overset{1}{\mathbb{P}} \,H^{Inc} \right\Vert_{\mathbb{L}^{2}(\Omega)} = \mathcal{O}\left( \frac{k^6 \, \eta_0}{c_r^{3} \, c_0} \right).
\end{eqnarray*}
This implies, 
\begin{eqnarray*}
E^\infty_{\mathrm{eff}, -}\left( \hat{x} \right) \, 
 & = & - \, \frac{i \, k^3 \, \eta_0}{\pm \, c_0} \, c_r^{-3} \, \left(1 \, + \, \frac{4}{\pi^3} \, \xi \right)^{-1} \, \frac{12}{\pi^3} \, \int_{\Omega} \, e^{- \, i \, k \, \hat{x}\cdot z} \, \hat{x} \, \times \overset{3}{\mathbb{P}} \,H^{\mathring{\mu_{r}}}(z)\,dz + \mathcal{O}\left( \frac{k^6 \, \eta_0}{c_r^{3} \, c_0} \right)\\  
 &=&  - \, \frac{i \, k^3 \, \eta_0}{\pm \, c_0} \, c_r^{-3} \, \, \left(1 \, + \, \frac{4}{\pi^3} \, \xi \right)^{-1} \, \frac{12}{\pi^3} \,  \int_{\Omega} \, e^{- \, i \, k \, \hat{x}\cdot z} \, \hat{x} \, \times \, e_{m_{0}}^{(3)}(z) \; dz \;\; \langle H^{\mathring{\mu_{r}}}, e_{m_{0}}^{(3)} \rangle_{\mathbb{L}^{2}(\Omega)} \\ &-&  \, \underbrace{\frac{i \, k^3 \, \eta_0}{\pm \, c_0} \, c_r^{-3} \, \, \left(1 \, + \, \frac{4}{\pi^3} \, \xi \right)^{-1} \, \frac{12}{\pi^3} \, \sum_{n \neq m_{0}} \, \int_{\Omega} \, e^{- \, i \, k \, \hat{x}\cdot z} \, \hat{x} \, \times \, e_{n}^{(3)}(z) \; dz \;\; \langle H^{\mathring{\mu_{r}}}, e_{n}^{(3)} \rangle_{\mathbb{L}^{2}(\Omega)} }_{I_{n\neq m_0}}   +  \mathcal{O}\left( \frac{k^{6} \, \eta_0}{c_r^{3} \, c_0} \right).
\end{eqnarray*}
We estimate the series $I_{n\neq m_0}$, which gives us
\begin{equation*}
    \left\vert  I_{n\neq m_0}\right\vert  \lesssim  \frac{k^3 \, \eta_0}{c_r^{3} \, c_0} \,  \left( \sum_{n \neq m_{0}} \, \left\vert \int_{\Omega} \, e^{- \, i \, k \hat{x}\cdot z} \, \hat{x} \, \times \, e_{n}^{(3)}(z) \; dz \right\vert^{2} \right)^{\frac{1}{2}} \;\; \left( \sum_{n \neq m_{0}} \left\vert \langle  \,H^{\mathring{\mu_{r}}}, e_{n}^{(3)} \rangle_{\mathbb{L}^{2}(\Omega)} \right\vert^{2} \right)^{\frac{1}{2}}
\end{equation*}
It is obvious to see from \eqref{comp-coro} that the $\mathbb{L}^{2}(\Omega)$-norm of $\overset{3}{\mathbb{P}} \,H^{\mathring{\mu_{r}}}$ is bounded by the $\mathbb{L}^{2}(\Omega)$-norm of $k \, H^{Inc}$, when $k$ is away from the resonance frequency, which is of order $\mathcal{O}\left( k \right)$. Then we can derive the coming estimation
\begin{equation*}
    \left\vert I_{n\neq m_0} \right\vert \; \lesssim \;  \frac{k^4 \, \eta_0}{c_r^{3} \, c_0} \,  \left( \sum_{n \neq m_{0}} \, \left\vert \int_{\Omega} \, e^{- \, i \, k \hat{x}\cdot z} \, \hat{x} \, \times \,  e_{n}^{(3)}(z) \; dz \right\vert^{2} \right)^{\frac{1}{2}} \,  = \, \mathcal{O}\left( \frac{k^4 \, \eta_0}{c_r^{3} \, c_0} \right). 
\end{equation*}
Then, by utilizing the fact that $k$ is small, we end up with:
\begin{eqnarray*}
E^{\infty}_{\mathrm{eff}, -}(\hat{x})  &=&  - \, \frac{i \, k^3 \, \eta_0}{\pm \, c_0} \, c_r^{-3} \, \left(1 \, + \, \frac{4}{\pi^3} \, \xi \right)^{-1} \, \frac{12}{\pi^3} \, \int_{\Omega} \, e^{- i k \hat{x}\cdot z} \, \hat{x} \, \times \, e_{m_{0}}^{(3)}(z) \; dz \;\; \langle H^{\mathring{\mu_{r}}}, e_{m_{0}}^{(3)} \rangle_{\mathbb{L}^{2}(\Omega)}  +  \mathcal{O}\left( \frac{k^4 \, \eta_0}{c_r^{3} \, c_0} \right) \\  
E^{\infty}_{\mathrm{eff}, -}(\hat{x}) &\overset{(\ref{Texpk})}{=}&  - \, \frac{i \, k^3 \, \eta_0}{\pm \, c_0} \, c_r^{-3} \, \left(1 \, + \, \frac{4}{\pi^3} \, \xi \right)^{-1} \, \frac{12}{\pi^3}  \, \hat{x} \, \times \int_{\Omega}  e_{m_{0}}^{(3)}(z) \; dz \;\; \langle H^{\mathring{\mu_{r}}}, e_{m_{0}}^{(3)} \rangle_{\mathbb{L}^{2}(\Omega)} \\ &+& - \, \frac{i \, k^3 \, \eta_0}{\pm \, c_0} \, c_r^{-3} \, \left(1 \, + \, \frac{4}{\pi^3} \, \xi \right)^{-1} \, \frac{12}{\pi^3} \, \int_{\Omega} \sum_{j \geq 1} \frac{\left( - i k \hat{x}\cdot z \right)^{j}}{j!}  \, \hat{x} \, \times \, e_{m_{0}}^{(3)}(z) \; dz \;\; \langle H^{\mathring{\mu_{r}}}, e_{m_{0}}^{(3)} \rangle_{\mathbb{L}^{2}(\Omega)} \\ &+& \mathcal{O}\left( \frac{k^4 \, \eta_0}{c_r^{3} \, c_0} \right).
\end{eqnarray*}
Next, we estimate the second term on the R.H.S, as 
\begin{eqnarray*}
    \left\vert \cdots \right\vert & \lesssim & \left\vert - \, \frac{i \, k^3 \, \eta_0}{\pm \, c_0} \, c_r^{-3} \, \left(1 \, + \, \frac{4}{\pi^3} \, \xi \right)^{-1} \, \frac{12}{\pi^3}  \, \hat{x} \, \times \int_{\Omega} \sum_{j \geq 1} \frac{\left( - i k \hat{x}\cdot z \right)^{j}}{j!}  \, e_{m_{0}}^{(3)}(z) \; dz \;\; \langle H^{\mathring{\mu_{r}}}, e_{m_{0}}^{(3)} \rangle_{\mathbb{L}^{2}(\Omega)} \right\vert \\ 
    & \lesssim &  \frac{k^4 \, \eta_0}{c_0 \, c_r^{3}}  \; \left\vert \langle H^{\mathring{\mu_{r}}}, e_{m_{0}}^{(3)} \rangle_{\mathbb{L}^{2}(\Omega)} \right\vert \overset{(\ref{comp-coro})}{=} \mathcal{O}\left( \frac{k^5 \, \eta_0}{c_0 \, c_r^{3} \, \left\vert \beta \right\vert} \right).  
\end{eqnarray*}
This implies, 
\begin{eqnarray}\label{E-new}
\nonumber
E^{\infty}_{\mathrm{eff}, -}(\hat{x}) &=&  - \, \frac{i \, k^3 \, \eta_0}{\pm \, c_0} \, c_r^{-3} \, \left(1 \, + \, \frac{4}{\pi^3} \, \xi \right)^{-1} \, \frac{12}{\pi^3}  \, \hat{x} \, \times \int_{\Omega}  e_{m_{0}}^{(3)}(z) \; dz \;\; \langle H^{\mathring{\mu_{r}}}, e_{m_{0}}^{(3)} \rangle_{\mathbb{L}^{2}(\Omega)} \\ &+& \mathcal{O}\left( \frac{k^5 \, \eta_0}{c_0 \, c_r^{3} \, \left\vert \beta \right\vert} \right) + \mathcal{O}\left( \frac{k^4 \, \eta_0}{c_r^{3} \, c_0} \right).
\end{eqnarray}
The estimation of $E^{\infty}_{\mathrm{eff}, -}(\hat{x})$ is delayed until we obtain the dominant term associated to the Fourier coefficient $\langle H^{\mathring{\mu_{r}}}, e_{m_{0}}^{(3)} \rangle_{\mathbb{L}^{2}(\Omega)}$. To achieve this, we observe from \eqref{comp-coro} that the projection of $H^{\mathring{\mu_r}}$ onto the third subspace dominates and this gives us
\begin{eqnarray}\label{EqWillUS}
\nonumber
\left\langle H^{\mathring{\mu_r}}, e_{m_0}^{(3)} \right\rangle_{\mathbb{L}^{2}(\Omega)} \, &=& \, \frac{i \, k}{f(\lambda_{m_0}^{(3)}(\Omega), \beta)}\left\langle \left({\bf I}-\frac{ \overset{3}{\mathbb{P}} \, {\bf K}}{f(\lambda_{m_0}^{(3)}(\Omega), \beta)}\right)^{-1} H^{Inc}, e_{m_0}^{(3)} \right\rangle_{\mathbb{L}^{2}(\Omega)} \\ \nonumber
&+& \underbrace{\frac{i \, k}{f(\lambda_{m_0}^{(3)}(\Omega), \beta)}\left\langle \left({\bf I}-\frac{ \overset{3}{\mathbb{P}} \, {\bf K}}{f(\lambda_{m_0}^{(3)}(\Omega), \beta)}\right)^{-1} \overset{1}{\mathbb{P}}H^{Inc}, e_{m_0}^{(3)} \right\rangle_{\mathbb{L}^{2}(\Omega)}}_{I_{\beta,1}} \\
&+& \, \underbrace{\frac{1}{f(\lambda_{m_0}^{(3)}(\Omega), \beta)}\left\langle \left({\bf I}-\frac{\overset{3}{\mathbb{P}}{\bf K}}{f(\lambda_{m_0}^{(3)}(\Omega), \beta)}\right)^{-1}\overset{3}{\mathbb{P}}{\bf K}\left(\overset{1}{\mathbb{P}} H^{\mathring{\mu_r}} \right), e_{m_0}^{(3)} \right\rangle_{\mathbb{L}^{2}(\Omega)}}_{I_{\beta,2}}.
\end{eqnarray}
Now, we estimate $I_{\beta,1}$ and $I_{\beta,2}$  appearing on the R.H.S of the previous formula. 
\begin{enumerate}
    \item[] 
    \item Estimation of $I_{\beta,1}$.
    \begin{equation*}
       \left\vert I_{\beta,1} \right\vert   \lesssim  \frac{k}{f(\lambda_{m_0}^{(3)}(\Omega), \beta)} \, \left\vert \sum_{n} \left\langle \overset{1}{\mathbb{P}}H^{Inc}, e_{n}^{(1)} \right\rangle_{\mathbb{L}^{2}(\Omega)} \, \left\langle \left({\bf I}-\frac{ \overset{3}{\mathbb{P}} \, {\bf K}}{f(\lambda_{m_0}^{(3)}(\Omega), \beta)}\right)^{-1} e_{n}^{(1)}, e_{m_0}^{(3)} \right\rangle_{\mathbb{L}^{2}(\Omega)} \right\vert.
    \end{equation*}
    Regarding the dominant term of $\left({\bf I}-\dfrac{ \overset{3}{\mathbb{P}} \, {\bf K}}{f(\lambda_{m_0}^{(3)}(\Omega), \beta)}\right)^{-1} e_{n}^{(1)}$, we have the following approximation, 
    \begin{equation*}
        \left({\bf I}-\frac{ \overset{3}{\mathbb{P}} \, {\bf K}}{f(\lambda_{m_0}^{(3)}(\Omega), \beta)}\right)^{-1} e_{n}^{(1)} \, = \, e_{n}^{(1)} \, + \, \frac{1}{f(\lambda_{m_0}^{(3)}(\Omega), \beta)} \, \frac{12}{\pi^{3} \, + \, 4 \, \xi} \, k^{4} \, \left({\bf I}-\frac{ \overset{3}{\mathbb{P}} \, {\bf K}}{f(\lambda_{m_0}^{(3)}(\Omega), \beta)}\right)^{-1}\overset{3}{\mathbb{P}}\int_{\Omega} \left\vert \cdot - y \right\vert e_{n}^{(1)}(y) \, dy,  
    \end{equation*}
    which implies, by taking the orthogonality fact between $e_{n}^{(1)}$ and $e_{m_{0}}^{(3)}$, 
    \begin{equation*}
       \left\vert I_{\beta,1} \right\vert   \lesssim  \frac{k^{5}}{f^{2}(\lambda_{m_0}^{(3)}(\Omega), \beta)} \, \left\vert \sum_{n} \left\langle \overset{1}{\mathbb{P}}H^{Inc}, e_{n}^{(1)} \right\rangle_{\mathbb{L}^{2}(\Omega)} \, \left\langle \left({\bf I}-\frac{ \overset{3}{\mathbb{P}} \, {\bf K}}{f(\lambda_{m_0}^{(3)}(\Omega), \beta)}\right)^{-1}\overset{3}{\mathbb{P}}\int_{\Omega} \left\vert \cdot - y \right\vert e_{n}^{(1)}(y) \, dy, e_{m_0}^{(3)} \right\rangle_{\mathbb{L}^{2}(\Omega)} \right\vert.
    \end{equation*}
    Successively, taking the adjoint operator of  $\left({\bf I}-\dfrac{ \overset{3}{\mathbb{P}} \, {\bf K}}{f(\lambda_{m_0}^{(3)}(\Omega), \beta)}\right)^{-1}$, using the self-adjointness of $\overset{3}{\mathbb{P}}$ and the adjoint operator related to the operator $V(\cdot) \, \rightarrow \, \int_{\Omega} \left\vert \cdot - y \right\vert \, V(y) \, dy$, we end up with the coming estimation
     \begin{equation*}
       \left\vert I_{\beta,1} \right\vert   \lesssim  \frac{k^{5}}{f^{2}(\lambda_{m_0}^{(3)}(\Omega), \beta)} \, \left\vert \sum_{n} \left\langle \overset{1}{\mathbb{P}}H^{Inc}, e_{n}^{(1)} \right\rangle_{\mathbb{L}^{2}(\Omega)} \, \left\langle  e_{n}^{(1)}, \int_{\Omega} \left\vert \cdot - y \right\vert \, \overset{3}{\mathbb{P}} \left({\bf I}-\frac{ \left( \overset{3}{\mathbb{P}} \, {\bf K} \right)^{\star}}{f(\lambda_{m_0}^{(3)}(\Omega), \beta)}\right)^{-1} e_{m_0}^{(3)}(y) \, dy \right\rangle_{\mathbb{L}^{2}(\Omega)} \right\vert.
    \end{equation*}
    As the vector field $\int_{\Omega} \left\vert \cdot - y \right\vert \, \overset{3}{\mathbb{P}} \left({\bf I}-\frac{ \left( \overset{3}{\mathbb{P}} \, {\bf K} \right)^{\star}}{f(\lambda_{m_0}^{(3)}(\Omega), \beta)}\right)^{-1} e_{m_0}^{(3)}(y) \, dy$ is moderate one, with respect to the frequency $k$, we end up with the following estimate
        \begin{eqnarray*}
       \left\vert I_{\beta,1} \right\vert 
                & \lesssim &  \frac{k^{5}}{f^{2}(\lambda_{m_0}^{(3)}(\Omega), \beta)} \, \left\Vert  \overset{1}{\mathbb{P}}H^{Inc} \right\Vert_{\mathbb{L}^{2}(\Omega)} \,  \lesssim   \frac{k^{6}}{f^{2}(\lambda_{m_0}^{(3)}(\Omega), \beta)} \, \left\Vert  E^{Inc} \right\Vert_{\mathbb{L}^{2}(\Omega)} \, \overset{(\ref{small-f})}{=} \, \mathcal{O}\left( \frac{k^{6}}{\left\vert \beta \right\vert^{2}} \, \right). 
    \end{eqnarray*}
    \item[] 
    \item Estimation of $I_{\beta,2}$.
\begin{eqnarray*}
    \left\vert I_{\beta,2} \right\vert & \leq & \left\vert \frac{1}{f(\lambda_{m_0}^{(3)}(\Omega), \beta)}\left\langle \left({\bf I}-\frac{\overset{3}{\mathbb{P}}{\bf K}}{f(\lambda_{m_0}^{(3)}(\Omega), \beta)}\right)^{-1}\overset{3}{\mathbb{P}}{\bf K}\left(\overset{1}{\mathbb{P}} H^{\mathring{\mu_r}} \right), e_{m_0}^{(3)} \right\rangle_{\mathbb{L}^{2}(\Omega)} \right\vert \\
    & \overset{(\ref{small-f})}{\lesssim} & \left\vert  \beta \right\vert^{-1} \, \left\Vert \left({\bf I}-\frac{\overset{3}{\mathbb{P}}{\bf K}}{f(\lambda_{m_0}^{(3)}(\Omega), \beta)}\right)^{-1}\overset{3}{\mathbb{P}} \right\Vert_{\mathcal{L}(\mathbb{L}^{2}(\Omega);\mathbb{L}^{2}(\Omega))} \, \left\Vert {\bf K}\left(\overset{1}{\mathbb{P}} H^{\mathring{\mu_r}} \right)  \right\Vert_{\mathbb{L}^{2}(\Omega)},
\end{eqnarray*}
and, by using the fact that the operator norm $\left\Vert \cdots \right\Vert_{\mathcal{L}(\mathbb{L}^{2}(\Omega);\mathbb{L}^{2}(\Omega))} \, \sim \, 1$, we get
\begin{equation*}
    \left\vert I_{\beta,2} \right\vert  \lesssim  \left\vert  \beta \right\vert^{-1}  \, \left\Vert {\bf K}\left(\overset{1}{\mathbb{P}} H^{\mathring{\mu_r}} \right)  \right\Vert_{\mathbb{L}^{2}(\Omega)}. 
\end{equation*}
Knowing that the operator ${\bf K}$, when reduced to $\mathbb{H}_{0}(\div = 0)$, behaves proportional to $k^{2} \, {\bf N}^{k}$, we reduce the previous estimation to 
\begin{equation*}
    \left\vert I_{\beta,2} \right\vert  \lesssim \left\vert  \beta \right\vert^{-1} \, k^2 \; \left\Vert {\bf N}^{k} \left( \overset{1}{\mathbb{P}} H^{\mathring{\mu_r}} \right) \right\Vert_{\mathbb{L}^{2}(\Omega)}  
    \lesssim  \left\vert  \beta \right\vert^{-1} \, k^2 \, \left\Vert {\bf N}^{k}   \right\Vert_{\mathcal{L}(\mathbb{L}^{2}(\Omega);\mathbb{L}^{2}(\Omega))} \;  \left\Vert \overset{1}{\mathbb{P}} H^{\mathring{\mu_r}} \right\Vert_{\mathbb{L}^{2}(\Omega)}.  
\end{equation*}
As $\left\Vert {\bf N}^{k}   \right\Vert_{\mathcal{L}(\mathbb{L}^{2}(\Omega);\mathbb{L}^{2}(\Omega))}$ is of order one, and by the use of $(\ref{Equa426})$, we get
\begin{equation*}
    \left\vert I_{\beta,2} \right\vert  \lesssim   \left\vert  \beta \right\vert^{-1} \, k^2 \, k \;  \left\Vert \overset{1}{\mathbb{P}} H^{Inc} \right\Vert_{\mathbb{L}^{2}(\Omega)} \, = \, \mathcal{O}\left( \frac{k^4}{\left\vert  \beta \right\vert} \right).  
\end{equation*}
\end{enumerate}
Hence, by recalling $(\ref{EqWillUS})$, we obtain
\begin{equation}\label{Equa0858}
\left\langle H^{\mathring{\mu_r}}, e_{m_0}^{(3)} \right\rangle_{\mathbb{L}^{2}(\Omega)} \, = \, \frac{i \; k}{f(\lambda_{m_0}^{(3)}(\Omega), \beta)}\left\langle \left({\bf I}-\frac{ \overset{3}{\mathbb{P}} \, {\bf K}}{f(\lambda_{m_0}^{(3)}(\Omega), \beta)}\right)^{-1} H^{Inc}, e_{m_0}^{(3)} \right\rangle_{\mathbb{L}^{2}(\Omega)} + \mathcal{O}\left( \frac{k^6}{\left\vert \beta \right\vert^{2}} \right) \, + \, \mathcal{O}\left( \frac{k^4}{\left\vert \beta \right\vert} \right).
\end{equation}
Now, by going back to $(\ref{E-new})$, using $(\ref{Equa0858})$ and $(\ref{def-xi})$, and recalling the fact that $\xi \sim 1$, we get 
\begin{eqnarray*}
E^{\infty}_{\mathrm{eff}, -}(\hat{x}) &=&   \frac{\pm \, 12 \, k^{2} \, \xi}{\pi^3 \, + \, 4 \, \xi} \,  \hat{x} \, \times \, \int_{\Omega} e_{m_{0}}^{(3)}(z) \; dz \;\; \frac{1}{f(\lambda_{m_0}^{(3)}(\Omega), \beta)}\left\langle \left({\bf I}-\frac{ \overset{3}{\mathbb{P}} \, {\bf K}}{f(\lambda_{m_0}^{(3)}(\Omega), \beta)}\right)^{-1} H^{Inc}, e_{m_0}^{(3)} \right\rangle_{\mathbb{L}^{2}(\Omega)} \\ &+&  \mathcal{O}\left( \frac{k^7}{\left\vert \beta \right\vert^2} \right) \, + \, \mathcal{O}\left(  k^2 \right) \, + \, \mathcal{O}\left(\frac{k^3}{\left\vert \beta \right\vert} \right).
\end{eqnarray*}
Moreover, as $k^{2} \lesssim \left\vert \beta \right\vert$, see $(\ref{kscalebeta})$, we reduce the error term of the previous term to $\mathcal{O}\left(\dfrac{k^3}{\left\vert \beta \right\vert} \right)$. Hence,
\begin{eqnarray*}
E^{\infty}_{\mathrm{eff}, -}(\hat{x}) &=&  \pm \, \frac{12 \, k^{2} \, \xi}{\pi^3 \, + \, 4 \, \xi}    \, \hat{x} \, \times \, \int_{\Omega} e_{m_{0}}^{(3)}(z) \; dz \;\; \frac{1}{f(\lambda_{m_0}^{(3)}(\Omega), \beta)}\left\langle \left({\bf I}-\frac{ \overset{3}{\mathbb{P}} \, {\bf K}}{f(\lambda_{m_0}^{(3)}(\Omega), \beta)}\right)^{-1} H^{Inc}, e_{m_0}^{(3)} \right\rangle_{\mathbb{L}^{2}(\Omega)} \\ &+& \mathcal{O}\left(\frac{k^3}{\left\vert \beta \right\vert} \right) \\
& \overset{(\ref{xn0=...+betan0})}{=} & \pm \, \frac{12 \, k^{2} \, \xi_{m_{0}}}{\pi^3 \, + \, 4 \, \xi_{m_{0}}}    \, \hat{x} \, \times \, \int_{\Omega} e_{m_{0}}^{(3)}(z) \; dz \;\; \frac{1}{f(\lambda_{m_0}^{(3)}(\Omega), \beta)}\left\langle \left({\bf I}-\frac{ \overset{3}{\mathbb{P}} \, {\bf K}}{f(\lambda_{m_0}^{(3)}(\Omega), \beta)}\right)^{-1} H^{Inc}, e_{m_0}^{(3)} \right\rangle_{\mathbb{L}^{2}(\Omega)} \\ &+&  \frac{\pm 3 k^{2} \pi^{6} \beta}{( \pi^3 + 4 \xi) ( \pi^3  +  4 \xi_{m_{0}})}  \hat{x} \times  \int_{\Omega} e_{m_{0}}^{(3)}  dz  \frac{1}{f(\lambda_{m_0}^{(3)}(\Omega), \beta)}\left\langle \left({\bf I}-\frac{ \overset{3}{\mathbb{P}} \, {\bf K}}{f(\lambda_{m_0}^{(3)}(\Omega), \beta)}\right)^{-1} H^{Inc}, e_{m_0}^{(3)} \right\rangle_{\mathbb{L}^{2}(\Omega)} \\ &+& \mathcal{O}\left(\frac{k^3}{\left\vert \beta \right\vert} \right).
\end{eqnarray*}
We estimate the second term on the R.H.S as 
\begin{equation*}
 \left\vert \frac{\pm 3  k^{2} \pi^{6}  \beta}{( \pi^3  +  4  \xi) ( \pi^3  +  4  \xi_{m_{0}})}  \hat{x}  \times  \int_{\Omega} e_{m_{0}}^{(3)} dz \frac{1}{f(\lambda_{m_0}^{(3)}(\Omega), \beta)}\left\langle \left({\bf I}-\frac{ \overset{3}{\mathbb{P}} \, {\bf K}}{f(\lambda_{m_0}^{(3)}(\Omega), \beta)}\right)^{-1} H^{Inc}, e_{m_0}^{(3)} \right\rangle_{\mathbb{L}^{2}(\Omega)} \right\vert = \mathcal{O}\left( k^{2} \right). 
\end{equation*}
Then, 
\begin{eqnarray*}
E^{\infty}_{\mathrm{eff}, -}(\hat{x}) &=&  \pm \, \frac{12 \, k^{2} \, \xi_{m_{0}}}{\pi^3 \, + \, 4 \, \xi_{m_{0}}}    \, \hat{x} \, \times \, \int_{\Omega} e_{m_{0}}^{(3)}(z) \; dz \;\; \frac{1}{f(\lambda_{m_0}^{(3)}(\Omega), \beta)}\left\langle \left({\bf I}-\frac{ \overset{3}{\mathbb{P}} \, {\bf K}}{f(\lambda_{m_0}^{(3)}(\Omega), \beta)}\right)^{-1} H^{Inc}, e_{m_0}^{(3)} \right\rangle_{\mathbb{L}^{2}(\Omega)} \\ &+& \mathcal{O}\left( k^{2} \right) + \mathcal{O}\left(\frac{k^3}{\left\vert \beta \right\vert} \right) \\
& \overset{(\ref{Expressionximo0})}{=} &  \pm \, \frac{k^{2}}{\lambda_{m_{0}}^{(3)}(\Omega)}    \, \hat{x} \, \times \, \int_{\Omega} e_{m_{0}}^{(3)}(z) \; dz \;\; \frac{1}{f(\lambda_{m_0}^{(3)}(\Omega), \beta)}\left\langle \left({\bf I}-\frac{ \overset{3}{\mathbb{P}} \, {\bf K}}{f(\lambda_{m_0}^{(3)}(\Omega), \beta)}\right)^{-1} H^{Inc}, e_{m_0}^{(3)} \right\rangle_{\mathbb{L}^{2}(\Omega)} \\ &+& \mathcal{O}\left( k^{2} \right) + \mathcal{O}\left(\frac{k^3}{\left\vert \beta \right\vert} \right) \\
& \overset{(\ref{WBMX})}{=}&    \frac{\pm \, 3 \, k^{2}}{\left\vert \beta \right\vert \; \left\vert 3 \, \lambda_{m_{0}}^{(3)}\left( \Omega \right) - 1 \right\vert^{2}}  \, \hat{x} \, \times \, \int_{\Omega} e_{m_{0}}^{(3)}(z) \; dz  \; \left\langle \left({\bf I}-\frac{ \overset{3}{\mathbb{P}} \, {\bf K}}{f(\lambda_{m_0}^{(3)}(\Omega), \beta)}\right)^{-1} H^{Inc}, e_{m_0}^{(3)} \right\rangle_{\mathbb{L}^{2}(\Omega)} \\ &+& \mathcal{O}\left( k^{2} \right) + \mathcal{O}\left(\frac{k^3}{\left\vert \beta \right\vert} \right).
\end{eqnarray*}
At this stage, by combining the previous formula with \eqref{add-coro-4}, we obtain 
\begin{eqnarray*}
\nonumber
	E^\infty(\hat{x}) &=& \frac{\pm \, 3 \, k^{2}}{\left\vert \beta \right\vert \; \left\vert 3 \, \lambda_{m_{0}}^{(3)}\left( \Omega \right) - 1 \right\vert^{2}}  \, \hat{x} \, \times \, \int_{\Omega} e_{m_{0}}^{(3)}(z) \; dz  \; \left\langle \left({\bf I}-\frac{ \overset{3}{\mathbb{P}} \, {\bf K}}{f(\lambda_{m_0}^{(3)}(\Omega), \beta)}\right)^{-1} H^{Inc}, e_{m_0}^{(3)} \right\rangle_{\mathbb{L}^{2}(\Omega)} \\
	&+& \mathcal{O}\left( k^{2} \right) + \mathcal{O}\left(\frac{k^3}{\left\vert \beta \right\vert} \right) + \Oh\left( c(k, \eta_0, c_r, c_0) \; g_{c_r}(a)  \right).
\end{eqnarray*}
Based on the fact that $[H^{\mathring{\mu_r}}]_{C^{0, \alpha}(\overline{\Omega})} $ is of order one with respect to $a$, the definition of $c(k, \eta_0, c_r, c_0)$ and the estimation of $g_{c_r}(a)$, given by $(\ref{add-constant})$, we deduce that 
\begin{equation*}
    \mathcal{O}\left( k^{2} \right) \, + \, \mathcal{O}\left(\frac{k^3}{\left\vert \beta \right\vert} \right) \, + \,  
 \Oh\left( c(k, \eta_0, c_r, c_0) \; [H^{\mathring{\mu_r}}]_{C^{0, \alpha}(\overline{\Omega})} \; g_{c_r}(a)  \right) \; = \;     \mathcal{O}\left( k^{2} \right) + \mathcal{O}\left(\frac{k^3}{\left\vert \beta \right\vert} \right).
\end{equation*}
Hence, by taking into account that we have assumed that $\lambda_{m_0}^{(3)}(\Omega) \, > \, \dfrac{1}{3}$, we end up with the coming formula
\begin{eqnarray}\label{E-new-2}
\nonumber
	E^\infty(\hat{x}) &=& \frac{\pm \, 3 \, k^{2}}{\left\vert \beta \right\vert \; \left( 3 \, \lambda_{m_{0}}^{(3)}\left( \Omega \right) - 1 \right)^{2}}  \, \hat{x} \, \times \, \int_{\Omega} e_{m_{0}}^{(3)}(z) \; dz  \; \left\langle \left({\bf I}-\frac{ \overset{3}{\mathbb{P}} \, {\bf K}}{f(\lambda_{m_0}^{(3)}(\Omega), \beta)}\right)^{-1} H^{Inc}, e_{m_0}^{(3)} \right\rangle_{\mathbb{L}^{2}(\Omega)} \\
	&+& \mathcal{O}\left( k^{2} \right) + \mathcal{O}\left(\frac{k^3}{\left\vert \beta \right\vert} \right).
\end{eqnarray}
Recalling that $H^{Inc}(x) := \left(\theta \, \times \, \mathrm{p} \right) \, e^{i \, k \, \theta \cdot x}$, using the smallness of the frequency $k$ and the boundedness of the operator $\left({\bf I}-\frac{ \overset{3}{\mathbb{P}} \, {\bf K}}{f(\lambda_{m_0}^{(3)}(\Omega), \beta)}\right)^{-1}$, to get
\begin{equation*}
    \left\langle \left({\bf I}-\frac{ \overset{3}{\mathbb{P}} \, {\bf K}}{f(\lambda_{m_0}^{(3)}(\Omega), \beta)}\right)^{-1} H^{Inc}, e_{m_0}^{(3)} \right\rangle_{\mathbb{L}^{2}(\Omega)} \, = \, \left\langle \left({\bf I}-\frac{ \overset{3}{\mathbb{P}} \, {\bf K}}{f(\lambda_{m_0}^{(3)}(\Omega), \beta)}\right)^{-1} \left(\theta \, \times \, \mathrm{p} \right), e_{m_0}^{(3)} \right\rangle_{\mathbb{L}^{2}(\Omega)} \, + \, \mathcal{O}\left( k \right).
\end{equation*}
And, this allows us to rewrite $(\ref{E-new-2})$ as
\begin{eqnarray}\label{Eq129}
\nonumber
	E^\infty(\hat{x}) &=& \frac{\pm \, 3 \, k^{2}}{\left\vert \beta \right\vert \; \left( 3 \, \lambda_{m_{0}}^{(3)}\left( \Omega \right) - 1 \right)^{2}}  \, \hat{x} \, \times \, \int_{\Omega} e_{m_{0}}^{(3)}(z) \; dz  \; \left\langle \left({\bf I}-\frac{ \overset{3}{\mathbb{P}} \, {\bf K}}{f(\lambda_{m_0}^{(3)}(\Omega), \beta)}\right)^{-1} \left(\theta \, \times \, \mathrm{p} \right), e_{m_0}^{(3)} \right\rangle_{\mathbb{L}^{2}(\Omega)} \\
	&+& \mathcal{O}\left( k^{2} \right) + \mathcal{O}\left(\frac{k^3}{\left\vert \beta \right\vert} \right).
\end{eqnarray}
By utilizing the Born series and keeping the dominant term associated with the operator ${\bf K}$, we can demonstrate that 
\begin{equation*}
    \left({\bf I} \, - \, \dfrac{ \overset{3}{\mathbb{P}} \, {\bf K}}{f(\lambda_{m_0}^{(3)}(\Omega), \beta)}\right)^{-1} \, = \, \begin{LARGE} \textbf{N} \end{LARGE}^{-1} \, + \, \mathcal{O}\left( k \right) \, + \, \mathcal{O}\left( \beta \right),
\end{equation*}
where $\begin{LARGE} \textbf{N} \end{LARGE}$ is a self-adjoint operator given by
\begin{equation*}
    \begin{LARGE} \textbf{N} \end{LARGE} := \left({\bf I} \, - \, \dfrac{\pm \, k^{2}}{2 \, \left\vert \beta \right\vert \, \left( 3 \, \lambda_{m_{0}}^{(3)}(\Omega) \, - \, 1 \right)^{2}} \left( N + N^{\prime} \right) \overset{3}{\mathbb{P}} \, \right)
\end{equation*}
with $N$ is the Newtonian operator and $N^{\prime}$ is the operator defined by $(\ref{DefNprimeTHM})$. Using this, we rewrite $(\ref{Eq129})$ as 
\begin{equation*}
	E^\infty(\hat{x}) = \frac{\pm \, 3 \, k^{2}}{\left\vert \beta \right\vert \; \left( 3 \, \lambda_{m_{0}}^{(3)}\left( \Omega \right) - 1 \right)^{2}}  \, \hat{x} \, \times \, \int_{\Omega} e_{m_{0}}^{(3)}(z) \; dz  \; \left\langle     \begin{LARGE} \textbf{N} \end{LARGE}^{-1} \left(\theta \, \times \, \mathrm{p} \right), e_{m_0}^{(3)} \right\rangle_{\mathbb{L}^{2}(\Omega)} + \mathcal{O}\left( k^{2} \right) + \mathcal{O}\left(\frac{k^3}{\left\vert \beta \right\vert} \right).
\end{equation*}
Now, by taking the adjoint operator of $    \begin{LARGE} \textbf{N} \end{LARGE}^{-1}$, which is itself, and remembering that $\langle ; \rangle_{\mathbb{L}^{2}(\Omega)}$ is a complex inner product, we get 
\begin{equation}\label{MNP}
	E^\infty(\hat{x}) = \frac{\pm \, 3 \, k^{2}}{\left\vert \beta \right\vert \; \left( 3 \, \lambda_{m_{0}}^{(3)}\left( \Omega \right) - 1 \right)^{2}}  \, \hat{x} \, \times \, \int_{\Omega} e_{m_{0}}^{(3)}(z) \; dz  \; \left( \theta \, \times \, \mathrm{p} \right) \cdot \int_{\Omega} \overline{     \begin{LARGE} \textbf{N} \end{LARGE}^{-1}\left(e_{m_0}^{(3)}\right)}(x) \, dx  + \mathcal{O}\left( k^{2} \right) + \mathcal{O}\left(\frac{k^3}{\left\vert \beta \right\vert} \right).
\end{equation}
We set $\boldsymbol{\mathcal{X}}_{m_{0}}(\lambda_{m_0}^{(3)}(\Omega), \beta)$ to be the tensor defined by
\begin{equation}\label{DefTensorX}
    \boldsymbol{\mathcal{X}}_{m_{0}}(\lambda_{m_0}^{(3)}(\Omega), \beta) := \int_{\Omega} e_{m_{0}}^{(3)}(z) \; dz   \otimes \int_{\Omega} \overline{  \begin{LARGE} \textbf{N} \end{LARGE}^{-1}\left(e_{m_0}^{(3)}\right)}(x) \, dx.  
\end{equation}
Finally, by the use of the definition $    \boldsymbol{\mathcal{X}}_{m_{0}}(\lambda_{m_0}^{(3)}(\Omega), \beta)$, the formula $(\ref{MNP})$ can be rewritten like
\begin{equation}\label{AnLoIs}
	E^\infty(\hat{x}) = \frac{\pm \, 3 \, k^{2}}{\left\vert \beta \right\vert \; \left( 3 \, \lambda_{m_{0}}^{(3)}\left( \Omega \right) - 1 \right)^{2}}  \, \hat{x} \, \times \, \left( \boldsymbol{\mathcal{X}}_{m_{0}}(\lambda_{m_0}^{(3)}(\Omega), \beta) \cdot \left( \theta \, \times \, \mathrm{p} \right) \right) + \mathcal{O}\left( k^{2} \right) + \mathcal{O}\left(\frac{k^3}{\left\vert \beta \right\vert} \right).
\end{equation}
Therefore, for $k$ fulfilling $(\ref{kfreqplasmon})$, we see that the first term in (\ref{E-new-2}) is the dominant part and it behaves like $\dfrac{k^{2}}{\beta}$
which indicates that with $\beta \, \ll \, 1$ and $\eta_0 \, \gg \, 1$, we deduce a giant amplification of the magnetic field $H^{Inc}$. In other words, the incident frequency
\begin{equation}\notag
	k_{n_0}^2 \, := \, \frac{1}{\eta_0 \,\, \lambda_{n_0}^{(1)}(B)},
\end{equation}
is a low frequency plasmonic resonance generated in $\Omega$ by the cluster of the all-dielectric nano-particles in the quasi-static regime. 
\newline 
\medskip
\newline 
The previous proof supports the following two remarks.
\begin{remark}
\phantom{}
    \begin{enumerate}
        \item[]
        \item The effective permeability expression. Recall, from $(\ref{new-coeff-neg})$, that 
\begin{equation*}
\mathring{\mu_r} \, = \, \frac{\pi^{3} \, - \, 8 \, \xi }{\pi^{3} \, + \, 4 \, \xi} \; {\bf I}  \overset{(\ref{xn0=...+betan0})}{=}  \frac{1 \, - \, 2  \, \left[ \dfrac{1}{\left( 3 \, \lambda_{m_{0}}^{(3)}(\Omega) \, - \, 1\right)} \, + \, \beta \right] }{1 \, +   \, \left[ \dfrac{1}{\left( 3 \, \lambda_{m_{0}}^{(3)}(\Omega) \, - \, 1\right)} \, + \, \beta \right]} \, {\bf I} = \left( 1 \, - \, \frac{1}{\lambda_{m_{0}}^{(3)}(\Omega)} + \mathrm{T}\left( \beta \right) \right) \, {\bf I}
\end{equation*}
where 
\begin{equation*}
    \mathrm{T}\left( \beta \right) := \frac{- \, \beta \, \left( 3 \, \lambda_{m_{0}}^{(3)}(\Omega) \, - \, 2 \, + \, \dfrac{1}{3 \, \lambda_{m_{0}}^{(3)}(\Omega)} \right)}{\lambda_{m_{0}}^{(3)}(\Omega) \, + \, \beta \, \left(\lambda_{m_{0}}^{(3)}(\Omega) \, - \, \dfrac{1}{3} \right)} \, = \, \mathcal{O}\left( \beta \right).
\end{equation*}
        \item[]
        \item[] 
        \item The scale of $\beta$ relative to $k$. We have seen, from $(\ref{kscalebeta})$, that $\beta \sim k^{2}$. But, if $\beta$ can be taken to satisfy $\beta \, \sim \, k^{\sigma}$, with $0 \, < \sigma \, < 2$, the following approximation holds
\begin{equation*}
     \left({\bf I}-\frac{ \left( \overset{3}{\mathbb{P}} \, {\bf K} \right)^{\star} }{f(\lambda_{m_0}^{(3)}(\Omega), \beta)}\right)^{-1}\left( e_{m_{0}}^{(3)} \right)  \; = \; e_{m_{0}}^{(3)} \; + \; \mathcal{O}\left( k^{2-\sigma} \right).   
\end{equation*}
And, $(\ref{DefTensorX})$ can be approximated by
\begin{equation*}
    \boldsymbol{\mathcal{X}}_{m_{0}}(\lambda_{m_0}^{(3)}(\Omega), \beta) \; = \; {\bf Q}^{(3)}_{m_{0}} \; + \; \mathcal{O}\left( k^{2-\sigma} \right).  
\end{equation*}
where
\begin{equation}\label{DefQm0Q1}
{\bf Q}^{(3)}_{m_{0}} \, := \, \int_{\Omega}e_{m_{0}}^{(3)}(z) \; dz \otimes \int_{\Omega} \overline{e_{m_0}^{(3)}}(z) \, dz.
\end{equation}
Consequently, 
\begin{equation*}
	E^\infty(\hat{x}) = \frac{\pm \, 3 \, k^{2}}{\left\vert \beta \right\vert \, \left( 3 \, \lambda_{m_0}^{(3)}(\Omega) \, - \, 1 \right)^{2}} \;  \hat{x} \, \times \left( {\bf Q}^{(3)}_{m_{0}} \cdot \left( \theta \, \times \, \mathrm{p} \right) \right) \, + \, \mathcal{O}\left(   k^{\min\left(4-2\, \sigma ; 2 \; 3 - \sigma \right)} \right).
\end{equation*}
In the case of plasmonics, the tensor ${\bf Q}^{(3)}_{m_{0}}$ plays the same role as the tensor ${\bf P}_0$, given in $(\ref{defP0})$.
    \end{enumerate}
\end{remark}
This ends the proof of the fourth part. \newline \medskip \newline 

Observe that by construction of $    \boldsymbol{\mathcal{X}}_{m_{0}}(\lambda_{m_0}^{(3)}(\Omega), \beta)$, see for instance $(\ref{DefTensorX})$, if $\int_{\Omega}e_{m_{0}}^{(3)}(z) \; dz = 0$, the far field formula given by $(\ref{AnLoIs})$ will be meaningless (i.e. the dominant term vanishes). Unfortunately, depending on the $\Omega$'s shape, it may happen that $\int_{\Omega}e_{m_{0}}^{(3)}(z) \; dz$ may be a vanishing integral. For example, based on \cite[Section 1.3]{GS}, we know that for the case $\Omega$ is the unit ball, we have $\int_{\Omega = B(0,1)}e_{n}^{(3)}(z) \; dz \, = \, 0$, for $n \geq 2$. The coming lines comes to overcome this issue and brings more details on the case $\Omega$ is the unit ball. 
\newline \smallskip \newline 
In the sequel, we assume that $\Omega$ is the unit ball and we take the parameter $\xi$ to be large enough, i.e. $\xi \gg 1$. Equations $(\ref{Eq455})$ and $(\ref{add-coro-E})$  illustrate that choosing the lower sign for ${\bf T}^{\mathring{\mu_r}}$ is crucial for creating a resonance phenomenon related to eigenvalues $\lambda_{n}^{(3)}(\Omega) \, > \, \dfrac{1}{3}$, see for instance $(\ref{ExcSeqxn})$. Our subsequent focus, guided by an essential detail that we will highlight later, is on the case where  $\lambda_{n}^{(3)}(\Omega) = \dfrac{1}{3}$ and ${\bf T}^{\mathring{\mu_r}}$ generated by the upper sign. We start by recalling, from $(\ref{Eq455})$, the far-field formula
\begin{equation*}
    E^{\infty}_{eff,-}(\hat{x})  =  - \, \pm \, i \, k \, \xi \, \frac{12}{(\pi^{3} \, - \, 4 \, \xi)} \int_{\Omega} e^{- \, i \, k \, \hat{x} \cdot z} \hat{x} \times H^{\mu}(z) \, dz,
\end{equation*}
which by Taylor expansion, for the function $z \, \rightarrow \, e^{- \, i \, k \, \hat{x} \cdot z}$, and the fact that we take $\xi \, \gg \, 1$, we obtain
\begin{equation*}
           E^{\infty}_{eff,-}(\hat{x}) =  - \, \pm \, i \, k \, \xi \, \frac{12}{(\pi^{3} \, - \, 4 \, \xi)} \int_{\Omega} \hat{x} \times H^{\mu}(z) \, dz + \mathcal{O}\left(k^2 \, \left\Vert H^{\mu} \right\Vert_{\mathbb{L}^{2}(\Omega)} \right), 
\end{equation*}
and using the fact that 
\begin{equation*}
    \int_{\Omega} e_{n}^{(1)}(x) \, dx \, = 
    \int_{\Omega} e_{n+1}^{(3)}(x) \, dx \, = \, 0, \quad \text{for} \quad n \geq 1, 
\end{equation*}
we rewrite the previous expression of the far-field like
\begin{equation}\label{FFEI}
               E^{\infty}_{eff,-}(\hat{x}) =  - \, \pm \, i \, k \, \xi \, \frac{12}{(\pi^{3} \, - \, 4 \, \xi)} \int_{\Omega} \hat{x} \times e_{1}^{(3)}(z) \, dz \; \langle H^{\mu}; e_{1}^{(3)} \rangle_{\mathbb{L}^{2}(\Omega)} + \mathcal{O}\left(k^2 \,\left\Vert H^{\mu} \right\Vert_{\mathbb{L}^{2}(\Omega)} \right).
\end{equation}
Now, we need to estimate $\langle H^{\mu}; e_{1}^{(3)} \rangle_{\mathbb{L}^{2}(\Omega)}$. From the L.S.E given by $(\ref{LS-vec})$, we have  
\begin{equation}\label{LSE13}
H^\mu \, + \, \frac{12 \, \xi}{(\pi^3 - 4 \xi)} \,  \nabla {\bf M}(H^\mu) \, = \, i \, k \, H^{Inc} \, + \,\frac{12 \, \xi}{(\pi^3 - 4 \xi)} \, K(H^{\mu}), 
\end{equation} 
where $K(H^{\mu}) :=   - \left( \nabla {\bf M^k} \, - \, \nabla {\bf M} \right)(H^\mu) \, + \, k^2 \, {\bf N^k}(H^\mu)$.  Moreover, from $(\ref{mid-diff})$, we know that
\begin{eqnarray}\label{ANC}
\nonumber
 K(H^{\mu}) &=&  \frac{k^2}{2} \, N(H^\mu) \, + \, \frac{k^2}{2}N^{\prime}(H^\mu) \, + \frac{i \, k^{3}}{6 \, \pi} \int_{\Omega} H^\mu(x) \, dx \\  &+&  \frac{1}{4 \, \pi} \sum_{n \geq 3} \, \frac{(i \, k)^{n+1}}{(n+1)!} \, \int_{\Omega} \nabla \nabla \left\vert \cdot - y \right\vert^{n} \cdot H^\mu(y) \, dy + \frac{k^{2}}{4 \, \pi} \sum_{n \geq 1} \, \frac{(i \, k)^{n+1}}{(n+1)!} \, \int_{\Omega}  \left\vert \cdot - y \right\vert^{n} \, H^\mu(y) \, dy, 
\end{eqnarray}
with $N^{\prime}(\cdot)$ is the operator defined by $(\ref{DefNprimeTHM})$. Hence, from $(\ref{LSE13})$ and the fact that $\lambda_{1}^{(3)}(\Omega) = \dfrac{1}{3}$, we get
\begin{equation*}
    \langle H^{\mu}; e_{1}^{(3)} \rangle_{\mathbb{L}^{2}(\Omega)} \; \frac{\pi^{3}}{\left(\pi^{3} \, - \, 4 \, \xi \right)} \, = \, i \, k \,  \langle H^{Inc}; e_{1}^{3} \rangle_{\mathbb{L}^{2}(\Omega)} \, + \frac{12 \, \xi}{(\pi^3 - 4 \xi)} \, \langle K(H^{\mu}); e_{1}^{(3)} \rangle_{\mathbb{L}^{2}(\Omega)}.
\end{equation*}
Based on $(\ref{Proj2=0})$, we split $H^{\mu} \, = \, \overset{1}{\mathbb{P}}(H^{\mu}) \, + \, \overset{3}{\mathbb{P}}(H^{\mu})$ and we plug it into the previous equation to obtain
\begin{eqnarray*}
        \langle H^{\mu}; e_{1}^{(3)} \rangle_{\mathbb{L}^{2}(\Omega)} \; \frac{\pi^{3}}{\left(\pi^{3} \, - \, 4 \, \xi \right)} &=& i \, k \,  \langle H^{Inc}; e_{1}^{3} \rangle_{\mathbb{L}^{2}(\Omega)} \, + \frac{12 \, \xi}{(\pi^3 - 4 \xi)} \, \langle K(\overset{1}{\mathbb{P}}(H^{\mu})); e_{1}^{(3)} \rangle_{\mathbb{L}^{2}(\Omega)} \\
        &+&  \frac{12 \, \xi}{(\pi^3 - 4 \xi)} \, \sum_{j \geq 1} \, \langle H^{\mu}; e_{j}^{(3)} \rangle_{\mathbb{L}^{2}(\Omega)} \, \langle K(e_{j}^{(3)}); e_{1}^{(3)} \rangle_{\mathbb{L}^{2}(\Omega)}. 
\end{eqnarray*}
This implies, 
\begin{eqnarray}\label{IHIP}
\nonumber
        \langle H^{\mu}; e_{1}^{(3)} \rangle_{\mathbb{L}^{2}(\Omega)} &=&  \frac{\left(\pi^{3} \, - \, 4 \, \xi \right)}{\left( \pi^{3} \, - \, 12 \, \xi \, \langle K(e_{1}^{(3)}); e_{1}^{(3)} \rangle_{\mathbb{L}^{2}(\Omega)} \right)}  i \, k \,  \langle H^{Inc}; e_{1}^{(3)} \rangle_{\mathbb{L}^{2}(\Omega)} \\ \nonumber &+& \frac{12 \, \xi}{\left( \pi^{3} \, - \, 12 \, \xi \, \langle K(e_{1}^{(3)}); e_{1}^{(3)} \rangle_{\mathbb{L}^{2}(\Omega)} \right)}  \, \langle K(\overset{1}{\mathbb{P}}(H^{\mu})); e_{1}^{(3)} \rangle_{\mathbb{L}^{2}(\Omega)} \\
        &+&  \frac{12 \, \xi}{\left( \pi^{3} \, - \, 12 \, \xi \, \langle K(e_{1}^{(3)}); e_{1}^{(3)} \rangle_{\mathbb{L}^{2}(\Omega)} \right)} \, \sum_{j \geq 2} \, \langle H^{\mu}; e_{j}^{(3)} \rangle_{\mathbb{L}^{2}(\Omega)} \, \langle K(e_{j}^{(3)}); e_{1}^{(3)} \rangle_{\mathbb{L}^{2}(\Omega)}. 
\end{eqnarray}
The dispersion equation related to this case will be given by  
\begin{equation*}
    0 \,  =  \, \left( \pi^{3} \, - \, 12 \, \xi \, \langle K(e_{1}^{(3)}); e_{1}^{(3)} \rangle_{\mathbb{L}^{2}(\Omega)} \right),
\end{equation*}    
which can be rewritten using $(\ref{ANC})$, as 
\begin{eqnarray}\label{Eq0408}
\nonumber
    0 & = & \pi^{3} \, - \, 6 \, \xi \, k^2 \, \Bigg[ \langle N(e_{1}^{(3)}); e_{1}^{(3)} \rangle_{\mathbb{L}^{2}(\Omega)} \, + \,  \langle N^{\prime}(e_{1}^{(3)}); e_{1}^{(3)} \rangle_{\mathbb{L}^{2}(\Omega)} \, + \frac{i \, k}{3 \, \pi} \left\vert \int_{\Omega} e_{1}^{(3)}(x) \, dx \right\vert^{2} \\  \nonumber
    & & \qquad \qquad \qquad - \frac{1}{2 \, \pi} \,  \, \sum_{n \geq 3} \, \frac{(i \, k)^{n-1}}{(n+1)!} \, \langle \int_{\Omega} \nabla \nabla \left\vert \cdot - y \right\vert^{n} \cdot e_{1}^{(3)}(y) \, dy; e_{1}^{(3)} \rangle_{\mathbb{L}^{2}(\Omega)} \\ && \qquad \qquad \qquad + \frac{1}{2 \, \pi} \sum_{n \geq 1} \, \frac{(i \, k)^{n+1}}{(n+1)!} \, \langle \int_{\Omega}  \left\vert \cdot - y \right\vert^{n} \, e_{1}^{(3)}(y) \, dy; e_{1}^{(3)} \rangle_{\mathbb{L}^{2}(\Omega)} \Bigg]. 
\end{eqnarray}
\newline \medskip \newline
Let us postpone the computations and mention first the following comment. In contrast to the case of eigenvalues $\lambda_{n}^{(3)}(\Omega) > \dfrac{1}{3}$, where the choice of a lower sign for the tensor ${\bf T}^{\mathring{\mu_r}}$ was important to generate resonances. For the case of $\lambda_{1}^{(3)}(\Omega) \, = \, \dfrac{1}{3}$, we made the choice to take the lower sign for the tensor ${\bf T}^{\mathring{\mu_r}}$. Since $k^2$ is proportional to $\xi$, this choice of sign produces a dispersion equation that does not have a monomial of order one in $\xi$. Thus, the dominant part of the dispersion equation will take the following form 
\begin{equation*}
    a_0 \, + \, a_1 \, \xi^2 \, + \, a_2 \, \xi^{\frac{5}{2}}, \quad \text{with} \quad a_{0}, a_{1}, a_{2} \, \in \mathbb{C}. 
\end{equation*}
Then, by cancelling the quadratic part and taking into account the fact that $\xi \gg 1$, we will have an amplification of the incident field. More details will be given later. \newline \medskip \newline
Now, by going back to $(\ref{Eq0408})$, assuming that 
\begin{equation*}
    \langle \left( N \, + \, N^{\prime}\right)(e_{1}^{(3)}); e_{1}^{(3)} \rangle_{\mathbb{L}^{2}(\Omega)} \, \neq \, 0,
\end{equation*}
the solution, that we denote by $\xi_{1}$, to the following reduced dispersion equation
\begin{equation*}
  0 \, = \,  \pi^{3} \, - \, 6 \, \xi \, k^2 \,  \langle \left( N \, + \,  N^{\prime} \right)(e_{1}^{(3)}); e_{1}^{(3)} \rangle_{\mathbb{L}^{2}(\Omega)}  \, \overset{(\ref{def-xi})}{=}  \, 
        \pi^{3} \, - \, 6 \, \xi^{2} \, \frac{c_{0} \, c_{r}^{3}}{\eta_{0}} \, \langle \left( N \, + \,  N^{\prime} \right)(e_{1}^{(3)}); e_{1}^{(3)} \rangle_{\mathbb{L}^{2}(\Omega)},
\end{equation*}
is given by
\begin{equation}\label{SolRDEq}
    \xi_{1}^{2} \, = \, \frac{ \pi^{3} \, \eta_{0}}{6 \, c_{0} \, c_{r}^{3} \,  \langle \left( N \, + \, N^{\prime} \right)(e_{1}^{(3)}); e_{1}^{(3)}  \rangle_{\mathbb{L}^{2}(\Omega)} }.
\end{equation}
For $\xi \gg 1$, the coefficient of the third term in $(\ref{IHIP})$ is of order one with respect to the frequency $k$. To improve the estimation of $    \langle H^{\mu}; e_{1}^{(3)} \rangle_{\mathbb{L}^{2}(\Omega)}$, in $(\ref{IHIP})$, we need to compute $\langle H^{\mu}; e_{j}^{(3)} \rangle_{\mathbb{L}^{2}(\Omega)}$, for $j \geq 2$. Again, by using the L.S.E given by $(\ref{LSE13})$, we get 
\begin{equation}\label{West}
    \langle H^\mu; e_{j}^{(3)} \rangle_{\mathbb{L}^{2}(\Omega)}   \, = \, \frac{(\pi^3 - 4 \xi)\,   i \, k}{(\pi^3 - 4 \xi) \, + \, 12 \, \xi \, \lambda_{j}^{(3)}(\Omega)}  \, \langle H^{Inc}; e_{j}^{(3)} \rangle_{\mathbb{L}^{2}(\Omega)} \, + \,\, \frac{12 \, \xi}{(\pi^3 - 4 \xi) \, + \, 12 \, \xi \, \lambda_{j}^{(3)}(\Omega)}\, \langle K(H^{\mu}); e_{j}^{(3)} \rangle_{\mathbb{L}^{2}(\Omega)}. 
\end{equation}
Then, by plugging \eqref{West} into $(\ref{IHIP})$, we obtain
\begin{eqnarray*}
        \langle H^{\mu}; e_{1}^{(3)} \rangle_{\mathbb{L}^{2}(\Omega)} &=&  \frac{1}{\left( \pi^{3} \, - \, 12 \, \xi \, \langle K(e_{1}^{(3)}); e_{1}^{(3)} \rangle_{\mathbb{L}^{2}(\Omega)} \right)} \Bigg[\left(\pi^{3} \, - \, 4 \, \xi \right) i k  \langle H^{Inc}; e_{1}^{3} \rangle_{\mathbb{L}^{2}(\Omega)} +  12 \xi  \langle K(\overset{1}{\mathbb{P}}(H^{\mu})); e_{1}^{(3)} \rangle_{\mathbb{L}^{2}(\Omega)} \\ \nonumber
        &+&  12 \, \xi \, \sum_{j \geq 2} \, \frac{(\pi^3 - 4 \xi)\,   i \, k}{(\pi^3 - 4 \xi) \, + \, 12 \, \xi \, \lambda_{j}^{(3)}(\Omega)}  \, \langle H^{Inc}; e_{j}^{(3)} \rangle_{\mathbb{L}^{2}(\Omega)} \,  \, \langle K(e_{j}^{(3)}); e_{1}^{(3)} \rangle_{\mathbb{L}^{2}(\Omega)} \\
        &+& 12 \, \xi  \, \sum_{j \geq 2} \,  \frac{ 12 \, \xi }{(\pi^3 - 4 \xi) \, + \, 12 \, \xi \, \lambda_{j}^{(3)}(\Omega)} \, \langle K(H^{\mu}); e_{j}^{(3)} \rangle_{\mathbb{L}^{2}(\Omega)}   \, \langle K(e_{j}^{(3)}); e_{1}^{(3)} \rangle_{\mathbb{L}^{2}(\Omega)}\Bigg]. 
\end{eqnarray*}
By using the previous formula and going back to $(\ref{FFEI})$, the far-field takes the following form
\begin{equation}\label{EqSN}
    E^{\infty}_{eff,-}(\hat{x})  =   \pm  \, \frac{12 \,  \xi \, k^{2}}{\left( \pi^{3} - 12 \, \xi  \, \langle K\left(e_{1}^{3}\right); e_{1}^{3} \rangle \right)} \int_{\Omega} \hat{x} \times e_{1}^{(3)}(z) \, dz \;  \langle H^{Inc}; e_{1}^{3} \rangle_{\mathbb{L}^{2}(\Omega)} + Err_0 + \mathcal{O}\left(k^2 \, \left\Vert H^{\mu} \right\Vert_{\mathbb{L}^{2}(\Omega)} \right),
\end{equation}
where 
\begin{eqnarray*}
   Err_{0} &:=& \mp \, \frac{i \, k \, 144 \, \xi^{2}}{(\pi^{3}-4 \, \xi) \, \left( \pi^{3} \, - \, 12 \, \xi \, \langle K(e_{1}^{(3)}); e_{1}^{(3)} \rangle_{\mathbb{L}^{2}(\Omega)} \right)} \, \int_{\Omega} \hat{x} \times e_{1}^{(3)}(z) \, dz \, \Bigg[  \langle K(\overset{1}{\mathbb{P}}(H^{\mu})); e_{1}^{(3)} \rangle_{\mathbb{L}^{2}(\Omega)} \\ \nonumber
        &+&   \sum_{j \geq 2} \, \frac{(\pi^3 - 4 \xi)\,   i \, k}{(\pi^3 - 4 \xi) \, + \, 12 \, \xi \, \lambda_{j}^{(3)}(\Omega)}  \, \langle H^{Inc}; e_{j}^{(3)} \rangle_{\mathbb{L}^{2}(\Omega)} \,  \langle K(e_{j}^{(3)}); e_{1}^{(3)} \rangle_{\mathbb{L}^{2}(\Omega)} \\
        &+&   \, \sum_{j \geq 2} \,  \frac{ 12 \, \xi }{(\pi^3 - 4 \xi) \, + \, 12 \, \xi \, \lambda_{j}^{(3)}(\Omega)} \, \langle K(H^{\mu}); e_{j}^{(3)} \rangle_{\mathbb{L}^{2}(\Omega)}   \, \langle K(e_{j}^{(3)}); e_{1}^{(3)} \rangle_{\mathbb{L}^{2}(\Omega)} \Bigg].
\end{eqnarray*}
Now, by taking the modulus on the both sides of the previous formula, we get
\begin{eqnarray*}
  \left\vert Err_{0} \right\vert & \lesssim &  \frac{ k \, \, \xi}{\left\vert \pi^{3} \, - \, 12 \, \xi \, \langle K(e_{1}^{(3)}); e_{1}^{(3)} \rangle_{\mathbb{L}^{2}(\Omega)} \right\vert} \,  \Bigg[  \left\vert \langle K(\overset{1}{\mathbb{P}}(H^{\mu})); e_{1}^{(3)} \rangle_{\mathbb{L}^{2}(\Omega)} \right\vert \\ \nonumber
        &+&   \sum_{j \geq 2} k \, \left\vert \langle H^{Inc}; e_{j}^{(3)} \rangle_{\mathbb{L}^{2}(\Omega)} \right\vert  \, \left\vert \langle K(e_{j}^{(3)}); e_{1}^{(3)} \rangle_{\mathbb{L}^{2}(\Omega)} \right\vert +   \, \sum_{j \geq 2}  \left\vert \langle K(H^{\mu}); e_{j}^{(3)} \rangle_{\mathbb{L}^{2}(\Omega)} \right\vert  \, \left\vert \langle K(e_{j}^{(3)}); e_{1}^{(3)} \rangle_{\mathbb{L}^{2}(\Omega)} \right\vert \Bigg] \\
        & \lesssim &  \frac{ k \, \, \xi}{\left\vert \pi^{3} \, - \, 12 \, \xi \, \langle K(e_{1}^{(3)}); e_{1}^{(3)} \rangle_{\mathbb{L}^{2}(\Omega)} \right\vert} \,  \Bigg[  \left\Vert K(\overset{1}{\mathbb{P}}(H^{\mu})) \right\Vert_{\mathbb{L}^{2}(\Omega)}  + \left( k \,  +   \,   \left\Vert  K(H^{\mu}) \right\Vert_{\mathbb{L}^{2}(\Omega)} \right) \, \left\Vert  K(e_{1}^{(3)})  \right\Vert_{\mathbb{L}^{2}(\Omega)} \Bigg].
\end{eqnarray*}
And, by keeping the dominant part of the operator $K(\cdot)$, which is $\dfrac{k^2}{2} \, \left( N \, + \, N^{\prime} \right)$, see for instance $(\ref{ANC})$, and using the fact that $N^{\prime}(\cdot)$ behaves like $N(\cdot)$ from singularity analysis point of view we get from the previous formula 
\begin{equation*}
    \left\vert Err_{0} \right\vert  \lesssim   \frac{ k \, \, \xi}{\left\vert \pi^{3} \, - \, 12 \, \xi \, \langle K(e_{1}^{(3)}); e_{1}^{(3)} \rangle_{\mathbb{L}^{2}(\Omega)} \right\vert} \,  \left[ k^{2} \, \left\Vert N(\overset{1}{\mathbb{P}}(H^{\mu})) \right\Vert_{\mathbb{L}^{2}(\Omega)}  + \, \left( k^{3} \,  +  k^{4} \,   \left\Vert  N(H^{\mu}) \right\Vert_{\mathbb{L}^{2}(\Omega)} \right) \, \left\Vert  N(e_{1}^{(3)})  \right\Vert_{\mathbb{L}^{2}(\Omega)} \right].
\end{equation*}
Moreover, using the continuity of the Newtonian operator, we obtain 
\begin{equation*}
    \left\vert Err_{0} \right\vert  \lesssim   \frac{ k^{3} \, \, \xi}{\left\vert \pi^{3} \, - \, 12 \, \xi \, \langle K(e_{1}^{(3)}); e_{1}^{(3)} \rangle_{\mathbb{L}^{2}(\Omega)} \right\vert} \,  \left[ \left\Vert \overset{1}{\mathbb{P}}(H^{\mu}) \right\Vert_{\mathbb{L}^{2}(\Omega)}  +  k +  k^{2} \, \left\Vert  H^{\mu} \right\Vert_{\mathbb{L}^{2}(\Omega)} \right].
\end{equation*}
As $\left\Vert H^{\mu} \right\Vert_{\mathbb{L}^{2}(\Omega)} = \left\Vert \overset{1}{\mathbb{P}}(H^{\mu}) \right\Vert_{\mathbb{L}^{2}(\Omega)} + \left\Vert \overset{3}{\mathbb{P}}(H^{\mu}) \right\Vert_{\mathbb{L}^{2}(\Omega)}$ and using the smallness of the frequency $k$, we obtain from the previous equation
\begin{equation}\label{USUK}
    \left\vert Err_{0} \right\vert  \lesssim   \frac{ k^{3} \, \, \xi}{\left\vert \pi^{3} \, - \, 12 \, \xi \, \langle K(e_{1}^{(3)}); e_{1}^{(3)} \rangle_{\mathbb{L}^{2}(\Omega)} \right\vert} \,  \left[ \left\Vert \overset{1}{\mathbb{P}}(H^{\mu}) \right\Vert_{\mathbb{L}^{2}(\Omega)}  +  k +  k^{2} \, \left\Vert \overset{3}{\mathbb{P}}(H^{\mu}) \right\Vert_{\mathbb{L}^{2}(\Omega)} \right].
\end{equation}
To estimate the R.H.S of the above equation, we need to estimate $\left\Vert \overset{1}{\mathbb{P}}(H^{\mu}) \right\Vert_{\mathbb{L}^{2}(\Omega)}$ and $\left\Vert \overset{3}{\mathbb{P}}(H^{\mu}) \right\Vert_{\mathbb{L}^{2}(\Omega)}$. 
\begin{enumerate}
    \item[] 
    \item Estimation of $\left\Vert \overset{1}{\mathbb{P}}(H^{\mu}) \right\Vert_{\mathbb{L}^{2}(\Omega)}$. To achieve this, we go back to the L.S.E, given by $(\ref{LSE13})$, to get 
    \begin{equation*}
    \langle H^\mu; e_{n}^{(1)}\rangle_{\mathbb{L}^{2}(\Omega)} = \, i \, k \, \langle H^{Inc}; e_{n}^{(1)}\rangle_{\mathbb{L}^{2}(\Omega)} \, + \,\frac{12 \, \xi}{(\pi^3 - 4 \xi)} \, \langle K(H^{\mu}); e_{n}^{(1)}\rangle_{\mathbb{L}^{2}(\Omega)}, 
\end{equation*}
which, after taking the modulus, the dominant part of the operator $K(\cdot)$ and the fact that $\xi \gg 1$, we obtain
   \begin{eqnarray*}
   \left\vert \langle H^\mu; e_{n}^{(1)}\rangle_{\mathbb{L}^{2}(\Omega)} \right\vert \, & \lesssim & \, k \, \left\vert \langle H^{Inc}; e_{n}^{(1)}\rangle_{\mathbb{L}^{2}(\Omega)} \right\vert \, + k^{2}  \, \left\vert \langle N(H^{\mu}); e_{n}^{(1)}\rangle_{\mathbb{L}^{2}(\Omega)} \right\vert \\
   \left\Vert \overset{1}{\mathbb{P}}(H^{\mu}) \right\Vert_{\mathbb{L}^{2}(\Omega)} \, & \lesssim & \, k \,     \left\Vert \overset{1}{\mathbb{P}}(H^{Inc}) \right\Vert_{\mathbb{L}^{2}(\Omega)} \, + \, k^{2} \,    \left\Vert H^{\mu} \right\Vert_{\mathbb{L}^{2}(\Omega)} \\
    \left\Vert \overset{1}{\mathbb{P}}(H^{\mu}) \right\Vert_{\mathbb{L}^{2}(\Omega)} \, & \lesssim & \, k \,     \left\Vert \overset{1}{\mathbb{P}}(H^{Inc}) \right\Vert_{\mathbb{L}^{2}(\Omega)} \, + \, k^{2} \,    \left\Vert \overset{3}{\mathbb{P}}(H^{\mu}) \right\Vert_{\mathbb{L}^{2}(\Omega)}
\end{eqnarray*}
    \item[]
    \item Estimation of $\left\Vert \overset{3}{\mathbb{P}}(H^{\mu}) \right\Vert_{\mathbb{L}^{2}(\Omega)}$. To achieve this, we go back to $(\ref{West})$, to get 
    \begin{eqnarray*}
    \left\vert \langle H^\mu; e_{j}^{(3)} \rangle_{\mathbb{L}^{2}(\Omega)}   \right\vert \, & \lesssim & \, k  \, \left\vert \langle H^{Inc}; e_{j}^{(3)} \rangle_{\mathbb{L}^{2}(\Omega)} \right\vert \, + k^{2} \, \left\vert \langle N(H^{\mu}); e_{j}^{(3)} \rangle_{\mathbb{L}^{2}(\Omega)} \right\vert \\
    \left\Vert \overset{3}{\mathbb{P}}(H^{\mu}) \right\Vert_{\mathbb{L}^{2}(\Omega)} \, & \lesssim & \, k \,     \left\Vert \overset{3}{\mathbb{P}}(H^{Inc}) \right\Vert_{\mathbb{L}^{2}(\Omega)} \, + \, k^{2} \, \left\Vert H^{\mu} \right\Vert_{\mathbb{L}^{2}(\Omega)} \\
    \left\Vert \overset{3}{\mathbb{P}}(H^{\mu}) \right\Vert_{\mathbb{L}^{2}(\Omega)} \, & \lesssim & \, k \,     \left\Vert \overset{3}{\mathbb{P}}(H^{Inc}) \right\Vert_{\mathbb{L}^{2}(\Omega)} \, + \, k^{2} \, \left\Vert \overset{1}{\mathbb{P}}(H^{\mu}) \right\Vert_{\mathbb{L}^{2}(\Omega)}. 
\end{eqnarray*}
\end{enumerate}
This gives us 
\begin{equation*}
    \left\Vert \overset{1}{\mathbb{P}}(H^{\mu}) \right\Vert_{\mathbb{L}^{2}(\Omega)} = \mathcal{O}\left( k \right) \quad \text{and} \quad     \left\Vert \overset{3}{\mathbb{P}}(H^{\mu}) \right\Vert_{\mathbb{L}^{2}(\Omega)} = \mathcal{O}\left( k \right).
\end{equation*}
And, by going back to $(\ref{USUK})$, we obtain 
\begin{equation*}
    \left\vert Err_{0} \right\vert  \lesssim   \frac{ k^{4} \, \, \xi}{\left\vert \pi^{3} \, - \, 12 \, \xi \, \langle K(e_{1}^{(3)}); e_{1}^{(3)} \rangle_{\mathbb{L}^{2}(\Omega)} \right\vert},
\end{equation*}
which implies, by recalling $(\ref{EqSN})$,
\begin{eqnarray*}
    E^{\infty}_{eff,-}(\hat{x}) & = &  \pm  \, \frac{12 \,  \xi \, k^{2}}{\left( \pi^{3} - 12 \, \xi  \, \langle K\left(e_{1}^{3}\right); e_{1}^{3} \rangle_{\mathbb{L}^{2}(\Omega)} \right)}  \hat{x} \times \int_{\Omega} e_{1}^{(3)}(z) \, dz \;  \langle H^{Inc}; e_{1}^{3} \rangle_{\mathbb{L}^{2}(\Omega)} \\ &+&  \mathcal{O}\left( \frac{ k^{4} \, \, \xi}{\left\vert \pi^{3} \, - \, 12 \, \xi \, \langle K(e_{1}^{(3)}); e_{1}^{(3)} \rangle_{\mathbb{L}^{2}(\Omega)} \right\vert} \right) \, + \, \mathcal{O}\left(k^{3} \right) \\
    & = &  \pm  \, \frac{12 \,  \xi \, k^{2}}{\left( \pi^{3} - 12 \, \xi  \, \langle K\left(e_{1}^{3}\right); e_{1}^{3} \rangle_{\mathbb{L}^{2}(\Omega)} \right)}  \hat{x} \times \int_{\Omega} e_{1}^{(3)}(z) \, dz \; \left(\theta \times \mathrm{p}\right) \cdot \int_{\Omega} e_{1}^{(3)}(z) \, dz \,  \\ &+& 
    \pm  \, \frac{12 \,  \xi \, k^{2}}{\left( \pi^{3} - 12 \, \xi  \, \langle K\left(e_{1}^{3}\right); e_{1}^{3} \rangle_{\mathbb{L}^{2}(\Omega)} \right)}  \hat{x} \times \int_{\Omega} e_{1}^{(3)}(z) \, dz \;  \int_{\Omega} \left(\theta \times \mathrm{p} \right) \cdot \sum_{n \geq 1} \frac{(i \, k \, \theta \cdot z)^{n}}{n!} \cdot e_{1}^{3}(z) \, dz  \\
    &+& \mathcal{O}\left( \frac{ k^{4} \, \, \xi}{\left\vert \pi^{3} \, - \, 12 \, \xi \, \langle K(e_{1}^{(3)}); e_{1}^{(3)} \rangle_{\mathbb{L}^{2}(\Omega)} \right\vert} \right) \, + \, \mathcal{O}\left(k^{3} \right).
\end{eqnarray*}
The second term on the R.H.S can be estimated as 
\begin{equation*}
    \left\vert \cdots \right\vert \, \lesssim \,   \frac{\xi \, k^{3}}{\left\vert \pi^{3} - 12 \, \xi  \, \langle K\left(e_{1}^{3}\right); e_{1}^{3} \rangle_{\mathbb{L}^{2}(\Omega)} \right\vert}  
\end{equation*}
Hence, 
\begin{eqnarray*}
    E^{\infty}_{eff,-}(\hat{x}) & = &  \pm  \, \frac{12 \,  \xi \, k^{2}}{\left( \pi^{3} - 12 \, \xi  \, \langle K\left(e_{1}^{3}\right); e_{1}^{3} \rangle_{\mathbb{L}^{2}(\Omega)} \right)}  \hat{x} \times \int_{\Omega} e_{1}^{(3)}(z) \, dz \; \left(\theta \times \mathrm{p} \right) \cdot \int_{\Omega} e_{1}^{(3)}(z) \, dz \,  \\ 
    &+& \mathcal{O}\left( \frac{ k^{3} \, \, \xi}{\left\vert \pi^{3} \, - \, 12 \, \xi \, \langle K(e_{1}^{(3)}); e_{1}^{(3)} \rangle_{\mathbb{L}^{2}(\Omega)} \right\vert} \right) \, + \, \mathcal{O}\left(k^{3} \right),
\end{eqnarray*}
which can be rewritten like
\begin{eqnarray}\label{Eq1614}
\nonumber
    E^{\infty}_{eff,-}(\hat{x})    & = &  \pm  \, \frac{12 \,  \xi \, k^{2}}{\left( \pi^{3} - 12 \, \xi  \, \langle K\left(e_{1}^{3}\right); e_{1}^{3} \rangle_{\mathbb{L}^{2}(\Omega)} \right)}  \hat{x} \times \left( \textbf{Q}_{1} \cdot 
 \left(\theta \times \mathrm{p} \right) \right) \\ &+& \, \mathcal{O}\left( \frac{ k^{3} \, \, \xi}{\left\vert \pi^{3} \, - \, 12 \, \xi \, \langle K(e_{1}^{(3)}); e_{1}^{(3)} \rangle_{\mathbb{L}^{2}(\Omega)} \right\vert} \right) \, + \, \mathcal{O}\left(k^{3} \right),
\end{eqnarray}
where the tensor $\textbf{Q}_{1}$ is given by $(\ref{DefQm0Q1})$.
In particular, when $\xi = \xi_1$, with $\xi_1$ is given by $(\ref{SolRDEq})$, the dominant term in $(\ref{Eq1614})$ behaves like $\dfrac{1}{k}$, we recall that the frequency $k$ is taken to be small, and the error term will be of order one. More precisely, we get
\begin{equation*}
    E^{\infty}_{eff,-}(\hat{x})  =  \frac{\pm \, i \, 6 \, \pi}{k \, \left\vert \int_{\Omega} e_{1}^{(3)}(x) \, dx \right\vert^2} \, \hat{x} \times \left( \textbf{Q}_{1} \cdot 
 \left(\theta \times \mathrm{p}\right) \right) \, + \,  \mathcal{O}\left( 1 \right).
\end{equation*}
In addition, as we are dealing with a ball and thanks to \cite[Section 4.5.3]{AhceneThesis}, we know that\footnote{In this particular case of the unit ball, we have 
\begin{equation*}
    \int_{\Omega = B(0,1)} e_{1}^{(3)}(x) \, dx \; = \; \frac{\pi^2}{9} \, \int_{B(0,1)} \phi_{1}(x) \, dx,
\end{equation*}
where the R.H.S is given in $(\ref{PCB})$.}
\begin{equation*}
    \int_{\Omega} e_{1}^{(3)}(x) \, dx \; = \; \begin{pmatrix} 
    0 \\
    \\
- \, i \, \dfrac{2 \, \sqrt{6 \, \pi}}{9} \\ 
\\
\dfrac{2}{3} \, \sqrt{\dfrac{\pi}{3}}  
\end{pmatrix}.
\end{equation*}
This implies 
\begin{equation*}
    \left\vert \int_{\Omega} e_{1}^{(3)}(x) \, dx \right\vert \; = \; \frac{2 \, \sqrt{\pi}}{3} \quad \text{and} \quad  \textbf{Q}_{1} \; = \; \frac{4 \, \pi}{27} \, \textbf{I}.
\end{equation*}
Hence, 
\begin{equation*}
    E^{\infty}_{eff,-}(\hat{x}) \, = \, \pm \, i \, 2 \, \pi \, k^{-1} \, \hat{x} \times \left(\theta \times \mathrm{p} \right)  \, + \,  \mathcal{O}\left( 1 \right),
\end{equation*}
which, combined with $(\ref{add-coro-4})$, gives us the coming formula
\begin{equation*}
    E^{\infty}(\hat{x}) \, = \, \pm \, i \, 2 \, \pi \, k^{-1} \, \hat{x} \times \left(\theta \times \mathrm{p} \right)  \, + \,  \mathcal{O}\left( 1 \right).
\end{equation*}
\bigskip

This concludes the proof of the fifth part and ends the proof of \textbf{Theorem \ref{coro-plas-resonance}}.

\section{A-priori estimates for the point-interaction approximations}\label{sec5:proof}

In this section, we prove the a-priori estimates presented in \textbf{Proposition \ref{prop-es-LS}}. 
We recall the L.S.E introduced in \eqref{LS-1}. By taking the average value with respect to $x$ over $S_m$, there holds
	\begin{eqnarray*}
        \nonumber
		\frac{1}{|S_m|}\int_{S_m}H^{\mathring{\mu_{r}}}(x)\,dx \, &-& \, \frac{\eta_0\, k^2}{\pm \, c_0 \, c_r^{3}}\frac{1}{|S_m|} \, \int_{S_m} \, \left[ - \, \underset{x}\nabla {\bf M^{k}}_{\Omega}\left( {\bf T}^{\mathring{\mu_{r}}} \cdot H^{\mathring{\mu_{r}}} \right) \, + \, k^{2} \, {\bf N^{k}}_{\Omega}\left({\bf T}^{\mathring{\mu_{r}}} \cdot H^{\mathring{\mu_{r}}} \right) \right](x) \, dx \\
		& = & \, i \, k \frac{1}{|S_m|}\int_{S_m} H^{Inc}(x)\,dx.
	\end{eqnarray*}
	 where $\nabla {\bf M^{k}}_{\Omega}$ is the Magnetization operator defined in $\Omega$, ${\bf N^{k}}_{\Omega}$ is the Newtonian operator defined in $\Omega$ and $S_m$ is the largest ball contained in $\Omega_m$. As $\Omega \, = \, \left( \underset{m=1}{\overset{\aleph}{\cup}} \Omega_{m} \right) \cup \left( \Omega \setminus \left( \underset{m=1}{\overset{\aleph}{\cup}} \Omega_{m} \right) \right),$ we rewrite the previous equation as 
     \begin{eqnarray}\label{EqLK1}
        \nonumber
		\frac{1}{|S_m|}\int_{S_m}H^{\mathring{\mu_{r}}}(x)\,dx \, &-& \, \frac{\eta_0\, k^2}{\pm \, c_0 \, c_r^{3}}\frac{1}{|S_m|} \, \int_{S_m} \, \sum_{j=1}^{\aleph} \left[ - \, \underset{x}{\nabla} {\bf M^{k}}_{\Omega_{j}}\left( {\bf T}^{\mathring{\mu_{r}}} \cdot H^{\mathring{\mu_{r}}} \right) \, + \, k^{2} \, {\bf N^{k}}_{\Omega_{j}}\left({\bf T}^{\mathring{\mu_{r}}} \cdot H^{\mathring{\mu_{r}}} \right) \right](x) \, dx \\
        & = & \, i \, k \frac{1}{|S_m|}\int_{S_m} H^{Inc}(x)\,dx \, + \, C_{m},  
	\end{eqnarray}
  where 
  \begin{eqnarray*}
      C_{m} &:=& \, \frac{\eta_0\, k^2}{\pm \, c_0 \, c_r^{3}}\frac{1}{|S_m|} \, \int_{S_m} \, \left[ - \, \underset{x}{\nabla} {\bf M^{k}}_{\Omega \setminus \left( \underset{m=1}{\overset{\aleph}{\cup}} \Omega_{m} \right)} \left( {\bf T}^{\mathring{\mu_{r}}} \cdot H^{\mathring{\mu_{r}}} \right) \, + \, k^{2} \, {\bf N^{k}}_{\Omega \setminus \left( \underset{m=1}{\overset{\aleph}{\cup}} \Omega_{m} \right)}\left({\bf T}^{\mathring{\mu_{r}}} \cdot H^{\mathring{\mu_{r}}} \right) \right](x) \, dx \\
      &=& \, \frac{\eta_0\, k^2}{\pm \, c_0 \, c_{r}^{3}} \, \frac{1}{|S_m|}\int_{\Omega\backslash\underset{m=1}{\overset{\aleph}{\cup}}\Omega_m} \, \left[  - \, \nabla {\bf M^{k}}_{S_{m}}\left( {\bf T}^{\mathring{\mu_{r}}} \cdot H^{\mathring{\mu_{r}}}\right) \, + \, k^{2} \, {\bf N^{k}}_{S_{m}}\left( {\bf T}^{\mathring{\mu_{r}}} \cdot H^{\mathring{\mu_{r}}}\right) \right](x) \, dx.
  \end{eqnarray*}
  Using Taylor expansion for $H^{Inc}(\cdot)$ and splitting $\underset{j=1}{\overset{\aleph}{\cup}} \Omega_{j} = \underset{j=1 \atop j \neq m}{\overset{\aleph}{\cup}} \Omega_{j} \, \cup \, \Omega_{m}$, we rewrite $(\ref{EqLK1})$ as 
  \begin{eqnarray}\label{EqLK2}
        \nonumber
		\frac{1}{|S_m|}\int_{S_m}H^{\mathring{\mu_{r}}}(x)\,dx \, &-& \, \frac{\eta_0\, k^2}{\pm \, c_0 \, c_r^{3}}\frac{1}{|S_m|} \, \int_{S_m} \, \sum_{j=1 \atop j \neq m}^{\aleph} \left[ - \, \underset{x}{\nabla} {\bf M^{k}}_{\Omega_{j}}\left( {\bf T}^{\mathring{\mu_{r}}} \cdot H^{\mathring{\mu_{r}}} \right) \, + \, k^{2} \, {\bf N^{k}}_{\Omega_{j}}\left({\bf T}^{\mathring{\mu_{r}}} \cdot H^{\mathring{\mu_{r}}} \right) \right](x) \, dx \\
        & = & \, i \, k \, H^{Inc}(z_{m}) \, + \, T_{m} \, + \, B_{m} \, + \, C_{m},  
	\end{eqnarray}
 where 
 \begin{eqnarray*}
     T_m \, &:=& \, \frac{ik}{\left\vert S_{m} \right\vert} \int_{S_{m}} \int_{0}^{1} \nabla H^{Inc}(z_{m}+t(x-z_{m})) \cdot (x - z_{m}) \, dt \, dx \\
     B_m \, &:=& \, \frac{\eta_0\, k^2}{\pm \, c_0 \, c_{r}^{3}} \, \frac{1}{|S_m|} \, \int_{S_m} \, \left[ - \, \underset{x}{\nabla} {\bf M^{k}}_{\Omega_{m}} \left( {\bf T}^{\mathring{\mu_{r}}} \cdot H^{\mathring{\mu_{r}}} \right) \, + \, k^{2} \, {\bf N^{k}}_{\Omega_{m}}\left({\bf T}^{\mathring{\mu_{r}}} \cdot H^{\mathring{\mu_{r}}} \right) \right](x) \, dx.
 \end{eqnarray*}
 We add and subtract its discrete version to obtain the second term on the L.H.S of $(\ref{EqLK2})$. More precisely, we get 
	 \begin{eqnarray}\label{eq-ls-aver0}
       \nonumber
	\frac{1}{|S_m|}\int_{S_m}H^{\mathring{\mu_{r}}}(x)\,dx &-& \frac{\eta_0\, k^2}{\pm \, c_0 \, c_{r}^{3}} \, \frac{1}{\left\vert S_{m} \right\vert} \, \sum_{j=1 \atop j \neq m}^\aleph \Upsilon_k(z_m, z_j) \cdot {\bf T}^{\mathring{\mu_{r}}} \cdot \frac{1}{|S_j|}\int_{S_j}H^{\mathring{\mu_{r}}}(y)\,dy \\	&=& \, i \, k \, H^{Inc}(z_m) + A_{m} + T_{m} + B_{m} + C_{m}, 
   \end{eqnarray}
   where
   	 \begin{eqnarray*}
	 	A_m \, &:=& \, \frac{\eta_0\, k^2}{\pm \, c_0 \, c_r^{3}} \frac{1}{|S_m|} \, \int_{S_m} \,  \sum_{j=1 \atop j \neq m}^{\aleph} \left[ - \, \underset{x}{\nabla} {\bf M^{k}}_{\Omega_{j}}\left( {\bf T}^{\mathring{\mu_{r}}} \cdot H^{\mathring{\mu_{r}}} \right) \, + \, k^{2} \, {\bf N^{k}}_{\Omega_{j}}\left({\bf T}^{\mathring{\mu_{r}}} \cdot H^{\mathring{\mu_{r}}} \right) \right](x) \, dx  \\
	 	&-& \frac{\eta_0\, k^2}{\pm \, c_0 \, c_{r}^{3}} \, \frac{1}{\left\vert S_{m} \right\vert} \, \sum_{j=1 \atop j \neq m}^\aleph \Upsilon_k(z_m, z_j) \cdot {\bf T}^{\mathring{\mu_{r}}} \cdot \frac{1}{|S_j|}\int_{S_j}H^{\mathring{\mu_{r}}}(y)\,dy.
	 \end{eqnarray*}
  For shortness, we set the following notation
   \begin{equation}\label{ErrTABC}
       Error_{m}^{\star} := T_{m} + A_{m} + B_{m} + C_{m},
   \end{equation}
   and we rewrite $(\ref{eq-ls-aver0})$ as 
   	 \begin{equation}\label{eq-ls-aver}
	\frac{1}{|S_m|}\int_{S_m}H^{\mathring{\mu_{r}}}(x)\,dx \, - \, \frac{\eta_0\, k^2}{\pm \, c_0 \, c_{r}^{3}} \, \frac{1}{\left\vert S_{m} \right\vert} \, \sum_{j=1 \atop j \neq m}^\aleph \Upsilon_k(z_m, z_j) \cdot {\bf T}^{\mathring{\mu_{r}}} \cdot \frac{1}{|S_j|}\int_{S_j}H^{\mathring{\mu_{r}}}(y)\,dy \, = \, i \, k \, H^{Inc}(z_m) \, + \, Error_{m}^{\star}. 
   \end{equation}
   Our aim is to evaluate $Error_{m}^{\star}$ by estimating $T_m$, $A_m$, $B_m$ and $C_m$.
 \begin{enumerate}
     \item[] 
     \item Estimation of $T_m$. We have, 
     \begin{eqnarray*}
         T_m &:=& \frac{ik}{\left\vert S_{m} \right\vert} \int_{S_{m}} \int_{0}^{1} \nabla H^{Inc}(z_{m}+t(x-z_{m})) \cdot (x - z_{m}) \, dt \, dx \\
        \left\vert T_m \right\vert & \leq & \frac{k}{\left\vert S_{m} \right\vert} \left\Vert  \int_{0}^{1} \nabla H^{Inc}(z_{m}+t(\cdot -z_{m})) \, dt \right\Vert_{\mathbb{L}^{2}(S_{m})} \, \left\Vert  (\cdot - z_{m})  \right\Vert_{\mathbb{L}^{2}(S_{m})} = \mathcal{O}\left( k \, d \right).
     \end{eqnarray*}
     Hence, using the fact $\aleph = \mathcal{O}\left( d^{-3} \right)$, 
     \begin{equation}\label{EstTm}
         \underset{m=1}{\overset{\aleph}{\sum}} \left\vert T_{m} \right\vert^{2} \, \lesssim \, \aleph \, d^{2} = \mathcal{O}\left( d^{-1} \right). 
     \end{equation}
     \item[] 
     \item Estimation of $A_m$. We start by rewriting $A_{m}$ as 
     \begin{equation}\label{Am=Am1+Am2+Am3+Am4}
         A_{m} = \frac{\eta_0\, k^2}{\pm \, c_0 \, c_{r}^{3}} \, \frac{1}{|S_m|} \; \left[A_{m,1} + A_{m,2} + A_{m,3} + A_{m,4} \right],
     \end{equation}
     and then we set and we estimate each term appearing on the R.H.S of \eqref{Am=Am1+Am2+Am3+Am4}, separately. 
     \begin{enumerate}
         \item[] 
         \item Estimation of $A_{m,1}$.
         \begin{equation*}
             A_{m,1} := - \, \sum_{j=1 \atop j \neq m}^{\aleph} \, \int_{\Omega_j} \, \nabla {\bf M}_{S_{m}} \left( {\bf T}^{\mathring\mu_{r}} \cdot \left(H^{\mathring{\mu_{r}}}(\cdot) \, - \, \frac{1}{|S_j|}\int_{S_j}H^{\mathring{\mu_r}}(y)\,dy\right)\, \right)(z) \,dz, 
         \end{equation*}
          Then, by taking the modulus on the both sides of the previous equation, based on \textbf{Corollary \ref{CoroHolderH}}, we get 
        \begin{equation*}
             \left\vert A_{m,1} \right\vert \lesssim \, \left\vert {\bf T}^{\mathring\mu_{r}} \right\vert \; [H^{\mathring{\mu_{r}}}]_{C^{0,\alpha}(\overline{\Omega})} \, d^{\alpha}  \; \sum_{j=1 \atop j \neq m}^{\aleph} \int_{\Omega_j} \int_{S_m} \frac{1}{\left\vert x - z \right\vert^{3}} \,dx\,dz. 
         \end{equation*}
         Indeed, with the Taylor expansion near the centers $(z_m, z_j)$ for the function over the double integral, by using the counting lemma given by \textbf{Lemma \ref{conting1}}, since $|\Omega_j|\sim d^3$, we obtain that 
        \begin{equation*}
             \left\vert A_{m,1} \right\vert \lesssim \, \left\vert {\bf T}^{\mathring\mu_{r}} \right\vert \; [H^{\mathring{\mu_{r}}}]_{C^{0,\alpha}(\overline{\Omega})} \, d^{\alpha}  \; \left\vert S_{m} \right\vert \, \left\vert \log(d) \right\vert. 
         \end{equation*}
         Hence, using the fact that $\aleph = \mathcal{O}\left( d^{-3} \right)$, 
         \begin{equation}\label{EstimationAm1}
             \sum_{m=1}^{\aleph} \frac{1}{\left\vert S_{m} \right\vert^{2}} \; \left\vert A_{m,1} \right\vert^{2} \lesssim  \, \left\vert {\bf T}^{\mathring\mu_{r}} \right\vert^{2} \; [H^{\mathring{\mu_{r}}}]^{2}_{C^{0,\alpha}(\overline{\Omega})} \, d^{2 \, \alpha - 3}  \;  \left\vert \log(d) \right\vert^{2}.
         \end{equation}
         \item[] 
         \item Estimation of $A_{m,2}$. 
         \begin{equation*}
             A_{m,2} \, := \, \sum_{j=1 \atop j \neq m}^{\aleph} \, \int_{\Omega_j} \, \left[ - \, \nabla {\bf M^{k}}_{S_{m}} + k^{2} \, {\bf N^{k}}_{S_{m}} + \nabla {\bf M}_{S_{m}} \right]\left({\bf T}^{\mathring\mu_{r}}\cdot\left(H^{\mathring{\mu_{r}}}(\cdot) \, - \, \frac{1}{|S_j|}\int_{S_j}H^{\mathring{\mu_r}}(y)\,dy\right)\right)\,dx. 
         \end{equation*}
          Regarding the singularity analysis of the kernel of the operator $\left[ - \, \nabla {\bf M^{k}}_{S_{m}} + k^{2} \, {\bf N^{k}}_{S_{m}} + \nabla {\bf M}_{S_{m}} \right]$,  for any $x\in S_m \subset \Omega_m$ and $z\in \Omega_j$, $j\neq m$, we have
	 \begin{equation}\label{Upsk-Ups0}
	 	\left\vert \left(\Upsilon_k-\Upsilon_0\right)(x,z) \right\vert \lesssim \frac{k^2}{d_{mj}},
	 \end{equation}
see \cite[Inequality (2.4.22)]{BouzekriThesis}. Then, by using $(\ref{distribute})$, $(\ref{Upsk-Ups0})$, \textbf{Corollary \ref{CoroHolderH}} and \textbf{Lemma \ref{conting1}}, we obtain
\begin{equation*}
\left\vert A_{m,2} \right\vert \lesssim k^{2} \,  \lvert {\bf T}^{\mathring{\mu_r}} \rvert \; [H^{\mathring{\mu_{r}}}]_{C^{0, \alpha}(\overline{\Omega})}  d^\alpha \, \left\vert S_{m} \right\vert. 
\end{equation*}
Hence, 
    \begin{equation}\label{EstimationAm2}
    \sum_{m=1}^{\aleph} \frac{1}{\left\vert S_{m} \right\vert^{2}} \, \left\vert A_{m,2} \right\vert^{2} \, \lesssim \, k^{4} \,  \lvert {\bf T}^{\mathring{\mu_r}} \rvert^{2} \; [H^{\mathring{\mu_{r}}}]^{2}_{C^{0, \alpha}(\overline{\Omega})} \; d^{2 \, \alpha - 3}. 
    \end{equation}
         \item[]  
         \item Estimation of $A_{m,3}$.
         \begin{equation*}
             A_{m,3} :=  \sum_{j=1 \atop j \neq m}^{\aleph}\left(\int_{\Omega_j}\int_{S_m}\Upsilon_0(x,z)\,dx\,dz \, - d^{3} \, \left\vert S_{m} \right\vert \, \Upsilon_0(z_m, z_j)\right)\cdot{\bf T}^{\mathring\mu_{r}}\cdot\frac{1}{|S_j|}\int_{S_j}H^{\mathring{\mu_r}}(y)\,dy.
         \end{equation*}
         Since $\Upsilon_{0}(x,z)$, for $x$ and $z$ in two disjoint domains, is an harmonic function, we get by the Mean Value Theorem,
         \begin{equation*}
             \int_{S_{m}} \Upsilon_{0}(x,z) \, dx = \left\vert S_{m} \right\vert \, \Upsilon_{0}(z_{m}, z), \quad z \in \Omega_{j}.  
         \end{equation*} 
         This implies, 
                  \begin{equation*}
             A_{m,3} =  \left\vert S_{m} \right\vert \, \sum_{j=1 \atop j \neq m}^{\aleph}\left(\int_{\Omega_j} \Upsilon_0(z_{m}, z) \, dz \, - \, d^{3}  \, \Upsilon_0(z_m, z_j)\right)	\cdot {\bf T}^{\mathring\mu_{r}} \cdot \frac{1}{|S_j|}\int_{S_j}H^{\mathring{\mu_r}}(y)\,dy.
         \end{equation*}
        Now, for any  fixed $m$, with $m \in \left\{  1, \cdots, \aleph \right\}$, knowing that  
              $\left\vert \Omega_{j} \right\vert \sim d^{3}$ and the fact that $|S_m|=|S_j|$, for any $m, j\in \left\{1, \cdots, \aleph\right\}$,  we can further derive that
        \begin{equation*}
             A_{m,3}
              =    \, \sum_{j=1 \atop j \neq m}^{\aleph}\int_{\Omega_j} \left( \Upsilon_0(z_{m},z) -  \Upsilon_0(z_m, z_j)\right)	\, dz \, \cdot {\bf T}^{\mathring\mu_{r}} \cdot \int_{S_j}H^{\mathring{\mu_r}}(y)\,dy,
         \end{equation*}
         Now, for any  fixed $m$, with $m = 1, \cdots, \aleph$, we split $\Omega$ into two parts as 
	  \begin{equation}\notag
	  	\Omega=\Omega_0^m\cup \left(\Omega\backslash \Omega_0^m\right),
	  \end{equation}
	 where $\Omega_0^m$ is a small neighborhood of $\Omega_m$ such that there exists $\beta\in (0,1)$ fulfilling
	 \begin{equation}\label{Omega-0m}
	 	|\Omega_0^m|=\Oh\left(d^{3\beta}\right),\quad
	 	\mathrm{diam}(\Omega_0^m)=\Oh(d^\beta ),
	 \end{equation}
	 see Figure \ref{fig:distribution2} for a schematic illustration.
	 
	\begin{figure}[htbp]
		\centering
		\includegraphics[width=0.48\linewidth]{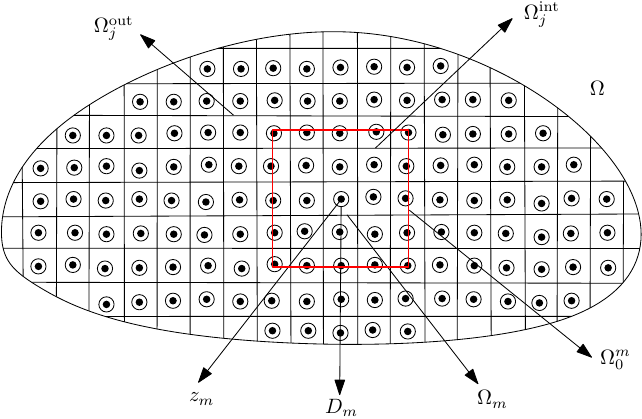}
		\caption{A schematic illustration for the way of splitting $\Omega$ in terms of $\Omega_m$.}
		\label{fig:distribution2}
	\end{figure}

	 Denote the number of those $\Omega_j's$ with $\Omega_j\cap \Omega_0^m \neq \{ \emptyset \}$ by $N_{(1)}$ and the corresponding subdomains as $\Omega_j^{\mathrm{int}}$, $j=1, \cdots, N_{(1)}$. We set the number of the rest of $\Omega_j's$ as ${N}_{(2)}$ and we denote those subdomains as $\Omega_j^{\mathrm{out}}$, then we have
	 \begin{equation}\label{OmegamOmegaint}
	 	\Omega_0^m=\underset{j=1}{\overset{N_{(1)}}{\cup}}\left( \Omega_j^{\mathrm{int}}\cap\Omega_0^m\right)\quad\mbox{and}\quad
	 	\Omega\backslash\Omega_0^m=\left(\underset{j=1}{\overset{{N}_{(2)}}{\cup}}\Omega_j^{\mathrm{out}}\right)\cup \left(\underset{j=1}{\overset{N_{(1)}}{\cup}}\Omega_j^{\mathrm{int}}\backslash \Omega_0^m\right).
	 \end{equation}
  Thus, there holds
  \begin{equation}\label{Am3=Am31+Am32+Am33}
      A_{m,3} = A_{m,3,1} + A_{m,3,2} + A_{m,3,3},  
  \end{equation}
  where 
  \begin{eqnarray*}
      A_{m,3,1} &:=&   \, \sum_{j=1}^{N_{(2)}}\int_{\Omega_j^{out}} \left( \Upsilon_0(z_{m},z) -  \Upsilon_0(z_m, z_j)\right)	\, dz \, \cdot {\bf T}^{\mathring\mu_{r}} \cdot \int_{S_j}H^{\mathring{\mu_r}}(y)\,dy \\
      A_{m,3,2} &:=&  \,  \sum_{j=1 \atop j \neq m}^{N_{(1)}}\int_{\Omega_j^{int} \setminus \Omega_{0}^{m}} \left( \Upsilon_0(z_{m},z) -  \Upsilon_0(z_m, z_j)\right)	\, dz \, \cdot {\bf T}^{\mathring\mu_{r}} \cdot \int_{S_j}H^{\mathring{\mu_r}}(y)\,dy \\
      A_{m,3,3} &:=&  \,  \sum_{j=1 \atop j \neq m}^{N_{(1)}}\int_{\Omega_j^{int} \cap \Omega_{0}^{m}} \left( \Upsilon_0(z_{m},z) -  \Upsilon_0(z_m, z_j)\right)	\, dz \, \cdot {\bf T}^{\mathring\mu_{r}} \cdot \int_{S_j}H^{\mathring{\mu_r}}(y)\,dy.
  \end{eqnarray*}
  Since
\begin{equation}\notag
	\underset{j=1}{\overset{N_{(1)}}{\cup}}\Omega_j^{\mathrm{int}}\backslash\Omega_0^m\subset \Omega\backslash\Omega_0^m\quad\mbox{and}\quad
	N_{(1)}\ll{N}_{(2)}\mbox{ with }{N}_{(2)}\sim d^{-3},
\end{equation}
the estimation of $A_{m,3,2}$ is dominated by the estimation of $A_{m,3,1}$. Then we have 
\begin{equation*}
    \left\vert A_{m,3,1} + A_{m,3,2} \right\vert \; \lesssim \; \left\vert A_{m,3,1} \right\vert.
\end{equation*}
To estimate $A_{m,3,1}$, we use Taylor expansion near $z_{j}$ to obtain  
\begin{equation*}
          A_{m,3,1} =  \,  \sum_{j=1}^{N_{(2)}}\int_{\Omega_j^{out}} 
 \int_{0}^{1} \nabla \Upsilon_0(z_{m},z_{j}+t(z-z_{j})) \cdot (z-z_{j}) \, dt \, dz  \cdot  {\bf T}^{\mathring\mu_{r}} \cdot \int_{S_j}H^{\mathring{\mu_r}}(y)\,dy,
\end{equation*}
which can be estimated as 
\begin{equation*}
  \left\vert  A_{m,3,1} \right\vert \lesssim   \, \sum_{j=1}^{N_{(2)}} \frac{1}{d_{mj}^{4}} 
 \, \int_{\Omega_j^{out}} 
  \left\vert z - z_{j} \right\vert  \, dz \; \left\vert {\bf T}^{\mathring\mu_{r}} \right\vert \; \left\vert S_{j} \right\vert \; \left\Vert H^{\mathring{\mu_r}} \right\Vert_{\mathbb{L}^{\infty}(S_{j})}.
\end{equation*}
Since $\left\vert \Omega_{j}^{out} \right\vert \sim d^{3}$, for $1 \leq j \leq N_{(2)}$, we can further obtain that 
\begin{equation*}
  \left\vert  A_{m,3,1} \right\vert \lesssim  \left\vert S_{m} \right\vert \, d^{4} \; \left\vert {\bf T}^{\mathring\mu_{r}} \right\vert \; \left\Vert H^{\mathring{\mu_r}} \right\Vert_{\mathbb{L}^{\infty}(\Omega)} \, \sum_{j=1}^{N_{(2)}} \frac{1}{d_{mj}^{4}}.
\end{equation*}
From \textbf{Lemma \ref{conting1}}, we know that
\begin{equation*}
 \sum_{j=1}^{N_{(2)}} \frac{1}{d_{mj}^{4}} \lesssim \frac{1}{\left( \underset{m \neq j}{\min} \;\; d_{mj} \right)^{4}}, 
\end{equation*}
where $d_{mj} \geq \underset{m \neq j}{\min} \;\; d_{mj} \gtrsim d^{\beta}$ for $z_m \in \Omega_{m}$ and $z_{j} \in \Omega_{j}^{out}$, $1\leq j\leq N_{(2)}$, which implies, 
\begin{equation*}
 \sum_{j=1}^{N_{(2)}} \frac{1}{d_{mj}^{4}} \lesssim \, d^{-4 \, \beta}.   
\end{equation*}
Then, 
\begin{equation*}
  \left\vert  A_{m,3,1} \right\vert \lesssim  \left\vert S_{m} \right\vert  \, d^{4 - 4 \, \beta} \; \left\vert {\bf T}^{\mathring\mu_{r}} \right\vert \; \left\Vert H^{\mathring{\mu_r}} \right\Vert_{\mathbb{L}^{\infty}(\Omega)}.
\end{equation*}
and thus 
\begin{equation}\label{EstAm31}
 \sum_{m=1}^{\aleph} \frac{1}{\left\vert S_{m} \right\vert^{2}} \; \left\vert  A_{m,3,1} \right\vert^{2} \, \lesssim  \, d^{5 - 8 \, \beta} \; \left\vert {\bf T}^{\mathring\mu_{r}} \right\vert^{2} \; \left\Vert H^{\mathring{\mu_r}} \right\Vert^{2}_{\mathbb{L}^{\infty}(\Omega)},
\end{equation}
due to the relation that $\aleph \sim d^{-3} $.
\medskip
\newline 
For $A_{m,3,3}$, by taking the Taylor expansion of the function $\left( \Upsilon_{0}(z_{m}, z) - \Upsilon_{0}(z_{m},z_{j}) \right)$ at $z=z_j$ and using the fact that $\left\vert \Omega_j^{\mathrm{int}}\cap \Omega_0^m \right\vert \sim d^{3}$, for $1 \leq j \leq N_{(1)}$,we can know that
\begin{equation*}
	\left\vert A_{m,3,3} \right\vert \, \lesssim \, \left\vert S_{m} \right\vert \; \lvert {\bf T}^{\mathring{\mu_{r}}}\rvert \, \lVert H^{\mathring{\mu_{r}}}\rVert_{\mathbb L^\infty(\Omega)} \; d^4 \; \sum_{j=1\atop j\neq m}^{N_{(1)}}\frac{1}{d_{mj}^{4}}, 
\end{equation*}
and\footnote{We have 
\begin{equation*}
    \sum_{m=1}^{\aleph} \, \sum_{j=1\atop j\neq m}^{N_{(1)}} \, \frac{1}{d_{mj}^{8}} \, = \,   \sum_{j=1}^{N_{(1)}} \, \sum_{m=1 \atop m \neq j}^{\aleph} \frac{1}{d_{mj}^{8}}.
\end{equation*}
}
\begin{eqnarray*}
\sum_{m=1}^{\aleph} \frac{\left\vert A_{m,3,3} \right\vert^{2}}{\left\vert S_{m} \right\vert^{2}}  \, & \lesssim & \;  \lvert {\bf T}^{\mathring{\mu_{r}}}\rvert^{2} \, \lVert H^{\mathring{\mu_{r}}}\rVert^{2}_{\mathbb L^\infty(\Omega)} \; d^8 \; N_{(1)} \; \sum_{m=1}^{\aleph}  \, \sum_{j=1\atop j\neq m}^{N_{(1)}}\frac{1}{d_{mj}^{8}} \\
& \lesssim & \,   \lvert {\bf T}^{\mathring{\mu_{r}}}\rvert^{2} \, \lVert H^{\mathring{\mu_{r}}}\rVert^{2}_{\mathbb L^\infty(\Omega)} \; d^{8} \; N_{(1)} \; \,   \sum_{j=1}^{N_{(1)}} \, \sum_{m=1 \atop m \neq j}^{\aleph} \frac{1}{d_{mj}^{8}} \, \lesssim  \,  \lvert {\bf T}^{\mathring{\mu_{r}}}\rvert^{2} \, \lVert H^{\mathring{\mu_{r}}}\rVert^{2}_{\mathbb L^\infty(\Omega)}  \; N_{(1)}^{2},
\end{eqnarray*}
by utilizing \textbf{Lemma} \ref{conting1}.
For $N_{(1)}$, since $\Omega_m\subset\Omega_0^m$ and for any $\Omega_j^{\mathrm{int}}$, with  $1 \leq j \leq 
 N_{(1)}$, we have $\Omega_j^{\mathrm{int}} \cap \Omega_0^m \neq \{ \emptyset \}$, then there exists a positive constant $\delta_0$ such that
\begin{equation}\notag
	\left| \Omega_j^{\mathrm{int}}\cap\Omega_0^m\right|\geq \delta_0\left|\Omega_j^{\mathrm{int}}\right|>0,
\end{equation}
which gives us
\begin{equation}\notag
	d^{3\beta} \overset{(\ref{Omega-0m})}{=} |\Omega_0^m| \overset{(\ref{OmegamOmegaint})}{=} \sum_{j=1}^{N_{(1)}}\left| \Omega_j^{\mathrm{int}}\cap\Omega_0^m\right|\geq N_{(1)} \delta_0 d^3,
\end{equation}
indicating that
\begin{equation}\notag
	N_{(1)}\lesssim \delta_0^{-1}d^{3\beta-3}.
\end{equation}
Then, 
\begin{equation}\label{EstAm33}
\sum_{m=1}^{\aleph} \frac{1}{\left\vert S_{m} \right\vert^{2}} \; \left\vert A_{m,3,3} \right\vert^{2} \, \lesssim \;  \lvert {\bf T}^{\mathring{\mu_{r}}} \rvert^{2} \, \lVert H^{\mathring{\mu_{r}}}\rVert^{2}_{\mathbb L^\infty(\Omega)} \, \delta_{0}^{-2} \; d^{6 \, \beta \, - \, 6}.
\end{equation}
Now, by recalling the expression $A_{m,3}$, given by $(\ref{Am3=Am31+Am32+Am33})$, using the fact that $A_{m,3,2}$ is dominated by $A_{m,3,1}$, together with $(\ref{EstAm31})$ and $(\ref{EstAm33})$, and by taking the parameter\footnote{The parameter $\beta$ is chosen to equal the exponents of the parameter $d$ in the estimations $(\ref{EstAm31})$ and $(\ref{EstAm33})$.} $\beta = \frac{11}{14}$, we obtain
\begin{equation}\label{EstAm3}
\sum_{m=1}^{\aleph} \frac{1}{\left\vert S_{m} \right\vert^{2}} \; \left\vert A_{m,3} \right\vert^{2} \, \lesssim  \;  \lvert {\bf T}^{\mathring{\mu_{r}}} \rvert^{2} \, \lVert H^{\mathring{\mu_{r}}}\rVert^{2}_{\mathbb L^\infty(\Omega)}  \; d^{-\frac{9}{7}}.
\end{equation}

         \item[] 
         \item Estimation of $A_{m,4}$.
         \begin{equation*}
         A_{m,4} := \sum_{j=1 \atop j \neq m}^{\aleph}\left(\int_{\Omega_j}\int_{S_m}\left(\Upsilon_k-\Upsilon_0\right)(x,z)\,dx\,dz \, - \, d^{3} \, \left\vert S_{m} \right\vert \, \left(\Upsilon_k-\Upsilon_0\right)(z_m, z_j)\right) 
          \cdot  \, {\bf T}^{\mathring\mu_{r}} \cdot \frac{1}{|S_j|}\int_{S_j}H^{\mathring{\mu_r}}(y)\,dy. 
     \end{equation*}
     Knowing that $\left\vert \Omega_{j} \right\vert \sim d^{3},$ 
     we rewrite $A_{m, 4}$ as
            \begin{equation}\label{Am4Formula}
         A_{m,4} = c_r^{-3} \, \sum_{j=1 \atop j \neq m}^{\aleph}\left(\int_{\Omega_j}\int_{S_m}\left(\Upsilon_k-\Upsilon_0\right)(x,z)\,dx\,dz - \left\vert \Omega_{j} \right\vert \, \left\vert S_{m} \right\vert \, \left(\Upsilon_k-\Upsilon_0\right)(z_m, z_j)\right) \cdot {\bf T}^{\mathring\mu_{r}} \cdot \frac{1}{|S_j|}\int_{S_j}H^{\mathring{\mu_r}}(y)\,dy. 
     \end{equation}
    By Taylor expansion of the function $\left( \Upsilon_k-\Upsilon_0 \right)(x, z)$ at $(z_m ,z_j)$, we have
	 \begin{equation}\label{CMPTaylor}
	 	(\Upsilon_k-\Upsilon_0)(x,z)=(\Upsilon_k-\Upsilon_0)(z_m,z_j)+R_{m,j}(x,z),
	 \end{equation}
	 where the matrix $R_{m,j}(x, z)$ is given by
     \begin{eqnarray*}\label{|Rmj|}
     \nonumber
     	R_{m, j}(x,z) &:=& \int_{0}^{1}\underset{x}{\nabla}(\Upsilon_k-\Upsilon_0)(z_m+t(x-z_m),z_j) \cdot (x-z_m)\,dt \\ \nonumber &+& \int_{0}^{1}\underset{z}{\nabla}(\Upsilon_k-\Upsilon_0)(z_m,z_j+t(z-z_j)) \cdot (z-z_j)\,dt, 
     \end{eqnarray*}
     with
     \begin{equation}\label{|Rmj|}
         \left\vert 	R_{m, j}(x,z) \right\vert \overset{(\ref{Upsk-Ups0})}{\lesssim} \frac{1}{d_{mj}^{2}} \, \left[ \left\vert x - z_m \right\vert + \left\vert z - z_j \right\vert \right].
     \end{equation}
     Then, by plugging $(\ref{CMPTaylor})$ into $(\ref{Am4Formula})$, we can obtain
     \begin{eqnarray*}
         A_{m,4} &=& \, \sum_{j=1 \atop j \neq m}^{\aleph} \int_{\Omega_j} \int_{S_m} R_{m,j}(x,z)\,dx\,dz  \cdot  \, {\bf T}^{\mathring\mu_{r}} \cdot \frac{1}{|S_j|}\int_{S_j}H^{\mathring{\mu_r}}(y)\,dy \\
        \left\vert A_{m,4} \right\vert &\overset{(\ref{|Rmj|})}{\lesssim}&  \, \left\vert {\bf T}^{\mathring\mu_{r}} \right\vert \, \left\Vert H^{\mathring{\mu_{r}}} \right\Vert_{\mathbb L^{\infty}\left( \Omega \right)} \, \sum_{j=1 \atop j \neq m}^{\aleph} \frac{1}{d_{mj}^{2}} \, \left(\left\vert S_{m} \right\vert \, \int_{\Omega_j} \left\vert z - z_{j} \right\vert \, dz + \left\vert \Omega_{j} \right\vert \int_{S_m} \left\vert x - z_{m} \right\vert \, dx \right)\\
        & \lesssim & \, d \, \left\vert S_{m} \right\vert \, \left\lvert {\bf T}^{\mathring{\mu_{r}}}\right\rvert \; \lVert H^{\mathring{\mu_{r}}} \rVert_{\mathbb L^{\infty}\left( \Omega \right)}, 
 \end{eqnarray*}
by \textbf{Lemma} \ref{conting1}.
 Hence, 
\begin{equation}\label{EstimationAm4}
    \sum_{m=1}^{\aleph} \frac{1}{\left\vert S_{m} \right\vert^{2}} \, \left\vert A_{m,4} \right\vert^{2} \, \lesssim  \, d^{-1} \,  \left\lvert {\bf T}^{\mathring{\mu_{r}}}\right\rvert^{2} \; \lVert H^{\mathring{\mu_{r}}} \rVert^{2}_{\mathbb L^{\infty}\left( \Omega \right)}. 
 \end{equation}
     \end{enumerate}
Therefore, for $A_{m}$, combining with the estimations $(\ref{EstimationAm1}), (\ref{EstimationAm2}), (\ref{EstAm3})$ and $(\ref{EstimationAm4})$, we deduce that 
     \begin{eqnarray}\label{Am}
     \nonumber
        \sum_{m=1}^{\aleph} \left\vert A_{m} \right\vert^{2} \; & \lesssim & \; \sum_{m=1}^{\aleph} \; \sum_{j=1}^{4} \; \frac{1}{\left\vert S_{m} \right\vert^{2}} \; \left\vert A_{m,j} \right\vert^{2} \\
        & \lesssim &   \frac{\eta_0^2 \, k^4}{c_0^2 \, c_r^{6}} \lvert {\bf T}^{\mathring{\mu_{r}}} \rvert^{2}  \; \left( [H^{\mathring{\mu_r}}]^2_{C^{0, \alpha}(\overline{\Omega})} \, d^{2\alpha-3} \, \left\vert \log d \right\vert^2 \, + \, \lVert H^{\mathring{\mu_r}}\rVert_{\mathbb L^{\infty}(\Omega)}^{2} \, d^{-\frac{9}{7}} \right).
\end{eqnarray}
\item[]
\item Estimation of $B_m$. We start by rewriting $B_{m}$ as $B_{m} = B_{m,1} + B_{m,2}$, where  
\begin{eqnarray*}
    B_{m,1} &:=& \frac{\eta_0\, k^2}{\pm c_0}\frac{c_r^{-3}}{|S_m|}  \int_{S_m}\left[- \nabla {\bf M^{k}}_{\Omega_{m}} \, + \, k^{2} \, {\bf N^{k}}_{\Omega_{m}} \right]\left( {\bf T}^{\mathring{\mu_r}} \cdot \left(H^{\mathring{\mu_r}}(\cdot)-\frac{1}{|S_m|}\int_{S_m}H^{\mathring{\mu_r}}(y)\,dy\right) \right)(x) \, dx \\
    B_{m,2} &:=& \frac{\eta_0\, k^2}{\pm c_0}\frac{c_r^{-3}}{|S_m|} \int_{S_m}\left[- \nabla {\bf M^{k}}_{\Omega_{m}} \, + \, k^{2} \, {\bf N^{k}}_{\Omega_{m}} \right]\left( {\bf T}^{\mathring{\mu_r}} \cdot \frac{1}{|S_m|}\int_{S_m}H^{\mathring{\mu_r}}(y)\,dy \right)(x) \,dx.
\end{eqnarray*}
Next, we estimate $B_{m,1}$ and $B_{m,2}$, separately. 
\begin{enumerate}
    \item[] 
    \item Estimation of $B_{m,1}$. Since $B_{m, 1}$ can be written as\footnote{We have used the fact that, for an arbitrary vector field $F$, 
  \begin{equation*}
      \int_{\Omega_{m}} \int_{S_{m}} \Upsilon_0(x, z) \cdot F(z) \, dx \, dz = \int_{S_{m}} \int_{\Omega_{m}} \Upsilon_0(x, z) \cdot F(z) \, dz \, dx = \int_{S_{m}} \nabla \nabla {\bf N}_{\Omega_{m}}\left(F\right)(x) \, dx. 
  \end{equation*} 
  }
	 \begin{eqnarray}\label{ImBBm1}
  \nonumber
	 	B_{m,1} &=& \frac{\eta_0 k^2}{\pm c_0}\frac{c_r^{-3}}{|S_m|} \Bigg[  \int_{S_m} \nabla \nabla {\bf N}_{\Omega_{m}}\left( {\bf T}^{\mathring{\mu_r}} \cdot \left(H^{\mathring{\mu_r}}(\cdot)-\frac{1}{|S_m|}\int_{S_m}H^{\mathring{\mu_r}}(y)\,dy\right) \right)(x) \,dx \\
	 	&+& \int_{S_m} \left[- \nabla {\bf M^{k}}_{\Omega_{m}}  + \nabla {\bf M}_{\Omega_{m}} +  k^{2} \, {\bf N^{k}}_{\Omega_{m}} \right]\left( {\bf T}^{\mathring{\mu_r}} \cdot \left(H^{\mathring{\mu_r}}(\cdot) - \frac{1}{|S_m|}\int_{S_m}H^{\mathring{\mu_r}}(y) dy\right)\right)(x) dx \Bigg].
	 \end{eqnarray}
And, by taking the modulus on the both sides of \eqref{ImBBm1}, we get
  \begin{eqnarray}\label{Imk-B}
  \nonumber
	 \left\vert B_{m,1} \right\vert & \lesssim & \frac{\eta_0 \, k^2}{c_0} \frac{c_r^{-3}}{|S_m|} \, |S_m|^{\frac{1}{q}} \, \left\Vert \nabla \nabla {\bf N} \left( {\bf T}^{\mathring{\mu_r}}\left(H^{\mathring{\mu_r}}(\cdot)-\frac{1}{|S_m|}\int_{S_m}H^{\mathring{\mu_r}}(y)\,dy\right) \chi_{\Omega_{m}}(\cdot)\right) \right\Vert_{\mathbb{L}^{p}(S_{m})}  \\ \nonumber
  & + & \frac{\eta_0 \, k^2}{c_0}c_r^{-3}\lvert {\bf T}^{\mathring{\mu_r}}\rvert[H^{\mathring{\mu_r}}]_{C^{0, \alpha}(\overline{\Omega})} \, d^{\alpha} \, \frac{1}{|S_m|}  \int_{\Omega_m}\int_{S_m} \left\vert \Upsilon_k - \Upsilon_0\right\vert(x, z) \,dx\,dz,
  \end{eqnarray}
  where $p, q>1$ such that $\dfrac{1}{p} \, + \, \dfrac{1}{q} \, = \, 1$. By utilizing the Calderon-Zygmund inequality, see \cite[Theorem 9.9]{GT}, we can further get  
    \begin{eqnarray}\label{Imk-B}
  \nonumber
	 \left\vert B_{m,1} \right\vert & \lesssim & \frac{\eta_0 \, k^2}{c_0} \frac{c_r^{-3}}{|S_m|} \, |S_m|^{\frac{1}{q}} \, \left\Vert  {\bf T}^{\mathring{\mu_r}}\left(H^{\mathring{\mu_r}}(\cdot)-\frac{1}{|S_m|}\int_{S_m}H^{\mathring{\mu_r}}(y)\,dy\right) \chi_{\Omega_{m}} (\cdot)\right\Vert_{\mathbb{L}^{p}(S_{m})}  \\ \nonumber
  & + & \frac{\eta_0 \, k^2}{c_0}c_r^{-3}\lvert {\bf T}^{\mathring{\mu_r}}\rvert[H^{\mathring{\mu_r}}]_{C^{0, \alpha}(\overline{\Omega})} \, d^{\alpha} \, \frac{1}{|S_m|}  \int_{\Omega_m}\int_{S_m} \left\vert \Upsilon_k - \Upsilon_0\right\vert(x, z) \,dx\,dz,
  \end{eqnarray}
  which can be reduced, with \textbf{Corollary $\ref{CoroHolderH}$}, to 
  \begin{equation*}
 	 \left\vert B_{m,1} \right\vert  \lesssim  \frac{\eta_0 \, k^2}{c_0 \, c_r^{3}}   \, \left\lvert  {\bf T}^{\mathring{\mu_r}} \right\rvert [ H^{\mathring{\mu_r}} ]_{C^{0,\alpha}(\overline{\Omega})} \, d^{\alpha} \, \left[ 1 + \frac{1}{|S_m|} \, \int_{\Omega_m}\int_{S_m} \left\vert \Upsilon_k - \Upsilon_0\right\vert(x, z) \,dx\,dz \right].
  \end{equation*}
  Now, similar to $(\ref{Upsk-Ups0})$, for any $x \in S_m \subset \Omega_m$ and $z \in \Omega_m$, we have 
      \begin{equation}\label{add-singu1}
	 	|(\Upsilon_k-\Upsilon_0)(x,z)| \lesssim \frac{k^2}{\left\vert x-z \right\vert},
	 \end{equation}
  hence,  
  \begin{equation*}
      \frac{1}{|S_m|}\int_{\Omega_m}\int_{S_m} \left\vert \Upsilon_k - \Upsilon_0\right\vert(x, z) \,dx\,dz \lesssim \frac{k^2}{|S_m|} \, \int_{\Omega_m}\int_{\Omega_m} \frac{1}{ \left\vert x - z \right\vert} \,dx\,dz \lesssim k^{2} \, d^{2}.
  \end{equation*}
Then, 
 \begin{equation}\label{EstimationBm1}
 	\left\vert B_{m,1} \right\vert  \lesssim  \frac{\eta_0 \, k^2}{c_0} c_r^{-3}  \, \left\lvert  {\bf T}^{\mathring{\mu_r}} \right\rvert [ H^{\mathring{\mu_r}} ]_{C^{0,\alpha}(\overline{\Omega})} \, d^{\alpha} \, \left[ 1 + k^{2} \, d^{2}  \right]  =  \mathcal{O}\left( \frac{\eta_0 \, k^2}{c_0} c_r^{-3}  \, \lvert  {\bf T}^{\mathring{\mu_r}} \rvert [ H^{\mathring{\mu_r}} ]_{C^{0,\alpha}(\overline{\Omega})} \, d^{\alpha}  \right).
  \end{equation}
 \item[]
    \item Estimation of $B_{m,2}$. We have,  
\begin{eqnarray*}
	B_{m,2} &=& \underbrace{- \, \frac{\eta_0 k^2}{\pm \, c_0 \, c_r^{3}} \frac{1}{|S_m|} \, \int_{S_m} \nabla {\bf M}_{\Omega_m} \left( {\bf T}^{\mathring{\mu_r}} \cdot \frac{1}{|S_m|} \int_{S_m}H^{\mathring{\mu_r}}(y)\,dy \, \right)(x) \, dx }_{B_{m, 2, 1}} \\
	&+& \underbrace{\frac{\eta_0 k^2}{\pm \, c_0 \, c_r^{3}} \frac{1}{|S_m|} \, \int_{S_m} \left[- \, \nabla {\bf M^{k}}_{\Omega_m} \, + \, \nabla {\bf M}_{\Omega_m} \, + \, k^{2} \, {\bf N^{k}}_{\Omega_m} \right]\left( {\bf T}^{\mathring{\mu_r}} \cdot \frac{1}{|S_m|}\int_{S_m}H^{\mathring{\mu_r}}(y)\,dy \right)(x) \,dx}_{B_{m, 2, 2}}.
\end{eqnarray*}
For $B_{m,2,2}$, similar to $(\ref{ImBBm1})$, there holds
\begin{eqnarray}\label{Bm22}
\nonumber
\left\vert B_{m,2,2} \right\vert & \lesssim &  \frac{\eta_0 k^4}{c_0}c_r^{-3}\lvert{\bf T}^{\mathring{\mu_r}}\rvert\lVert H^{\mathring{\mu_r}}\rVert_{\mathbb L^{\infty}}\frac{1}{|S_m|}\int_{\Omega_m}\int_{S_m}\frac{1}{|x-z|}\,dx\,dz \\ &\lesssim& \frac{\eta_0 k^4}{c_0}\lvert{\bf T}^{\mathring{\mu_r}}\rvert\lVert H^{\mathring{\mu_r}}\rVert_{\mathbb L^{\infty}} c_r^{-3}\, d^2. 
\end{eqnarray}
For $B_{m,2,1}$, as the function 
\begin{equation*}
    x \longrightarrow  \nabla {\bf M}_{\Omega_m} \left( {\bf T}^{\mathring{\mu_r}} \cdot \frac{1}{|S_m|} \int_{S_m}H^{\mathring{\mu_r}}(y)\,dy \, \right)(x),  
\end{equation*}
is harmonic one, by using the Mean Value Theorem, we get
\begin{equation*}
    \frac{1}{|S_m|} \, \int_{S_m} \nabla {\bf M}_{\Omega_m} \left( {\bf T}^{\mathring{\mu_r}} \cdot \frac{1}{|S_m|} \int_{S_m}H^{\mathring{\mu_r}}(y)\,dy \, \right)(x) \, dx \, = \, \nabla {\bf M}_{\Omega_m} \left( {\bf T}^{\mathring{\mu_r}} \cdot \frac{1}{|S_m|} \int_{S_m}H^{\mathring{\mu_r}}(y)\,dy \, \right)(z_{m}),
\end{equation*}
and we end up with
\begin{eqnarray*}
 B_{m,2,1} \, &=& \, - \, \frac{\eta_0 \, k^2}{\pm \, c_0 \, c_r^{3}} \, \nabla {\bf M}_{\Omega_{m}}\left({\bf T}^{\mathring{\mu_r}} \cdot \frac{1}{|S_m|}\int_{S_m}H^{\mathring{\mu_r}}(y)\,dy \right)(z_{m}) \\
 &=& \, - \, \frac{\eta_0 \, k^2}{\pm \, c_0 \, c_r^{3}} \, \nabla {\bf M}_{\Omega_{m}}\left( \boldsymbol{I} \right)(z_{m}) \cdot \left( {\bf T}^{\mathring{\mu_r}} \cdot \frac{1}{|S_m|}\int_{S_m}H^{\mathring{\mu_r}}(y)\,dy \right),
\end{eqnarray*}
where due to the constant vector nature of ${\bf T}^{\mathring{\mu_r}} \cdot \frac{1}{|S_m|}\int_{S_m}H^{\mathring{\mu_r}}(y)\,dy$, the second equality is result of it. By referring to \cite[Table 1]{Yaghjian}, we know that 
\begin{equation*}
    \nabla {\bf M}_{\Omega_{m}}\left( \boldsymbol{I} \right)(x) \, = \,  \frac{1}{3} \, \boldsymbol{I}, \quad \text{for} \quad x \in \Omega_{m}. 
\end{equation*}
Hence,   
\begin{equation}\label{Jm0-B}
        B_{m,2,1} = -\frac{\eta_0 k^2}{\pm 3c_0}c_r^{-3}{\bf T}^{\mathring{\mu_r}} \cdot \frac{1}{|S_m|}\int_{S_m}H^{\mathring{\mu_r}}(y)\,dy. 
\end{equation}
Thus, for $B_m$, by using the estimates \eqref{EstimationBm1}, \eqref{Bm22} and the formula \eqref{Jm0-B}, we have
\begin{align*}
    B_{m} &= -\frac{\eta_0 k^2}{\pm 3c_0}c_r^{-3}{\bf T}^{\mathring{\mu_r}} \cdot \frac{1}{|S_m|}\int_{S_m}H^{\mathring{\mu_r}}(y)\,dy + B_{m,1} + B_{m,2,2}\notag\\
     &= -\frac{\eta_0 k^2}{\pm 3c_0}c_r^{-3}{\bf T}^{\mathring{\mu_r}} \cdot \frac{1}{|S_m|}\int_{S_m}H^{\mathring{\mu_r}}(y)\,dy + \mathcal{O}\left( \frac{\eta_0 \, k^2}{c_0} c_r^{-3} \, \lvert{\bf T}^{\mathring{\mu_r}}\rvert \, [ H^{\mathring{\mu_r}}]_{C^{0, \alpha}(\overline\Omega)} \, d^{\alpha} \right),
\end{align*} 
where the remainder term admits the following estimation
\begin{eqnarray}\label{ErrBm}
    \sum_{m=1}^{\aleph} \left\vert \mathcal{O}\left( \cdots \right) \right\vert^{2} \,  \lesssim  \, \mathcal{O}\left( \aleph \,  \frac{\eta_0^{2} \, k^4}{c_0^{2}} \, \lvert{\bf T}^{\mathring{\mu_r}}\rvert^{2} \, [H^{\mathring{\mu_r}}]^{2}_{C^{0, \alpha}(\overline\Omega)} \, c_r^{-6} \, d^{2\alpha} \right) =  \mathcal{O}\left( \frac{\eta_0^{2} \, k^4}{c_0^{2}} \, \lvert{\bf T}^{\mathring{\mu_r}}\rvert^{2} \, [ H^{\mathring{\mu_r}}]^{2}_{C^{0, \alpha}(\Omega)} \, c_r^{-6} \, d^{2\alpha-3} \right).
\end{eqnarray} 
\end{enumerate}
\item Estimation of $C_m$. We write $C_m$ as
\begin{eqnarray*} 
	C_m &=& - \, \frac{\eta_0\, k^2}{\pm \, c_0 \, c_r^{3}}\frac{1}{|S_m|}\int_{\Omega\backslash\underset{m=1}{\overset{\aleph}{\cup}}\Omega_m} \nabla {\bf M}_{S_m}\left( {\bf T}^{\mathring{\mu_r}} \cdot H^{\mathring{\mu_r}}\right)(x)\,dx \\
	&+& \frac{\eta_0\, k^2}{\pm \, c_0 \, c_r^{3}} \frac{1}{|S_m|} \int_{\Omega\backslash\underset{m=1}{\overset{\aleph}{\cup}}\Omega_m} \left[- \nabla {\bf M^{k}}_{S_{m}} \, + \, \nabla {\bf M}_{S_{m}} \, + \, k^{2} \, {\bf N^{k}}_{S_{m}}  \right]\left({\bf T}^{\mathring{\mu_r}} \cdot H^{\mathring{\mu_r}} \right)(x)\,dx,
\end{eqnarray*}
which, by utilizing the Mean Value Theorem, can be further formulated as  
\begin{eqnarray}\label{DefCm} 
\nonumber
	C_m &=& - \, \frac{\eta_0\, k^2}{\pm \, c_0 \, c_r^{3}} \, \nabla {\bf M}_{\Omega\backslash\underset{m=1}{\overset{\aleph}{\cup}}\Omega_m}\left({\bf T}^{\mathring{\mu_r}} \cdot H^{\mathring{\mu_r}}\right)(z_{m}) \\
	&+& \frac{\eta_0\, k^2}{\pm \, c_0 \, c_r^{3}} \frac{1}{|S_m|} \int_{\Omega\backslash\underset{m=1}{\overset{\aleph}{\cup}}\Omega_m} \left[- \nabla {\bf M^{k}}_{S_{m}} \, + \, \nabla {\bf M}_{S_{m}} \, + \, k^{2} \, {\bf N^{k}}_{S_{m}}  \right]\left({\bf T}^{\mathring{\mu_r}} \cdot H^{\mathring{\mu_r}} \right)(x)\,dx.
\end{eqnarray}
Since for any $x\in S_m\subset \Omega_m$, $z\in \Omega\backslash\underset{m=1}{\overset{\aleph}{\cup}}\Omega_m$, similar to \eqref{add-singu1}, there holds
\begin{equation*}
	\left\vert (\Upsilon_k-\Upsilon_0)(x,z) \right\vert \lesssim \frac{k^2}{\left\vert x - z \right\vert},
\end{equation*}
then, by taking the modulus on the both sides of $(\ref{DefCm})$, we obtain
\begin{eqnarray}\label{Cm=Cm1+Cm2}
\nonumber
\left\vert C_m \right\vert & \lesssim & \frac{\eta_0\, k^2}{c_0 \, c_r^{3}} \,  \, \int_{\Omega\backslash\underset{m=1}{\overset{\aleph}{\cup}}\Omega_m} \frac{1}{\left\vert z_{m} - z \right\vert^{3}} \; dz \; \lvert {\bf T}^{\mathring{\mu_r}} \rvert \, \left\Vert H^{\mathring{\mu_r}} \right\Vert_{\mathbb{L}^{\infty}(\Omega)}  \\
	&+& \underbrace{\frac{\eta_0\, k^2}{c_0 \, c_r^{3}}\frac{1}{|S_m|}\int_{\Omega\backslash\underset{m=1}{\overset{\aleph}{\cup}}\Omega_m}\int_{S_m} \frac{k^{2}}{\left\vert x - z \right\vert}  \, \lvert {\bf T}^{\mathring{\mu_r}} \rvert \; \left\Vert H^{\mathring{\mu_r}} \right\Vert_{\mathbb{L}^{\infty}(\Omega)} \,dx\,dz}_{C_{m, 2}}.
\end{eqnarray}
To estimate $C_{m, 2}$, we consider the following two cases in terms of the split of  $\Omega\backslash\underset{m=1}{\overset{\aleph}{\cup}}\Omega_m$, for any fixed $m$.
\begin{enumerate}
    \item We denote the set $V_1\subset\Omega\backslash\underset{m=1}{\overset{\aleph}{\cup}}\Omega_m$, such that any $z \in V_1$ locates away from $x\in S_m$. In this case, the function $\dfrac{1}{|x-z|}$ is bounded. Moreover, from \textbf{Remark \ref{rem-vol-d}}, we know that $|V_1|=\Oh(d)$ .
    \item[] 
    \item We denote the set $V_2\subset\Omega\backslash\underset{m=1}{\overset{\aleph}{\cup}}\Omega_m$ such that $x\in S_m$ locates near one of the $\Omega_m's$ touching the boundary $\partial\Omega$ of $V_2$. In this case, it is obvious that there exist a positive constant $c_+$ such that $V_2\subset B(x, c_+d)$. 
    \item[] 
\end{enumerate}
The figure below gives us a schematic illustration of the decomposition 
\begin{equation}\label{V1V2}
 \Omega\backslash\underset{m=1}{\overset{\aleph}{\cup}}\Omega_m = V_{1} \cup V_{2}.   
\end{equation}

\begin{figure}[htbp]
	\centering
	\includegraphics[width=0.7\linewidth]{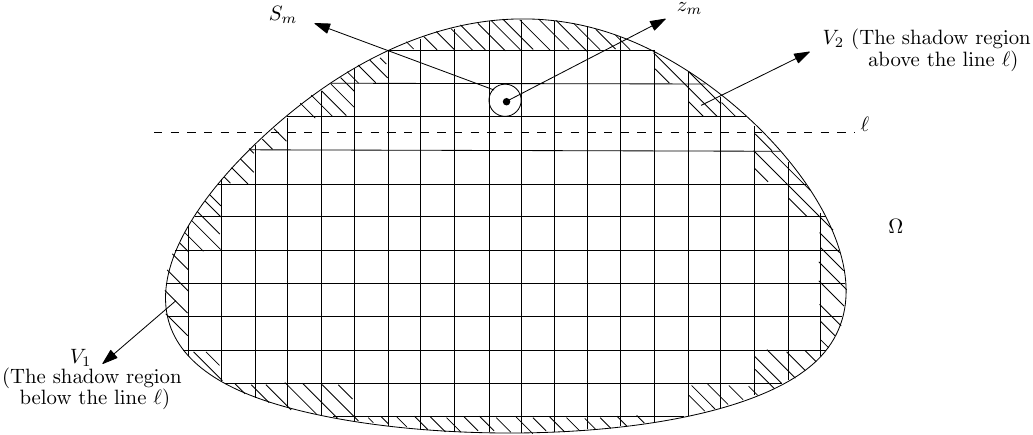}
	\caption{A schematic illustration for the split of the region $\Omega\backslash\underset{m=1}{\overset{\aleph}{\cup}}\Omega_m$.}
	\label{fig:es-cmk-bound}
\end{figure}

Using the decomposition $(\ref{V1V2})$, we get
\begin{eqnarray*}
    \left\vert C_{m,2} \right\vert & \lesssim &  \frac{\eta_0\, k^4}{c_0 \,  c_r^{3}}\frac{1}{|S_m|} \int_{V_1 \cup V_2} \int_{S_m} \frac{1}{|x-z|}\,dx\,dz  \, \lvert {\bf T}^{\mathring{\mu_r}}\rvert\lVert H^{\mathring{\mu_r}}\rVert_{\mathbb L^{\infty}(\Omega)} \\
	& \lesssim & \frac{\eta_0\, k^4}{c_0 \, c_r^{3}}\frac{1}{|S_m|}\left( |V_1| \, |S_m|+ \int_{B(x, c_+d)} \, \int_{S_{m}} \, \frac{1}{|x-z|} \, dx \,dz \, \right) \, \lvert {\bf T}^{\mathring{\mu_r}}\rvert\lVert H^{\mathring{\mu_r}}\rVert_{\mathbb L^{\infty}(\Omega)},
\end{eqnarray*}
where
\begin{equation*}
    \int_{B(x, c_+d)} \, \int_{S_{m}} \, \frac{1}{|x-z|} \, dx \,dz = \mathcal{O}\left( \left\vert S_{m} \right\vert \; d^{2} \right), 
\end{equation*}
together with the fact that $\left\vert V_{1} \right\vert = \mathcal{O}\left( d \right)$, which gives us
\begin{equation*}
 \left\vert C_{m,2} \right\vert	\lesssim \frac{\eta_0\, k^4}{c_0 \, c_r^{3}} \; \lvert {\bf T}^{\mathring{\mu_r}} \rvert \; \lVert H^{\mathring{\mu_r}}\rVert_{\mathbb{L}^{\infty}(\Omega)}  \, d.
\end{equation*}
Now, the estimation for $C_m$ in \eqref{Cm=Cm1+Cm2} becomes
\begin{equation}\label{5.30*}
\left\vert C_m \right\vert  \lesssim  \frac{\eta_0\, k^2}{c_0 \, c_r^{3}}  \, \left[  \int_{\Omega\backslash\underset{m=1}{\overset{\aleph}{\cup}}\Omega_m} \frac{1}{\left\vert z_{m} - z \right\vert^{3}} \; dz + k^{2} \, d \right] \; \left\vert {\bf T}^{\mathring{\mu_r}} \right\vert \, \left\Vert H^{\mathring{\mu_r}} \right\Vert_{\mathbb{L}^{\infty}(\Omega)}.
\end{equation}
It is obvious that the first term on the R.H.S. of \eqref{5.30*} is more dominant than the second one. Then, we can derive from \eqref{5.30*} that
\begin{equation*}
\sum_{m=1}^{\aleph} \left\vert C_m \right\vert^{2}  \lesssim  \; \frac{\eta_0^2\, k^4}{c_0^2 \, c_{r}^{6}} \, \left\vert {\bf T}^{\mathring{\mu_r}} \right\vert^{2} \, \left\Vert H^{\mathring{\mu_r}} \right\Vert^{2}_{\mathbb{L}^{\infty}(\Omega)} \, \sum_{m=1}^{\aleph} \left\vert   \int_{\Omega\backslash\underset{m=1}{\overset{\aleph}{\cup}}\Omega_m} \frac{1}{\left\vert z_{m} - z \right\vert^{3}} \; dz \right\vert^{2}.
\end{equation*}
By utilizing \textbf{Lemma \ref{lem-count-boundary}}, formula $(\ref{count-bound})$, which provides a counting criterion near $\partial\Omega$, we obtain

\begin{equation}\label{Cm0}
	\sum_{m=1}^{\aleph} \left\vert C_{m} \right\vert^{2} \; \lesssim \; \frac{\eta_0^2 \, k^4}{c_0^2 \, c_{r}^{6}}  \, \lvert{\bf T}^{\mathring{\mu_r}} \rvert^{2} \; \lVert H^{\mathring{\mu_r}}\rVert^2_{\mathbb{L}^{\infty}(\Omega)} \; d^{-1}.
\end{equation}
\end{enumerate}

Now, recall $Error_{m}^{\star}$ given by \eqref{ErrTABC}, which can be written as
\begin{equation*}
    Error_{m}^{\star} = -\frac{\eta_0 k^2}{\pm 3c_0}c_r^{-3}{\bf T}^{\mathring{\mu_r}} \cdot \frac{1}{|S_m|}\int_{S_m}H^{\mathring{\mu_r}}(y)\,dy + Error_{m},
\end{equation*}
with 
\begin{equation}\label{DefErrorm}
    Error_{m} := T_{m} + A_{m} + \mathcal{O}\left( \frac{\eta_0 \, k^4}{c_0} \, \lvert{\bf T}^{\mathring{\mu_r}}\rvert \, \lVert H^{\mathring{\mu_r}}\rVert_{\mathbb L^{\infty}(\Omega)} \, c_r^{-3} \, d^{2} \right) + C_{m}. 
\end{equation}
Using the expression of $Error_{m}$, given by \eqref{DefErrorm}, the L.S.E introduced in $(\ref{eq-ls-aver})$ takes the following form
	 \begin{equation}\label{FF}
\digamma_{m} \, - \, \frac{\eta_0\, k^2}{\pm c_0}a^{3-h}\sum_{j=1 \atop j \neq m}^\aleph \Upsilon_k(z_m, z_j) \cdot {\bf P}_{0} \cdot \digamma_{j}  = \, i \, k \, H^{Inc}(z_m)  + Error_{m},  
   \end{equation}
where
\begin{equation*}
 \digamma_{m} :=    \left( \boldsymbol{I} + \frac{\eta_0 k^2}{\pm 3c_0}c_r^{-3}{\bf T}^{\mathring{\mu_r}} \right) \cdot \frac{1}{|S_m|}\int_{S_m}H^{\mathring{\mu_{r}}}(x)\,dx.
\end{equation*}
Taking into account the estimations $(\ref{EstTm})$, $(\ref{Am})$, $(\ref{ErrBm})$ and $(\ref{Cm0})$, it is straightforward to derive the estimate associated with $Error_m$ presented in \eqref{DefErrorm} as
\begin{eqnarray}\label{5.31*}
    \sum_{m=1}^{\aleph} \left\vert Error_{m} \right\vert^{2} \, & \lesssim & \,     \sum_{m=1}^{\aleph} \, \left[  \left\vert T_{m} \right\vert^{2} \, +   \left\vert A_{m} \right\vert^{2} \, +      \left\vert \mathcal{O}(\cdots) \right\vert^{2} \, + \left\vert C_{m} \right\vert^{2} \right] \notag\\ & \lesssim & \frac{\eta_0^2 \, k^4}{c_0^2}\,\lvert  {\bf T}^{\mathring{\mu_{r}}} \rvert^{2} \; c_r^{-6} \; \left( [H^{\mathring{\mu_r}}]^2_{C^{0, \alpha}(\overline{\Omega})} \, d^{2\alpha-3} \, \left\vert \log\left(d\right) \right\vert^2 \, + \, \lVert H^{\mathring{\mu_r}}\rVert_{\mathbb L^{\infty}\left( \Omega \right)}^2 \, d^{-\frac{9}{7}} \right). 
\end{eqnarray}
In addition, by subtracting \eqref{alg-dis-00} from \eqref{FF}, we get 
	\begin{equation}\label{diff-la}
	\left( \digamma_{m} \, - \, U_{m} \, \right) \, - \,  \frac{\eta_0\, k^2}{\pm c_0}a^{3-h}\sum_{j=1 \atop j \neq m}^{\aleph} \Upsilon_k(z_m, z_j) \cdot {\bf P_0} \cdot \left( \digamma_{j} \, - \, U_{j} \right) = {Error}_m, 
	\end{equation}
where ${Error}_m$ is given by $(\ref{DefErrorm})$. We know, from $(\ref{contrast-epsilon})$, that $\eta_{0} = \eta \, a^{2}$ and by plugging this relation into $(\ref{diff-la})$, we derive the same algebraic system as the one given by $(\ref{linear-discrete})$, up to a switching operation\footnote{From a singularity analysis point of view, we have
\begin{equation*}
    \Upsilon_{k}(\cdot,\cdot) \cdot {\bf P_0} \; \sim \; {\bf P_0} \cdot \Upsilon_{k}(\cdot,\cdot). 
\end{equation*}
} between the matrix $\Upsilon_{k}(\cdot,\cdot)$ and the tensor ${\bf P_0}$. To avoid redundancy steps, based on the proof of \cite[Theorem 1.3]{CGS}, we deduce that  
\begin{eqnarray*}
    \left( \sum_{m=1}^{\aleph} \, \left\vert \digamma_{m} \, - \, U_{m} \right\vert^{2} \, \right)^{\frac{1}{2}} \; & \lesssim & \; \frac{c_0\, c_r^3}{c_{r}^{3} \, c_{0} \, - k^{2} \, \eta_{0} \, \left\vert {\bf P_0} \right\vert } \;\; \left( \sum_{m=1}^{\aleph} \, \left\vert Error_{m} \right\vert^{2} \, \right)^{\frac{1}{2}} \\
 & \overset{(\ref{5.31*})}{\lesssim} & \frac{\eta_0\, k^2 \, \left\vert {\bf T}^{\mathring{\mu_r}} \right\vert}{\left( c_0 \, c_r^3 \, - \, k^2 \, \eta_0 \, \lvert{\bf P_0} \rvert \right)} \; \left( [H^{\mathring\mu_{r}}]^{2}_{C^{0, \alpha}(\overline{\Omega})}d^{2 \, \alpha -3} \, \left\vert \log (d) \right\vert^{2}
	 \, + \, \lVert H^{\mathring\mu_{r}} \rVert^{2}_{\mathbb{L}^{\infty}(\Omega)} \,  d^{-\frac{9}{7}} \right)^{\frac{1}{2}}.
	\end{eqnarray*}
This proves \eqref{es-l2} and ends the proof of \textbf{Proposition \ref{prop-es-LS}}.

\section{Quantitative estimates of the $C^{0, \alpha}$-regularity}
Here, we show the detailed proof of the $C^{0, \alpha}$- H\"{o}lder regularity of the solution to the electromagnetic scattering problem for the effective medium \eqref{model-equi} with $\mathring{\mu_{r}}$ being negative definite, which is stated by \textbf{Proposition \ref{prop-regu-neg}}. We end this section by proving  \textbf{Proposition \ref{AddedLemma}}.

\subsection{Proof of Proposition \ref{prop-regu-neg}.}\label{subsec-proof 4.1}
In this subsection, we prove \textbf{Proposition \ref{prop-regu-neg}} for the H\"{o}lder regularity, up to the boundary, of the solution to the electromagnetic scattering problem \eqref{model-equi} with negative definite effective permeability $\mathring{\mu_{r}}$. The proof shall be divided into the following two steps.

\begin{enumerate}
    \item[] 
    \item We show that $E^{\mathring{\epsilon_{r}}}, H^{\mathring{\mu_r}}\in \mathbb L^2(\Omega)$, $\Omega\subset\mathbb{R}^3$, excluding the plasmonic eigenvalues. \newline \\
Remember from $(\ref{compa-bdm})$ that we have
\begin{equation}\label{compa-bdmp}
	\frac{12 \; \xi \; c(k)}{\pi \, \left\vert \pi^{3} \mp 4 \, \xi \right\vert } \, \left\vert \Omega \right\vert \, < \, \min\left\{ 1, \lvert \mathring{\mu_{r}}\rvert, \lvert A_{\mathring{\mu_r}} \rvert \right\}.
\end{equation}
The tensors $\mathring{\mu_{r}}$ and $A_{\mathring{\mu_r}}$  are both characterized by a chosen sign, which allows us to distinguish two cases. In the sequel, we assume that $\xi > 2 \, \delta \, \pi^{3}$, where $\delta$ is a real positive parameter which will be defined later for simplification. It is direct to see from \eqref{ls} and \eqref{us} that for $\xi>2\delta \pi^3$, $\mathring{\mu_r}$ is negative definite.
\begin{enumerate}
    \item[]
    \item By taking the lower sign. The inequality $(\ref{compa-bdmp})$ becomes:
    \begin{equation*}\label{compa1}
	\frac{12 \, \xi \, c(k)}{\pi \, \left( \pi^{3} + 4 \, \xi \right)} \, \left\vert \Omega \right\vert < \min\left\{
	 1;\;\; \frac{8 \, \xi \, - \, \pi^{3}}{4 \, \xi \, + \, \pi^{3}}; \;\; \underset{n}{Inf} \;\; \frac{\left\vert \pi^3 + 4 \, \xi \, \left(1 \, - \, 3 \, \lambda_n^{(3)}(\Omega) \right) \right\vert}{4 \, \xi \, + \, \pi^{3}} \right\}.
\end{equation*}
We denote by $\lambda_{n_{1}}^{(3)}(\Omega)$ the closest eigenvalue from the right to $\dfrac{1}{3}$. Set\footnote{The accumulation point for the sequence $\{ \lambda_{n}^{(3)}(\Omega) \}_{n \in \mathbb{N}}$ is $\dfrac{1}{2}$, so we can infer that $0 < \delta_{\star} < \dfrac{1}{6}$.} $\delta_{\star} := \lambda_{n_{1}}^{(3)}(\Omega) - \dfrac{1}{3}$ and $\delta := \left[ \dfrac{1}{12 \, \delta_{\star}} \right] + 1$, where $[\cdot]$ is the integer part function. The distribution of the sequence of eigenvalues allows us to distinguish two cases. 
\begin{enumerate}
    \item[] 
    \item If $ 0 \, < \, \lambda_n^{(3)}(\Omega) \, \leq \, \dfrac{1}{3}$, we obtain: 
    \begin{equation*}
 \underset{n}{Inf} \; \left\vert \pi^3 + 4 \, \xi \, \left(1 \, - \, 3 \, \lambda_n^{(3)}(\Omega) \right) \right\vert =  \pi^3.   
\end{equation*}
Hence, $(\ref{compa1})$ takes the following form
\begin{equation*}
	\frac{12 \, \xi \, c(k)}{\pi \, \left( \pi^{3} + 4 \, \xi \right)} \, \left\vert \Omega \right\vert \, < \, \min\left\{
	 1;\;\; \frac{8 \, \xi \, - \, \pi^{3}}{4 \, \xi \, + \, \pi^{3}}; \;\; \frac{\pi^3}{4 \, \xi \, + \, \pi^{3}} \right\} = \frac{\pi^3}{4 \, \xi \, + \, \pi^{3}}.
\end{equation*}
Hence, 
\begin{equation}\label{Ineq1}
	 c(k) < \frac{\pi^{4}}{12 \, \xi \, \left\vert \Omega \right\vert}. 
\end{equation}
\item[] 
\item If $\dfrac{1}{3} + \delta_{\star} < \lambda_n^{(3)}(\Omega) \leq 1$, we obtain\footnote{We have used the fact that 
\begin{equation*}
    x < 1 + [x], \quad x \in \mathbb{R}^{+}.
\end{equation*}
}:   
\begin{equation*}
    \pi^{3} - 8 \, \xi \leq \pi^{3} + 4 \, \xi \, \left(1  - 3 \, \lambda_n^{(3)}(\Omega) \right) < \pi^{3} \, \left( 1 - 24 \, \delta_{\star} \, - \, 24 \, \delta_{\star} \, \left[ \frac{1}{12 \, \delta_{\star}} \right] \right) < 0.
\end{equation*}
Consequently,
\begin{equation*}
 \underset{n}{Inf} \; \left\vert \pi^3 + 4 \, \xi \, \left(1 \, - \, 3 \, \lambda_n^{(3)}(\Omega) \right) \right\vert \, = \,\underset{n}{Inf} \;\left(-\, \pi^3\, -\, 4\, \xi\, \left(1\, -\, 3\, \lambda_n^{(3)}(\Omega)\right)\right)
 =\, - \, \pi^3 + 12 \, \xi \, \delta_{\star}.   
\end{equation*} 
Hence, from $(\ref{compa1})$, a sufficient condition can be derived as
    \begin{equation*}
	\frac{12 \, \xi \, c(k)}{\pi \, \left( \pi^{3} + 4 \, \xi \right)} \, \left\vert \Omega \right\vert < \min\left\{
	 1;\;\; \frac{8 \, \xi \, - \, \pi^{3}}{4 \, \xi \, + \, \pi^{3}};  \;\; \frac{- \, \pi^3 + 12 \, \xi \, \delta_{\star}}{4 \, \xi \, + \, \pi^{3}} \right\} = \frac{- \, \pi^3 + 12 \, \xi \, \delta_{\star}}{4 \, \xi \, + \, \pi^{3}}.
\end{equation*}
Then,
\begin{equation}\label{Ineq2}
    c(k) < \frac{\left( - \,\pi^{3} + 12 \, \xi \, \delta_{\star} \right)\, \pi}{12 \, \xi \, \left\vert \Omega \right\vert}.
\end{equation}
\end{enumerate}
By combining with the inequalities $(\ref{Ineq1})$ and $(\ref{Ineq2})$, we deduce that
   \begin{equation}\label{Ineq5}
	 c(k) <  \min\left\{ \frac{\pi^{4}}{12 \, \xi \, \left\vert \Omega \right\vert} ; \frac{\pi \, \left( - \pi^3 + 12 \, \xi \, \delta_{\star} \right)}{12 \, \xi \, \left\vert \Omega \right\vert} \right\} = \frac{\pi^{4}}{12 \, \xi \, \left\vert \Omega \right\vert}.
\end{equation}
    \item[] 
    \item By taking the upper sign. The inequality $(\ref{compa-bdmp})$ becomes:
    \begin{equation*}
	\frac{12 \, \xi \, c(k)}{\pi \, \left(   4 \, \xi \, - \, \pi^{3} \right)} \, \left\vert \Omega \right\vert < \min\left\{
	 1;\;\; \frac{8 \, \xi \, + \, \pi^{3}}{4 \, \xi \, - \, \pi^{3}}; \;\; \underset{n}{Inf} \;\; \frac{\left\vert \pi^3 + 4 \, \xi \, \left(3 \, \lambda_n^{(3)}(\Omega) \, - \, 1 \, \right) \right\vert}{4 \, \xi \, - \, \pi^{3}} \right\},
\end{equation*}
which can be reduced to     \begin{equation}\label{FAF1155}
	\frac{12 \, \xi \, c(k)}{\pi \, \left(   4 \, \xi \, - \, \pi^{3} \right)} \, \left\vert \Omega \right\vert < \min\left\{
	 1; \;\; \underset{n}{Inf} \;\; \frac{\left\vert \pi^3 + 4 \, \xi \, \left(3 \, \lambda_n^{(3)}(\Omega) \, - \, 1 \, \right) \right\vert}{4 \, \xi \, - \, \pi^{3}} \right\},
\end{equation}
since 
\begin{equation*}
    \frac{8\, \xi\, +\, \pi^3}{4\, \xi\, -\, \pi^3}>1.
\end{equation*} 
We denote by $\lambda_{n_{2}}^{(3)}(\Omega)$ the closest eigenvalue from the left to $\dfrac{1}{3}$. We set $\delta_{\star} := \dfrac{1}{3} - \lambda_{n_{2}}^{(3)}(\Omega)$ and we set $\delta := 1 + \left[ \dfrac{1}{12 \, \delta_{\star}} \right]$. Two cases shall be studied as follows. 
\begin{enumerate}
    \item[] 
    \item If $ 0 < \lambda_n^{(3)}(\Omega) \leq \dfrac{1}{3} - \delta_{\star}$, we obtain: 
    \begin{equation*}
        \pi^{3} - 4 \, \xi < \pi^{3} + 4 \, \xi \, \left( 3 \, \lambda_n^{(3)}(\Omega) - \, 1 \right) < \pi^{3} \, \left(1 - 12 \, \delta \, \delta_{\star} \right) < 0.
    \end{equation*}
    Then, 
    \begin{equation*}
 \underset{n}{Inf} \; \left\vert \pi^3 + 4 \, \xi \, \left(3 \, \lambda_n^{(3)}(\Omega) \, - \, 1 \right) \right\vert = \underset{n}{Inf} \; \left(-\, \pi^3 - 4\, \xi\, \left(3\, \lambda_n^{(3)}(\Omega)\, -\, 1\right)\right)
 = - \, \pi^3 + 12 \, \xi \, \delta_{\star}.   
\end{equation*}  
Hence, $(\ref{FAF1155})$ takes the following form
\begin{equation*}\label{ITBF1}
	\frac{12 \, \xi \, c(k)}{\pi \, \left(4 \, \xi \, - \pi^{3} \right)} \, \left\vert \Omega \right\vert < \min\left\{
	 1;\;\; \frac{- \, \pi^3 + 12 \, \xi \, \delta_{\star}}{4 \, \xi \, - \, \pi^{3}} \right\} = \frac{- \, \pi^3 + 12 \, \xi \, \delta_{\star}}{4 \, \xi \, - \, \pi^{3}},
\end{equation*}
where, for the last estimation, we have used the fact that $\delta_{\star} \leq \frac{1}{3}$. Then, 
\begin{equation}\label{Ineq3}
    c(k) < \frac{\pi \, \left( - \, \pi^3 + 12 \, \xi \, \delta_{\star} \right)}{12 \, \xi \, \left\vert \Omega \right\vert}.
\end{equation}
\item[] 
\item If $\dfrac{1}{3} < \lambda_n^{(3)}(\Omega) \leq 1$, we have  
\begin{equation*}
     \underset{n}{Inf} \; \left\vert \pi^3 + 4 \, \xi \, \left(3 \, \lambda_n^{(3)}(\Omega) \, - \, 1 \right) \right\vert =  \, \pi^3.
\end{equation*}
Then, 
\begin{equation*}
	\frac{12 \, \xi \, c(k)}{\pi \, \left(   4 \, \xi \, - \, \pi^{3} \right)} \, \left\vert \Omega \right\vert < \min\left\{
	 1; \; \frac{\pi^3}{4 \, \xi \, - \, \pi^{3}} \right\} = \frac{\pi^3}{4 \, \xi \, - \, \pi^{3}}.
\end{equation*}
Hence, 
\begin{equation}\label{Ineq4}
    	 c(k) \, < \, \frac{\pi^4}{12 \, \xi \, \left\vert \Omega \right\vert}.
\end{equation}
\end{enumerate}
By combining with the inequalities $(\ref{Ineq3})$ and $(\ref{Ineq4})$, as done in $(\ref{Ineq5})$, we deduce 
\begin{equation}\label{Ineq6}
    	 c(k) \, < \, \min\left\{\frac{\pi\left(-\, \pi^3\, +\, 12\, \xi\, \delta_{\star}\right)}{12\, \xi\, |\Omega|},\, \frac{\pi^4}{12\, \xi\, |\Omega|}\right\}\,<\, \frac{\pi^4}{12 \, \xi \, \left\vert \Omega \right\vert}.
\end{equation} 
\item[]
\item For the critical case, where $\lambda_{n}^{(3)} = \frac{1}{3}$ is an eigenvalue of the Magnetization operator, we can prove that $(\ref{Ineq6})$ still holds.
\end{enumerate}
Finally, for both of the lower sign and the upper sign cases, we deduce that under the condition
\begin{equation}\label{cdtxi}
    2 \, \delta \, \pi^{3} < \xi < \frac{\pi^{4}}{12 \, \left\vert \Omega \right\vert \, c(k)},
\end{equation}
 the operator ${\bf I}\pm \xi\nabla{\bf M}{\bf T}^{\mathring{\mu_{r}}}$, in $(\ref{LS-vec})$, dominates the operator ${\bf K}$ and the corresponding coercivity can be attained, which shows that $H^{\mathring{\mu_r}}\in \mathbb{L}^2(\Omega)$. In a similar way, we can prove that $E^{\mathring{\epsilon_{r}}}\in \mathbb{L}^2(\Omega)$.

\item[] 
\item We prove that $E^{\mathring{\epsilon_{r}}}, H^{\mathring{\mu_r}} \in C^{0, \alpha}(\overline{\Omega})$, for $0<\alpha<1$.

We recall, from the electromagnetic scattering problem for the effective medium, that
\begin{equation}\label{effect-medium}
	\begin{cases}
	 \mathrm{Curl} E^{\mathring{\epsilon_{r}}}-i k \mathring{\mu_r} H^{\mathring{\mu_r}}=0 &\mbox{in }\mathbb{R}^3 \\
	 \mathrm{Curl} H^{\mathring{\mu_r}}+i k \mathring{\epsilon_{r}} E^{\mathring{\epsilon_{r}}}=0 &\mbox{in }\mathbb{R}^3 \\
	 \nu\times E^{\mathring{\epsilon_{r}}}_+|_{\partial\Omega}=\nu\times E^{\mathring{\epsilon_{r}}}_-|_{\partial\Omega} \,\, \text{and} \,\, \nu\times H^{\mathring{\mu_{r}}}_+|_{\partial\Omega}=\nu\times H^{\mathring{\mu_{r}}}_-|_{\partial\Omega}& \mbox{on }\partial\Omega,
	\end{cases}
\end{equation}
where 
\begin{equation}\notag
	\mathring{\epsilon_{r}}={\bf I}, \mbox{ in }\mathbb{R}^3, \quad \text{and} \quad
	\mathring{\mu_{r}}=
	\begin{cases}
	{\bf I}+\dfrac{\eta_0 k^2}{\pm c_0}c_r^{-3}{\bf T}^{\mathring{\mu_{r}}}&\mbox{ in }\Omega,\\
                     & \\
	{\bf I}&\mbox{ in }\mathbb{R}^3\backslash \Omega,
	\end{cases}
\end{equation}
with $H_+^{\mathring{\mu_{r}}}$ (resp. $E_+^{\mathring{\epsilon_{r}}}$) and  $H^{\mathring{\mu_{r}}}_-$ (resp. $E^{\mathring{\epsilon_{r}}}_-$) denoting the magnetic (resp. electric) fields outside $\Omega$ and inside $\Omega$, respectively. It is direct  to see that $E_+^{\mathring{\epsilon_{r}}}$ and $H_+^{\mathring{\mu_{r}}}$ are smooth enough for $\mathring{\epsilon_{r}}=\mathring{\mu_{r}}={\bf I}$ in $\mathbb{R}^3\backslash \Omega$. Now, let $R>0$, such that $\Omega\subset B(0, R)\subset B(0, 2R)$, where $B(0, R)$ is the ball centered at the origin with radius $R$. Suppose $\psi_R\in C^\infty(\mathbb{R}^3\backslash \Omega)$ fulfilling that
\begin{equation}\notag
\psi_R=
	\begin{cases}
	1 &\mbox{in}\quad B(0, R)\backslash \Omega,\\
          & \\
	0 &\mbox{in}\quad \mathbb{R}^3\backslash B(0, 2R).
	\end{cases}
\end{equation}
Thus it is easy to get that
\begin{equation}\label{embed-out1}
	\psi_R H^{\mathring{\mu_{r}}_+}\in C^\infty(\mathbb{R}^3\backslash\Omega)\hookrightarrow \mathbb{W}^{k, p}(\mathbb{R}^3\backslash\Omega)\quad\mbox{for any}\quad k, \, p>0,\quad \psi_R H^{\mathring{\mu_{r}}}_+|_{\partial\Omega}=H^{\mathring{\mu_{r}}}_+|_{\partial\Omega},
\end{equation}
where $\mathbb{W}^{k, p}$ is the normal Sobolev space. It suffices to take $k=1$ in \eqref{embed-out1}. Then from the definition of $\psi_R$, by using the Trace Theorem for the trace operator $Tr$ in a bounded domain with Lipschitz boundary, we have
\begin{equation}\notag
\lVert Tr(\psi_RH^{\mathring{\mu_{r}}}_+)\rVert_{\mathbb{W}^{1-\frac{1}{p}, p}(\partial\Omega)}=\lVert H^{\mathring{\mu_{r}}}_+|_{\partial\Omega}\rVert_{\mathbb{W}^{1-\frac{1}{p}, p}(\partial\Omega)}\leq c\lVert\psi_RH^{\mathring{\mu_{r}}}_+\rVert_{\mathbb{W}^{1, p}(\mathbb{R}^3\backslash \Omega)},
\end{equation}
which directly indicates that
\begin{equation}\notag
	H^{\mathring{\mu_{r}}}_+|_{\partial\Omega}\in \mathbb{W}^{1-\frac{1}{p}, p}(\partial\Omega).
\end{equation}
Since $\partial\Omega$ is $C^2$-regular, it is easy to see that in \eqref{effect-medium}
\begin{equation}\notag
	\nu\in C^{1}(\partial\Omega)\subset \mathbb{W}^{1, \infty}(\partial\Omega)\subset \mathbb{W}^{1-\frac{1}{p}, p}(\partial\Omega),\quad\mbox{for any } p.
\end{equation}
Thus, we can obtain that
\begin{equation}\notag
	\nu\times H^{\mathring{\mu_r}}_+|_{\partial\Omega}\in \mathbb{W}^{1-\frac{1}{p}, p}(\partial\Omega).
\end{equation}
By the boundary transmission condition in \eqref{effect-medium}, it yields that
\begin{equation}\label{bd-trans}
	\nu\times H^{\mathring{\mu_{r}}}_-|_{\partial\Omega}=\nu\times H^{\mathring{\mu_{r}}}_+|_{\partial\Omega}\in \mathbb{W}^{1-\frac{1}{p},  p}(\partial\Omega)\quad\mbox{for any }p.
\end{equation}
Then there holds inside $\Omega$ that
\begin{equation}\label{new-effec}
\begin{cases}
	\mathrm{Curl} E^{\mathring{\epsilon_{r}}}_--ik \mathring{\mu_{r}} H^{\mathring{\mu_{r}}}_-=0 &\mbox{in }\Omega,\\
	\mathrm{Curl} H^{\mathring{\mu_{r}}}_-+i k \mathring{\epsilon_{r}} E^{\mathring{\epsilon_{r}}}_-=0&\mbox{in }\Omega,\\
	\nu\times H^{\mathring{\mu_{r}}}_-|_{\partial\Omega} \,\, \text{and} \,\, \nu\times E^{\mathring{\epsilon_{r}}}_-|_{\partial\Omega}\in \mathbb{W}^{1-\frac{1}{p}, p}(\partial\Omega).
\end{cases}
\end{equation}
Recall that under the condition $(\ref{cdtxi})$, we have proved that $H^{\mathring{\mu_{r}}} \in \mathbb L^2(\Omega)$ and $E^{\mathring{\epsilon_{r}}} \in \mathbb L^2(\Omega)$, which indicates by \eqref{new-effec} that
\begin{equation*}
	\mathrm{Curl} H^{\mathring{\mu_{r}}}_- \in \mathbb L^2(\Omega) \,\,\, \text{and} \,\,\, \mathrm{Curl} E^{\mathring{\epsilon_{r}}}_-\in \mathbb L^2(\Omega).
\end{equation*} 
Moreover, by taking advantage of the fact that $\mathring{\mu_{r}}$ is constant, we can further derive from \eqref{new-effec} that $\mathrm{div} H^{\mathring{\mu_{r}}}_{-} =0$.
Therefore, based on the facts that
\begin{equation}\notag
	H^{\mathring{\mu_{r}}}_-, \; \mathrm{Curl} H^{\mathring{\mu_{r}}}_-, \; \mathrm{div} H^{\mathring{\mu_{r}}}_-\in \mathbb L^2(\Omega)\quad\mbox{and}\quad \nu\times H^{\mathring{\mu_{r}}}_-|_{\partial\Omega}\in \mathbb{W}^{1-\frac{1}{p}, p}(\partial\Omega), \mbox{ for any }p,
\end{equation} 
and by utilizing \cite[Corollary 5.3]{AS}, with $p=2$, we deduce that
\begin{equation}\notag
	H^{\mathring{\mu_{r}}}_-\in \mathbb{W}^{1, 2}(\Omega).
\end{equation}
By the Sobolev Embedding Theorem, there holds
\begin{equation}\notag
	H^{\mathring{\mu_{r}}}_-\in \mathbb{W}^{1, 2}(\Omega)\hookrightarrow \mathbb L^6 (\Omega).
\end{equation}
Similarly, from \eqref{new-effec}, we have
\begin{equation}\notag
	H^{\mathring{\mu_{r}}}_-, \; \mathrm{Curl} H^{\mathring{\mu_{r}}}_-, \; \mathrm{div}H^{\mathring{\mu_{r}}}_-\in \mathbb L^6(\Omega).
\end{equation}
Together with \eqref{bd-trans}, we know that $H^{\mathring{\mu_{r}}}_-\in  \mathbb{W}^{1, 6}(\Omega)$ by using \cite[Corollary 5.3]{AS} again. Thus, it is easy to get
\begin{equation}\notag
	H^{\mathring{\mu_{r}}}_- \in \mathbb{W}^{1, 6}(\Omega)\hookrightarrow C^{0, \frac{1}{2}}(\Omega)\subset \mathbb L^p(\Omega), \quad\mbox{for any }p,
\end{equation}
for the bounded domain $\Omega$, by the Sobolev Embedding Theorem, which further indicates by \eqref{new-effec} that
\begin{equation}\notag
	H^{\mathring{\mu_{r}}}_-, \; \mathrm{Curl} H^{\mathring{\mu_{r}}}_-, \; \mathrm{div}H^{\mathring{\mu_{r}}}_-\in \mathbb L^p(\Omega), \quad\mbox{for any }p.
\end{equation}
Combining with \eqref{bd-trans} and the Sobolev Embedding Theorem, we can deduce that
\begin{equation}\notag
	H^{\mathring{\mu_{r}}}_-\in \mathbb{W}^{1, p}(\Omega)\hookrightarrow C^{0, \alpha}(\Omega)\quad\mbox{for any }p>3\mbox{ and } 0<\alpha<1,
\end{equation}
up to the boundary. With a similar argument, we can derive the result that $H^{\mathring{\mu_{r}}}_-, E^{\mathring{\epsilon_{r}}}_-\in C^{0, \alpha}(\overline\Omega)$ for any $0<\alpha<1$. This ends the proof of \textbf{Proposition \ref{prop-regu-neg}}.
\end{enumerate}

\subsection{Proof of \textbf{Proposition \ref{AddedLemma}}}\label{AddedSubSecLemma}
In \textbf{Proposition \ref{prop-regu-neg}}, we have shown the qualitative $C^{0, \alpha}$-regularity property. Here, we provide quantitative estimates of this $C^{0, \alpha}$-regularity in terms of the parameters $k$ and $\beta$ (and $a$ therefore).  To justify $(\ref{ToproveinAppendix})$ we start by recalling the L.S.E given by $(\ref{LS-1})$, 
\begin{eqnarray*}
     H^{\mathring{\mu}_r}(x) \, - \, \frac{\eta_0 \, k^2}{\pm \, c_0 \, c_r^{3}}  \, \left[ - \, \nabla {\bf M^{k}} \left( {\bf T}^{\mathring\mu_{r}} \cdot H^{\mathring{\mu}_r}\right)(x) \, + \, k^{2} \, {\bf N^{k}}\left( {\bf T}^{\mathring\mu_{r}} \cdot H^{\mathring{\mu}_r}\right)(x)\right] \, &=& \, i\, k \, H^{Inc}(x, \theta) \\ \nonumber
 	H^{\mathring{\mu}_r}(x) \, - \, \pm \, \xi  \, {\bf T}^{\mathring\mu_{r}} \, \left[ - \, \nabla {\bf M^{k}} \left( H^{\mathring{\mu}_r}\right)(x) \, + \, k^{2} \,  {\bf N^{k}}\left( H^{\mathring{\mu}_r}\right)(x)\right] \, &=& \, i\, k \, H^{Inc}(x, \theta). 
\end{eqnarray*}
Based on $(\ref{IBA})$ and $(\ref{Equa0758})$, we choose the lower sign to obtain 
\begin{eqnarray}\label{Eq0259}
\label{EqB6}
 	H^{\mathring{\mu}_r} \, + \, \xi  \, {\bf T}^{\mathring\mu_{r}} \, \left[ - \, \nabla {\bf M^{k}}\left( H^{\mathring{\mu}_r}\right) \, + \, k^{2} \, {\bf N^{k}}\left( H^{\mathring{\mu}_r}\right) \right] \, &=& \, i\, k \, H^{Inc}  \\  \nonumber
   	H^{\mathring{\mu}_r} \, + \, \frac{\xi  \, {\bf T}^{\mathring\mu_{r}}}{\left( 1 \, - \, \xi  \, {\bf T}^{\mathring\mu_{r}} \right)}  \, Curl^{2}  {\bf N^{k}}\left( H^{\mathring{\mu}_r} \right) \, &=& \, \frac{i\, k}{\left( 1 \, - \, \xi  \, {\bf T}^{\mathring\mu_{r}} \right)} \, H^{Inc} \\ \nonumber
    H^{\mathring{\mu}_r} \, - \, \frac{\xi  \, {\bf T}^{\mathring\mu_{r}}}{\left( 1 \, - \, \xi  \, {\bf T}^{\mathring\mu_{r}} \right)} \, Curl \, SL^{k}\left( \nu \times H^{\mathring{\mu}_r} \right) \, &=& \, \frac{i\, k}{\left( 1 \, - \, \xi  \, {\bf T}^{\mathring\mu_{r}} \right)} \, H^{Inc} \\ &-& \frac{\xi  \, {\bf T}^{\mathring\mu_{r}}}{\left( 1 \, - \, \xi  \, {\bf T}^{\mathring\mu_{r}} \right)}  \, Curl  {\bf N^{k}}\left( Curl \left( H^{\mathring{\mu}_r} \right) \right), \quad \text{in} \quad \Omega, 
\end{eqnarray}
where $SL^{k}(\cdot)$ is the vectorial Single Layer operator defined by
\begin{equation*}
SL^{k}(E)(x) \, := \, \int_{\partial \Omega} \Phi_{k}(x,y) \, E(y) \, d\sigma(y), \quad x \in \Omega. 
\end{equation*}
And, by going to the boundary and using the fact that 
\begin{equation}\label{CurlSl}
    Curl  \, SL^{k}\left( \nu \times H^{\mathring{\mu}_r} \right) \, = \, \widetilde{\boldsymbol{A}}_{k}\left( \nu \times H^{\mathring{\mu}_r} \right) \, + \, \frac{1}{2} \, \nu \, \times \, \left( \nu \times H^{\mathring{\mu}_r} \right), \quad \text{on} \; \partial \Omega,
\end{equation}
where 
\begin{equation}\label{DefAk}
    \widetilde{\boldsymbol{A}}_{k}\left( \nu \times H^{\mathring{\mu}_r} \right)(x) \, := \, p.v. \, \int_{\partial \Omega} \underset{x}{\nabla} \Phi_{k}(x,y) \, \times \, \left( \nu \times H^{\mathring{\mu}_r} \right)(y) \, d\sigma(y), \quad x \in \partial \Omega,  
\end{equation}
we obtain
\begin{equation}\label{EqAS}
    \frac{\left( 2 \, - \, \xi  \, {\bf T}^{\mathring\mu_{r}} \right)}{2 \, \xi  \, {\bf T}^{\mathring\mu_{r}}} \, \nu \times H^{\mathring{\mu}_r} \, - \, \nu \times  \widetilde{\boldsymbol{A}}_{k}\left( \nu \times H^{\mathring{\mu}_r} \right) \, = \, \frac{i \, k}{\xi  \, {\bf T}^{\mathring\mu_{r}}} \, \nu \times H^{Inc} \, - \, \nu \times Curl {\bf N^{k}} \left( Curl \left(  H^{\mathring{\mu}_r} \right) \right).
\end{equation}
Then, 
\begin{eqnarray*}
  \nu \times H^{\mathring{\mu}_r}  \, &=& \,  \left[     \frac{\left( 2 \, - \, \xi  \, {\bf T}^{\mathring\mu_{r}} \right)}{2 \, \xi  \, {\bf T}^{\mathring\mu_{r}}}  \, - \, \nu \times  \widetilde{\boldsymbol{A}}_{0} \right]^{-1} \, \Bigg[ \nu \times \left(\widetilde{\boldsymbol{A}}_{k} - \widetilde{\boldsymbol{A}}_{0} \right)\left( \nu \times H^{\mathring{\mu}_r} \right) + \frac{i \, k}{\xi  \, {\bf T}^{\mathring\mu_{r}}} \, \nu \times H^{Inc} \\ && \qquad \qquad \qquad \qquad \qquad \qquad \qquad \qquad \qquad - \, \nu \times Curl {\bf N^{k}} \left( Curl \left(  H^{\mathring{\mu}_r} \right) \right) \Bigg].
\end{eqnarray*}
Then, by taking the $\left\Vert \cdot \right\Vert_{\mathbb{L}^{2}(\partial \Omega)}$-norm in both sides, we obtain
\begin{eqnarray}\label{Eqbdist}
\nonumber
  \left\Vert \nu \times H^{\mathring{\mu}_r} \right\Vert_{\mathbb{L}^{2}(\partial \Omega)} \, & \leq & \frac{1}{\text{dist}} \; \Bigg[ \left\Vert \nu \times \left(\widetilde{\boldsymbol{A}}_{k} - \widetilde{\boldsymbol{A}}_{0} \right)\left( \nu \times H^{\mathring{\mu}_r} \right) \right\Vert_{\mathbb{L}^{2}(\partial \Omega)} \, + \, \frac{k}{\xi  \, {\bf T}^{\mathring\mu_{r}}} \, \left\Vert  \nu \times H^{Inc} \right\Vert_{\mathbb{L}^{2}(\partial \Omega)} \\ && \qquad  \qquad \qquad \qquad \qquad \qquad \qquad + \left\Vert \, \nu \times Curl {\bf N^{k}} \left( Curl \left(  H^{\mathring{\mu}_r} \right) \right) \right\Vert_{\mathbb{L}^{2}(\partial \Omega)} \Bigg],
\end{eqnarray}
where 
\begin{equation}\label{distbLemma}
    \text{dist} := \text{dist} \left[ \frac{\left( 2 \, - \, \xi  \, {\bf T}^{\mathring\mu_{r}} \right)}{2 \, \xi  \, {\bf T}^{\mathring\mu_{r}}} \,  ; \, \sigma\left( \nu \times  \widetilde{\boldsymbol{A}}_{0} \right) \right] 
     \overset{(\ref{exp-T-neg})}{=} 
    \text{dist} \left[ \frac{\left( \pi^{3} \, - \, 2 \, \xi \right)}{12 \, \xi} \,  \, ; \, \sigma\left( \nu \times  \widetilde{\boldsymbol{A}}_{0} \right) \right].
\end{equation}
To continue with the estimation of $\text{dist}$, we need to investigate the spectrum of the operator $\nu \times  \widetilde{\boldsymbol{A}}_{0}$ denoted by $\sigma\left( \nu \times \widetilde{\boldsymbol{A}}_{0} \right)$. The coming lemma gives more clarification on this.  
\begin{lemma}
    The following relation holds, 
    \begin{equation}\label{ShiftSpectrum}
        \sigma\left( \nu \times \widetilde{\boldsymbol{A}}_{0} \right) \; = \; \sigma\left( \nabla {\bf M} \right) \; - \; \frac{1}{2},  
    \end{equation}
    where $\nabla {\bf M}$ is the Magnetization operator and $\sigma\left(  \nabla {\bf M} \right)$ is its spectrum. 
\end{lemma}
\begin{proof}
Let $\lambda_n \, \in \, \sigma\left( \nabla {\bf M} \right)$, i.e. $\exists \, V_n \, \in \, \nabla$-Harmonic such that 
\begin{eqnarray*}
    \nabla {\bf M} \left( V_{n} \right) \; &=& \; \lambda_n \, V_{n}, \quad \text{in} \; \Omega. \\
    - \, \nabla \, \div {\bf N} \left( V_{n} \right) \; &=& \; \lambda_n \, V_{n}, \quad \text{in} \; \Omega. \\
    Curl^{2}  \, {\bf N} \left( V_{n} \right) \, + \,  \Delta  \, {\bf N} \left( V_{n} \right) \; &=& \, - \; \lambda_n \, V_{n}, \quad \text{in} \; \Omega. \\
    Curl^{2}  \, {\bf N} \left( V_{n} \right) \, - \,   V_{n}  \; &=& \, - \; \lambda_n \, V_{n}, \quad \text{in} \; \Omega. \\
    Curl^{2}  \, {\bf N} \left( V_{n} \right)  \; &=& \, \left( 1 \, - \, \lambda_n \right) \, V_{n}, \quad \text{in} \; \Omega. \\
    Curl  \, \left[ {\bf N} \left( Curl V_{n} \right) \, - \, SL\left( \nu \times V_{n} \right) \right]   \; &=& \, \left( 1 \, - \, \lambda_n \right) \, V_{n}, \quad \text{in} \; \Omega. \\
    Curl \,SL\left( \nu \times V_{n} \right) \; &=& \, \left( - \, 1 \, + \, \lambda_n \right) \, V_{n}, \quad \text{in} \; \Omega. \\
    \widetilde{\boldsymbol{A}}_{0}\left( \nu \times V_{n} \right) \, + \, \frac{1}{2} \, \nu \, \times \, \left( \nu \times V_{n} \right) \; &\overset{(\ref{CurlSl})}{=}& \, \left( - \, 1 \, + \, \lambda_n \right) \, V_{n}, \quad \text{on} \; \partial \Omega. \\
    \nu \times \widetilde{\boldsymbol{A}}_{0}\left( \nu \times V_{n} \right) \, + \, \frac{1}{2} \, \nu \times \, \left[ \nu \, \times \, \left( \nu \times V_{n} \right) \right] \; &=& \, \left( - \, 1 \, + \, \lambda_n \right) \, \left( \nu \times V_{n} \right), \quad \text{on} \; \partial \Omega. \\
    \nu \times \widetilde{\boldsymbol{A}}_{0}\left( \nu \times V_{n} \right) \, - \, \frac{1}{2} \,  \left( \nu \times V_{n} \right)  \; &=& \, \left( - \, 1 \, + \, \lambda_n \right) \, \left( \nu \times V_{n} \right), \quad \text{on} \; \partial \Omega. \\
    \nu \times \widetilde{\boldsymbol{A}}_{0}\left( \nu \times V_{n} \right)   \; &=& \, \left( - \, \frac{1}{2} \, + \, \lambda_n \right) \, \left( \nu \times V_{n} \right), \quad \text{on} \; \partial \Omega.
\end{eqnarray*}
This concludes the proof of the announced Lemma. 
\end{proof}
Using $(\ref{ShiftSpectrum})$, we rewrite $(\ref{distbLemma})$ as
\begin{eqnarray*}
    \text{dist} \, = \, \text{dist} \left[ \frac{\left( \pi^{3} \, - \, 2 \, \xi \right)}{12 \, \xi} \,  \, ; \, \sigma\left( \nabla {\bf M} \right) \, - \, \frac{1}{2} \right] \, &=& \, \text{dist} \left[ \frac{1}{3} \, + \,  \frac{\pi^{3}}{12 \, \xi} \,  \, ; \, \sigma\left( \nabla {\bf M} \right) \right] \\
& \overset{(\ref{xn0=...+betan0})}{=} & \, \text{dist} \left[ \dfrac{\lambda_{m_{0}}^{(3)}\left( \Omega \right) \, + \, \beta \, \left( \lambda_{m_{0}}^{(3)}\left( \Omega \right) \, - \, \dfrac{1}{3} \right)}{1 \, + \, \beta \, \left(3 \, \lambda_{m_{0}}^{(3)}\left( \Omega \right) \, - \, 1 \right)} \, ; \, \sigma\left( \nabla {\bf M} \right) \right]. 
\end{eqnarray*}
As $\beta$ is taken to be small, the closest eigenvalue to 
\begin{equation*}
\dfrac{\lambda_{m_{0}}^{(3)}\left( \Omega \right) \, + \, \beta \, \left( \lambda_{m_{0}}^{(3)}\left( \Omega \right) \, - \, \dfrac{1}{3} \right)}{1 \, + \, \beta \, \left(3 \, \lambda_{m_{0}}^{(3)}\left( \Omega \right) \, - \, 1 \right)}, 
\end{equation*}
is given by $\lambda_{m_{0}}^{(3)}\left( \Omega \right)$. Hence, 
\begin{equation}\label{Estdist}
\text{dist} \; = \; \text{dist} \left[  \dfrac{\lambda_{m_{0}}^{(3)}\left( \Omega \right) \, + \, \beta \, \left( \lambda_{m_{0}}^{(3)}\left( \Omega \right) \, - \, \dfrac{1}{3} \right)}{1 \, + \, \beta \, \left(3 \, \lambda_{m_{0}}^{(3)}\left( \Omega \right) \, - \, 1 \right)} \, ; \, \lambda_{m_{0}}^{(3)}\left( \Omega \right) \right] = \, \frac{\left\vert \beta \right\vert \, \left( 3 \, \lambda_{m_{0}}^{(3)}\left( \Omega \right) \, - \, 1 \right)^{2}}{3 \, \left[ 1 \, + \, \beta \, \left(3 \, \lambda_{m_{0}}^{(3)}\left( \Omega \right) \, - \, 1 \right) \right]} \, = \, \mathcal{O}\left( \left\vert \beta \right\vert \right). 
\end{equation}
Then, by going back to $(\ref{Eqbdist})$, and using $(\ref{Estdist})$, we derive that 
\begin{eqnarray}\label{EqB13}
\nonumber
  \left\Vert \nu \times H^{\mathring{\mu}_r} \right\Vert_{\mathbb{L}^{2}(\partial \Omega)} \, & \leq & \frac{3 \, \left[ 1 \, + \, \beta \, \left(3 \, \lambda_{m_{0}}^{(3)}\left( \Omega \right) \, - \, 1 \right) \right]}{\left\vert \beta \right\vert \, \left( 3 \, \lambda_{m_{0}}^{(3)}\left( \Omega \right) \, - \, 1 \right)^{2}} \; \Bigg[ \left\Vert \nu \times \left(\widetilde{\boldsymbol{A}}_{k} - \widetilde{\boldsymbol{A}}_{0} \right)\left( \nu \times H^{\mathring{\mu}_r} \right) \right\Vert_{\mathbb{L}^{2}(\partial \Omega)} \\ &+& \, \frac{k}{\xi  \, {\bf T}^{\mathring\mu_{r}}} \, \left\Vert  \nu \times H^{Inc} \right\Vert_{\mathbb{L}^{2}(\partial \Omega)} + \left\Vert \, \nu \times Curl {\bf N^{k}} \left( Curl \left(  H^{\mathring{\mu}_r} \right) \right) \right\Vert_{\mathbb{L}^{2}(\partial \Omega)} \Bigg].
\end{eqnarray}
We need to estimate the first term on the R.H.S. To do this, from $(\ref{DefAk})$, we have   
\begin{eqnarray*}
    \nu \times \left( \widetilde{\boldsymbol{A}}_{k} \, - \, \widetilde{\boldsymbol{A}}_{0} \right)\left( \nu \times H^{\mathring{\mu}_r} \right)(x) \, &=& \, \nu \times \int_{\partial \Omega} \underset{x}{\nabla} \left( \Phi_{k} \, - \, \Phi_{0}\right)(x,y) \, \times \, \left( \nu \times H^{\mathring{\mu}_r} \right)(y) \, d\sigma(y) \\
     &=& \, - \frac{k^{2}}{8 \, \pi} \, \nu \times \int_{\partial \Omega}   \underset{x}{\nabla} \left\vert x - y \right\vert  \, \times \, \left( \nu \times H^{\mathring{\mu}_r} \right)(y) \, d\sigma(y) \\
     &+&  \frac{1}{4 \, \pi} \, \nu \times \int_{\partial \Omega} \left(\sum_{j \geq 3} \frac{(i \, k)^{j}}{j!} \underset{x}{\nabla} \left\vert x - y \right\vert^{j-1} \right) \times \, \left( \nu \times H^{\mathring{\mu}_r} \right)(y) \, d\sigma(y). 
\end{eqnarray*}
Next, we set $\beth$ to be the operator defined by 
\begin{eqnarray*}
    \beth \; : \; \mathbb{L}^{2}\left(\partial \Omega \right) & \longrightarrow & \mathbb{L}^{2}\left(\partial \Omega \right) \\ 
    E & \longrightarrow & \beth \left( E \right)(\cdot) :=  \, \nu \times \int_{\partial \Omega}   \nabla \left\vert \cdot - y \right\vert  \, \times \, E(y) \, d\sigma(y).
\end{eqnarray*}
Thus, 
\begin{eqnarray*}
   \left\Vert \nu \times \left( \widetilde{\boldsymbol{A}}_{k} \, - \, \widetilde{\boldsymbol{A}}_{0} \right)\left( \nu \times H^{\mathring{\mu}_r} \right) \right\Vert_{\mathbb{L}^{2}(\partial \Omega)} 
     & \leq &  \frac{k^{2}}{8 \, \pi} \, \left\Vert \beth \right\Vert_{\mathcal{L}\left(\mathbb{L}^{2}(\partial \Omega); \mathbb{L}^{2}(\partial \Omega)\right)} \, \left\Vert  \nu \times H^{\mathring{\mu}_r}  \right\Vert_{\mathbb{L}^{2}(\partial \Omega)} \\
     &+& k^{3} \, C_{1}(\Omega, k) \, \left\Vert \nu \times H^{\mathring{\mu}_r} \right\Vert_{\mathbb{L}^{2}(\partial \Omega)},
\end{eqnarray*}
where 
\begin{equation*}
    C_{1}(\Omega, k) \, = \, \frac{1}{4 \ \pi} \, \left[ \int_{\partial \Omega} \, \int_{\partial \Omega} \left\vert \sum_{j \geq 3} \frac{i^{j} \, k^{j-3}}{j !} \nabla \left\vert x - y \right\vert^{j-1} \right\vert^{2} \, d\sigma(y) \, d\sigma(x) \right],
\end{equation*}
which is of order one with respect to the parameter $k$. In addition, we have the following estimation
\begin{eqnarray*}
\left\Vert \beth \left( E \right) \right\Vert^{2}_{\mathbb{L}^{2}(\partial \Omega)} \, &=&  \, \int_{\partial \Omega} \left\vert \nu(x) \times \int_{\partial \Omega}   \nabla \left\vert x - y \right\vert  \, \times \, E(y) \, d\sigma(y) \, \right\vert^{2} \, d \sigma(x) \\
& \leq &  \, \int_{\partial \Omega} \left\vert  \int_{\partial \Omega}   \nabla \left\vert x - y \right\vert  \, \times \, E(y) \, d\sigma(y) \, \right\vert^{2} \, d \sigma(x) \\
& \leq &  \, \int_{\partial \Omega}   \int_{\partial \Omega}  \left\vert \nabla \left\vert x - y \right\vert \, \right\vert^{2} \, d\sigma(y)  \, d \sigma(x) \; \left\Vert E \right\Vert^{2}_{\mathbb{L}^{2}(\partial \Omega)}  = \, \left\vert \partial \Omega \right\vert^{2}  \; \left\Vert E \right\Vert^{2}_{\mathbb{L}^{2}(\partial \Omega)}. 
\end{eqnarray*}
Then, 
\begin{equation*}
    \left\Vert \beth \right\Vert_{\mathcal{L}\left(\mathbb{L}^{2}(\partial \Omega); \mathbb{L}^{2}(\partial \Omega)\right)} \, \leq \, \left\vert \partial \Omega \right\vert,
\end{equation*}
and this implies, 
\begin{equation}\label{Eq1237}
   \left\Vert \nu \times \left( \widetilde{\boldsymbol{A}}_{k} \, - \, \widetilde{\boldsymbol{A}}_{0} \right)\left( \nu \times H^{\mathring{\mu}_r} \right) \right\Vert_{\mathbb{L}^{2}(\partial \Omega)} \leq  \frac{k^{2}}{8 \, \pi} \, \left\vert \partial \Omega \right\vert \, \left\Vert  \nu \times H^{\mathring{\mu}_r}  \right\Vert_{\mathbb{L}^{2}(\partial \Omega)} \, + \, k^{3} \, C_{1}(\Omega, k) \, \left\Vert \nu \times H^{\mathring{\mu}_r} \right\Vert_{\mathbb{L}^{2}(\partial \Omega)}.
\end{equation}
Next, by using $(\ref{Eq1237})$, the inequality $(\ref{EqB13})$ becomes 
\begin{eqnarray}\label{bk2}
\nonumber
  \left\Vert \nu \times H^{\mathring{\mu}_r} \right\Vert_{\mathbb{L}^{2}(\partial \Omega)} \, & \leq & \frac{3 \, \left[ 1 \, + \, \beta \, \left(3 \, \lambda_{m_{0}}^{(3)}\left( \Omega \right) \, - \, 1 \right) \right]}{\left\vert \beta \right\vert \, \left( 3 \, \lambda_{m_{0}}^{(3)}\left( \Omega \right) \, - \, 1 \right)^{2}} \; \Bigg[ \frac{k^{2}}{8 \, \pi} \,\left\vert \partial \Omega \right\vert \, \left\Vert  \nu \times H^{\mathring{\mu}_r}  \right\Vert_{\mathbb{L}^{2}(\partial \Omega)} \, + \, k^{3} \, C_{1}(\Omega, k) \, \left\Vert \nu \times H^{\mathring{\mu}_r} \right\Vert_{\mathbb{L}^{2}(\partial \Omega)} \\ &+& \, \frac{k}{\xi  \, {\bf T}^{\mathring\mu_{r}}} \, \left\Vert  \nu \times H^{Inc} \right\Vert_{\mathbb{L}^{2}(\partial \Omega)} + \left\Vert \, \nu \times Curl {\bf N^{k}} \left( Curl \left(  H^{\mathring{\mu}_r} \right) \right) \right\Vert_{\mathbb{L}^{2}(\partial \Omega)} \Bigg].
\end{eqnarray}
Using the fact that 
\begin{equation*}
    k^{2} \, \leq \, \frac{\pi \, \left\vert \beta \right\vert \, \left(3 \, \lambda_{m_{0}}^{(3)}\left( \Omega \right) \, - \, 1 \right)^{2}}{3 \, \left\vert \Omega \right\vert}, 
\end{equation*}
see $(\ref{WBMX})$ and $(\ref{kscalebeta})$, we rewrite the previous inequality $(\ref{bk2})$ as   
\begin{eqnarray*}
  \left\Vert \nu \times H^{\mathring{\mu}_r} \right\Vert_{\mathbb{L}^{2}(\partial \Omega)} \, & \leq &  \, \frac{\left\vert \partial \Omega \right\vert}{8 \, \left\vert \Omega \right\vert}  \, \left\Vert \nu \times H^{\mathring{\mu}_r} \right\Vert_{\mathbb{L}^{2}(\partial \Omega)} +  \, \pi \, k \, C_{1}(\Omega, k)   \, \left\Vert \nu \times H^{\mathring{\mu}_r} \right\Vert_{\mathbb{L}^{2}(\partial \Omega)} \\
  & + & \, \frac{\beta \, \left(3 \, \lambda_{m_{0}}^{(3)}\left( \Omega \right) \, - \, 1 \right)}{\left\vert \Omega \right\vert} \, \Bigg[ \, \frac{1}{8} \,\left\vert \partial \Omega \right\vert  \, + \, \pi \, k \, C_{1}(\Omega, k) \Bigg]  \, \left\Vert \nu \times H^{\mathring{\mu}_r} \right\Vert_{\mathbb{L}^{2}(\partial \Omega)} \\ &+& \, \frac{3 \, \left[ 1 \, + \, \beta \, \left(3 \, \lambda_{m_{0}}^{(3)}\left( \Omega \right) \, - \, 1 \right) \right]}{\left\vert \beta \right\vert \, \left( 3 \, \lambda_{m_{0}}^{(3)}\left( \Omega \right) \, - \, 1 \right)^{2}} \;  \frac{k}{\xi  \, {\bf T}^{\mathring\mu_{r}}} \, \left\Vert  \nu \times H^{Inc} \right\Vert_{\mathbb{L}^{2}(\partial \Omega)}
  \\ &+& \, \frac{3 \, \left[ 1 \, + \, \beta \, \left(3 \, \lambda_{m_{0}}^{(3)}\left( \Omega \right) \, - \, 1 \right) \right]}{\left\vert \beta \right\vert \, \left( 3 \, \lambda_{m_{0}}^{(3)}\left( \Omega \right) \, - \, 1 \right)^{2}} \;  \left\Vert \, \nu \times Curl {\bf N^{k}} \left( Curl \left(  H^{\mathring{\mu}_r} \right) \right) \right\Vert_{\mathbb{L}^{2}(\partial \Omega)},
\end{eqnarray*}
which, under the condition 
\begin{equation}\label{CdtOmega8}
    \frac{\left\vert \partial \Omega \right\vert}{8 \, \left\vert \Omega \right\vert} \, < \, 1,
\end{equation}
the smallness of $k$ and $\beta$ and the fact that $\xi  \, {\bf T}^{\mathring\mu_{r}} \, \sim \, 1$, gives us 
\begin{eqnarray*}
  \left\Vert \nu \times H^{\mathring{\mu}_r} \right\Vert_{\mathbb{L}^{2}(\partial \Omega)} \, & \lesssim &   \, \frac{k}{\left\vert \beta \right\vert} \;   \left\Vert  \nu \times H^{Inc} \right\Vert_{\mathbb{L}^{2}(\partial \Omega)}
  \\ &+& \, \frac{1}{\left\vert \beta \right\vert} \; \left\Vert \, \nu \times Curl {\bf N^{k}}  \right\Vert_{\mathcal{L}\left(\mathbb{L}^{2}( \Omega);\mathbb{L}^{2}(\partial \Omega)\right)} \; \left\Vert Curl \left(  H^{\mathring{\mu}_r} \right)  \right\Vert_{\mathbb{L}^{2}(\Omega)}. 
\end{eqnarray*}
Now, using the fact that 
\begin{equation*}
    \left\Vert \, \nu \times Curl {\bf N^{k}}  \right\Vert_{\mathcal{L}\left(\mathbb{L}^{2}( \Omega);\mathbb{L}^{2}(\partial \Omega)\right)} \; \sim \; 1,
\end{equation*}
we reduce the previous estimation to
\begin{equation}\label{DH}
  \left\Vert \nu \times H^{\mathring{\mu}_r} \right\Vert_{\mathbb{L}^{2}(\partial \Omega)} \,  \lesssim    \, \frac{k}{\left\vert \beta \right\vert} \;   \left\Vert  \nu \times H^{Inc} \right\Vert_{\mathbb{L}^{2}(\partial \Omega)}
+ \, \frac{1}{\left\vert \beta \right\vert}  \; \left\Vert Curl \left(  H^{\mathring{\mu}_r} \right)  \right\Vert_{\mathbb{L}^{2}(\Omega)}. 
\end{equation}
We need to estimate $\left\Vert Curl \left(  H^{\mathring{\mu}_r} \right)  \right\Vert_{\mathbb{L}^{2}(\Omega)}$. To do this, by recalling $(\ref{EqB6})$ and taking the Curl operator in its both sides, we obtain
\begin{eqnarray}\label{Eq0957109}
\nonumber
     Curl\left(H^{\mathring{\mu}_r} \right)   \, &=& \, i\, k \, Curl\left(H^{Inc}\right) \, - \, \xi  \, {\bf T}^{\mathring\mu_{r}}  \, k^{2} \,  Curl\left({\bf N^{k}}\left( H^{\mathring{\mu}_r}\right)\right) \\
    Curl\left(H^{\mathring{\mu}_r} \right) \, & \overset{(\ref{IncWave})}{=} & \, k^{2} \, E^{Inc} \, - \, \xi  \, {\bf T}^{\mathring\mu_{r}}  \, k^{2} \,  Curl\left({\bf N^{k}}\left( H^{\mathring{\mu}_r}\right)\right) \\ \nonumber
    \left\Vert Curl\left(H^{\mathring{\mu}_r} \right) \right\Vert_{\mathbb{L}^{2}(\Omega)} \, & \lesssim &  \, k^{2} \, \left\Vert E^{Inc} \right\Vert_{\mathbb{L}^{2}(\Omega)} \, +  \, k^{2} \, \left\Vert  Curl {\bf N^{k}} \right\Vert_{\mathcal{L}\left( \mathbb{L}^{2}(\Omega); \mathbb{L}^{2}(\Omega) \right)} \, \left\Vert H^{\mathring{\mu}_r} \right\Vert_{\mathbb{L}^{2}(\Omega)} \\ \nonumber
        \left\Vert Curl\left(H^{\mathring{\mu}_r} \right) \right\Vert_{\mathbb{L}^{2}(\Omega)} \, & \lesssim &  \, k^{2} \, \left\Vert E^{Inc} \right\Vert_{\mathbb{L}^{2}(\Omega)} \, +  \, k^{2}  \, \left\Vert H^{\mathring{\mu}_r} \right\Vert_{\mathbb{L}^{2}(\Omega)} \\ \nonumber
             \left\Vert Curl\left(H^{\mathring{\mu}_r} \right) \right\Vert_{\mathbb{L}^{2}(\Omega)} \,   & \overset{(\ref{es-coro1})}{\lesssim} & \, k^{2} \, \left\Vert E^{Inc} \right\Vert_{\mathbb{L}^{2}(\Omega)} \, +  \, \frac{k^{3}}{\left\vert \beta \right\vert}  \, \left\Vert H^{Inc} \right\Vert_{\mathbb{L}^{2}(\Omega)}.
\end{eqnarray}
As $\beta \sim k^{2}$, with $k$ taken to be small, and both $\left\Vert H^{Inc} \right\Vert_{\mathbb{L}^{2}(\Omega)}$ and $\left\Vert E^{Inc} \right\Vert_{\mathbb{L}^{2}(\Omega)}$ are of order one with respect to the parameters $k$ and $\beta$, we reduce the previous estimation into the following one
\begin{equation*}
    \left\Vert Curl\left(H^{\mathring{\mu}_r} \right) \right\Vert_{\mathbb{L}^{2}(\Omega)} \, \lesssim \,  \, \frac{k^{3}}{\left\vert \beta \right\vert}  \, \left\Vert H^{Inc} \right\Vert_{\mathbb{L}^{2}(\Omega)}.
\end{equation*}
Plugging the previous estimation into $(\ref{DH})$, we deduce that  
\begin{equation}\label{EqB22}
  \left\Vert \nu \times H^{\mathring{\mu}_r} \right\Vert_{\mathbb{L}^{2}(\partial \Omega)} \,  \lesssim    \, \frac{k}{\left\vert \beta \right\vert} \;   \left\Vert  \nu \times H^{Inc} \right\Vert_{\mathbb{L}^{2}(\partial \Omega)}
+   \; \frac{k^{3}}{\left\vert \beta \right\vert^{2}}  \, \left\Vert H^{Inc} \right\Vert_{\mathbb{L}^{2}(\Omega)}. 
\end{equation}
From $(\ref{EqAS})$, we recall that  
\begin{eqnarray*}
\nu \times H^{\mathring{\mu}_r} \, &=& \, \frac{2 \, \xi  \, {\bf T}^{\mathring\mu_{r}}}{\left( 2 \, - \, \xi  \, {\bf T}^{\mathring\mu_{r}} \right)} \, \left[ \nu \times  \widetilde{\boldsymbol{A}}_{k}\left( \nu \times H^{\mathring{\mu}_r} \right) \, + \, \frac{i \, k}{\xi  \, {\bf T}^{\mathring\mu_{r}}} \, \nu \times H^{Inc} \, - \, \nu \times Curl {\bf N^{k}} \left( Curl \left(  H^{\mathring{\mu}_r} \right) \right) \right] \\
\left\Vert \nu \times H^{\mathring{\mu}_r} \right\Vert_{\mathbb{H}^{1}\left( \partial \Omega \right)} \, & \lesssim & \, \left\Vert \nu \times  \widetilde{\boldsymbol{A}}_{k}\left( \nu \times H^{\mathring{\mu}_r} \right) \right\Vert_{\mathbb{H}^{1}\left( \partial \Omega \right)} \, + \, k \, \left\Vert \nu \times H^{Inc} \right\Vert_{\mathbb{H}^{1}\left( \partial \Omega \right)} \\ &+& \, \left\Vert \nu \times Curl {\bf N^{k}} \left( Curl \left(  H^{\mathring{\mu}_r} \right) \right) \right\Vert_{\mathbb{H}^{1}\left( \partial \Omega \right)}.
\end{eqnarray*}
In addition, by using the continuity of the operator $\nu \times  \widetilde{\boldsymbol{A}}_{k}$ from
$\mathbb{L}^{2}_{t}(\partial \Omega)$ to $\mathbb{H}^{1}_{t}(\partial \Omega)$, see \cite[Theorem 4.2]{Kirsch1989}, where we recall that 
\begin{equation*}
\mathbb{H}_{t}^{m}(\partial \Omega) \, := \, \left\{ E : \partial \Omega \longrightarrow \mathbb{C}^{3}; \; E \in \mathbb{H}^{m}(\partial \Omega); \, \nu \cdot E \, = \, 0 \; \text{on} \; \partial \Omega \right\}, \quad m \in \mathbb{N},   
\end{equation*}
we obtain  
\begin{equation*}
\left\Vert \nu \times H^{\mathring{\mu}_r} \right\Vert_{\mathbb{H}^{1}\left( \partial \Omega \right)} \,  \lesssim  \, \left\Vert  \nu \times H^{\mathring{\mu}_r} \right\Vert_{\mathbb{L}^{2}\left( \partial \Omega \right)} \, + \, k \, \left\Vert \nu \times H^{Inc} \right\Vert_{\mathbb{H}^{1}\left( \partial \Omega \right)} + \, \left\Vert \nu \times Curl {\bf N^{k}} \left( Curl \left(  H^{\mathring{\mu}_r} \right) \right) \right\Vert_{\mathbb{H}^{1}\left( \partial \Omega \right)}. 
\end{equation*}
Additionally, the operator $\nu \times Curl {\bf N^{k}}$ is continuous from $\mathbb{H}^{\frac{1}{2}}(\Omega)$ to $\mathbb{H}^{1}(\partial \Omega)$. Then,
\begin{equation}\label{EqbEstCurlH}
\left\Vert \nu \times H^{\mathring{\mu}_r} \right\Vert_{\mathbb{H}^{1}\left( \partial \Omega \right)} \,  \lesssim  \, \left\Vert  \nu \times H^{\mathring{\mu}_r} \right\Vert_{\mathbb{L}^{2}\left( \partial \Omega \right)} \, + \, k \, \left\Vert \nu \times H^{Inc} \right\Vert_{\mathbb{H}^{1}\left( \partial \Omega \right)} + \, \left\Vert  Curl \left(  H^{\mathring{\mu}_r} \right) \right\Vert_{\mathbb{H}^{\frac{1}{2}}\left( \Omega \right)}. 
\end{equation}
The estimation of $\left\Vert  Curl \left(  H^{\mathring{\mu}_r} \right) \right\Vert_{\mathbb{H}^{\frac{1}{2}}\left( \Omega \right)}$ is needed. To achieve this, we recall from $(\ref{Eq0957109})$ that 
\begin{eqnarray}\label{EqB2024}
    Curl\left(H^{\mathring{\mu}_r} \right) \, & = & \, k^{2} \, E^{Inc} \, - \, \xi  \, {\bf T}^{\mathring\mu_{r}}  \, k^{2} \,  Curl\left({\bf N^{k}}\left( H^{\mathring{\mu}_r}\right)\right) \\ \nonumber
   \left\Vert Curl\left(H^{\mathring{\mu}_r} \right) \right\Vert_{\mathbb{H}^{\frac{1}{2}}\left( \Omega \right)} \, &  \lesssim & \, k^{2} \, \left\Vert E^{Inc} \right\Vert_{\mathbb{H}^{\frac{1}{2}}\left( \Omega \right)} \, + \, k^{2} \, \left\Vert Curl\left({\bf N^{k}}\left( H^{\mathring{\mu}_r}\right)\right) \right\Vert_{\mathbb{H}^{\frac{1}{2}}\left( \Omega \right)}. 
\end{eqnarray}
Furthermore, as the operator $Curl {\bf N^{k}}$ is continuous from $\mathbb{L}^{2}(\Omega)$ to $\mathbb{H}^{1}\left( \Omega \right) \, \subset \, \mathbb{H}^{\frac{1}{2}}\left( \Omega \right)$, we get
\begin{eqnarray*}
\nonumber
  \left\Vert Curl\left(H^{\mathring{\mu}_r} \right) \right\Vert_{\mathbb{H}^{\frac{1}{2}}\left( \Omega \right)} \, &  \lesssim & \, k^{2} \, \left\Vert E^{Inc} \right\Vert_{\mathbb{H}^{\frac{1}{2}}\left( \Omega \right)} \, + \, k^{2} \, \left\Vert  H^{\mathring{\mu}_r}  \right\Vert_{\mathbb{L}^{2}\left( \Omega \right)} \\
  & \overset{(\ref{es-coro1})}{\lesssim} & \, k^{2} \, \left\Vert E^{Inc} \right\Vert_{\mathbb{H}^{\frac{1}{2}}\left( \Omega \right)} \, +  \, \frac{k^{3}}{\left\vert \beta \right\vert}  \, \left\Vert H^{Inc} \right\Vert_{\mathbb{L}^{2}(\Omega)}.
\end{eqnarray*}
By plugging the previous equation into $(\ref{EqbEstCurlH})$, we obtain
\begin{eqnarray*}
\left\Vert \nu \times H^{\mathring{\mu}_r} \right\Vert_{\mathbb{H}^{1}\left( \partial \Omega \right)} \, & \lesssim & \, \left\Vert  \nu \times H^{\mathring{\mu}_r} \right\Vert_{\mathbb{L}^{2}\left( \partial \Omega \right)} \, + \, k \, \left\Vert \nu \times H^{Inc} \right\Vert_{\mathbb{H}^{1}\left( \partial \Omega \right)} + \, k^{2} \, \left\Vert E^{Inc} \right\Vert_{\mathbb{H}^{\frac{1}{2}}\left( \Omega \right)} \, +  \, \frac{k^{3}}{\left\vert \beta \right\vert}  \, \left\Vert H^{Inc} \right\Vert_{\mathbb{L}^{2}(\Omega)} \\
 & \overset{(\ref{EqB22})}{\lesssim}  &   \, \frac{k}{\left\vert \beta \right\vert} \;   \left\Vert  \nu \times H^{Inc} \right\Vert_{\mathbb{L}^{2}(\partial \Omega)}
+   \; \frac{k^{3}}{\left\vert \beta \right\vert^{2}}  \, \left\Vert H^{Inc} \right\Vert_{\mathbb{L}^{2}(\Omega)} \, + \, k \, \left\Vert \nu \times H^{Inc} \right\Vert_{\mathbb{H}^{1}\left( \partial \Omega \right)} + \, k^{2} \, \left\Vert E^{Inc} \right\Vert_{\mathbb{H}^{\frac{1}{2}}\left( \Omega \right)}. 
\end{eqnarray*}
Knowing that $\left\Vert  \nu \times H^{Inc} \right\Vert_{\mathbb{L}^{2}(\partial \Omega)}, \;
\left\Vert H^{Inc} \right\Vert_{\mathbb{L}^{2}(\Omega)}, \;  \left\Vert \nu \times H^{Inc} \right\Vert_{\mathbb{H}^{1}\left( \partial \Omega \right)}$ and $\left\Vert E^{Inc} \right\Vert_{\mathbb{H}^{\frac{1}{2}}\left( \Omega \right)}$ are of order one with respect to the parameters $k$ and $\beta$, and by taking into account the fact that $\beta \sim k^{2}$, we reduce the previous inequality into 
\begin{equation*}
\left\Vert \nu \times H^{\mathring{\mu}_r} \right\Vert_{\mathbb{H}^{1}\left( \partial \Omega \right)} \,  \lesssim   \, \frac{k}{\left\vert \beta \right\vert} \;   \left\Vert  \nu \times H^{Inc} \right\Vert_{\mathbb{L}^{2}(\partial \Omega)}
+   \; \frac{k^{3}}{\left\vert \beta \right\vert^{2}}  \, \left\Vert H^{Inc} \right\Vert_{\mathbb{L}^{2}(\Omega)}. 
\end{equation*}
Furthermore, by Sobolev embeddings, see \cite[Theorem 9.16]{Brezis}, we deduce that $\nu \times H^{\mathring{\mu}_r} \in \mathbb{L}^{r}(\Omega)$, with $r \geq 2$. Hence, 
\begin{equation}\label{ASEq1012}
\left\Vert \nu \times H^{\mathring{\mu}_r} \right\Vert_{\mathbb{L}^{r}\left( \partial \Omega \right)} \,  \lesssim   \, \frac{k}{\left\vert \beta \right\vert} \;   \left\Vert  \nu \times H^{Inc} \right\Vert_{\mathbb{L}^{2}(\partial \Omega)}
+   \; \frac{k^{3}}{\left\vert \beta \right\vert^{2}}  \, \left\Vert H^{Inc} \right\Vert_{\mathbb{L}^{2}(\Omega)}, \quad r \geq 2. 
\end{equation}
Going back to $(\ref{Eq0259})$ and taking the $\left\Vert \cdot \right\Vert_{\mathbb{L}^{r}(\Omega)}$-norm in its both sides, we get  
\begin{equation*}
   \left\Vert H^{\mathring{\mu}_r} \right\Vert_{\mathbb{L}^{r}(\Omega)} \, \lesssim \,  \left\Vert Curl \, SL^{k}\left( \nu \times H^{\mathring{\mu}_r} \right) \right\Vert_{\mathbb{L}^{r}(\Omega)} \, + \, k \, \left\Vert H^{Inc} \right\Vert_{\mathbb{L}^{r}(\Omega)} \, + \, \left\Vert Curl  {\bf N^{k}}\left( Curl \left( H^{\mathring{\mu}_r} \right) \right) \right\Vert_{\mathbb{L}^{r}(\Omega)}, 
\end{equation*}
which, by using the continuity of the operator $Curl SL^{k}$ from $\mathbb{L}^{r}(\partial \Omega)$ to $\mathbb{L}^{r}(\partial \Omega)$ and the continuity of the operator $Curl {\bf N^{k}}$ from $\mathbb{L}^{r}(\Omega)$ to $\mathbb{L}^{r}(\Omega)$, the last inequality can be reduced to  
\begin{eqnarray*}
   \left\Vert H^{\mathring{\mu}_r} \right\Vert_{\mathbb{L}^{r}(\Omega)} \, & \lesssim & \,  \left\Vert  \nu \times H^{\mathring{\mu}_r}  \right\Vert_{\mathbb{L}^{r}(\partial \Omega)} \, + \, k \, \left\Vert H^{Inc} \right\Vert_{\mathbb{L}^{r}(\Omega)} \, + \, \left\Vert Curl \left( H^{\mathring{\mu}_r} \right)  \right\Vert_{\mathbb{L}^{r}(\Omega)}. 
\end{eqnarray*}
Using $(\ref{EqB2024})$ and the continuity of the operator $Curl {\bf N^{k}}$ from $\mathbb{L}^{r}(\Omega)$ to $\mathbb{L}^{r}(\Omega)$, we obtain
\begin{equation}\label{Eq07591012}
    \left\Vert Curl \left( H^{\mathring{\mu}_r} \right)  \right\Vert_{\mathbb{L}^{r}(\Omega)} \, \lesssim \, k^{2} \, \left\Vert E^{Inc} \right\Vert_{\mathbb{L}^{r}(\Omega)} \, +  \, k^{2} \,  \left\Vert  H^{\mathring{\mu}_r} \right\Vert_{\mathbb{L}^{r}(\Omega)}
\end{equation}
and hence
\begin{eqnarray*}
      \left\Vert H^{\mathring{\mu}_r} \right\Vert_{\mathbb{L}^{r}(\Omega)} \, & \lesssim & \left\Vert  \nu \times H^{\mathring{\mu}_r}  \right\Vert_{\mathbb{L}^{r}(\partial \Omega)} \, + \, k \, \left\Vert H^{Inc} \right\Vert_{\mathbb{L}^{r}(\Omega)} \, + \,  \, k^{2} \, \left\Vert E^{Inc} \right\Vert_{\mathbb{L}^{r}(\Omega)} \, +  \, k^{2} \,  \left\Vert H^{\mathring{\mu}_r} \right\Vert_{\mathbb{L}^{r}(\Omega)}.
\end{eqnarray*}
As the parameter $k$ is taken to be small and both $\left\Vert H^{Inc} \right\Vert_{\mathbb{L}^{r}(\Omega)}$ and $\left\Vert E^{Inc} \right\Vert_{\mathbb{L}^{r}(\Omega)}$ are of order one we obtain
\begin{eqnarray}\label{Eq19141012}
\nonumber
   \left\Vert H^{\mathring{\mu}_r} \right\Vert_{\mathbb{L}^{r}(\Omega)} \, & \lesssim  & \,   \left\Vert  \nu \times H^{\mathring{\mu}_r}  \right\Vert_{\mathbb{L}^{r}(\partial \Omega)} \, + \, k \, \left\Vert H^{Inc} \right\Vert_{\mathbb{L}^{r}(\Omega)} \\
   & \overset{(\ref{ASEq1012})}{\lesssim} & \, \frac{k}{\left\vert \beta \right\vert} \;   \left\Vert  \nu \times H^{Inc} \right\Vert_{\mathbb{L}^{2}(\partial \Omega)}
+   \; \frac{k^{3}}{\left\vert \beta \right\vert^{2}}  \, \left\Vert H^{Inc} \right\Vert_{\mathbb{L}^{2}(\Omega)}, \quad r \geq 2.
\end{eqnarray}
In addition, by plugging the previous estimation into $(\ref{Eq07591012})$, we derive 
\begin{eqnarray*}
\nonumber
    \left\Vert Curl \left( H^{\mathring{\mu}_r} \right)  \right\Vert_{\mathbb{L}^{r}(\Omega)} \, & \lesssim & \, k^{2} \, \left\Vert E^{Inc} \right\Vert_{\mathbb{L}^{r}(\Omega)} \, +  \, k^{2} \, \left[ \, \frac{k}{\left\vert \beta \right\vert} \;   \left\Vert  \nu \times H^{Inc} \right\Vert_{\mathbb{L}^{2}(\partial \Omega)}
+   \; \frac{k^{3}}{\left\vert \beta \right\vert^{2}}  \, \left\Vert H^{Inc} \right\Vert_{\mathbb{L}^{2}(\Omega)} \right] \\
 & \lesssim & \frac{k^{3}}{\left\vert \beta \right\vert} \;   \left\Vert  \nu \times H^{Inc} \right\Vert_{\mathbb{L}^{2}(\partial \Omega)}
+   \; \frac{k^{5}}{\left\vert \beta \right\vert^{2}}  \, \left\Vert H^{Inc} \right\Vert_{\mathbb{L}^{2}(\Omega)}, \quad r \geq 2,
\end{eqnarray*}
which gives us an estimation of $Curl \left( H^{\mathring{\mu}_r} \right)$ in $\mathbb{L}^{r}(\Omega)$, for $r \geq 2$. At this stage we use  \cite[Corollary 5.3]{AS} to derive the following inequality 
\begin{equation*}
    \left\Vert  H^{\mathring{\mu}_r}  \right\Vert_{\mathbb{W}^{m,p}(\Omega)} \,  \lesssim  \, \left\Vert  H^{\mathring{\mu}_r}   \right\Vert_{\mathbb{L}^{p}(\Omega)} \, + \, \left\Vert Curl \left( H^{\mathring{\mu}_r} \right)  \right\Vert_{\mathbb{W}^{m-1,p}(\Omega)} \, + \, \left\Vert  \nu \times H^{\mathring{\mu}_r}  \right\Vert_{\mathbb{W}^{m-\frac{1}{p},p}(\partial \Omega)}.
\end{equation*}
And, in particular, for $m \,= \, \dfrac{5}{2}$ and $p=2$, we get
\begin{equation}\label{HZ18461012}
    \left\Vert  H^{\mathring{\mu}_r}  \right\Vert_{\mathbb{H}^{\frac{5}{2}}(\Omega)} \,  \lesssim  \, \left\Vert  H^{\mathring{\mu}_r}   \right\Vert_{\mathbb{L}^{2}(\Omega)} \, + \, \left\Vert Curl \left( H^{\mathring{\mu}_r} \right)  \right\Vert_{\mathbb{H}^{\frac{3}{2}}(\Omega)} \, + \, \left\Vert  \nu \times H^{\mathring{\mu}_r}  \right\Vert_{\mathbb{H}^{2}(\partial \Omega)}.
\end{equation}
Then, in order to get an estimation for $\left\Vert  H^{\mathring{\mu}_r}  \right\Vert_{\mathbb{H}^{\frac{5}{2}}(\Omega)}$, we need to estimate $\left\Vert Curl \left( H^{\mathring{\mu}_r} \right)  \right\Vert_{\mathbb{H}^{\frac{3}{2}}(\Omega)}$ and $\left\Vert  \nu \times H^{\mathring{\mu}_r}  \right\Vert_{\mathbb{H}^{2}(\partial \Omega)}$.
\begin{enumerate}
    \item[] 
    \item Estimation of $\left\Vert Curl \left( H^{\mathring{\mu}_r} \right)  \right\Vert_{\mathbb{H}^{\frac{3}{2}}(\Omega)}$. \\
    To do this, by taking the $\left\Vert \cdot \right\Vert_{\mathbb{H}^{\frac{3}{2}}(\Omega)}$-norm in both sides of $(\ref{EqB2024})$, we derive that 
\begin{equation*}
   \left\Vert Curl\left(H^{\mathring{\mu}_r} \right) \right\Vert_{\mathbb{H}^{\frac{3}{2}}(\Omega)} \,  \lesssim  \, k^{2} \, \left\Vert E^{Inc} \right\Vert_{\mathbb{H}^{\frac{3}{2}}(\Omega)} \, + \,  k^{2} \, \left\Vert Curl\left({\bf N^{k}} \left( H^{\mathring{\mu}_r}\right)\right) \right\Vert_{\mathbb{H}^{\frac{3}{2}}(\Omega)},
\end{equation*}
which, by using the continuity of the operator $Curl \, {\bf N^{k}}$ from $\mathbb{H}^{\frac{1}{2}}(\Omega)$ to $\mathbb{H}^{\frac{3}{2}}(\Omega)$, can be reduced to  
\begin{equation}\label{EBEq}
   \left\Vert Curl\left(H^{\mathring{\mu}_r} \right) \right\Vert_{\mathbb{H}^{\frac{3}{2}}(\Omega)} \,  \lesssim  \, k^{2} \, \left\Vert E^{Inc} \right\Vert_{\mathbb{H}^{\frac{3}{2}}(\Omega)} \, + \,  k^{2} \, \left\Vert  H^{\mathring{\mu}_r} \right\Vert_{\mathbb{H}^{\frac{1}{2}}(\Omega)}.
\end{equation}
    \item[]
    \item Estimation of $\left\Vert  \nu \times H^{\mathring{\mu}_r}  \right\Vert_{\mathbb{H}^{2}(\partial \Omega)}$.\\
To achieve this, we go back to $(\ref{EqAS})$ to obtain
\begin{eqnarray*}
 \left\Vert \nu \times H^{\mathring{\mu}_r} \right\Vert_{\mathbb{H}^{2}(\partial \Omega)}  \, & \lesssim & \,  \left\Vert \nu \times  \widetilde{\boldsymbol{A}}_{k}\left( \nu \times H^{\mathring{\mu}_r} \right) \right\Vert_{\mathbb{H}^{2}(\partial \Omega)} \, + \,  k \, \left\Vert \nu \times H^{Inc} \right\Vert_{\mathbb{H}^{2}(\partial \Omega)}  \\ &+& \left\Vert \nu \times Curl {\bf N^{k}} \left( Curl \left(  H^{\mathring{\mu}_r} \right) \right) \right\Vert_{\mathbb{H}^{2}(\partial \Omega)}.
\end{eqnarray*}
And, using the continuity of the operator $\nu \times Curl {\bf N^{k}}$ from $\mathbb{H}^{\frac{3}{2}}(\Omega)$ to $\mathbb{H}^{2}(\partial \Omega)$, we obtain
\begin{equation*}
 \left\Vert \nu \times H^{\mathring{\mu}_r} \right\Vert_{\mathbb{H}^{2}(\partial \Omega)}  \,  \lesssim  \,  \left\Vert \nu \times  \widetilde{\boldsymbol{A}}_{k}\left( \nu \times H^{\mathring{\mu}_r} \right) \right\Vert_{\mathbb{H}^{2}(\partial \Omega)} \, + \,  k \, \left\Vert \nu \times H^{Inc} \right\Vert_{\mathbb{H}^{2}(\partial \Omega)} + \left\Vert  Curl \left(  H^{\mathring{\mu}_r} \right)  \right\Vert_{\mathbb{H}^{\frac{3}{2}}(\Omega)}.
\end{equation*}
Thanks to \cite[Theorem 4.3]{Kirsch1989} we know that under the $C^{4,\alpha}$-regularity\footnote{This regularity condition on $\Omega$ might be reduced.} of $\partial \Omega$, with $0 < \alpha < 1$, the operator $\nu \times  \widetilde{\boldsymbol{A}}_{k}$ is continuous from $\mathbb{H}^{1}(\partial \Omega)$ to $\mathbb{H}^{2}(\partial \Omega)$. Hence, 
\begin{eqnarray*}
 \left\Vert \nu \times H^{\mathring{\mu}_r} \right\Vert_{\mathbb{H}^{2}(\partial \Omega)}  \,  & \lesssim &  \,  \left\Vert  \nu \times H^{\mathring{\mu}_r}  \right\Vert_{\mathbb{H}^{1}(\partial \Omega)} \, + \,  k \, \left\Vert \nu \times H^{Inc} \right\Vert_{\mathbb{H}^{2}(\partial \Omega)} + \left\Vert  Curl \left(  H^{\mathring{\mu}_r} \right)  \right\Vert_{\mathbb{H}^{\frac{3}{2}}(\Omega)} \\
 & \overset{(\ref{EBEq})}{\lesssim} &  \,  \left\Vert  \nu \times H^{\mathring{\mu}_r}  \right\Vert_{\mathbb{H}^{1}(\partial \Omega)} \, + \,  k \, \left\Vert \nu \times H^{Inc} \right\Vert_{\mathbb{H}^{2}(\partial \Omega)} \, + \, k^{2} \, \left\Vert E^{Inc} \right\Vert_{\mathbb{H}^{\frac{3}{2}}(\Omega)} \, + \,  k^{2} \, \left\Vert  H^{\mathring{\mu}_r} \right\Vert_{\mathbb{H}^{\frac{1}{2}}(\Omega)}.
\end{eqnarray*}
As $\left\Vert \nu \times H^{Inc} \right\Vert_{\mathbb{H}^{2}(\partial \Omega)}$ and $\left\Vert E^{Inc} \right\Vert_{\mathbb{H}^{\frac{3}{2}}(\Omega)}$ are both of order one, and the parameter $k$ is taken to be small, we reduce the previous estimation into
\begin{equation}\label{EqFB}
 \left\Vert \nu \times H^{\mathring{\mu}_r} \right\Vert_{\mathbb{H}^{2}(\partial \Omega)}  \,   \lesssim \,  \left\Vert  \nu \times H^{\mathring{\mu}_r}  \right\Vert_{\mathbb{H}^{1}(\partial \Omega)} \, + \,  k \, \left\Vert \nu \times H^{Inc} \right\Vert_{\mathbb{H}^{2}(\partial \Omega)} \, + \,  k^{2} \, \left\Vert  H^{\mathring{\mu}_r} \right\Vert_{\mathbb{H}^{\frac{1}{2}}(\Omega)}.
\end{equation}
  \item[]
\end{enumerate}
Consequently, based on $(\ref{EBEq})$ and $(\ref{EqFB})$, the estimation $(\ref{HZ18461012})$ becomes, 
\begin{equation*}
    \left\Vert  H^{\mathring{\mu}_r}  \right\Vert_{\mathbb{H}^{\frac{5}{2}}(\Omega)} \,  \lesssim  \,  \left\Vert  H^{\mathring{\mu}_r}  \right\Vert_{\mathbb{L}^{2}(\Omega)} \, + \, k^{2} \,     \left\Vert  E^{Inc}  \right\Vert_{\mathbb{H}^{\frac{3}{2}}(\Omega)} \, + \,   k^{2} \,  \left\Vert  H^{\mathring{\mu}_r}  \right\Vert_{\mathbb{H}^{\frac{1}{2}}(\Omega)} \, + \, \left\Vert \nu \times  H^{\mathring{\mu}_r}  \right\Vert_{\mathbb{H}^{1}(\partial \Omega)} \, + \, k \, \left\Vert \nu \times  H^{Inc}  \right\Vert_{\mathbb{H}^{2}(\partial \Omega)}.
\end{equation*}
Knowing that $\mathbb{H}^{\frac{5}{2}}(\Omega) \, \subset \, \mathbb{H}^{\frac{1}{2}}(\Omega)$, we deduce that $\left\Vert  H^{\mathring{\mu}_r}  \right\Vert_{\mathbb{H}^{\frac{1}{2}}(\Omega)} \, \lesssim \, \left\Vert  H^{\mathring{\mu}_r}  \right\Vert_{\mathbb{H}^{\frac{5}{2}}(\Omega)}$, and by the fact that $k$ is a small parameter, we obtain
\begin{equation*}
    \left\Vert  H^{\mathring{\mu}_r}  \right\Vert_{\mathbb{H}^{\frac{5}{2}}(\Omega)} \,  \lesssim  \,  \left\Vert  H^{\mathring{\mu}_r}  \right\Vert_{\mathbb{L}^{2}(\Omega)} \, + \, k^{2} \,     \left\Vert  E^{Inc}  \right\Vert_{\mathbb{H}^{\frac{3}{2}}(\Omega)} \, + \,  \left\Vert \nu \times  H^{\mathring{\mu}_r}  \right\Vert_{\mathbb{H}^{1}(\partial \Omega)} \, + \, k \, \left\Vert \nu \times  H^{Inc}  \right\Vert_{\mathbb{H}^{2}(\partial \Omega)}.
\end{equation*}
In addition, as $\left\Vert  E^{Inc}  \right\Vert_{\mathbb{H}^{\frac{3}{2}}(\Omega)}$ and $ \left\Vert \nu \times  H^{Inc}  \right\Vert_{\mathbb{H}^{2}(\partial \Omega)}$ are both of order one, we reduce the previous estimation into the following one
\begin{equation}\label{SSI}
    \left\Vert  H^{\mathring{\mu}_r}  \right\Vert_{\mathbb{H}^{\frac{5}{2}}(\Omega)} \,  \lesssim  \,  \left\Vert  H^{\mathring{\mu}_r}  \right\Vert_{\mathbb{L}^{2}(\Omega)} \,  + \,  \left\Vert \nu \times  H^{\mathring{\mu}_r}  \right\Vert_{\mathbb{H}^{1}(\partial \Omega)} \, + \, k \, \left\Vert \nu \times  H^{Inc}  \right\Vert_{\mathbb{H}^{2}(\partial \Omega)}.
\end{equation}
Regarding the previous estimation, as $(\ref{HZ18461012})$, to finish with the estimation of $\left\Vert  H^{\mathring{\mu}_r}  \right\Vert_{\mathbb{H}^{\frac{5}{2}}(\Omega)}$, we need to estimate $\left\Vert \nu \times  H^{\mathring{\mu}_r}  \right\Vert_{\mathbb{H}^{1}(\partial \Omega)}$. To achieve this, we go back to $(\ref{EqAS})$ to obtain
\begin{eqnarray*}
 \left\Vert \nu \times H^{\mathring{\mu}_r} \right\Vert_{\mathbb{H}^{1}(\partial \Omega)}  \, & \lesssim & \,  \left\Vert \nu \times  \widetilde{\boldsymbol{A}}_{k}\left( \nu \times H^{\mathring{\mu}_r} \right) \right\Vert_{\mathbb{H}^{1}(\partial \Omega)} \, + \,  k \, \left\Vert \nu \times H^{Inc} \right\Vert_{\mathbb{H}^{1}(\partial \Omega)}  \\ &+& \left\Vert \nu \times Curl {\bf N^{k}} \left( Curl \left(  H^{\mathring{\mu}_r} \right) \right) \right\Vert_{\mathbb{H}^{1}(\partial \Omega)}.
\end{eqnarray*}
And, using the continuity of the operator $\nu \times Curl {\bf N^{k}}$ from $\mathbb{H}^{\frac{1}{2}}(\Omega)$ to $\mathbb{H}^{1}(\partial \Omega)$, we obtain
\begin{eqnarray*}
 \left\Vert \nu \times H^{\mathring{\mu}_r} \right\Vert_{\mathbb{H}^{1}(\partial \Omega)}  \,  & \lesssim & \,  \left\Vert \nu \times  \widetilde{\boldsymbol{A}}_{k}\left( \nu \times H^{\mathring{\mu}_r} \right) \right\Vert_{\mathbb{H}^{1}(\partial \Omega)} \, + \,  k \, \left\Vert \nu \times H^{Inc} \right\Vert_{\mathbb{H}^{1}(\partial \Omega)} + \left\Vert  Curl \left(  H^{\mathring{\mu}_r} \right)  \right\Vert_{\mathbb{H}^{\frac{1}{2}}(\Omega)} \\
 & \lesssim & \,  \left\Vert \nu \times  \widetilde{\boldsymbol{A}}_{k}\left( \nu \times H^{\mathring{\mu}_r} \right) \right\Vert_{\mathbb{H}^{1}(\partial \Omega)} \, + \,  k \, \left\Vert \nu \times H^{Inc} \right\Vert_{\mathbb{H}^{1}(\partial \Omega)} + \left\Vert  Curl \left(  H^{\mathring{\mu}_r} \right)  \right\Vert_{\mathbb{H}^{\frac{3}{2}}(\Omega)} \\
  & \overset{(\ref{EBEq})}{\lesssim}& \left\Vert \nu \times  \widetilde{\boldsymbol{A}}_{k}\left( \nu \times H^{\mathring{\mu}_r} \right) \right\Vert_{\mathbb{H}^{1}(\partial \Omega)} \, + \,  k \, \left\Vert \nu \times H^{Inc} \right\Vert_{\mathbb{H}^{1}(\partial \Omega)} \, + \,  k^{2} \, \left\Vert E^{Inc} \right\Vert_{\mathbb{H}^{\frac{3}{2}}(\Omega)} \, + \,  k^{2} \, \left\Vert  H^{\mathring{\mu}_r} \right\Vert_{\mathbb{H}^{\frac{1}{2}}(\Omega)}\\
  & \lesssim & \left\Vert \nu \times  \widetilde{\boldsymbol{A}}_{k}\left( \nu \times H^{\mathring{\mu}_r} \right) \right\Vert_{\mathbb{H}^{1}(\partial \Omega)} \, + \,  k \, \left\Vert \nu \times H^{Inc} \right\Vert_{\mathbb{H}^{1}(\partial \Omega)} \, + \ \,  k^{2} \, \left\Vert  H^{\mathring{\mu}_r} \right\Vert_{\mathbb{H}^{\frac{1}{2}}(\Omega)}\\
   & \lesssim & \left\Vert \nu \times  \widetilde{\boldsymbol{A}}_{k}\left( \nu \times H^{\mathring{\mu}_r} \right) \right\Vert_{\mathbb{H}^{1}(\partial \Omega)} \, + \,  k \, \left\Vert \nu \times H^{Inc} \right\Vert_{\mathbb{H}^{1}(\partial \Omega)} \, + \ \,  k^{2} \, \left\Vert  H^{\mathring{\mu}_r} \right\Vert_{\mathbb{H}^{\frac{5}{2}}(\Omega)}. 
\end{eqnarray*}
Thanks to \cite[Theorem 4.3]{Kirsch1989}, we know that the operator $\nu \times  \widetilde{\boldsymbol{A}}_{k}$ is continuous from $\mathbb{L}^{2}(\partial \Omega)$ to $\mathbb{H}^{1}(\partial \Omega)$. Hence, 
\begin{eqnarray*}
     \left\Vert \nu \times H^{\mathring{\mu}_r} \right\Vert_{\mathbb{H}^{1}(\partial \Omega)}  \,   & \lesssim &  \,  \left\Vert  \nu \times H^{\mathring{\mu}_r}  \right\Vert_{\mathbb{L}^{2}(\partial \Omega)} \, + \,  k \, \left\Vert \nu \times H^{Inc} \right\Vert_{\mathbb{H}^{1}(\partial \Omega)} + k^{2} \, \left\Vert H^{\mathring{\mu}_r}   \right\Vert_{\mathbb{H}^{\frac{5}{2}}(\Omega)} \\
     & \overset{(\ref{EqB22})}{\lesssim} & \, \frac{k}{\left\vert \beta \right\vert} \;   \left\Vert  \nu \times H^{Inc} \right\Vert_{\mathbb{L}^{2}(\partial \Omega)} +   \; \frac{k^{3}}{\left\vert \beta \right\vert^{2}}  \, \left\Vert H^{Inc} \right\Vert_{\mathbb{L}^{2}(\Omega)} \, + \, k^{2} \, \left\Vert H^{\mathring{\mu}_r}   \right\Vert_{\mathbb{H}^{\frac{5}{2}}(\Omega)}.
\end{eqnarray*}
Going back to $(\ref{SSI})$ and using the previous estimation we end up with the following estimation
\begin{equation*}
    \left\Vert  H^{\mathring{\mu}_r}  \right\Vert_{\mathbb{H}^{\frac{5}{2}}(\Omega)} \,  \lesssim  \,  \left\Vert  H^{\mathring{\mu}_r}  \right\Vert_{\mathbb{L}^{2}(\Omega)} \,  + \,  \, \frac{k}{\left\vert \beta \right\vert} \;   \left\Vert  \nu \times H^{Inc} \right\Vert_{\mathbb{L}^{2}(\partial \Omega)} +   \; \frac{k^{3}}{\left\vert \beta \right\vert^{2}}  \, \left\Vert H^{Inc} \right\Vert_{\mathbb{L}^{2}(\Omega)}.
\end{equation*}
Moreover, thanks to $(\ref{Eq19141012})$,  we obtain   
\begin{equation*}
    \left\Vert  H^{\mathring{\mu}_r}  \right\Vert_{\mathbb{H}^{2}(\Omega)} \,  \lesssim  \,   \frac{k^{3}}{\left\vert \beta \right\vert^{2}}  \, \left\Vert H^{Inc} \right\Vert_{\mathbb{L}^{2}(\Omega)} \, + \,   \frac{k}{\left\vert \beta \right\vert} \, \left\Vert \nu \times H^{Inc} \right\Vert_{\mathbb{H}^{\frac{3}{2}}(\partial \Omega)}.
\end{equation*}
In addition, the following embedding result holds
\begin{equation*}
    \mathbb{H}^{\frac{5}{2}}(\Omega) \; \hookrightarrow \; C^{0, \alpha}(\overline{\Omega}), \quad \text{with} \quad 0 \, < \, \alpha \, < \, 1, 
\end{equation*}
see \cite[Section 4.12]{Adams}, and this implies 
\begin{equation*}
    \left[ H^{\mathring{\mu}_r} \right]_{\mathcal{C}^{0, \alpha}(\overline{\Omega})} \; \lesssim \;  \frac{k^{3}}{\left\vert \beta \right\vert^{2}}  \, \left\Vert H^{Inc} \right\Vert_{\mathbb{L}^{2}(\Omega)} \, + \,   \frac{k}{\left\vert \beta \right\vert} \, \left\Vert \nu \times H^{Inc} \right\Vert_{\mathbb{H}^{\frac{3}{2}}(\partial \Omega)}, \quad \text{with} \quad 0 \, < \, \alpha \, < \, 1.
\end{equation*}
This ends the proof of \textbf{Proposition \ref{AddedLemma}}.

\section{Appendix: Counting Lemmas}

Two lemmas are used in the proof of our main theorems to calculate the sums of inverse distances between nano-particles inside $\Omega$ and near $\partial \Omega$.
\smallskip
\begin{lemma}\label{conting1}
	Suppose ${z}_m$, $m=1, \cdots, \aleph$, are defined in \eqref{scaling motion}. For $k>0$, we have the following formula:
 \begin{equation*}
     \sum_{j=1\atop j\neq m}^\aleph| z_j- z_m|^{-k} = \begin{cases}
		  \Oh\left(d^{-3}\right), & \text{\quad if $k<3$,}\\
                          & \\
            \Oh\left(d^{-3}\left(1 \, + \, \left\vert \log(d) \right\vert \right)\right), & \text{\quad if $k=3$,} \\
            & \\
            \Oh\left(d^{-k}\right), & \text{\quad if $k>3$.}
    \end{cases}
 \end{equation*}	
 \medskip
\end{lemma}

\textbf{Lemma \ref{conting1}} provides us with a straightforward way to count the number of the particles and evaluate the distance between any two points $z_j\in D_j$ and $z_m\in D_m$,
 we refer to \cite[Sectoin 3.3]{ACP} for more details.
\medskip
\newline 
The following lemma is concerned with the counting of the particles which are very close to the boundary $\partial \Omega$. Let $z_m$ be an arbitrary fixed position (of an arbitrary particle). From $z_m$, we split the space into $d$-equidistant cubes $\Sigma_\ell$, with $F_\ell$ being their corresponding surfaces in a fixed direction such that $F_\ell$ support some of the points $(z_\ell)_{\ell=1, \ell\neq m}^\aleph$. We know, from \cite[Lemma A.1]{AM}, that the number of such cubes $\Sigma_\ell$ is at most of order $\Oh(\aleph^{\frac{1}{3}})$. Then, for the particles counting in the area $\Omega \, \setminus \, \underset{m=1}{\overset{\aleph}{\cup}} \Omega_{m}$, we can first assume that any small part of $\partial\Omega$ is flat under the condition that $\partial\Omega$ is $C^{1,\alpha}$-regular and $a$ is sufficiently small. Next, we divide the region in $\Omega \, \setminus \, \underset{m=1}{\overset{\aleph}{\cup}} \Omega_{m}$ near the considered flat part into concentric layers such that there are at most $(2n+1)^2$ small cubes $\Omega_n$, with $n=1, \cdots, \left[ \aleph^{\frac{1}{3}} \right]$, up to the $n$-th layer intersecting with the surface. It is obvious that  the number of the particles located in the $n$-th layer is at most of order $\Oh((2n+1)^2-(2n-1)^2)$ and the distance from any small cube $\Omega_n$ containing the particles to $z_m$ is at least $(n+\ell)d$. Thus we have the following counting lemma presenting an upper bound with respect to the particles distances. 

\begin{lemma}\label{lem-count-boundary}
	For any fixed $z_m$, $m=1, \cdots, \aleph$,  the following formula holds,
	\begin{equation}\label{count-bound}
		\sum_{m=1}^\aleph \left( \int_{\Omega \setminus  \underset{m=1}{\overset{\aleph}{\cup}} \Omega_{m}} \; \frac{1}{|z_m-z|^3}\,dz\right)^2=\Oh\left( d^{-1}\right).
	\end{equation}
\end{lemma}

\begin{proof}
	It is clear from the distribution described above \textbf{Lemma \ref{lem-count-boundary}} that
	\begin{eqnarray*}
		\boldsymbol{\varkappa} \, &:=& \,  \sum_{m=1}^\aleph \left( \int_{\Omega\backslash\cup_{m=1}^\aleph \Omega_m}\frac{1}{|z_m-z|^3}\,dz\right)^{2} \\
		& \leq & \sum_{\ell=1}^{\aleph^{\frac{1}{3}}}\sum_{z_\ell\in F_\ell}\left( \sum_{n=1}^{\aleph^{\frac{1}{3}}} \frac{ \left\vert \Omega_n \right\vert \, \left[ (2n+1)^{2} \, - \, (2n-1)^{2} \right]}{(n+\ell)^3d^3} \right)^2 
		 \lesssim  \sum_{\ell=1}^{\aleph^{\frac{1}{3}}}\left( \sum_{n=1}^{\aleph^{\frac{1}{3}}} \frac{n\ell}{(n+\ell)^3}\right)^{2}.
   \end{eqnarray*}
   By referring to the precise computations derived in \cite[Section 4, formula (4.35)]{Cao-Sini}, we get 
   \begin{eqnarray*}
 \boldsymbol{\varkappa} \,  & \lesssim &
 \left(3 \, \bar{c} \, + \, \frac{9}{2}\right) \, d^{-1} \, + \, \frac{2 \, (3 \, d^{-2} \, - \, d^{-1} \, - \, 2)}{1 \, + \, d^{-1}} \, - \, 6 \, d^{-1} \, \log\left( \frac{2 \, d^{-1}}{1 \, + \, d^{-1}} \right) + \frac{3\, d \, (1-d)}{(1 \, + \, d^3)(1 \, + \,d^2)},
	\end{eqnarray*}
 where $\bar{c}$ is a positive constant and $d$ is given by $(\ref{basic-ad})$. Then $(\ref{count-bound})$ can be obtained as a direct consequence of the fact that $d$ is a small parameter.
\end{proof}

\end{document}